\documentclass[12pt,a4paper]{article}

\usepackage[english,french]{babel}
\usepackage[T1]{fontenc}
\usepackage{latexsym,amsfonts,amssymb,amsmath,amsthm,mathrsfs, mathtools}
\usepackage{tikz}
\usepackage[all]{xy}
\usepackage{ulem}
\usepackage{tabularx}
\usepackage{times}
\usepackage{pdflscape}
\usepackage{chngcntr}
\usepackage{graphicx}
\graphicspath{ {figures/} }
\usepackage{array}
\usepackage{makeidx}
\usepackage[hidelinks, backref=page]{hyperref}
\hypersetup{
    colorlinks = true,
    citecolor={blue},
}

\DeclareFontShape{OT1}{cmtt}{bx}{n}{<5><6><7><8><9><10><10.95><12><14.4><17.28><20.74><24.88>cmttb10}{}

\newtheorem{thm}{Theorem} %%%%Theorem in Intro
\newtheorem{Theorem}[subsubsection]{Theorem}
\newtheorem{Lemma}[subsubsection]{Lemma}
\newtheorem{Proposition}[subsubsection]{Proposition}
\theoremstyle{definition}
\newtheorem{Definition}[subsubsection]{Definition}
\newtheorem{Corollary}[subsubsection]{Corollary}
\theoremstyle{remark}
\newtheorem{Remark}[subsubsection]{Remark}
\theoremstyle{remark}
\newtheorem{Example}{Example}

\def\N{\mathbb{N}}
\def\Q{\mathbb{Q}}
\def\Z{\mathbb{Z}}
\def\S{\mathbb{S}}
\def\R{\mathbb{R}}

\def\A{\mathcal{A}}
\def\Hom{\mathrm{Hom}}

\def\H{\mathrm{H}}
\def\G{\mathbb{G}}

\def\M{\mathcal{M}}
\def\R{\mathrm{R}}
\def\W{\mathbb{W}}

\def\FGL{\mathcal{FGL}}
\def\F{\mathbb{F}_{2}}
\def\FF{\mathbb{F}}
\def\Ext{\mathrm{Ext}}

\def\EG{\mathcal{EG}}
\def\Ho{\mathrm{Ho}}

\def\E{\mathrm{E}}
\def\ovk{\overline{\kappa}}
\def\lim{\mathrm{lim}}

\def\hocolim{\mathrm{hocolim}}
\def\Gal{\mathrm{Gal}}
\def\Def{\mathrm{Def}}
\def\Aut{\mathrm{Aut}}
\def\Ring{\mathfrak{Ring}}
\def\spec{\mathrm{spec}}
\def\Groupoid{\mathfrak{Groupoid}}
\def\aff{Aff^{\acute{e}t}_{\M}}

\renewcommand*{\backref}[1]{}
\renewcommand*{\backrefalt}[4]{%
        \ifcase #1 (Not cited.)%
        \or        (Cited on page~#2.)%
        \else      (Cited on pages~#2.)%
        \fi}
\def\EC2{E_C^{hC_{2}}}
\def\EG24{E_C^{hG_{24}}}

\def\lim{\mathrm{lim}}

\makeatletter
\def\@cite#1#2{{\normalfont[{\bfseries#1\if@tempswa , #2\fi}]}}
\makeatother

\begin{document}
\selectlanguage{english}
\title{ Homotopy groups of $E_{C}^{hG_{24}}\wedge A_1$ }
\author{Viet-Cuong Pham
}
\date{}
\maketitle
\begin{abstract}
\noindent
  Let $A_1$ be any spectrum in a class of finite spectra whose mod $2$ cohomology is isomorphic to a free module of rank one over the subalgebra $\A(1)$ of the Steenrod algebra. Let $E_{C}$ be the second Morava-$E$ theory associated to a universal deformation of the formal completion of the supersingular elliptic curve $(C) : y^{2}+y = x^{3}$ defined over $\FF_{4}$ and $G_{24}$ a maximal finite subgroup of automorphism group $\mathbb{S}_{C}$ of the formal completion of $C$. In this paper, we compute the homotopy groups of $E_{C}^{hG_{24}}\wedge A_1$ by means of the homotopy fixed point spectral sequence. 
\end{abstract}
\let\thefootnote\relax\footnote{This work was partially supported by the ANR Project ChroK, ANR-16-CE40-0003 }
\tableofcontents
%%%%%%%%%%%%%%%%%%%%%%%%%%%%%%%%

\section*{Introduction}\label{Chapzero}
\noindent
A central problem in stable homotopy theory is to understand the homotopy groups of the sphere spectrum localised at each prime $p$, $\pi_{*}(S_{(p)}^{0})$. A powerful tool for computing the latter is the Adams spectral sequence, whose $\E_2$-term is given by $\Ext_{\A_{p}}^{*,*}(\FF_p,\FF_p)$, the extension groups over the Steenrod algebra $\A_{p}$. However, this method only allows one to compute $\pi_*(S_{(p)}^{0})$ stem by stem. In the late 1950's, Toda in \cite{Tod59} and in the 1960's, Adams in \cite{Ada66}, in the study of the image of J,  showed the existence of infinite families of elements of $\pi_*(S^0)$ living in arbitrarily large stems. These were the first periodic families discovered, known as the $\alpha$-family, of the stable homotopy groups of the sphere. Adam's work and subsequent work by L. Smith, Toda and Miller-Mahowald-Wilson and others motivated and marked the beginning of chromatic homotopy theory. \\\\
\noindent
In the early 1980's, Ravenel published a series of conjectures which described the global structure of the stable homotopy category. Most of the conjectures were then resolved by Hopkins and his collaborators. In fact, the chromatic point of view offers a promising tool to analyse $\pi_{*}(S^{0}_{(p)})$ in a systematic way by decomposing it into smaller pieces. More precisely, let $L_{n}$ and $L_{K(n)}$ denote the Bousfield localisations with respect to the $n^{th}$ Johnson-Wilson theory $E(n)$ and $n^{th}$-Morava $K$-theory, respectively (here the prime $p$ is implicit in the notation). We have the chromatic convergence theorem.
 \begin{thm}[Hopkins-Ravenel, \cite{Rav92}] Let $X$ be a $p$-local finite spectrum. There is a tower 
 $$ ...\rightarrow L_{n}X\rightarrow L_{n-1}X\rightarrow ...\rightarrow L_0X \cong L_{\H\Q}X,$$
  such that $X$ is homotopy equivalent to its homotopy limit.
 \end{thm}
 \noindent
 Furthermore, the chromatic fracture square asserts that $L_{n}$ can be inductively determined from the Bousfield localisation $L_{K(m)}$ with respect to the $m^{th}$ Morava $K$-theory for $0\leq m\leq n$, via homotopy pull-back squares
\begin{thm}\cite{Rav92}. For any spectrum $X$ and all positive integers $n$, the following diagram is a homotopy pullback square
$$\xymatrix{ L_{n}X\ar[r]\ar[d] & L_{K(n)}X\ar[d] \\
		    L_{n-1}X\ar[r]	& L_{n-1}L_{K(n)}X.
}$$  
\end{thm}
\noindent
Therefore, in the chromatic approach to stable homotopy theory, it is crucial to understand the $K(n)$-local homotopy category at all primes and all natural numbers $n$, referred to as the chromatic level. For this purpose, a general strategy is to study the homotopy type of the $K(n)$-localisation of various finite spectra. A central result of the theory is the work of Devinatz and Hopkins \cite{DH04} which expresses the $K(n)$-localisation of a finite spectrum $X$ as the continuous homotopy fixed point spectrum $$L_{K(n)}X \simeq E_{n}^{h\mathbb{G}_{n}}\wedge X$$ where $\mathbb{G}_{n}$ is the extended Morava stabiliser group, which is profinite, and $E_n$ is the $n^{th}$ Morava $E$-theory. More generally, for any closed subgroup $F$ of $\mathbb{G}_{n}$, the continuous homotopy fixed point spectrum $E_n^{hF}$ can be formed. \\\\
\noindent
The study of chromatic level one was a great success: the homotopy groups of $L_{K(1)}S^{0}$ have been completely computed at all primes and, at the prime $2$, $L_{K(1)}S^0$ detects essentially the image of J.
Chromatic level two has also been thoroughly investigated at odd primes. It started with the computation by Shimomura and his collaborators of the $L_{2}$ localisation of various finite spectra  (see \cite{SY95}, \cite{Shi97}, \cite{Shi00}, \cite{SW02}). Later Goerss-Henn-Mahowald-Rezk in \cite{GHMR05} proposed a conceptual framework to organise the $K(2)$-local homotopy category at the prime $3$, in which the authors constructed a finite resolution of the $K(2)$-local sphere using higher real $K$-theories. See \cite{GHM01}, \cite{HKM13}, \cite{GH16} for further investigations at $n=2$ and $p=3$ and \cite{Beh12} for an exposition of $L_{2}S^{0}$ at $p\geq 5$.\\\\
\noindent
The situation of chromatic level two at the prime $2$ turns out to be much more complicated and we are only beginning to understand it better. Considerable effort has recently been made to understand the $K(2)$-local homotopy category at the prime $2$ by the community. In \cite{BG18}, Bobkova and Goerss established a finite resolution of a spectrum related to the $K(2)$-local sphere at the prime $2$ analogous to that of \cite{GHMR05}, which realised an algebraic resolution of $\mathbb{S}^1_2$, a certain closed subgroup of the second Morava stabiliser group, constructed by Beaudry \cite{Bea15}.\\\\
\noindent
One reason why the latter is hard to deal with lies largely in the fact that the cohomological properties of the group $\mathbb{G}_{2}$ are much more complicated at the prime $2$. However, one exciting feature of chromatic level $2$ is its close relationship with the theory of elliptic curves and modular forms, see Section \ref{Prelim}. At chromatic level $2$ and at the prime $2$, we can choose the Morava $E$-theory to be the Lubin-Tate theory associated to the formal group law of the elliptic curve $C: y^2+y=x^3$ over $\FF_4$. We denote by $E_C$ and $\G_C$ the corresponding Morava $E$-theory and Morava stabiliser group.
One of the main tools used to investigate the $K(2)$-local homotopy category is a certain finite resolution. There is a certain subgroup $\S^1_C$ of $\G_C$; let $G_{24}$ be the automorphism group of $C$ and $C_6$ be a cyclic subgroup of order $6$ of $G_{24}$ (see Section \ref{Prelim} for details).
\begin{thm}\cite{BG18}\label{TDSS Intro} There is a resolution of $E_C^{h\S^1_C}$, in the $K(2)$-local homotopy category at the prime $2$, of the following form
$$E_{C}^{h\mathbb{S}^{1}_{C}}\xrightarrow{\delta_0} \mathcal{E}_0\xrightarrow{\delta_1} \mathcal{E}_1\xrightarrow{\delta_2} \mathcal{E}_2\xrightarrow{\delta_3} \mathcal{E}_3$$
where $\mathcal{E}_0= E_C^{hG_{24}}$, $\mathcal{E}_1=\mathcal{E}_2 = E_C^{hC_{6}}$ and $\mathcal{E}_3 =\Sigma^{48}E_C^{hG_{24}}$.
\end{thm}
\noindent
This resolution is commonly called the topological duality resolution. The spectrum $E_C^{h\S^1_C}$ is used to build the spectrum $E_C^{h\S_C}$, where $\S_C$ is the Morava stabiliser group, via a certain cofiber sequence $$E_C^{h\S_C}\rightarrow E_C^{h\S_C^1}\xrightarrow{1-\pi} E_C^{h\S^1_C},$$ and $E_C^{h\S_C}$ only differs from $L_{K(2)}S^0$ by the Galois action, i.e., there is a homotopy equivalence $$L_{K(2)}S^0 \simeq (E_C^{h\S_C})^{h\Gal(\FF_4/\F)}.$$
Thus, this theorem offers a useful instrument to study the homotopy type of $L_{K(2)}X$ for finite spectra $X$ at the prime $2$. In particular, it produces a spectral sequence, known as the topological duality spectral sequence, abbreviated by TDSS, converging to $\pi_{*}(E_{C}^{h\mathbb{S}^{1}_{C}}\wedge X)$ 
\begin{equation}\phantomsection \label{TDSS}
\E_1^{p,q} \cong \pi_q (\mathcal{E}_p\wedge X) \Longrightarrow \pi_{q-p}(E_C^{h\S^1_C}\wedge X).
\end{equation} 
By now, it should be clear that judicious choices of finite spectra become important. Main players in this paper are finite spectra constructed by Davis and Mahowald in \cite{DM81}. Let $A_1$ denote a class of finite spectra whose mod $2$ cohomology is isomorphic, as a module over the subalgebra $\A(1)$ generated by $Sq^{1}, Sq^{2}$ of the Steenrod algebra $\A$, to a free module of rank one on a class of degree $0$. As shown in \cite{DM81}, the class $A_1$ contains four different homotopy types, which are distinguished by the structure of their mod-2 cohomology as modules over the Steenrod algebra. They are successively denoted by $A_1[00]$, $A_1[01]$, $A_1[10]$, $A_1[11]$, see Definition \ref{version A1}. The spectra $A_1[01]$ and $A_1[10]$ are Spanier-Whitehead self-dual, i.e., $D(A_1[01]) \simeq \Sigma^{-6}A_1[01]$ and $D(A_1[10]) \simeq \Sigma^{-6}A_1[10]$ and the spectra $A_1[00]$ and $A_1[11]$ are Spanier-Whitehead dual to each other, i.e., $D(A_1[00])\simeq \Sigma^{-6} A_1[11]$ (here $D(-)$ denotes the function spectra $F(-, S^0)$). By an abuse of language, we write $A_1$ to refer to any of these four spectra and refer to any of them as a version of $A_1$. In particular, we use this notation in the statement of results that are true for all versions. We emphasis, however, that all results are \textit{a priori} dependent on the version of $A_1$ and this is the case. 
The spectrum $A_1$ is constructed via three cofiber sequences starting from the sphere spectrum. First, let $V(0)$ be the mod $2$ Moore spectrum, i.e., the cofiber of multiplication by $2$ on the sphere. Next let $Y$ be the cofiber of multiplication by $\eta$, the first Hopf element, on $V(0)$. Davis and Mahowald show that $Y$ admits $v_1$-self maps, $v_1: \Sigma^2 Y\rightarrow Y$. Then $A_1$ is the cofiber of any of these $v_1$-self maps of $Y$. We note also that even though $Y$ admits eight $v_1$-self maps, the associated cofibers only have four different homotopy types.  \\\\
\noindent
One reason for working with $A_1$ is the fact that it is the cofiber of a $v_1$-self map of periodicity $1$, making a few computations simpler; this is in contrast with the generalised Moore spectrum $M(2,v_1^4)$ which is the cofiber of a $v_1$-self map of periodicity $4$ on the Moore spectrum $V(0)$. The second one is that a sufficient understanding of the homotopy type of $L_{K(2)}A_1$ might allow us to determine the Gross-Hopkins duality formula for the $K(2)$-local homotopy category at the prime $2$. In fact, the spectrum $A_1$ can be considered as an analog of the Toda-Smith complex $V(1)$ at the prime $3$ and as demonstrated in \cite{GH16}, computations of the homotopy groups of $L_{K(2)}V(1)$ allows one to characterise the Gross-Hopkins formula for the $K(2)$-local homotopy category at the prime $3$. The third reason is that $A_1$ is a "small" finite spectrum of type $2$ having only eight cells with the top cell being in dimension $6$, hence it is reasonable to expect that a study of the homotopy type of $A_1$ gives us valuable information about the homotopy groups of $S^0$, at least about the $v_2$-periodic families of $S^0$. More precisely, the authors of \cite{BEM17} show that $A_1$ admits a $v_2^{32}$-self map. Let $[(v_2^{32})^{-1}]A_1$ denote the associated telescope, i.e.,
$$[(v_2^{32})^{-1}]A_1 = \hocolim (A_1 \rightarrow \Sigma^{-192} A_1\rightarrow ...\rightarrow \Sigma^{-192k}A_1\rightarrow ...).$$
We note that the homotopy type of this telescope is independent on the choice of $v_2$-self map of $A_1$ by Nilpotence and Periodicity Technology, see \cite{Rav92}. Suppose that $x\in \pi_t([(v_2^{32})^{-1}]A_1)$ is a nontrivial element. This means that the composite 
$$S^{t+192k}\rightarrow \Sigma^{192k}A_1\xrightarrow{v_2^{32k}}A_1$$ 
is essential for $k\in \N$. This gives rise to a nontrivial element of $\pi_*S^0$ in stem $192k +t - i_{k}$ for some $0\leq i_{k}\leq 6$. \\\\
\noindent
Moreover, the $K(2)$-localisation of $A_1$ might be used to detect nontrivial elements of the homotopy groups of $[(v_2^{32})^{-1}]A_1$. In fact, the $K(2)$-localisation map $A_1\rightarrow L_{K(2)}A_1$ factors through $[(v_2^{32})^{-1}]A_1\rightarrow L_{K(2)}A_1.$ Ravenel's Telescope Conjecture predicts that the latter is a homotopy equivalence. As a key step towards the study of $\pi_*(L_{K(2)}A_1)$, as explained in the discussion following Theorem \ref{TDSS Intro}, we study the TDSS for $E_C^{h\S^1_C}\wedge A_1$.\\\\
 In this paper, we study the homotopy fixed point spectral sequence, abbreviated by HFPSS, for $E_C^{hG_{24}}\wedge A_1$, which constitutes an important part of the $\E_1$-term of the TDSS:
\begin{equation}\phantomsection \label{HFPSS2}
\mathrm{H}^{*}(G_{24},(E_{C})_*(A_1))\Longrightarrow \pi_{*}(E_{C}^{hG_{24}}\wedge A_1).
\end{equation} 
Here are qualitative versions of the main results of the paper; see Theorem \ref{Einfty A01} and \ref{Einfty A11} for more precise statements.
\\\\
There are classes $$\Delta^8\in \H^{0}(G_{24}, (E_C)_{192}), \ovk\in \H^{4}(G_{24}, (E_C)_{24}), \nu\in \H^{1}(G_{24}, (E_C)_{4}).$$
\begin{thm} As a module over the ring $\FF_4[\Delta^{\pm 8},\ovk, \nu]/(\nu\ovk)$, the $\E_{\infty}$-term of the HFPSS for $E_C^{hG_{24}}\wedge A_1[01]$ and $E_C^{hG_{24}}\wedge A_1[10]$  is a direct sum of $46$ explicitly known cyclic modules. 
\end{thm}
\begin{thm} As a module over the ring $\FF_4[\Delta^{\pm 8},\ovk, \nu]/(\nu\ovk)$, the $\E_{\infty}$-term of the HFPSS for $E_C^{hG_{24}}\wedge A_1[00]$ and $E_C^{hG_{24}}\wedge A_1[11]$  is a direct sum of $48$ explicitly known cyclic modules. 
\end{thm}
\noindent
One of the key ingredients, to this end, is a comparison between $tmf\wedge A_1$ and $E_C^{hG_{24}}\wedge A_1$, where $tmf$ denotes the connective spectrum of topological modular forms. In fact, there is a homotopy equivalence (Theorem \ref{compar}):
 $$(\Delta^{8})^{-1}tmf\wedge A_1 \simeq (E_{C}^{hG_{24}})^{h\Gal(\FF_{4}/\F)}\wedge A_1,$$ where $\Delta^{8}$ is the periodicity generator of $\pi_{*}tmf$. Based on the latter, we first analyse the homotopy groups of $tmf\wedge A_1$ by means of the Adams spectral sequence, abbreviated by ASS , then invert $\Delta^{8}$ to get information about the homotopy groups of $E_{C}^{hG_{24}}\wedge A_1$.  
We note that in \cite{BEM17}, Batacharya, Egger, Mahowald also discuss the $\E_2$-term of the ASS for $tmf\wedge A_1$; our method is, however, different (compare to \cite{BEM17}). Next, we summarise the contents of the paper.\\\\
\noindent
In Section \ref{Prelim} and Section \ref{DMSS_Const}, we discuss some background and tools used in our computation. We recollect on Lubin-Tate theories and topological modular forms; in particular, we sketch a proof of the relationship between topological modular forms and homotopy fixed point spectrum $E_C^{hG_{24}}$. We give a generalisation of the Davis-Mahowald spectral sequence, which is an important tool to analyse the cohomology of various Hopf algebras. In Section \ref{G24 2}, we discuss the Davis-Mahowald spectral sequence for $A_1$ and obtain the $\mathrm{E}_{2}$-term of the Adams spectral sequence for $tmf\wedge A_1$. In Section \ref{G24 3}, we study some differentials in the later and then extract some suitable information about $\pi_{*}(tmf\wedge A_1)$. In Section \ref{G24 4}, we finally study the homotopy fixed point spectral sequence for $E_{C}^{hG_{24}}\wedge A_1$. We emphasise that there are two different outcomes for the $\E_{\infty}$-term of the homotopy fixed point spectral sequence, depending on the version of $A_1$, see Theorem \ref{Einfty A01} and \ref{Einfty A11}.\\\\
\textbf{Convention and Notation.} Unless otherwise stated, all spectra are localised at the prime $2$. $\mathrm{H}^{*}(X)$ and $\mathrm{H}_{*}(X)$ denote the mod-$2$ cohomology and homology of the spectrum $X$, respectively. Given a Hopf algebra $A$ over a field $k$ and $M$ a $A$-comodule, we will often abbreviate $\Ext_{A}^{*}(k,M)$ by $\Ext_{A}^{*}(M)$. In general, we will write $C_{f}$ for the cofiber of a map $f : X\rightarrow Y$, except that we will write $V(0)$ for the Moore spectrum which is the cofiber of the multiplication by $2$ on the sphere.
\\\\
\textbf{Acknowlegements.} I would like to thank my PhD advisor, Hans-Werner Henn, for suggesting this project and for guiding me through many stages of it and for carefully reading earlier drafts of this paper. Special thanks go to Irina Bobkova for many fruitful exchanges, to Agn\`es Beaudry and Paul Goerss for their interest in this work and helpful conversations, to Lennart Meier for answering many of my questions about topological modular forms and for suggesting I include the proof of Theorem \ref{tmf-HFP}. I would also like to thank John Rognes for making available his course notes on "Adams spectral sequence", from which I learned many details about the Davis-Mahowald spectral sequence.
%%%%%%%%%%%%% 

\section{Recollection on chromatic homotopy theory}\label{Prelim}
\subsection{Lubin-Tate theories}
We recall some generalities on the deformation theory of formal group laws and Goerss-Hopkins-Miller theory. Let $\mathcal{FGL}$ be the category whose objects are pairs $(k,\Gamma)$ where $k$ is a perfect field of characteristic $p$ and $\Gamma$ is a formal group law over $k$ and morphisms between $(k,\Gamma)$ and $(k^{'}, \Gamma^{'})$ are pairs $(i, \phi)$ where $i :  k^{'}\rightarrow k$ is a homomorphism of fields and $\phi: \Gamma\xrightarrow{\cong} i^*\Gamma^{'}$ is a morphism of formal group laws. \\\\
Let $(k,\Gamma)\in \mathcal{FGL}$ with $\Gamma$ of height $n$. A deformation of $(k, \Gamma)$ to a complete local ring $R$ with maximal ideal $m$ is a pair $(F,\iota)$ where $F$ is a formal group law over $R$ and $\iota: k\rightarrow R/m$ is a map of fields such that $p^{*}F = \iota^*\Gamma$ with $p$ the canonical projection $R\rightarrow R/m$. A $\star$-isomorphism $\phi$ between two deformations to $R$ is an isomorphism between the underlying formal group laws which reduces to the identity over $R/m$, i.e, $\phi\equiv x\ \mbox{mod}\ (m)$. This defines a functor from the category of complete local rings $\Ring_{c,l}$ to small groupoids $\Groupoid$
$$\Def_{\Gamma}: \Ring_{c,l}\rightarrow \Groupoid$$
which associates to every complete local ring $R$ the category of deformations of $(k,\Gamma)$ over $R$ and $\star$-isomorphisms between them. By Lubin-Tate deformation theory, $\Def_{\Gamma}$ is co-representable, see \cite{LT66}. That is, there exists a complete local ring $E_{k,\Gamma}$, non-canonically isomorphic to $\W(k)[[u_1, u_2,...,u_{n-1}]]$, such that $$\Def_{\Gamma}(R) \cong \Hom_{\Ring_{c,l}}(E_{k,\Gamma}, R).$$ Here $\W(k)$ denotes the ring of Witt vectors on $k$.
Over $E_{k,\Gamma}$ lives a universal deformation $\tilde{\Gamma}$ of $\Gamma$. Consider the graded ring $E_{k,\Gamma}[u^{\pm 1}]$ where $|u_i|=0$ for $1\leq i\leq n-1$ and $|u| = -2$. Let $MU$ be the cobordism spectrum. A famous theorem of Quillen asserts that the coefficient rings $MU_*$ of $MU$ supports the universal group law. Thus, the formal group law $u^{-1}\tilde{\Gamma}(ux,uy)$ is classified by a map of graded rings $MU_*\rightarrow E_{k,\Gamma}[u^{\pm 1}].$ Define a functor from the category of pointed spaces to that of graded abelian groups:
 $$X\mapsto MU_*(X)\otimes_{MU_*} E_{k,\Gamma}[u^{\pm1}].$$
 The formal group $u^{-1}\widetilde{\Gamma}(ux,uy)$ satisfies the Landweber exact functor criterion, see \cite{Rez98}. By the Landweber exact functor theorem, the above functor is a homology functor. Thus, it is represented by a ring spectrum $E(k, \Gamma)$ with $$(E(k,\Gamma))_*\cong \W(k)[[u_1, u_2,..., u_{n-1}]][u^{\pm 1}].$$ The latter is known as a $n^{th}$ Morava $E$-theory or Lubin-Tate theory.
 \begin{Example}\label{Cmpx K} Let $(k,\Gamma) = (\F, \G_m)$ where $\G_m$ is the multiplicative formal group law, i.e, $\G_m(x,y) = x+y + xy$. Then $E(\F,\G_m) \simeq K\Z_{2}$, the $2$-completed complex $K$-theory and $\G(\F,\G_m) = \Z_2^{\times}$, the unit of the $2$-adic integers. Furthermore, the action of $\G(\F,\G_m)$ on $E(\F,\G_m)_*$ is determined by Adams operations.
 \end{Example}
 \noindent
The construction that associates to a formal group law $(k, \Gamma)$ the Morava $E$-theory $E(k,\Gamma)$ defines a functor from $\FGL$ to $\Ho(Sp)$, the stable homotopy category. Let us denote by $\G(k,\Gamma)$ the automorphism group of the pair $(k,\Gamma)$. We note that $\G(k,\Gamma)$ is a profinite group, see \cite{Goe08}, Section 7.2. By functoriality, the group $\G(k,\Gamma)$ acts on $E(k,\Gamma)$. This action is, however, defined only up to homotopy. The Goerss-Hopkins-Miller obstruction theory lifts this action to structured ring spectra.
\begin{Theorem}\cite{GH04} \label{GHM ob}The spectrum $E(k,\Gamma)$ has an essentially unique structure of $E_{\infty}$-ring. Furthermore, $\G(k,\Gamma)$ acts on $E(k,\Gamma)$ via $E_{\infty}$-ring maps.
\end{Theorem}
%%%%%%%%%%%%%%%%%% TMF
\subsection{Topological modular forms}\label{sub-tmf} An astute choice of Morava $E$-theory or equivalently a choice of formal group law of height $2$ will make the calculation easier. 
Let $C$ be the supersingular elliptic curve over $\FF_4$ given by the Weierstrass equation $y^2+y = x^3$. Denote by $F_{C}$ the formal completion of $C$ at the origin. The latter is a formal group law of height $2$. We abbreviate $ E(\FF_4, F_C)$ by $E_C$ and $\G(\FF_4, F_C)$ by $\G_C$. Let $\S_C$ denote the automorphism group of $F_C$. Let $\Gal$ denote the Galois group of $\FF_4$ over $\F$. There is a short exact sequence
$$1\rightarrow \S_C\rightarrow \G_C\rightarrow \Gal\rightarrow 1.$$
The image of $\S_C$ in $\G_C$ corresponds to the automorphisms of $(\FF_4, F_C)$ fixing $\FF_4$. 
Since $F_C$ is defined over $\F$, $\Gal$ fixes $F_C$, the above short exact sequence splits, i.e., $\G_C\cong \S_C\rtimes \Gal.$ The automorphism group of $C$ has order $24$ and these are all defined over $\FF_4$, more precisely, 
$$\Aut(C)=\Aut_{\FF_4}(C) \cong SL_2(\Z/3)\cong Q_{8}\rtimes C_3=:G_{24},$$ where $Q_8$ is the quaternion group and $C_3 = \langle \omega\rangle$ is a cyclic group of order $3$, see \cite{Sil09}. The group $Q_8$ has a representation $\langle i,j| i^4=1, i^2=j^2, iji^{-1}=j^{-1}\rangle$. The latter has $8$ elements  $\{1, i, j, k, -1, -i, -j,-k\}$ where $-1$ denotes $i^2=j^2=k^2$. The group $C_3$ acts on $Q_8$ by permuting $i, j $ and $k:=ij$$$\omega i\omega^2 = j,\ \ \ \omega j\omega^2 = k.$$ The elements $\omega$ and $i$ correspond to the automorphisms $\omega (x,y) = (\xi x, \xi^2 y)$ and $i(x,y) = (x+1, y+x+\xi^2)$, respectively.
\\\\
Since $C$ is already defined over $\F$, $\Gal$ acts on $\Aut(C)$. Denote by $G_{48}$ the semi-direct product $G_{24}\rtimes \Gal$. Moreover, the automorphism group $\Aut(C)$ of $C$ maps injectively to $\S_C$,  and $G_{48}$ maps injectively to $\G_C$. We view $G_{24}$ and $G_{48}$ as subgroups of $\S_C$ and $\G_C$, respectively. 
\\\\
%We see that $C_2 = \langle-1\rangle \leq Q_{8}$ is invariant under the action of $C_3$, and so $C_6 := C_{2}\times C_3$ is a subgroup of $G_{24}$. As an automorphism of $C$, $-1$ is given by $(x,y)\mapsto (x, y+1)$. We see immediately that both $C_2$ and $C_3$ are invariant by $\Gal$, and hence $G_{12}:= C_6\rtimes \Gal$ is a subgroup of $G_{48}$
%\\\\
%The homotopy fixed point spectra $E_C^{hG_{24}}$ and $E_{C}^{hC_{6}}$ as well as $E_C^{hG_{48}}$ and $E_C^{hG_{12}}$ will play a central role in this thesis, as already mentioned in the introduction.
%\\\\
The reasons for choosing the formal group law of the supersingular elliptic curve $C$ are two-fold. First, the geometric origin of $G_{48}$ allows one to have an explicit description of its action on $\pi_*(E_C)$, see \cite{Bea17} for more details and further references. Thus, it allows us to adequately compute the $\E_2$-term of various homotopy fixed point spectral sequences. Second, this choice of the Morava $E$-theory enables us to compare the associated homotopy fixed point spectrum with the spectrum of topological modular forms, hence providing us with more tools to understand the formers. 
\\\\
Next, we recall the construction of the spectrum of topological modular forms and show its closed relationship with the homotopy fixed point spectrum $E_C^{hG_{24}}$.
Let $\M$, $\M(3)$ be the moduli stack of elliptic curves and elliptic curves with a full level  $3$ structure over $\Z_{(2)}$, respectively. As functors of points on $\Z_{(2)}$-algebras, the former are described as follows. If $R$ is a $\Z_{(2)}$-algebra, then

- $\M(\spec(R))$ is the groupoid of elliptic curves over $\spec(R)$ and isomorphisms between them.

- $\M(3)(\spec(R))$ is the groupoid of pairs $(E,\phi)$ consisting of an elliptic curve $E$ with an isomorphism of group schemes $\phi: \Z/3\times \Z/3\rightarrow E[3]$ over $\spec(R)$, where $E[3]$ is the subscheme of $3$-torsion points of $E$ and isomorphisms between them. 

\begin{Theorem}[Goerss-Hopkins-Miller, see \cite{DFHH14}] There is an $E_{\infty}$-ring spectra-valued sheaf $\mathcal{O}^{top}$ on the affine \'etale site $\aff$ of $\M$ such that 
\begin{itemize}
\item[1.] The sheafification of $\pi_0\mathcal{O}^{top}$ is the structure sheaf of $\M$.
\item[2.] If $E: \spec(R)\rightarrow \M$ is an \'etale morphism, then $\mathcal{O}^{top}(\spec(R))$ is a spectrum associated to the formal completion of $E$ at its origin via the Landweber exact functor theorem.
\end{itemize}
\end{Theorem}
\begin{Remark} The spectra constructed by point $2.$ of the previous theorem are called elliptic spectra. They are even periodic spectra $R$ whose formal group law on $\pi_0(R)$ is the completion of an elliptic curve. These are $E(2)$-local, see \cite{DFHH14}, Chapter 6, Lemma 4.2.
\end{Remark}
\noindent
Let $G:=GL_2(\Z/3)$ denote the automorphism group of the constant group scheme $\Z/3\times \Z/3$ over $\Z_{(2)}$. Then $G$ acts on $\M(3)$ by precomposition with the level structure. The obvious forgetful functor gives rise to a finite \'etale morphism of stacks (because $3$ is invertible in $\Z_{(2)}$): 
\begin{equation}\label{etale map of stack}
\M(3)\rightarrow \M.
\end{equation}
Thus, one can evaluate $\mathcal{O}^{top}$ at $\M$ and $\M(3)$. Define

$$TMF = \mathcal{O}^{top}(\M):= \underset{U\in \aff}{\mathrm{holim}} \mathcal{O}^{top}(U),$$ 
%$$TMF_0(3) = \mathcal{O}^{top}(\M_0(3)):= \underset{U\in Aff^{\acute{e}t}_{\mathcal{M}_0(3)}}{\mathrm{holim}} \mathcal{O}^{top}(U),$$

$$TMF(3) = \mathcal{O}^{top}(\M(3)):= \underset{U\in Aff^{\acute{e}t}_{\mathcal{M}(3)}}{\mathrm{holim}} \mathcal{O}^{top}(U).$$
The morphism of (\ref{etale map of stack}) is a Galois cover with Galois group $G$, or a $G$-torsor.As a consequence of the fact that $\mathcal{O}^{top}$ satisfies descent, one obtains that 
\begin{equation}\label{descent TMF}
TMF \simeq TMF(3)^{hG}.
\end{equation}
It is known that $\M(3)$ is affine over the ring $\Z_{(2)}[ \zeta]$ where $\zeta$ is a primitive third root of unity, see \cite{DR73}, also \cite{Sto14}. Furthermore, up to isomorphism, there is a unique supersingular elliptic curve with a full level structure over $\FF_4$. This follows from the fact that there is a unique supersingular elliptic curve over $\FF_4$ (up to isomorphism) and that the automorphism group of the supersingular elliptic curve $C$ has order $48$, which is equal to that of $G$, the automorphism group of $\Z/3\times \Z/3$. In other words, the fiber of the morphism $\M(3)\rightarrow \M$ over the supersingular locus of $\M$ is isomorphic to $\spec(\FF_4)$, i.e., the following square is a pullback of stacks
$$\xymatrix{ \spec(\FF_4) \ar[rr]\ar[d]\ar@{}[drr]|{} && \M(3)\ar[d]\\
			\spec(\FF_4)//G_{48} \ar[rr] && \M
}$$
where the bottom is given by specifying a supersingular elliptic curve, for example $C$. Therefore, by the construction of $\mathcal{O}^{top}$, $L_{K(2)}\mathcal{O}^{top}(\M(3))$ is the Lubin-Tate theory associated to the paire $(\FF_4, F_C)$, see \cite{DFHH14}, Chapter 12. This means that there is a homotopy equivalence 
\begin{equation}\label{K(2)-loc TMF(3)}L_{K(2)}TMF(3)\xrightarrow{\simeq} E_C.
\end{equation}
Note that $G$ can be identified with $\Aut(C) = G_{48}$, such that the equivalence (\ref{K(2)-loc TMF(3)}) is equivariant with respect to the action of $G$ on the source and of $G_{48}$ on the target, as follows. Suppose the the map $\spec(\FF_4)\rightarrow \M(3)$ specifies the elliptic curve $C$ and a $3$ level structure $\Z/3^{\times 2}\xrightarrow{\Gamma} C$. Then for any $g\in G$, there is a unique $\phi(g)\in G_{48}$ making the following diagram commute
$$\xymatrix{ \Z/3^{\times 3} \ar[r]^-{\Gamma} &C\\
\Z/3^{\times 3}\ar[u]^{g} \ar[r]^-{\Gamma} & C\ar[u]_{\phi(g)}.
}$$
\begin{Theorem}\label{tmf-HFP} There is a homotopy equivalence
\begin{equation}\label{K(2)-loc TMF}
L_{K(2)}TMF \simeq E_C^{hG_{48}}.
\end{equation} 
\end{Theorem}
\begin{proof} Since an elliptic spectrum is $E(2)$-local, $TMF(3)$ is $E(2)$-local, being a homotopy limit of $E(2)$-local spectra. Using the equivalence (\ref{descent TMF}) and the fact that $K(2)$-localisation commutes with homotopy limit in the category of $E(2)$-local spectra, we obtain that 
$$L_{K(2)}TMF \cong L_{K(2)} (TMF(3)^{hG})\cong (L_{K(2)}TMF(3))^{hG}\cong E_C^{hG_{48}}.$$
\end{proof}
%Both $TMF$ and $TMF_{0}(3)$ have connective models in the following sense. \\\\
\noindent
\textbf{A connective model of $TMF$.} In \cite{DFHH14} a connective ring spectrum $tmf$ was constructed together with a map of ring spectra $tmf\rightarrow TMF$. There is an element $\Delta^8\in \pi_{192} tmf$ such that the latter map extends to a homotopy equivalence 
\begin{equation}\label{tmf connec model}[(\Delta^8)^{-1}] tmf \simeq TMF,
\end{equation} see \cite{DFHH14}. The (co)homology of $tmf$, as a module over the Steenrod algebra $\A$ (c.f Section \ref{DMSS_Const} for a recollection on the Steenrod algebra), was known earlier by Hopkins and Mahowald but its proof remained unpublished until \cite{Mat16}:
\begin{Theorem} There is an isomorphism of modules over the Steenrod algebra:
$$\H^*(tmf)\cong \A//\A(2),$$
where $\A(2)$ is the subalgebra of $\A$ generated by $Sq^1, Sq^2, Sq^4$.
Equivalently, there is an isomorphism of comodules over the dual $\A_*$ of Steenrod algebra 
$$\H_*(tmf)\cong\A_*\square_{\A(2)_*}\F,$$
where $\A(2)_*$ is the dual of $\A(2)$.
\end{Theorem}
\section{The Davis-Mahowald spectral sequence}\label{DMSS_Const}
We introduce a generalisation of the Davis-Mahowald spectral sequence, which is an useful tool for analysing $\Ext$-groups over various Hopf algebras. Initially, this spectral sequence was used by Davis and Mahowald in \cite{DM82} to compute $\Ext$-groups over the subalgebra $\A(2)$ of the Steenrod algebra.
\subsection{Construction of the Davis-Mahowald spectral sequence} Let $k$ be a field of characteristic $2$. We will later specialise to the case $k=\F$, the field of two elements. Let $(A,\Delta,\mu,\epsilon,\eta, \chi)$ be a commutative Hopf algebra over $k$ with $\Delta, \mu, \epsilon, \eta,\chi$ being coproduct, product, counit, unit, the conjugation, respectively. 
%%%%%%%%%%%%
\begin{Definition} Let $E$ be the graded exterior algebra on a finite dimensional $k$-vector space $V$ with all elements of $V$ having degree $1$. An $A$-comodule algebra structure on $E$ is called almost graded if the natural embedding $k\oplus V\rightarrow E $ is a map of $A$-comodules.
\end{Definition}
\noindent
This definition is motivated by the following examples which are of main interest in this paper. Recall that the Steenrod algebra $\A$ is generated by the Steenrod squares $Sq^{i}$ for $i\geq 0$, subject to the Adem relations 
$$Sq^{a}Sq^{b} = \sum\limits_{i=0}^{\lfloor \frac{a}{2} \rfloor} {b-i-1 \choose a-2i}  Sq^{a+b-i}Sq^{i}$$ 
for all $a,b >0$ and $a<2b$. Let $\A_*$ denote the dual of the Steenrod algebra. In \cite{Mil58}, Milnor determines the Hopf algebra structure of $\A_*$.
As a graded algebra, $\A_{*} = \F[\xi_{i}|i \geq 1]$ where $\xi_{i}$ is in degree $| \xi_{i} | = 2^{i} -1$. The coproduct is given by the formula 
$$\Delta(\xi_{k})=\sum_{i=0}^{k}\xi_{i}^{2^{k-i}}\otimes \xi_{k-i},$$ where $\xi_{0} =1$. Let us denote by $\zeta_i$ the conjugate of $\xi_{i}$. Then we have 
\begin{equation}\phantomsection \label{Diagonal}
\Delta(\zeta_k)=\sum\limits_{i+j=k}\zeta_{i}\otimes\zeta_{j}^{2^{i}}.
\end{equation} 
An Hopf ideal of a Hopf algebra $A$ is an ideal $I$ such that $\Delta(I)\subset I\otimes A + A\otimes I$. If $I$ is a Hopf ideal of $A$, then $A/I$ inherits a structure of Hopf algebra from $A$ such that the natural projection $A\rightarrow A/I$ is a map of Hopf algebras.
\begin{Example}\label{ExampA(n)} Let $\A(n)_{*}$
be the quotient of $\A_{*}$ by the Hopf ideal $I_{n}$ generated by $(\zeta_{1}^{2^{n+1}},\zeta_{2}^{2^{n}},...,\zeta_{n+1}^{2}, \zeta_{n+2},...)$. As an algebra, $$\A(n)_{*} =\mathbb{F}_{2}[\zeta_{1},\zeta_{2},...,\zeta_{n+1}]/(\zeta_{1}^{2^{n+1}},\zeta_{2}^{2^{n}},...,\zeta_{n+1}^{2}).$$
It is dual to the subalgebra $\A(n) = \langle Sq^{1}, Sq^{2},...,Sq^{2^{n}}\rangle$ of the Steenrod algebra $A$.
%%%%%%%%%%%%%%
The canonical projection $\pi :  \A(n)_{*}\rightarrow \A(n-1)_{*}$ induced by the inclusion $I_{n}\subset I_{n-1}$ of Hopf ideals is a map of Hopf algebras, hence induces on $\A(n)_{*}$ a structure of right $\A(n-1)_{*}$-comodule algebra: $$(id\otimes \pi)\Delta: \A(n)_{*}\rightarrow \A(n)_{*}\otimes \A(n)_{*}\rightarrow \A(n)_{*}\otimes \A(n-1)_{*}. $$ 
%%%%%%%%%%%%%%
An easy computation shows that the group of primitives $\A(n)_{*}\square_{\A(n-1)_{*}}\F$ of this coaction is given by
$$\A(n)_{*}\square_{\A(n-1)_{*}}\F = E(\zeta_{1}^{2^{n}},\zeta_{2}^{2^{n-1}},...,\zeta_{n+1})$$ which is abstractly isomorphic to $E_{n}=E(x_{1},...,x_{n+1})$ where $x_{i}$ stands for $\zeta_{i}^{2^{n+1-i}}$. Here and elsewhere in this paper, $E(X)$ denotes the exterior algebra on the $k$-vector space spanned by the set $X$. We see that the algebra $E(x_{1},x_{2},...,x_{n+1})$ inherits a left $\A(n)_{*}$-comodule algebra structure from $\A(n)_{*}$, namely,
$$\Delta(x_{k}) = \sum_{i=0}^{k}\zeta_{i}^{2^{n+1-k}}\otimes x_{k-i}, \ 1\leq k\leq n+1$$ where $x_{0}=1$ by convention. %Therefore, the pair $(\A(n)_{*},E_{n})$ fits into the above framework. 
This means that $ E_{n}$ is an almost graded $\A(n)_{*}$-comodule.
\end{Example}
%%%%%%%%%%%%%%%%%%%%
\begin{Example}\label{ExampB(n)} Let $B(n)_{*}$ be the quotient of $\A_{*}$ by the Hopf ideal $J_{n}$ generated by $(\zeta_{1}^{2^{n}},\zeta_{2}^{2^{n}},\zeta_{3}^{2^{n-1}},...,\zeta_{n+1}^{2}, \zeta_{n+2},...)$, so that  
$$B(n)_{*} = \mathbb{F}_{2}[\zeta_{1},\zeta_{2},...,\zeta_{n+1}]/(\zeta_{1}^{2^{n}},\zeta_{2}^{2^{n}},\zeta_{3}^{2^{n-1}},...,\zeta_{n+1}^{2}).$$
Similarly to \textit{Example} \ref{ExampA(n)}, the projection $B(n)_{*}\rightarrow \A(n-1)_{*}$ induced by the inclusion of Hopf ideals $J_{n}\subset I_{n-1}$ defines a structure of right $\A(n-1)_{*}$-comodule algebra on $B(n)_{*}$. A calculation shows that 
$$B(n)_{*}\square_{\A(n-1)_{*}}\F = E(\zeta_{2}^{2^{n-1}},\zeta_{3}^{2^{n-2}},...,\zeta_{n+1}),$$
which is abstractly isomorphic to $F_{n} := E(x_{2},...,x_{n+1})$. The notation is chosen to be coherent with that of \textit{Example} \ref{ExampA(n)}. We see that $F_{n}$ inherits a structure of left $B(n)_{*}$- comodule algebra from that of $B(n)_{*}$, namely,
$$\Delta(x_{k})=  \sum_{i=0, i\ne 1}^{k}\zeta_{i}^{2^{n+1-k}}\otimes x_{k-i}, \ 2\leq k\leq n+1$$ where $x_{0}=1$. Thus, $F_n$ is a almost graded $B(n)_*$-comodule.
\end{Example} 
\noindent
Let $E$ be an almost graded $A$-comodule exterior algebra on a finite dimensional $k$-vector space $V$. We will construct an $A$-comodule polynomial algebra, called the Koszul dual of $E$ as follows. Let $P$ be the graded polynomial algebra of $V$ with all elements of $V$ having degree $1$. Let us denote by $E_{i}$ and $P_{i}$ the subspace of elements of homogeneous degree $i$ for $i\geq 0$ of $E$ and $P$, respectively. Let us also denote by $E_{\leq i}$ the the direct sum $\bigoplus\limits_{j=0}^{j=i} E_{j}$. Notice that $P_{1}$ sits in a short exact sequence:
\begin{equation}\phantomsection \label{ExSq}
0\rightarrow k\rightarrow k\oplus E_{1}\xrightarrow{p} P_{1}\rightarrow 0.
\end{equation}
The embedding $k\rightarrow k\oplus E_{1}$ is clearly a map of left $A$-comodules. Thus $P_{1}$ admits a (unique) structure of left $A$-comodule such that $p :k\oplus E_{1} \rightarrow P_{1}$ is a map of $A$-comodules. 
%\begin{Lemma} $P_{1}^{\otimes n}$ has a left $A$-comodule structure with coaction given the composite \marginpar[check english]{give a more detailed proof}
%
%$$P_{1}^{\otimes n}\xrightarrow{\Delta^{\otimes n}} (A\otimes P_{1})^{\otimes n}\xrightarrow{shuffle}A^{\otimes n}\otimes P_{1}^{\otimes n}\rightarrow A\otimes P_{1}^{\otimes n}  $$ where the shuffle map keeps copies of $A$ and $P_{1}$ in the same order and the last map is the multiplication on $A$.
%\end{Lemma}
%\begin{proof} This is the usual $A$-comodule structure on the tensor product.
%The coassociativity of the coaction follows from the coassociativity of the coaction map of $P_{1}$ and the commutativity of $A$. 
%\end{proof}
%%%%%%%%%%%%%%%%%%%%%%%
\begin{Lemma} If $P_1^{\otimes n}$ is equipped with the usual structure of $A$-comodule of a tensor product, then $P_n$ admits a unique structure of $A$-comodule making the multiplication $P_1^{\otimes n}\rightarrow P_n$ a map of $A$-comodules.
\end{Lemma}
\begin{proof} This map is surjective and its kernel is spanned by elements of the form $y_{1}\otimes ...\otimes y_{n} - y_{\sigma(1)}\otimes...\otimes y_{\sigma(n)_{*}}$ where $\sigma$ is a permutation of the set $\{1, 2, ..., n\}$. Then, since $A$ is commutative, we see that the kernel is stable under the coaction of $A$. The lemma follows.
\end{proof}
\noindent
This lemma shows that $P=\bigoplus\limits_{i\geq 0}P_{i}$ admits a left $A$-comodule algebra structure.\\\\
Now, let us define a cochain complex, called the Koszul complex, \begin{equation}\phantomsection \label{Kos-CC}(E\otimes P, d)\end{equation}  with %\marginpar{Moche!!! Find a way to express d by a diagram}
\begin{itemize}
  \item[i)] $(E\otimes P)_{-1} = k$
  
  \item[ii)] $(E\otimes P)_{m}= E\otimes P_{m}$ for $m\geq 0$ 
  
  \item[iii)] $d : k=(E\otimes P)_{-1}\rightarrow E = (E\otimes P)_{0}$ being the unit of $E$
  
  \item[iv)] $d(\prod\limits_{j=1}^{n} x_{i_{j}}\otimes z) = \sum\limits_{t=1}^{n} \prod\limits_{j\ne t} x_{i_{j}}\otimes p(x_{i_{t}})z$ where $x_{i_{j}}\in E_{1}$, $z\in P_{m}$ and $p$ is the projection of (\ref{ExSq}).
\end{itemize}

\begin{Remark}  In other words, $d: E_{\leq n}\otimes P_{m}\rightarrow E_{\leq n-1}\otimes P_{m+1}$ is the unique homomorphism making the following diagram commute
\begin{equation}\phantomsection \label{diagram_d}
\xymatrix{ E_{\leq 1}^{\otimes n}\otimes P_{\leq m} \ar[rrrr]^{(\underset{\sigma}{\sum} (Id^{\otimes (n-1)}\otimes p)\circ \sigma)\otimes Id}\ar[d]^{\mu\otimes Id}& &&&E_{\leq 1}^{\otimes (n-1)}\otimes P_{1}\otimes P_{m}\ar[d]^{\mu\otimes \mu}\\
		  E_{\leq n}\otimes P_{m}\ar[rrrr]^{d} &&& &E_{\leq n-1}\otimes P_{m+1},
}
\end{equation}
where in the upper horizontal map, the sum is taken over all cyclic permutations on $n$ factors of $E_{1}$ in the tensor product $E_{1}^{\otimes n}$ and $p$ is the restriction on $E_{1}$ of the map of ($\ref{ExSq}$).
\end{Remark}

\begin{Proposition} The complex $(E\otimes P,d)$ is an exact sequence of $A$-comodules. Furthermore, $(E\otimes P,d)$ has a structure of differential graded algebra induced from the algebra structure of $E$ and $P$.
\end{Proposition}
\begin{proof} Let $x_{1}, ..., x_{n}$ be a basis of $E_{1}$. As a cochain complex over $k$, $(E\otimes P,d)$ is isomorphic to the tensor product of $(E(x_{i})\otimes k[y_{i}],d_{i})$ where $y_{i} = p(x_{i})$ for $1\leq i\leq n$. Here, each $(E(x_{i})\otimes k[y_{i}],d_{i})$ is defined in the same manner as $(E\otimes P,d)$ is. It is not hard to see that the cochain complex $(E(x_{i})\otimes k[y_{i}],d_{i})$ is exact. Hence, $(E\otimes P,d)$ is exact by the K\"unneth theorem. This proves the first part.\\\\
Let us check that $d$ is a map of $A$-comodules. In the diagram (\ref{diagram_d}), the two vertical maps are ones of $A$-comodules because $E$ and $P$ are $A$-comodule algebras. In addition, they are surjective. It remains to check that the upper horizontal map is a map of $A$-comodules. Or equivalently, each map $E_{\leq 1}^{\otimes n}\xrightarrow{(Id^{\otimes(n-1)}\otimes p)\circ \sigma} E_{\leq 1}^{\otimes (n-1)}\otimes P_{1}$ is a map of $A$-comodules where $\sigma$ is a cyclic permutation on $n$ elements. This is true because $\sigma$ is a map of $A$-comodules as $A$ is commutative and $p$ is a map of $A$-comodules by definition. The second part follows. \\\\
\noindent
Finally, it is straightforward from the formula of $d$ in (\ref{Kos-CC}.iv) that $d$ satisfies the Leibniz rule.
\end{proof}
\noindent
This lemma allows us to construct a spectral sequence of algebras converging to $\mathrm{Ext}_{A}^{s}(k)$ see (\cite{Rav86}, Theorem A1.3.2). 
\begin{Proposition} \phantomsection \label{Prop_SS}(1) There is a spectral sequence of algebras converging to $\mathrm{Ext}_{A}^{s}(k):$
\begin{equation}\phantomsection \label{Rav-SS}
\xymatrix{
                   \E_1^{s,t} = \mathrm{Ext}_{A}^{s}(k,E\otimes P_{t})\ar@{=>}[r]&\mathrm{Ext}_{A}^{s+t}(k,k)}.%%
\end{equation}          
%%%%%
(2) If $M$ is a $A$-comodule, then there is a spectral sequence converging to $\mathrm{Ext}_{A}^{s}(M)$
$$\xymatrix{\E_1^{s,t} = \mathrm{Ext}_{A}^{s}(k,E\otimes P_{t}\otimes M)\ar@{=>}[r]&\mathrm{Ext}_{A}^{s+t}(k,M).
}$$
%%%%%
Furthermore, this spectral sequence is a spectral sequence of modules over that of (\ref{Rav-SS}). 
\end{Proposition}
\noindent
\textbf{Terminology.} We will call these spectral sequences the Davis-Mahowald spectral sequences or DMSS for short, associated to the almost graded $A$-module algebra $E$. The first grading $s$ of the $\E_n$-term is referred to as the cohomological grading or degree and the second grading $t$ is referred to as the Davis-Mahowald grading or degree (or DM grading or degree for short). 
\\\\
In view of carrying out explicit computations of products in $\mathrm{Ext}_{A}^{*}(k)$ and the action of $\mathrm{Ext}_{A}^{*}(k)$ on $\mathrm{Ext}_{A}^{*}(M)$, we recall a double complex from which the above spectral sequence is derived.
\\\\
For each $t\geq 0$, let $(C^{s}(A,E\otimes P_{t}), d_{v} )_{s\geq 0}$ be the cobar complex whose cohomology is $\Ext_{A}^{*}(E\otimes P_{t})$, i.e., 
$$C^{s}(A,E\otimes P_{t})=A^{\otimes s}\otimes E\otimes P_{t}$$ 
and 
$d_{v} : A^{\otimes s}\otimes E\otimes P_{t}\rightarrow A^{\otimes s+1}\otimes E\otimes P_{t}$ is given by 
\begin{multline*}
d_{v}(a_{1}\otimes ...\otimes a_{s}\otimes m) =1\otimes a_{1}\otimes ...\otimes a_{s}\otimes m+\sum\limits_{i=1}^{s}a_{1}\otimes...\otimes a_{i-1}\otimes \Delta(a_{i})\otimes...\otimes a_{s}\otimes m \\
+ a_{1}\otimes ...\otimes a_{s}\otimes \Delta(m),
\end{multline*} 
where $a_i\in A$ for $1\leq i\leq s$ and $m\in E\otimes P_t$.
We will abbreviate $a_{1}\otimes ...\otimes a_{s}\otimes m$ by $[a_{1}|...|a_{s}|m]$. By an abuse of notation, we will denote by $d_{v}$ the differentials in the cobar complexes associated to $E\otimes P_{t}$ for different $t$. The fact that $d: E\otimes P_{t}\rightarrow E\otimes P_{t+1}$ is a map of $A$-comodules implies that the maps $ d_{h}= Id^{\otimes s}\otimes d: C^{s}(A,E\otimes P_{t}) \rightarrow C^{s}(A,E\otimes P_{t+1})$ assemble to give a map of cochain complexes 
$ d_{h}: (C^{s}(A,E\otimes P_{t}), d_{v})_{s\geq 0} \rightarrow (C^{s}(A,E\otimes P_{t+1}), d_{v} )_{s\geq0}$.
Finally, it is easily seen that the maps of cochain complexes assemble to form a double complex $(C^{s}(A,E\otimes P_{t}),d_{v},d_{h})_{s,t\geq 0}$

$$
\xymatrix{ E\ar[r]^{d_{h}}\ar[d]^{d_{v}}&E\otimes P_{1}\ar[r]^{d_{h}}\ar[d]^{d_{v}}&E\otimes P_{2}\ar[r]^{d_{h}}\ar[d]^{d_{v}}&E\otimes P_{3}\ar[r]^{d_{h}}\ar[d]^{d_{v}}&...\\
		A\otimes E\ar[r]^{d_{h}}\ar[d]^{d_{v}}& A\otimes E\otimes P_{1}\ar[r]^{d_{h}}\ar[d]^{d_{v}}&A\otimes E\otimes P_{2}\ar[r]^{d_{h}}\ar[d]^{d_{v}}&A\otimes E\otimes P_{3}\ar[r]^{d_{h}}\ar[d]^{d_{v}}&...\\
		A^{\otimes 2}\otimes E\ar[r]^{d_{h}}\ar[d]^{d_{v}}& A^{\otimes 2} \otimes E\otimes P_{1}\ar[r]^{d_{h}}\ar[d]^{d_{v}}&A^{\otimes 2}\otimes E\otimes P_{2}\ar[r]^{d_{h}}\ar[d]^{d_{v}}&A^{\otimes 2}\otimes E\otimes P_{3}\ar[r]^{d_{h}}\ar[d]^{d_{v}}&...\\
		...&...&...&...&...\\
}$$
We can see that the spectral sequence associated to the horizontal filtration has $\mathrm{E}_{1}$-term isomorphic to $(A^{s}\otimes k,d_{v})_{s\geq 0}$ which identifies with the cobar complex of the trivial $A$-comodule $k$. Thus this spectral sequence degenerates at the $\mathrm{E}_{2}$-term and the $\mathrm{E}_{\infty}=\mathrm{E}_{2}$-term identifies with $\Ext_{A}^{s}(k)$. Since there are no possible extension problems, the cohomology of the total complex is isomorphic to $\Ext_{A}^{s}(k)$. Now, the spectral sequence associated to the vertical filtration has $\mathrm{E}_{1}$-term isomorphic to $\Ext^{s}_{A}( E\otimes P_{t})$. This spectral sequence is exactly the one appearing in Proposition \ref{Prop_SS}.
\begin{Remark} \phantomsection \label{SSd_1}The differential $d_{1}: \Ext^{0}_{A}(E\otimes P_{t})\rightarrow \Ext^{0}_{A}(E\otimes P_{t+1})$ is the restriction of the derivation $d$ in (\ref{Kos-CC}) on the $A$-primitives of $E\otimes P_{t}$.
\end{Remark}
\subsection{Naturality of the Davis-Mahowald spectral sequence} We notice that the above construction is natural in pairs $(A,E)$ where $A$ is a commutative Hopf algebra and $E$ is an almost graded left $A$-comodule exterior algebra. This allows us to compare Davis-Mahowald spectral sequences associated to different pairs $(A,E)$. We will make use of this property to reduce computations in a crucial way. Let us first define morphisms between such pairs.

\begin{Definition} Let $(A, E)$ and $(B, F)$ be such that $A$ and $B$ are commutative Hopf algebras, $E$ and $F$ are almost graded exterior comodule algebras over $A$ and $B$, respectively. A morphism between $(A,E)$ and $(B,F)$ consists of $f_{1}: A\rightarrow B$ and $f_{2}: E\rightarrow F$ where $f_{1}$ is a map of Hopf algebras and $f_{2}$ is a map of $B$-comodule graded algebras with the $B$-comodule structure on $E$ being induced from $f_{1}$.
\end{Definition}
\begin{Remark} The map $f_2 : E\rightarrow F$ is determined by a map of $B$-comodules $k\oplus E_1\rightarrow k\oplus F_1$.
\end{Remark}
\begin{Proposition} A morphism between $(A,E)$ and $(B,F)$ induces a map between the associated Davis-Mahowald spectral sequences.
\end{Proposition}

\begin{proof} Let $P$ and $Q$ be the Koszul dual of $E$ and $F$, respectively. The map of $B$-comodule algebras $f_{2} : E\rightarrow F$ induces a map of graded $B$-comodule algebras $P\rightarrow Q$ such that the following diagram is commutative
$$\xymatrix{ k\oplus E_{1}\ar[r]^{p}\ar[d]^{f_{2}} &P_{1}\ar[d]\\
		    k\oplus F_{1}\ar[r]^{p}& Q_{1}.\\	
}$$
Then one can check that the induced map $E\otimes P\rightarrow F\otimes Q$ is a map of Koszul complexes. Therefore one obtains a map of double complexes $(A^{\otimes s}\otimes E\otimes P_{t}) \rightarrow (B^{\otimes s}\otimes F\otimes Q_{t})$, hence a map of spectral sequences.
 \end{proof}
 \noindent
\begin{Remark} Although we have only treated the ungraded situation so far, the construction carries over verbatim to the graded one. More precisely, suppose that $A$ and $E$ are graded algebras. We refer to this grading as the internal degree. We require that the structural maps in the $A$-comodule structure of $E$ to preserve the internal degree. Then we see that the Koszul dual $P$ of $E$ is internally graded and the Koszul complex is a graded cochain complex with respect to the internal degree. It follows that the associated DMSS is tri-graded with the third grading associated to the internal grading and the differentials preserve the internal degree.
\end{Remark}
\noindent
We continue with \textit{Example} \ref{ExampA(n)} and \ref{ExampB(n)}.
%%%%%%%%%%%%%%%%%%%%%%%%%%%%%
\begin{Example}\label{ExampA(n)-cont} Recall that $E_{n}$ is an almost graded $\A(n)_{*}$-comodule. Let $R_{n}$ denote the Koszul dual of $E_{n}$. In particular, it follows from Proposition \ref{Prop_SS} that for any graded left $\A(n)_{*}$-comodule $M$, the DMSS converging to $\Ext_{\A(n)_{*}}^{*,*}(\F,M)$ has  $\mathrm{E}_{1}$-term isomorphic to $$\E_1^{s,t,\sigma}\cong \Ext_{\A(n)_{*}}^{s,t}(E_{n}\otimes R_{n}^{\sigma}\otimes M),$$ where $s$ is the cohomological grading, $t$ is the internal grading and $\sigma$ is the Davis-Mahowald grading.
The change-of-rings isomorphism tells us that 
$$\mathrm{Ext}_{\A(n)_{*}}^{s,t}( E_{n}\otimes R_{n}^{\sigma}\otimes M)\cong \mathrm{Ext}_{\A(n-1)_{*}}^{s,t}( R_{n}^{\sigma}\otimes M),$$
see \cite{Rav84}, Appendix A1.3.13 for the change-of-rings isomorphism. This means that the problem of computing $\Ext_{\A(n)_{*}}^{s,t}(-)$ can be reduced to two steps: first computing $\Ext_{\A(n-1)_{*}}^{s,t}(-)$, then studying the corresponding Davis-Mahowald spectral sequence. We will demonstrate the efficiency of this method by carrying out explicit computations in the case $n=2$ and some relevant $M$.
\end{Example}
%%%%%%%%%%%%%%%%%%%%
\begin{Example}\label{ExampB(n)-cont} Recall that $F_n$ is an almost graded $B(n)_*$-comodule. Let $S_{n}$ denote the Koszul dual of $F_{n}$. The DMSS is the spectral sequence of algebras 
$$\xymatrix{ \E_1^{s, t, \sigma} = \mathrm{Ext}^{s,t}_{B(n)_{*}}(F_{n}\otimes S_{n}^{\sigma})\ar@{=>}[r]&\mathrm{Ext}^{s+\sigma,t}_{B(n)_{*}}(\mathbb{F}_{2})}.$$ 
By the change-of-rings theorem, the $\mathrm{E}_{1}$-term is isomorphic to $\mathrm{Ext}_{\A(n-1)_{*}}^{s,t}(S_{n}^{\sigma}),$ because $F_{n} = B(n)_{*}\square_{\A(n-1)_{*}}\F$. Moreover, for any graded left $B(n)_{*}$-comodule $M$, the DMSS for $\mathrm{Ext}^{s+\sigma,t}_{B(n)_{*}}(\mathbb{F}_{2})$ is a spectral sequence of modules over the above spectral sequence
 $$\xymatrix{ \mathrm{Ext}^{s,t}_{B(n)_{*}}(F_{n}\otimes S_{n}^{\sigma}\otimes M)\cong \mathrm{Ext}_{\A(n-1)_{*}}^{s,t}(S_{n}^{\sigma}\otimes M)\ar@{=>}[r]&\mathrm{Ext}^{s+\sigma,t}_{B(n)_{*}}(\mathbb{F}_{2})}.$$
 \end{Example} 
 \noindent
%%%%%%%%%%%%%%%%%%%%
\textbf{Comparison of DMSS.} There is a morphism between $(\A(n)_{*}, E_{n})$ and $(B(n)_{*},F_{n})$ given by the two projections 
$$
\A(n)_{*}\rightarrow B(n)_{*};  \zeta_{i}\mapsto\zeta_{i} $$
$$E_{n}\rightarrow F_{n}; x_{1}\mapsto 0, x_{i}\mapsto x_{i}\ \mbox{for} \ i\geq 2. $$
This induces a map of spectral sequences 
$$\xymatrix{ \mathrm{Ext}_{\A(n)_{*}}^{s,t}(E_{n}\otimes R_{n}^{\sigma}\otimes M)\ar@{=>}[d]\ar[r]&\mathrm{Ext}_{B(n)_{*}}^{s,t}(F_{n}\otimes S_{n}^{\sigma}\otimes M)\ar@{=>}[d]\\
\mathrm{Ext}^{s+\sigma,t}_{\A(n)_{*}}(M)\ar[r]&\mathrm{Ext}^{s+\sigma,t}_{B(n)_{*}}(M).
}$$  
As was mentioned earlier, this comparison allows us to transfer some computations in the former SS to the latter which are simpler because all modules involved in the latter are smaller. This observation will be made concrete in Section \ref{G24 2}. 
%%%%%%%%%%%%%%%%%%%%%%%%%%%%%%%%%%%%%%%%%%%% tmfA1
\section{The Davis-Mahowald spectral sequence for the $\A(2)_{*}$-comodule $A_1$}\label{G24 2}
\noindent
The goal of this section is to describe the structure of $\Ext_{\A(2)_{*}}^{*,*}(\H_*(A_1))$ as a module over $\Ext_{\A(2)_{*}}^{*,*}(\F)$ for different $\A(2)_{*}$-comodules $A_1$ that will be recalled in Subsection \ref{DMSS_A_1}. To achieve a part of this goal, we will study the DMSS 
$$\xymatrix{ \Ext_{\A(2)_{*}}^{s,t}(E_{2}\otimes S_{2}^{\sigma}\otimes A_1)\ar@{=>}[r]&\Ext_{\A(2)_{*}}^{s+\sigma,t}(\H_*(A_1))
}$$ as a spectral sequence of modules over the spectral sequence of algebras
$$\xymatrix{ \Ext_{\A(2)_{*}}^{s,t}(E_{2}\otimes S_{2}^{\sigma})\ar@{=>}[r]&\Ext_{\A(2)_{*}}^{s+\sigma,t}(\F).
}$$ 
We obtain then the structure of  $\Ext_{\A(2)_{*}}^{*,*}(\H_*(A_1))$ as a graded abelian group and a partial action of $\Ext_{\A(2)_{*}}^{*,*}(\F)$ on it. However, there is an important action of an element of $\Ext_{\A(2)_{*}}^{*,*}(\F)$ on some elements of $\Ext_{\A(2)_{*}}^{*,*}(\H_*(A_1))$ that cannot be seen at the $\mathrm{E}_{1}$-term of the DMSS. One way of understanding these exotic products is to carry out computations at the level of double complexes: find representatives of the cohomological classes in question in the double complexes from which the DMSS is derived and carry out products at that level. It turns out that a brute-force attack is messy. Instead, computations are simplified drastically by comparing the DMSS associated to $(\A(2)_{*}, E_{2})$ to that of $(B(2)_{*}, F_{2})$:
$$\xymatrix{ \mathrm{Ext}_{\A(2)_{*}}^{s,t}(E_{n}\otimes R_{2}^{\sigma}\otimes A_1)\ar@{=>}[d]\ar[r]&\mathrm{Ext}_{B(2)_{*}}^{s,t}(F_{n}\otimes S_{2}^{\sigma}\otimes (1))\ar@{=>}[d]\\
\mathrm{Ext}^{s+\sigma,t}_{\A(2)_{*}}(\H_*(A_1))\ar[r]&\mathrm{Ext}^{s+\sigma,t}_{B(2)_{*}}(\H_*(A_1)).
}$$  
\subsection{Recollections on the Davis-Mahowald spectral sequence for the $\A(2)_{*}$-comodule $\F$} 
\noindent
To fix notation, we recollect some information relevant for our purposes. This material was originally treated in \cite{DM82} and reviewed in unpublished course notes of Rognes \cite{Rog12}. As we will specialise to the case $n=2$, we will simplify the notation by writing $R, R_{\sigma}, S, S_{\sigma}$ for $R_{2}, R_{2}^{\sigma}, S_{2}, S_{2}^{\sigma}$ from \textit{Example} \ref{ExampA(n)} and \ref{ExampB(n)}, respectively.\\\\
Recall that $R$ is a homogenous graded polynomial algebra on three generators, say $y_{1}, y_{2}, y_{3}$ and $R_{\sigma}$ is its subspace of homogeneous elements of degree $\sigma$ for $\sigma \geq 0$. Let us first explicitly give the coaction of $\A(2)_{*}$ on $R = \F[y_{1},y_{2},y_{3}]$ with $|y_{1}| = 4$, $|y_{2}|=6$, $|y_{3}|=7$. From \textit{Example} \ref{ExampA(n)-cont}, we have
%%%%%%%%%%%%
\begin{alignat*}{1} 
\Delta(y_{1})&=1\otimes y_{1}\\
\Delta(y_{2})&=\xi_{1}^{2}\otimes y_{1}+1\otimes y_{2}\\
\Delta(y_{3})&=\zeta_2\otimes y_{1}+\xi_{1}\otimes y_{2}+1\otimes y_{3}.
\end{alignat*}
By the change-of-rings theorem, the $\mathrm{E}_{1}$-term of the DMSS for $\Ext_{\A(2)_{*}}^{*,*}(\F)$ is isomorphic to $\Ext_{\A(1)_*}^{s,t}(\bigoplus\limits_{\sigma\geq 0}R_{\sigma})$. The coaction of $\A(1)_*$ on $R_{1}$ is induced from that of $\A(2)_{*}$ and hence is given by
%%%%%%%%%%%%%%%%%
\begin{alignat*}{1}
\Delta(y_{1})&=1\otimes y_{1}\\
\Delta(y_{2})&=\xi_{1}^{2}\otimes y_{1}+1\otimes y_{2}\\
\Delta(y_{3})&=\zeta_2\otimes y_{1}+\xi_{1}\otimes y_{2}+1\otimes y_{3}.
\end{alignat*}
In particular, $y_{1}, y_{2}^{2}, y_{3}^{4}$ are $\A(1)_*$-primitives of $R$. % and so represents a class $v_{2}^{4}\in\Ext^{0,28}_{\A(1)_*}(\F,R_{4})$. 
Let $R_{\sigma}^{'}$ denote the $\A(1)_*$-subcomodule $\{y_{1}^{i}y_{2}^{j}y_{3}^{k}\in R_{\sigma}| k\leq3\}$ of $R_{\sigma}$. 

\begin{Lemma}\label {Lem_Decom} As an $\A(1)_*$-comodule, $R_{\sigma}$ can be decomposed as 

$$R_{\sigma}\cong\bigoplus\limits_{i\equiv \sigma (mod 4), i\leq\sigma}R_{i}^{'}\otimes \F\{y_{3}^{\sigma-i}\}.$$ 
Therefore, $$\bigoplus\limits_{\sigma\geq 0} R_{\sigma} = (\bigoplus\limits_{\sigma\geq 0} R_{\sigma}^{'})\otimes \F[y_{3}^{4}].$$
\end{Lemma}
\begin{proof} If one views $\F\{y_{3}^{\sigma-i}\}$ as a subvector space of $R_{\sigma-i}$, then the product of $R$ produces an isomorphism of vector spaces
$$\bigoplus\limits_{i\equiv \sigma (mod 4), i\leq\sigma}R_{i}^{'}\otimes \F\{y_{3}^{\sigma-i}\}\xrightarrow{\cong} R_{\sigma}.$$
Since $y_{3}^{4}$ is a $\A(1)_*$-primitive of $R_{\sigma}$, this map is also a map of $\A(1)_*$-comodules. The lemma follows.
\end{proof}
\noindent
Let us denote $\Ext_{\A(1)_*}^{*,*}(R_{\sigma}^{'})$ by $G_{\sigma}$, so that $$\Ext_{\A(1)_*}^{*,*}(R) \cong (\bigoplus\limits_{\sigma\geq 0} G_{\sigma} )\otimes \F[v_{2}^{4}],$$ where $v_2^4\in \Ext^{0,24}_{\A(1)_*}(R_4)$ represented by $y_3^4\in R_4$.
Determining the full multiplicative structure of $\Ext_{\A(1)_*}^{*,*}(R)$ is quite involved. Instead, we will work modulo $(v_{2}^{4})$. This will suffice for us to obtain a set of algebra generators of $\Ext_{\A(1)_*}^{*,*}(R)$. More precisely, since the product $R_{\sigma}^{'}\otimes R_{\tau}^{'}\rightarrow R_{\sigma+\tau}$ factorises through $R_{\sigma+\tau}^{'} \oplus (R_{\sigma+\tau -4}\otimes\F\{y_{3}^{4}\})$, we obtain a map $$G_{\sigma}\otimes G_{\tau}\rightarrow G_{\sigma+\tau}\oplus (G_{\sigma+\tau-4}\otimes\F\{v_{2}^{4}\}).$$ We will analyse the map $G_{\sigma}\otimes G_{\tau}\rightarrow G_{\sigma+\tau}$ which is the composite 
$$G_{\sigma}\otimes G_{\tau}\rightarrow G_{\sigma+\tau}\oplus (G_{\sigma+\tau-4}\otimes\F\{v_{2}^{4}\})\rightarrow G_{\sigma+\tau} $$ where the second map is the projection on the first factor.\\\\
In what follows, we compute $G_{i}$ for $i\geq 0$ as modules over $G_{0}$. For this, we decompose $R_{i}^{'}$ into smaller pieces, compute the $\Ext$ groups over $\A(1)_*$ of these pieces, then determine $G_{i}$ via long exact sequences. Next, we study the pairings $$G_{\sigma}\otimes G_{\tau}\rightarrow G_{\sigma+\tau},$$ which allows us to determine a set of algebra generators of the $\mathrm{E}_{1}$-term. Finally, we compute $d_{1}$-differentials on this set of algebra generators. We do not intend to describe completely the $\Ext_{\A(2)_{*}}^{*,*}(\F)$ but only a subalgebra in which we are interested.\\\\
Since $y_{1}$ is primitive, multiplication by $y_{1}$ induces injections of $\A(1)_*$-comodules 
$$\Sigma^{4}R_{\sigma}^{'}\rightarrow R_{\sigma+1}^{'}.$$
%%%%%%%%%%%%%%%%%
\begin{Lemma}\phantomsection \label{Lem_SES1} There are short exact sequences of $\A(1)_*$-comodules
\begin{itemize}
\item[(a)]$$0\rightarrow  \mathrm{H}_{*}(\Sigma^{12}C_{\eta})\rightarrow R_{2}^{'}\rightarrow  \Sigma^{8} (\A(1)_*\square_{\A(0)_*}\F)\rightarrow 0$$ where $\eta : S^{1}\rightarrow S^{0}$ is the Hopf map and the map $ \mathrm{H}_{*}(\Sigma^{12}C_{\eta})\rightarrow R_{2}^{'}$ sends the generators of $ \mathrm{H}_{12}(\Sigma^{12}C_{\eta})$ and $ \mathrm{H}_{14}(\Sigma^{12}C_{\eta})$  to $y_{2}^{2}$ and $y_{3}^{2}$, respectively.
\item[(b)] $$0\rightarrow \Sigma^{4}R_{1}^{'}\rightarrow R_{2}^{'}\rightarrow \Sigma^{12}V_{3}\rightarrow 0$$ where $V_{3} =\mathrm{H}_{*}(S^{0}\cup_{2}e^{1}\cup_{\eta}e^{2}).$
\end{itemize}
\end{Lemma} 
\begin{proof} 
For part $(a)$, the map $\Sigma^{12} \mathrm{H}_{*}(C_{\eta})\rightarrow R_{2}^{'}$ described in the statement of the Lemma \ref{Lem_SES1} is a map of $\A(1)_*$-comodules. Its quotient is isomorphic to $\F\{y_{1}^{2},y_{1}y_{2}, y_{1}y_{3}, y_{2}y_{3} \}$ with the $\A(1)_*$-comodule structure given by 
\begin{alignat*}{1}
\Delta(y_{2}y_{3}) &= 1\otimes y_{2}y_{3} + \xi_{1}^{2}\otimes y_{1}y_{3}+\xi_{2}\otimes y_{1}y_{2}+\zeta_2\xi_{1}^{2}\otimes y_{1}^{2}\\
\Delta(y_{1}y_{3}) &= 1\otimes y_{1}y_{3}+\xi_{1}\otimes y_{1}y_{2}+\zeta_2\otimes y_{1}^{2}\\
\Delta(y_{1}y_{2})&=1\otimes y_{1}y_{2}+\xi_{1}^2\otimes y_{1}^{2}\\
\Delta(y_{1}^{2})&=1\otimes y_{1}^{2}.
\end{alignat*} 
We can check that this module is isomorphic to $\Sigma^{8}(\A(1)_*\square_{\A(0)_*}\F)$ as $\A(1)_*$-comodules.\\\\
\noindent
For part $(b)$, the quotient of $R^{'}_{2}$ by $\Sigma^{4}R^{'}_{1}$ is isomorphic to $\F\{y_{2}^{2},y_{2}y_{3},y_{3}^{2}\}$ with $\A(1)_*$-comodule structure given by
\begin{alignat*}{1}
\Delta(y_{2}^{2}) &=1\otimes y_{2}^{2}\\
\Delta(y_{2}y_{3}) &= \xi_{1}\otimes y_{2}^{2}+1\otimes y_{2}y_{3}\\
\Delta(y_{3}^{2}) &= \xi_{1}^{2}\otimes y_{2}^{2}+1\otimes y_{3}^{2}.
\end{alignat*}
One can check that this quotient is isomorphic to $\Sigma^{12}V_{3}$.
\end{proof}

\begin{Lemma}\phantomsection \label{Lem_SES2} For every $\sigma\geq 3$, there is a short exact sequence of $\A(1)_*$-comodules
$$0\rightarrow \Sigma^{4}R_{\sigma-1}^{'}\xrightarrow{\times y_{1}} R_{\sigma}^{'}\rightarrow \Sigma^{6\sigma}V_{4}\rightarrow 0$$ where $V_{4}$ is $\mathrm{H}_{*}(V(0) \wedge C_{\eta})$. 
\end{Lemma} 
\noindent
\begin{Remark} The spectrum $V(0) \wedge C_{\eta}$ is homotopy equivalent to $Y$, introduced in the Introduction (c.f Section \ref{DMSS_A_1} for a presentation of $\H^{*}(Y)$.)
\end{Remark}
\begin{proof} The quotient of $R^{'}_{\sigma}$ by $\Sigma^{4}R^{'}_{\sigma-1}$ is isomorphic to $\F\{y_{2}^{\sigma},y_{2}^{\sigma-1}y_{3},y_{2}^{\sigma-2}y_{3}^{2},y_{2}^{\sigma-3}y_{3}^{3}\}$ with $\A(1)_*$-comodule structure given by
\begin{alignat*}{1}
\Delta(y_{2}^{\sigma}) &= 1\otimes y_{2}^{\sigma}\\
\Delta(y_{2}^{\sigma-1}y_{3}) &= \xi_{1}\otimes y_{2}^{\sigma} + 1\otimes y_{2}^{\sigma-1}y_{3}\\
\Delta(y_{2}^{\sigma-2}y_{3}^{2})&=\xi_{1}^{2}\otimes y_{2}^{\sigma}+1\otimes y_{2}^{\sigma-2}y_{3}^{2}\\
\Delta(y_{2}^{\sigma-3}y_{3}^{3})&=\xi_{1}^{3}\otimes y_{2}^{\sigma}+\xi_{1}^{2}\otimes y_{2}^{\sigma-1}y_{3}+\xi_{1}\otimes y_{2}^{\sigma-2}y_{3}^{2}+1\otimes y_{2}^{\sigma-3}y_{3}^{3}.
\end{alignat*} 
It can be easily seen that this quotient is isomorphic to $\Sigma^{6\sigma} V_{4}$.
\end{proof}
\noindent
Next we describe the $\mathrm{Ext}$ groups of some $\A(1)_*$-comodules as basic steps towards computing $G_{\sigma}$. These calculations are elementary and classical. %By Bott's periodicity, the groups $\Ext_{\A(1)_*}^{*,*}(M)$ for any $\A(1)_*$-comodule $M$ is a module over $\F[v_{1}^{4}]$. In charts representing these groups, we will often only plot their $\F[v_{1}^{4}]$-module generators.  
%\marginpar{produce charts for these $\Ext$ groups}
%1)$\Ext^{s,t}_{\A(1)_*}(\F)$, $\Ext^{s,t}_{\A(1)_*}(C_{2})$, $\Ext_{\A(1)_*}^{s,t}( C_{\eta})$ 
\begin{Proposition} There are classes $h_0\in \Ext^{1,1}$, $h_1\in \Ext^{1,2}$, $v\in \Ext^{3,7}$, $v_1^4\in \Ext^{4,12}$ such that there is an isomorphism of algebras
$$G_{0}:=\Ext^{s,t}_{\A(1)_*}(\F)\cong \F[h_{0}, h_{1}, v, v_{1}^{4}]/(h_{1}^{3}, h_{0}h_{1}, h_{1}v, v^{2}-h_{0}^{2}v_{1}^{4}).$$

\begin{figure}[h!]
\begin{center}
\begin{tikzpicture}[scale=0.5]
\clip(-1.5,-1.5) rectangle (10,6);
\draw[color=gray] (0,0) grid [step=1] (10,6);
\foreach \n in {0,1,...,10}
{
\def\nn{\n-0}
\node[below] at (\nn+0.5,0) {$\n$};
}
\foreach \s in {0,1,...,6}
{\def\ss{\s-0};
\node [left] at (-0.4,\ss+0.5,0){$\s$};
}
\foreach \s in {0,1,...,4}
\draw [fill] (0.5,\s+0.5) circle [radius=0.05];
\draw[->] (0.5,0.5)--(0.5,5.5);
\node at (0.2,1.5) {$h_{0}$};
\node at (2.2,1.5){$h_{1}$};
%%%%%%%%%%%%%%%%%%%
\draw [fill] (1.5,1.5) circle [radius=0.05];
\draw [fill] (2+0.5,2+0.5) circle [radius=0.05];
\draw[-] (0.5,0.5)--(2.5,2+0.5);
%%%%%%%%%%%
\draw [fill] (4+0.5,3+0.5) circle [radius=0.05];
\draw [fill] (4+0.5,4+0.5) circle [radius=0.05];
\draw[->] (4.5,3.5)--(4.5,5.5);
\node at (4.5,3) {$v$};
%%%%%%
\draw [fill] (8.5,4.5) circle [radius=0.05];
\draw[->] (8.5,4.5)--(8.5,5.5);
\draw (8.5,4.5) circle [radius=0.4 ];
\node at (8.5,3.5) {$v_{1}^{4}$};
\end{tikzpicture}
\caption{$\Ext^{*,*}_{\A(1)_*}(\F,\F)$ in the range $0\leq t-s \leq 8.$}
\end{center}
\end{figure}
\end{Proposition}
\noindent
See for example (\cite{Rav86}, Theorem 3.1.25).
\begin{Lemma}\phantomsection \label{Lem_Mod} As a module over $\Ext^{s,t}_{\A(1)_*}(\F)$,
\begin{itemize}
\item[(1)] $\Ext^{s,t}_{\A(1)_*}( \mathrm{H}_{*}(V(0)))$ is generated by $h^{0}\in \Ext^{0,0}$, $h^{1}\in \Ext^{1,3}$ with the following relations $h_{0}h^{0} = vh^{0} =v h^{1} =0$ and $h_{1}^{2}.h^{0} = h_{0}h^{1}$.
\item[(2)] $\Ext_{\A(1)_*}^{s,t}( \mathrm{H}_{*}(C_{\eta}))$ is generated by $\{h^{i}\in \Ext^{i,3i}|\ 0\leq i\leq 3\}$ with $h_{1}h^{i} = 0$, $vh^{0} = h_{0}h^{2}$, $vh^{1} = h_{0}h^{3}$.
\item[(3)] $\Ext_{\A(1)_*}^{s,t}(\mathrm{H}_{*}(S^{0}\cup_{2}e^{1}\cup_{\eta}e^{2}))$ is generated by $h^{0}\in \Ext^{0,0}$, $h^{1}\in\Ext^{1,3}$, $a^{1}\in\Ext^{1,3}$, $h^{2}\in\Ext^{2,6}$, $h^{3}\in\Ext^{3,9}$ with $h_{0}h^{0}=h_{1}h^{0}=h_{1}h^{1}=h_{0}a^{1}= va^{1}=h_{1}h^{2}= vh^{2}=h_{1}h^{3}=vh^{3}=0$ and $h_{0}h^{2}=h_{1}^{2}a^{1}$.
\item[(4)] $\Ext_{\A(1)_*}^{s,t}( \mathrm{H}_{*}(Y))$ is generated by $\{h^{i}|\ 0\leq i\leq 3\}$ with $h_{0}h^{i}=h_{1}h^{i} = vh^{i} = 0$.
\end{itemize}
\end{Lemma}
%--------------------------
\begin{figure}[h!]
\begin{minipage}[c]{.46\linewidth}
\begin{tikzpicture}[scale=0.5]
\clip(-1.5,-1.5) rectangle (5,4);
\draw[color=gray] (0,0) grid [step=1] (10,6);
\foreach \n in {0,1,...,5}
{
\def\nn{\n-0}
\node[below] at (\nn+0.5,0) {$\n$};
}
\foreach \s in {0,1,...,5}
{\def\ss{\s-0};
\node [left] at (-0.4,\ss+0.5,0){$\s$};
}
\draw [fill] (0.5,0.5) circle [radius=0.05];
\draw[fill] (1.5,1.5) circle [radius=0.05];
\draw[fill] (2.5,2.5) circle [radius =0.05];
\draw[fill] (2.5,1.5) circle [radius=0.05];
\draw[fill] (3.5,2.5) circle [radius =0.05];
\draw[fill] (4.5,3.5) circle [radius=0.05];
%flech
\draw[-] (0.5,0.5)--(1.5,1.5);
\draw[-](1.5,1.5)--(2.5,2.5);
\draw[-](2.5,1.5)--(3.5,2.5);
\draw[-](3.5,2.5)--(4.5,3.5);
\draw[-](2.5,1.5)--(2.5,2.5);
%%%%%%%%%
\node at (1.2,0.5) {$h^{0}$};
\node at (3.2,1.5) {$h^{1}$};
\end{tikzpicture}
\phantomsection \label{Fig_C2}
\caption{$\Ext^{s,t}_{\A(1)_*}(\H_*(V(0)))$ in the range $0\leq t-s\leq 4$.}
\end{minipage} \hfill
\begin{minipage}[c]{.46\linewidth}
 \begin{tikzpicture}[scale=0.5]
\clip(-1.5,-1.5) rectangle (8,5);
\draw[color=gray] (0,0) grid [step=1] (9,6);
\foreach \n in {0,1,...,6}
{
\def\nn{\n-0}
\node[below] at (\nn+0.5,0) {$\n$};
}
\foreach \s in {0,1,...,4}
{\def\ss{\s-0};
\node [left] at (-0.4,\ss+0.5,0){$\s$};
}
\foreach \s in {0,1,...,4}
\draw [fill] (0.5,\s+0.5) circle [radius=0.05];
\draw[->] (0.5,0.5)--(0.5,5.5);
\node at (1,0.5) {$h^{0}$};

\foreach \s in {1,2,...,4}
\draw [fill] (2.5,\s+0.5) circle [radius=0.05];
\draw[->] (2.5,1.5)--(2.5,5.5);
\node at (3,1.5) {$h^{1}$};

\foreach \s in {2,3,...,4}
\draw [fill] (4.5,\s+0.5) circle [radius=0.05];
\draw[->] (4.5,2.5)--(4.5,5.5);
\node at (5,2.5) {$h^{2}$};

\foreach \s in {3,4}
\draw [fill] (6.5,\s+0.5) circle [radius=0.05];
\draw[->] (6.5,3.5)--(6.5,5.5);
\node at (7,3.5) {$h^{3}$};
\end{tikzpicture}
\caption{$\Ext^{s,t}_{\A(1)_*}(\H_*(C_{\eta}))$ in the range $0\leq t-s\leq 6$.}
 \end{minipage}\hfill
\begin{minipage}[c]{.46\linewidth}
 \begin{tikzpicture}[scale=0.5]
\clip(-1.5,-1.5) rectangle (8,5);
\draw[color=gray] (0,0) grid [step=1] (9,6);
\foreach \n in {0,1,...,6}
{
\def\nn{\n-0}
\node[below] at (\nn+0.5,0) {$\n$};
}
\foreach \s in {0,1,...,4}
{\def\ss{\s-0};
\node [left] at (-0.4,\ss+0.5,0){$\s$};
}

\draw[fill] (0.5,0.5) circle [radius=0.05];
\node at (0.5,0.2) {$h^{0}$};
%%%%%%%%%
\foreach \s in{0,1,...,3}
\draw[fill] (2.5,\s+1.5) circle [radius=0.05];
\node at (2.4,1.2) {$h^{1}$};
\draw[-] (2.5,1.5)--(2.5,5.5);
%%%%%%%%%%
\draw[fill] (2.7,1.5) circle [radius=0.05];
\node at (3.2,1.4) {$a^{1}$};
\draw[fill] (3.5,2.5) circle [radius=0.05];
\draw[-] (2.7,1.5)--(3.5,2.5);
\draw[-] (3.5,2.5)--(4.5,3.5);
%%%%%%%%%%%
\draw[fill] (4.5,2.5) circle [radius=0.05];
\draw[fill] (4.5,3.5) circle [radius=0.05];
\node at (4.5,2.2) {$h^{2}$};
\draw[-] (4.5,2.5)--(4.5,3.5);
\draw[fill] (6.5,3.5) circle [radius=0.05];
\draw [fill] (6.5,4.5) circle [radius=0.05];
\node at (6.5,3.2) {$h^{3}$};
\draw[-] (6.5,3.5)--(6.5,5.5);

\end{tikzpicture}
\caption{$\Ext^{s,t}_{\A(1)_*}(\H_*(S^{0}\cup_{2}e^{1}\cup_{\eta}e^{2}))$ in the range $0\leq t-s\leq 6$.}
\end{minipage}\hfill
\begin{minipage}[c]{.46\linewidth}
 \begin{tikzpicture}[scale=0.5]
\clip(-1.5,-1.5) rectangle (8,5);
\draw[color=gray] (0,0) grid [step=1] (9,6);
\foreach \n in {0,1,...,6}
{
\def\nn{\n-0}
\node[below] at (\nn+0.5,0) {$\n$};
}
\foreach \s in {0,1,...,4}
{\def\ss{\s-0};
\node [left] at (-0.4,\ss+0.5,0){$\s$};
}
\draw [fill] (0.5,0.5) circle [radius=0.05];
\node at (1,0.5) {$h^{0}$};

\draw [fill] (2.5,1+0.5) circle [radius=0.05];
\node at (3,1.5) {$h^{1}$};

\draw [fill] (4.5,2+0.5) circle [radius=0.05];
\node at (5,2.5) {$h^{2}$};

\draw [fill] (6.5,3+0.5) circle [radius=0.05];
\node at (7,3.5) {$h^{3}$};
\end{tikzpicture}
\caption{$\Ext^{s,t}_{\A(1)_*}(\H_*(Y))$ in the range $0\leq t-s\leq 6$.}
\end{minipage}
\end{figure}
\noindent
See \cite{Rav86}, Theorem 3.1.27 for $(1)$ and $(4)$. The calculations for $(2)$ and $(3)$ are also elementary, so that we omit the detail.
\begin{Remark} We use the same notation $h^{i}$ for $i=0,1,2,3$ to denote certain generators of the above groups. This abuse of notation is justified by the fact that these generators have close relationships which are described in the next lemma. The context will clarify the use of the notation.
\end{Remark}
\noindent
Consider cell inclusions $V(0)\rightarrow Y$ and $S^{0}\cup_{2}e^{1}\cup_{\eta}e^{2}\rightarrow Y$. The induced homomorphisms in $\Ext$ over $\A(1)_*$ are described as follows.
\begin{Lemma} \phantomsection \label{Lem_inclu}
$(i)$The homomorphism $\Ext^{*,*}_{\A(1)_*}(\mathrm{H}_{*}(V(0)))\rightarrow \Ext^{*,*}_{\A(1)_*}(\mathrm{H}_{*}(Y))$ sends the classes $h^{0}$ and $h^{1}$ to the non-trivial classes of the same name. \\\\
\noindent
$(ii)$ The homomorphism $\Ext^{*,*}_{\A(1)_*}(\mathrm{H}_{*}(S^{0}\cup_{2}e^{1}\cup_{\eta}e^{2}))\rightarrow \Ext^{*,*}_{\A(1)_*}(\mathrm{H}_{*}(Y))$ sends the classes $h^{0}$, $h^{1}$, $h^{2}$, $h^{3}$ to the non-trivial classes of the same name.
\end{Lemma}
\begin{proof} For part $(i)$, consider the short exact sequence of $\A(1)_*$-comodules
$$0\rightarrow \mathrm{H}_{*}(V(0))\rightarrow \mathrm{H}_{*}(Y)\rightarrow \mathrm{H}_{*}(\Sigma^{2}V(0))\rightarrow 0$$
For degree reasons, the classes $h^{0}$ and $h^{1}$ of $\Ext^{*,*}_{\A(1)_*}(\mathrm{H}_{*}(V(0)))$ do not belong to the image of the connecting homomorphism  $$\Ext^{s-1,t}_{\A(1)_*}(\mathrm{H}_{*}(\Sigma^{2}V(0)))\rightarrow \Ext^{s,t}_{\A(1)_*}(\mathrm{H}_{*}(V(0))).$$ Therefore, they are sent to nontrivial classes of the same name in $\Ext^{*,*}_{\A(1)_*}(\mathrm{H}_{*}(Y))$.
For part $(ii)$, consider the short exact sequence of $\A(1)_*$-comodules
$$0\rightarrow \mathrm{H}_{*}(S^{0}\cup_{2}e^{1}\cup_{\eta}e^{2})\rightarrow \mathrm{H}_{*}(Y)\rightarrow \Sigma^{3}\F\rightarrow 0$$ and the resulting long exact sequence $$\Ext^{s-1,t}_{\A(1)_*}(\mathrm{H}_{*}(\Sigma^{3}\F)) \xrightarrow{\partial} \Ext^{s,t}_{\A(1)_*}(\mathrm{H}_{*}(S^{0}\cup_{2}e^{1}\cup_{\eta}e^{2})) \rightarrow  \Ext^{s,t}_{\A(1)_*}(\mathrm{H}_{*}(Y)). $$ For degree reasons, the classes $h^{0}$, $h^{2}$, $h^{3}$ of $\Ext^{s,t}_{\A(1)_*}(\mathrm{H}_{*}(S^{0}\cup_{2}e^{1}\cup_{\eta}e^{2}))$ are not in the image of the connecting homomorphism, and thus are sent to $h^{0}$, $h^{2}$, $h^{3}$ in $ \Ext^{*,*}_{\A(1)_*}(\mathrm{H}_{*}(Y)$, respectively. Next, for degree reasons, the classes $h_{0}h^{1}$ and $h_{1}a^{1}$ are sent to $0 \in \Ext^{*,*}_{\A(1)_*}( \mathrm{H}_{*}(Y))$. The only way for this to happen is that the connecting homomorphism sends $\Sigma^{3}1\in \Ext^{0,3}_{\A(1)_*}(\F, \mathrm{H}_{*}(\Sigma^{3}\F))$ to the sum $h^{1} + a^{1}$. It follows that $h^{1}$ is not in the image of the connecting homomorphism, and therefore is sent to $h^{1}\in \Ext^{1,3}_{\A(1)_*}( \mathrm{H}_{*}(Y))$
\end{proof}

\begin{Lemma}\phantomsection \label{Lem_Y} $\mathrm{H}_{*}(Y)$ has a structure of an $\A(1)_*$-comodule algebra. The resulting structure on $\Ext_{\A(1)_*}^{*,*}(\mathrm{H}_{*}(Y))$ is that of a polynomial algebra. 
\end{Lemma}
\begin{proof} It is not hard to see that $\mathrm{H}_{*}(Y)$ is isomorphic to $\A(1)_*\square_{E(1)_{*}}\F$ as $\A(1)_*$-comodules, where $E(1)_*$ is the Hopf quotient of $\A(1)_*$ by the Hopf ideal $(\zeta_1)$, i.e., $E(1)_*\cong \F[\zeta_2]/(\zeta_2^2)$. In particular, $\mathrm{H}_{*}(Y)$ has the structure of a $\A(1)_*$-comodule algebra. As a consequence, $\Ext_{\A(1)_*}^{*,*}( \mathrm{H}_{*}(Y))$ is an algebra and is furthermore isomorphic to $\Ext^{*,*}_{E(1)_{*}}(\F)$ by the change-of-rings isomorphism. It is well-known that the latter is a polynomial algebra on one variable.
\end{proof}
\noindent
%\textbf{Remark} This algebra structure does not come from a ring spectrum structure on $Y$. In fact, $\mathrm{H}_{*}(Y)$ is not a $\A_{*}$-comodule algebra. This remark is not relevant here. What is a good justification for this claim. $\mathrm{H}_{*}(Y)$ is definitely a $\A_{*}$-comodule algebra
We now compute $G_{\sigma} := \Ext_{\A(1)_*}^{*,*}(R_{\sigma}^{'})$. We denote by $\alpha_{s,t, \sigma}$ the non-trivial class of $ \Ext_{\A(1)_*}^{s,s+t}(R_{\sigma}^{'})$ whenever there is a unique such one.
\begin{Proposition}\phantomsection \label{Lem_G1}As a module over $G_{0}$, $G_{1} = \Ext_{\A(1)_*}^{*,*}(R^{'}_{1}) $ is generated by $\alpha_{0,4,1}\in \Ext_{\A(1)_*}^{0,4}(R^{'}_{1})$ and $\alpha_{1,8,1}\in \Ext_{\A(1)_*}^{1,9}(R^{'}_{1})$ with the relations $h_{1}\alpha_{0,4,1} = 0$ and $v\alpha_{0,4,1} = h_{0}^{2}\alpha_{1,8,1}$.
\end{Proposition}
\begin{proof} Consider the short exact sequence of $\A(1)_*$-comodules $$0\rightarrow\Sigma^{4}\F\rightarrow R_{1}^{'}\rightarrow \Sigma^{6}\mathrm{H}_{*}(V(0))\rightarrow 0.$$ The connecting homomorphism $$\partial: \Ext^{s,t-6}_{\A(1)_*}(V(0))\rightarrow \Ext^{s+1,t-4}_{\A(1)_*}(\F)$$ of the resulting long exact sequence sends $h^{0}$ to $h_{1}$ and $h^{1}$ to $0$. The latter follows from degree reasons and the former from the following map of short exact sequences of $\A(1)_*$-comodules and the naturality of the connecting homomorphism
$$\xymatrix{ 0\ar[r]&\Sigma^{4}\F\ar[r]&R_{1}\ar[r]&\Sigma^{6}\mathrm{H}_{*}(V(0))\ar[r]& 0\\
		    0\ar[r]&\Sigma^{4}\F\ar[r]\ar[u]&\mathrm{H}_{*}(\Sigma^{4}C_{\eta})\ar[r]\ar[u]&\Sigma^{6}\F\ar[r]\ar[u]& 0.	
}$$It follows that $G_{1}$ is $v_{1}^{4}$-periodic on the following generators (Figure \ref{G_1}) \\
\begin{figure}[h!]
\begin{center}
\begin{tikzpicture}[scale=0.5]
\clip(-1.5,-1.5) rectangle (7,6);
\draw[color=gray] (0,0) grid [step=1] (7,6);
\foreach \n in {4,5,...,11}
{
\def\nn{\n-0}
\node[below] at (\nn+0.5-4,0) {$\n$};
}
\foreach \s in {0,1,...,6}
{\def\ss{\s-0};
\node [left] at (-0.4,\ss+0.5,0){$\s$};
}
\foreach \s in {0,1,...,4}
\draw [fill,red] (0.5,\s+0.5) circle [radius=0.05];
\draw[->,red] (0.5,0.5)--(0.5,5.5);
%%%%%%%%%%%%%%%%%%%
\draw [fill] (4+0.5,1+0.5) circle [radius=0.05];
\draw [fill] (4+0.5,2+0.5) circle [radius=0.05];
\draw[-] (4+0.5,1+0.5)--(4+0.5,2+0.5);
%%%%%%%%%%%
\draw [fill,red] (4+0.5,3+0.5) circle [radius=0.05];
\draw [fill,red] (4+0.5,4+0.5) circle [radius=0.05];
\draw[->,red] (4.5,3.5)--(4.5,5.5);
%%%%%%
\draw [fill] (5.5,2.5) circle [radius=0.05];
\draw [fill] (6.5,3.5) circle [radius=0.05];
\draw[-] (4.5,1.5)--(5.5,2.5);
\draw[-] (5.5,2.5)--(6.5,3.5);
\draw[densely dashed] (4.5,2.5)--(4.5,3.5);
\end{tikzpicture}
\end{center}
\caption{$G_{1}$ - The red part is the contribution of $\Ext_{\A(1)_*}^{*,*}(\Sigma^4\F)$ and the black part from $\Ext_{\A(1)_*}^{*,*}(\Sigma^6\H_*(V(0)))$. }
\phantomsection \label{G_1}
\end{figure}
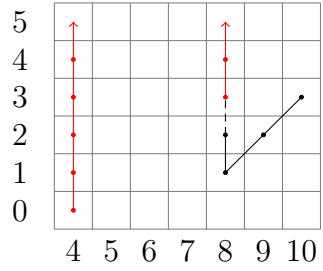

\noindent
What remains to be established is the multiplication by $h_{0}$ on the generator of bidegree $(s,t-s) = (2,8)$. This is done by a similar consideration of the connecting homomorphism associated to the short exact sequence of $\A(1)_*$-comodules 
$$0\rightarrow \Sigma^{4}C_{\eta}\rightarrow R_{1}^{'}\rightarrow \Sigma^{7}\F\rightarrow 0.$$ 

\end{proof}

\begin{Proposition}\phantomsection \label{Lem_G2}As a module over $G_{0}$, $\Ext_{\A(1)_*}^{*,*}(R^{'}_{2}) = G_{2}$ is generated by $\alpha_{s,t,2}\in \Ext^{s,s+t}$ where $(s,t)\in \{(0,8), (0,12), (1,14), (2,16), (3,18)\}$ with $$h_{1}\alpha_{s,t,2} = 0, v\alpha_{0,8,2} = h_{0}^{3}\alpha_{0,12,2}$$ $$v\alpha_{0,12,2} = h_{0}\alpha_{2,16,2}, v\alpha_{1,14,2} = h_{0}\alpha_{3,18,2}.$$
\end{Proposition}
\begin{proof}
The short exact sequence in part (a) of Lemma \ref{Lem_SES1} gives rise to the long exact sequence 
%%%%%%%%%%%%%%%%%
$$\rightarrow\Ext_{\A(1)_*}^{s,t-12}(\mathrm{H}^{*}(C_{\eta}))\rightarrow \Ext_{\A(1)_*}^{s,t}(R^{'}_{2}) \rightarrow \Ext_{\A(0)_*}^{s,t-8}(\F)\rightarrow \Ext_{\A(1)_*}^{s+1,t-12}(\mathrm{H}^{*}(C_{\eta}))\rightarrow$$ 
%%%%%%%%%%%%%%%%%%%%%%%%
Combining that $\Ext_{\A(0)_*}^{s,t}(\F)\cong\F[h_{0}]$ and the description of $\Ext_{\A(1)_*}^{s,t}(\mathrm{H}^{*}(C_{\eta}))$, we see that the connecting homomorphism is trivial for degree reasons.

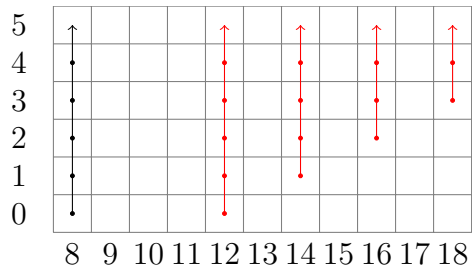
\begin{figure}[!h]
\begin{center}
\begin{tikzpicture}[scale=0.5]
\clip(-1.5,-1.5) rectangle (11,6);
\draw[color=gray] (0,0) grid [step=1] (11,6);
\foreach \n in {8,9,...,19}
{
\def\nn{\n-0}
\node[below] at (\nn+0.5-8,0) {$\n$};
}
\foreach \s in {0,1,...,6}
{\def\ss{\s-0};
\node [left] at (-0.4,\ss+0.5,0){$\s$};
}
\foreach \s in {0,1,...,4}
\draw [fill] (0.5,\s+0.5) circle [radius=0.05];
\draw[->] (0.5,0.5)--(0.5,5.5);
%%%%%%%%%%%%%%%%%%%
\foreach \t  in {0,1,..., 4}
     \draw [fill,red] (4.5,\t+0.5) circle [radius=0.05];
\foreach \t  in {0,1,..., 3}
      \draw[-,red] (4.5,\t+0.5)--(4.5,\t+1.5);
      \draw[->,red] (4.5,4.5)--(4.5,5.5); 
\foreach \t  in {1,2,..., 4}
     \draw [fill,red] (6.5,\t+0.5) circle [radius=0.05];
\foreach \t  in {1,2,..., 3}
      \draw[-,red] (6.5,\t+0.5)--(6.5,\t+1.5);
      \draw[->,red] (6.5,4.5)--(6.5,5.5); 
\foreach \t  in {2,3,..., 4}
     \draw [fill,red] (8.5,\t+0.5) circle [radius=0.05];
\foreach \t  in {2,3}
      \draw[-,red] (8.5,\t+0.5)--(8.5,\t+1.5);
      \draw[->,red] (8.5,4.5)--(8.5,5.5); 
\foreach \t  in {3,4}
     \draw [fill,red] (10.5,\t+0.5) circle [radius=0.05];
     \draw[-,red] (10.5,3.5)--(10.5,4.5);
      \draw[->,red] (10.5,4.5)--(10.5,5.5); 
\end{tikzpicture}
\end{center}
\caption{$G_{2}$ - The red part is the contribution of $\Ext_{\A(0)_*}^{s,t}(\F,\F)$ and the black one of $\Ext_{\A(1)_*}^{s,t}(\mathrm{H}_{*}(C_{\eta}))$}
\end{figure}
\noindent
What remains is to establish the $v_{1}^{4}$-multiplication on the class $\alpha_{0,8,2}$ of bidegree $(0,8)$. Consider the long exact sequence associated to the short exact sequence in part (b) of Lemma \ref{Lem_SES1}
\begin{equation} \label{connect-SES1}
\Ext_{\A(1)_*}^{s-1,t}(\Sigma^{12}V_{3})\xrightarrow{\partial} \Ext_{\A(1)_*}^{s,t}( \Sigma^{4}R_{1}^{'})\rightarrow \Ext_{\A(1)_*}^{s,t}(R_{2}^{'}).
\end{equation}
One can check that the class $\Sigma^{4}\alpha_{0,4,1} \in \Ext_{\A(1)_*}^{s,t}( \Sigma^{4}R_{1}^{'})$ is not in the image of $\partial$, and so is sent to $\alpha_{0,8,2}\in \Ext_{\A(1)_*}^{s,t}(R_{2}^{'})$. For degree reasons, we see that $v_{1}^{4}\Sigma^{4}\alpha_{0,4,1}$ is not in the image of $\partial$, thus $v_{1}^{4}\alpha_{0,8,2}$ is nontrivial in $G_{2}$. This completes the proof. 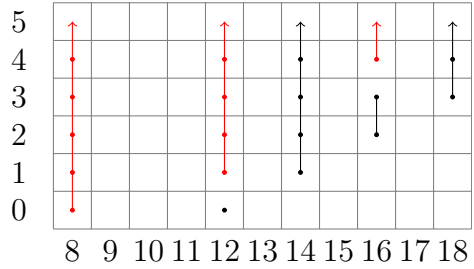
\begin{figure}[!h]
\begin{center}
\begin{tikzpicture}[scale=0.5]
\clip(-1.5,-1.5) rectangle (11,6);
\draw[color=gray] (0,0) grid [step=1] (11,6);
\foreach \n in {8,9,...,19}
{
\def\nn{\n-0}
\node[below] at (\nn+0.5-8,0) {$\n$};
}
\foreach \s in {0,1,...,6}
{\def\ss{\s-0};
\node [left] at (-0.4,\ss+0.5,0){$\s$};
}
\foreach \s in {0,1,...,4}
\draw [fill,red] (0.5,\s+0.5) circle [radius=0.05];
\draw[->,red] (0.5,0.5)--(0.5,5.5);
%%%%%%%%%%%%%%%%%%%
\foreach \t  in {1,2,..., 4}
     \draw [fill,red] (4.5,\t+0.5) circle [radius=0.05];
     \draw[->,red] (4.5,1.5)--(4.5,5.5); 
      \draw[fill] (4.5,0.5) circle [radius=0.05];
%%%%%%%%%%%%%%%%%%%%%%%%%
\foreach \t  in {1,2,..., 4}
     \draw [fill] (6.5,\t+0.5) circle [radius=0.05];
      \draw[->] (6.5,1.5)--(6.5,5.5); 
%%%%%%%%%%%%%%%%%%%%%%%%%%      
\foreach \t  in {2,3}
     \draw [fill] (8.5,\t+0.5) circle [radius=0.05];
\foreach \t  in {2}
      \draw[-] (8.5,\t+0.5)--(8.5,\t+1.5);
      \draw[->,red] (8.5,4.5)--(8.5,5.5); 
      \draw[fill,red] (8.5,4.5) circle [radius=0.05];
%%%%%%%%%%%%%%%%%%%%%%%%   
\foreach \t  in {3,4}
     \draw [fill] (10.5,\t+0.5) circle [radius=0.05];
      \draw[->] (10.5,3.5)--(10.5,5.5); 
\end{tikzpicture}
\end{center}
\caption{$G_{2}$ -The red part is the contribution of $G_{1}$ and the black one of $\Ext^{*,*}_{\A(1)_*}(V_{3})$.}
\phantomsection \label{G_2}
\end{figure}
\noindent
\end{proof}
\begin{Remark}
We can make a complete calculation of the connecting homomorphism of (\ref{connect-SES1}), which results to the chart Figure-\ref{G_2}.
\end{Remark}
\begin{Lemma}\phantomsection \label{Lem_G3}As a module over $G_{0}$, $\Ext_{\A(1)_*}^{*,*}(R^{'}_{3}) = G_{3}$ is generated by $\alpha_{s,t,3}$ of $\Ext^{s,s+t}$ where $(s,t)\in \{(0,12),(0,16), (0,18), (1,20), (2,22), (3,24)\}$ with $h_{1}\alpha_{s,t,3} = 0$, $v\alpha_{0,12,3} = h_{0}^{3}\alpha_{0,16,3}$, $v\alpha_{0,16,3} = h_{0}^{2}\alpha_{1,20,3}$, $v\alpha_{0,18,3} = h_{0}\alpha_{2,22,3}$, $v\alpha_{1,20,3} = h_{0}\alpha_{3,24,3}$.
\end{Lemma}
\begin{proof}The short exact sequence in Lemma \ref{Lem_SES2} gives the long exact sequence 
%%%%%%%%%%%%%%%%%%%%%%%
$$\rightarrow \Ext_{\A(1)_*}^{s,t}(\Sigma^{4}R^{'}_{2})\rightarrow \Ext_{\A(1)_*}^{s,t}(R^{'}_{3})\rightarrow \Ext_{\A(1)_*}^{s,t}(\Sigma^{18}V_{4})\rightarrow \Ext_{\A(1)_*}^{s+1,t}(\Sigma^{4}R^{'}_{2})\rightarrow$$
%%%%%%%%%%%%%%%%%%%%%%%
For degree reasons, the connecting homomorphism is trivial, hence we obtain the additive structure of $G_{3}$ as in Figure \ref{G_3}. We need to establish the non-trivial $h_{0}$-multiplication on the generators $\{\alpha_{s,18+2s,3}| \ s\geq0\}$. Taking the $v_{1}^{4}$-periodicity into account, we reduce to show this property for the generators of $$\alpha_{0,18,3}, \alpha_{1,20,3}, \alpha_{2,22,3}, \alpha_{3,24,3}.$$

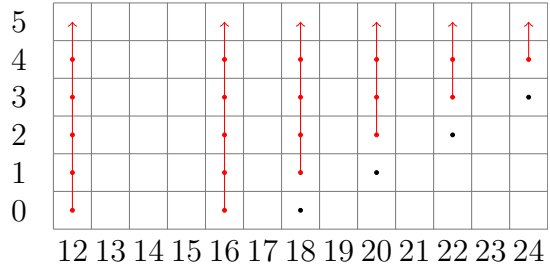
\begin{figure}[!h]
\begin{center}
\begin{tikzpicture}[scale=0.5]
\clip(-1.5,-1.5) rectangle (13,6);
\draw[color=gray] (0,0) grid [step=1] (13,6);
\foreach \n in {12,13,...,25}
{
\def\nn{\n-0}
\node[below] at (\nn+0.5-12,0) {$\n$};
}
\foreach \s in {0,1,...,6}
{\def\ss{\s-0};
\node [left] at (-0.4,\ss+0.5,0){$\s$};
}
\foreach \s in {0,1,...,4}
\draw [fill,red] (0.5,\s+0.5) circle [radius=0.05];
\draw[->,red] (0.5,0.5)--(0.5,5.5);
%%%%%%%%%%%%%%%%%%%
\foreach \t  in {0,1,..., 4}
     \draw [fill,red] (4.5,\t+0.5) circle [radius=0.05];
      \draw[->,red] (4.5,0.5)--(4.5,5.5); 
\foreach \t  in {1,2,..., 4}
     \draw [fill,red] (6.5,\t+0.5) circle [radius=0.05];
      \draw[->,red] (6.5,1.5)--(6.5,5.5); 
\foreach \t  in {2,3,..., 4}
     \draw [fill,red] (8.5,\t+0.5) circle [radius=0.05];
      \draw[->,red] (8.5,2.5)--(8.5,5.5); 
\foreach \t  in {3,4}
     \draw [fill,red] (10.5,\t+0.5) circle [radius=0.05];
      \draw[->,red] (10.5,3.5)--(10.5,5.5); 
\draw[fill,red] (12.5,4.5) circle [radius=0.05];      
\draw[->,red] (12.5,4.5)--(12.5,5.5);
\foreach \s in {0,1,...,3}
\draw[fill] (2*\s+6.5,\s+0.5) circle [radius=0.05];
\end{tikzpicture}
\end{center}
\caption{$G_{3}$ - The red part is the contribution of $G_{2}$ and the black one of $\Ext_{\A(1)_*}^{s,t}(V_4)$}
\phantomsection \label{G_3}
\end{figure}
\noindent
For this, we can check that there are the following short exact sequences:
$$0\rightarrow \Sigma ^{18} \mathrm{H}_{*}(C_{\eta}) \rightarrow R_{3} \rightarrow R_{3}/\Sigma ^{18} \mathrm{H}_{*}(C_{\eta})\rightarrow 0$$
and 
$$0\rightarrow \Sigma^{4} R_{2} \rightarrow R_{3}/\Sigma ^{18} \mathrm{H}_{*}(C_{\eta}) \rightarrow \Sigma^{19} \mathrm{H}_{*}(C_{\eta})\rightarrow 0$$
where, as a sub $\A(1)_*$-comodule of $R_{3}$, $\Sigma^{18}\mathrm{H}_{*}(C_{\eta})$ is equal to $\F\{y_{1}y_{3}^{2}+y_{2}^{3}, y_{2}y_{3}^{2}\}$ and the map $\Sigma^{4}R_{2}\rightarrow R_{3}/\Sigma ^{18}\mathrm{ C_{\eta}}$ is the composite $\Sigma^{4}R_{2}\xrightarrow{\times y_{1}} R_{3}\rightarrow R_{3}/\Sigma ^{18} \mathrm{H}_{*}(C_{\eta}) $.\\
\noindent
As a consequence, $\Ext_{\A(1)_*}^{*,*}( R_{3}/\Sigma ^{18} \mathrm{H}_{*}(C_{\eta}))$ sits in a long exact sequence
$$\Ext_{\A(1)_*}^{s-1,t}(\Sigma^{19}\mathrm{H}_{*}(C_{\eta}))\xrightarrow{\partial} \Ext_{\A(1)_*}^{s,t}(\Sigma^{4}R_{2})\rightarrow \Ext_{\A(1)_*}^{s,t}(\R_{3}/\Sigma ^{18} \mathrm{H}_{*}(C_{\eta}))\rightarrow.$$ 
Since $\partial$ is $G_{0}$-linear, one only needs to compute $\partial$ on the two generators of $\Ext^{0,19}_{\A(1)_*}(\Sigma^{19}\mathrm{H}_{*}(C_{\eta}))$ and $\Ext^{1,21}_{\A(1)_*}(\F,\Sigma^{19}\mathrm{H}_{*}(C_{\eta}))$. Direct computations show that $\partial$ act non-trivially on these classes. It follows that $\partial$ is a monomorphism and so $\Ext_{\A(1)_*}^{s,t}(\R_{3}/\Sigma ^{18} \mathrm{H}_{*}(C_{\eta}))$ is $v_{1}$-free on the generators depicted in Figure \ref{Ext_mid}.

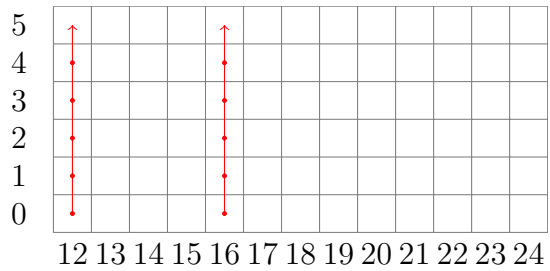
\begin{figure}[!h]
\begin{center}
\begin{tikzpicture}[scale=0.5]
\clip(-1.5,-1.5) rectangle (13,6);
\draw[color=gray] (0,0) grid [step=1] (13,6);
\foreach \n in {12,13,...,25}
{
\def\nn{\n-0}
\node[below] at (\nn+0.5-12,0) {$\n$};
}
\foreach \s in {0,1,...,6}
{\def\ss{\s-0};
\node [left] at (-0.4,\ss+0.5,0){$\s$};
}
\foreach \s in {0,1,...,4}
\draw [fill,red] (0.5,\s+0.5) circle [radius=0.05];
\draw[->,red] (0.5,0.5)--(0.5,5.5);
%%%%%%%%%%%%%%%%%%%
\foreach \t  in {0,1,..., 4}
     \draw [fill,red] (4.5,\t+0.5) circle [radius=0.05];
      \draw[->,red] (4.5,0.5)--(4.5,5.5); 

\end{tikzpicture}
\caption{$\Ext_{\A(1)_*}^{s,t}(\R_{3}/\Sigma ^{18} \mathrm{H}_{*}(C_{\eta})).$}
\phantomsection \label{Ext_mid}
\end{center}
\end{figure}
\noindent
It follows immediately from the exact sequence $$0\rightarrow \Sigma ^{18} \mathrm{H}_{*}(C_{\eta}) \rightarrow R_{3} \rightarrow R_{3}/\Sigma ^{18} \mathrm{H}_{*}(C_{\eta})\rightarrow 0$$ that $\Ext_{\A(1)_*}^{*,*}(R_{3})$ is as depicted in Figure \ref{G_3}. In particular, missing $h_{0}$-extensions are established. 

\begin{figure}[!h]
\begin{center}
\begin{tikzpicture}[scale=0.5]
\clip(-1.5,-1.5) rectangle (13,6);
\draw[color=gray] (0,0) grid [step=1] (13,6);
\foreach \n in {12,13,...,25}
{
\def\nn{\n-0}
\node[below] at (\nn+0.5-12,0) {$\n$};
}
\foreach \s in {0,1,...,6}
{\def\ss{\s-0};
\node [left] at (-0.4,\ss+0.5,0){$\s$};
}
\foreach \s in {0,1,...,4}
\draw [fill,red] (0.5,\s+0.5) circle [radius=0.05];
\draw[->,red] (0.5,0.5)--(0.5,5.5);
%%%%%%%%%%%%%%%%%%%
\foreach \t  in {0,1,..., 4}
     \draw [fill,red] (4.5,\t+0.5) circle [radius=0.05];
      \draw[->,red] (4.5,0.5)--(4.5,5.5); 
\foreach \t  in {0,1,..., 4}
     \draw [fill] (6.5,\t+0.5) circle [radius=0.05];
      \draw[->] (6.5,0.5)--(6.5,5.5); 
\foreach \t  in {1,2,..., 4}
     \draw [fill] (8.5,\t+0.5) circle [radius=0.05];
      \draw[->] (8.5,1.5)--(8.5,5.5); 
\foreach \t  in {2,3,4}
    \draw [fill] (10.5,\t+0.5) circle [radius=0.05];
      \draw[->] (10.5,2.5)--(10.5,5.5); 
\foreach \t  in {3,4}
     \draw [fill] (12.5,\t+0.5) circle [radius=0.05];
      \draw[->] (12.5,3.5)--(12.5,5.5); 

\end{tikzpicture}
\end{center}
\caption{$G_{3}$ -The red part is the contribution of $\Ext_{\A(1)_*}^{s,t}(\R_{3}/\Sigma ^{18} \mathrm{H}_{*}(C_{\eta}))$ and the black one of $\Ext_{\A(1)_*}^{s,t}(\Sigma^{18}\mathrm{H}_{*}(C_{\eta})).$}
\phantomsection \label{G_3}
\end{figure}
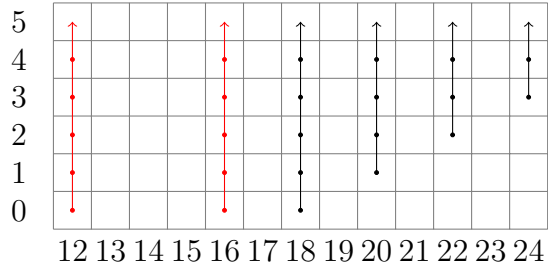
\end{proof}
\begin{Theorem}\phantomsection \label{Lem_G} As a module over $G_{0}$, we have
\begin{itemize}

\item[(a)] For every $\sigma\geq 2$, $\Ext_{\A(1)_*}^{*,*}(R^{'}_{\sigma}) = G_{\sigma}$ is generated by $\alpha_{s,t,\sigma}\in\Ext^{s,t+s}_{\A(1)_*}(R^{'}_{\sigma})$ where $(s,t)\in\{(0,4\sigma),(0,2j+4\sigma)|2\leq j\leq \sigma, (j,6\sigma+2j)|1\leq j\leq 3\}$ with $h_{1}\alpha_{s,t,\sigma} =0 $.
%\item[(b)] $\alpha_{0,4,1}\alpha_{s,2s+6t,\sigma}$ is a non-trivial element in $G_{\sigma +1}$
\item[(b) ]For all pairs of triples $(s_{1}, t_{1}, \sigma_{1})$ and $(s_{2}, t_{2},\sigma_{2})$ with $\sigma_{1}\geq 1$  and $\sigma_{2}\geq 1$ except for $(2,9,1)$ and $(3,10,1)$, we have that $$\alpha_{s_{1},t_{1},\sigma_{1}}\alpha_{s_{2},t_{2},\sigma_{2}}=\alpha_{s_{1}+s_{2},t_{1}+t_{2},\sigma_{1}+\sigma_{2}}.$$
\end{itemize}
\end{Theorem}  
\begin{figure}[!h]
\begin{center}
\begin{tikzpicture}[scale=0.5]
\clip(-1.5,-1.5) rectangle (17,6);
\draw[color=gray] (0,0) grid [step=1] (17,6);

\foreach \s in {0,1,...,6}
{\def\ss{\s-0};
\node [left] at (-0.4,\ss+0.5,0){$\s$};
}
\foreach \s in {0,1,...,4}
\draw [fill] (0.5,\s+0.5) circle [radius=0.05];
\draw[->] (0.5,0.5)--(0.5,5.5);
\node[below,scale=0.7] at (0.5,-0.5) {$4\sigma$};
%%%%%%%%%%%%%%%%%%%
\foreach \t  in {0,1,..., 4}
     \draw [fill] (4.5,\t+0.5) circle [radius=0.05];
      \draw[->] (4.5,0.5)--(4.5,5.5); 
      \node[below, scale=0.7] at (4.5,-0.5) {$4\sigma+4$};
\foreach \t  in {0,1,..., 4}
     \draw [fill] (6.5,\t+0.5) circle [radius=0.05];
      \draw[->] (6.5,0.5)--(6.5,5.5); 
      \node[below, scale=0.7] at (6.5,-0.5) {$4\sigma+6$};
\foreach \t  in {0,1,..., 4}
     \draw [fill] (10.5,\t+0.5) circle [radius=0.05];
      \draw[->] (10.5,0.5)--(10.5,5.5); 
      \node[below, scale=0.7] at (10.5,-0.5) {$6\sigma$};
\foreach \t  in {1,2,..., 4}
     \draw [fill] (12.5,\t+0.5) circle [radius=0.05];
      \draw[->] (12.5,1.5)--(12.5,5.5); 
\foreach \t  in {2,3,4}
    \draw [fill] (14.5,\t+0.5) circle [radius=0.05];
      \draw[->] (14.5,2.5)--(14.5,5.5); 
\foreach \t  in {3,4}
     \draw [fill] (16.5,\t+0.5) circle [radius=0.05];
      \draw[->] (16.5,3.5)--(16.5,5.5); 

\end{tikzpicture}
\end{center}
\caption{$G_{\sigma}$ for $\sigma \geq 2$}
\end{figure}
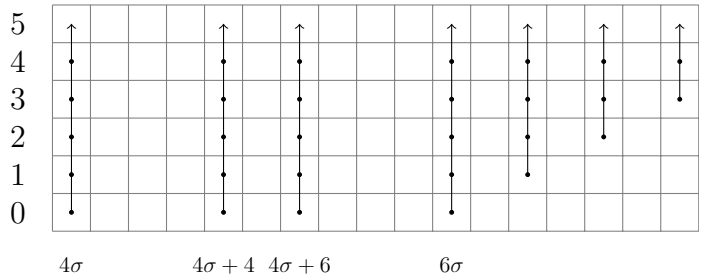

\begin{proof} $(a)$ The statement for $\sigma = 2$ is Lemma \ref{Lem_G2}. Let us prove the claim for $\sigma\geq 3$ by induction. The base case is Lemma \ref{Lem_G3}.\\\\ 
Suppose the claim is true for some $\sigma \geq 3$. The long exact sequence associated to the short exact sequence in Lemma \ref{Lem_SES2} reads
%%%%%%%%%%%%%%%%%%%%%%%%%
$$\rightarrow \Ext_{\A(1)_*}^{s,t}(R^{'}_{\sigma+1})\rightarrow \Ext_{\A(1)_*}^{s,t}(\Sigma^{6\sigma+6}V_{4})\rightarrow \Ext_{\A(1)_*}^{s+1,t}(\Sigma^{4}R^{'}_{\sigma})\rightarrow .$$
%%%%%%%%%%%%%%%%%
Combining the additive structure of $\Ext_{\A(1)_*}^{s,t}(\Sigma^{4}R^{'}_{\sigma})$ and that $$\Ext_{\A(1)_*}^{s,t}(\Sigma^{6\sigma+6}V_{4})\cong \Sigma^{6\sigma+6}\F[v_{1}],$$ we obtain the additive structure of $G_{\sigma+1}$ as described in the lemma because the connecting homomorphism vanishes for degree reasons. To establish the non-trivial $h_{0}$-multiplication on the generators $\{\alpha_{s,2s+6\sigma+6,\sigma+1}| \  s\geq 0  \}$, we use the following identities
\\
\begin{itemize}
\item[(i)] $G_{\sigma +1}\ni\alpha_{0,4,1}\alpha_{s,2s+6t,\sigma}\ne 0$ $\forall \sigma\geq 1$
\item[(ii)]$\alpha_{1,8,1}\alpha_{s,2s+6\sigma-6,\sigma-1} = \alpha_{s+1,2s+6\sigma+2,\sigma} $ $\forall \sigma\geq 2$
\item[(iii)]$\alpha_{0,12,2}\alpha_{s,2s+6\sigma-6,\sigma-1}=\alpha_{s,2s+6\sigma+6,\sigma+1}$ $\forall \sigma\geq 3$.
\end{itemize}
\noindent
\\
These identities are the content of part $(b)$. For the sake of the presentation, we postpone the proof to $(b)$; this is legitimate because, as we will see, the proof of $(b)$ only uses the additive structure of $G_{\sigma}'s$. Let us show how these identities allow us to conclude the proof of $(a)$. Indeed, the classes $\alpha_{s,2s+6\sigma-6,\sigma-1}$ exist (non-trivial) for all $\sigma\geq 3$ and $s\geq0$. Therefore, we have that, for all $\sigma\geq 3$,
\begin{align*}
h_{0}\alpha_{s,2s+6\sigma+6,\sigma+1}&= h_{0}\alpha_{0,12,2}\alpha_{s,2s+6\sigma-6,\sigma-1}\ (\mbox{multiplying both sides of (iii) by $h_{0}$})\\
							    & = \alpha_{0,4,1}\alpha_{1,8,1}\alpha_{s,2s+6\sigma-6,\sigma-1}\ (\mbox{because of (i)})\\
						            &=\alpha_{0,4,1}\alpha_{s+1,2s+2 +6\sigma,\sigma}\ (\mbox{because of (ii)}) \\
							    & \ne 0\  (\mbox{because of (i)}).
\end{align*}
\noindent
\\
%$(b)$ Since the generator $\alpha_{0,4,1}$ is represented by the primitive $y_{1}\in R_{1}$, multiplying by $\alpha_{0,4,1}$ is exactly the homomorphism $\Ext_{\A(1)_*}^{s,t}(\Sigma^{4}R_{t}^{'})\rightarrow \Ext_{\A(1)_*}^{s,t}(R^{'}_{t+1})$ induced by $\Sigma^{4}R^{'}_{t}\xrightarrow{.y_{1}}R^{'}_{t+1}$. This homomorphism being injective by the inductive step in the proof of $(a)$, the claim of $(b)$ follows.\\\\
\noindent
%%%%%%%%%%
$(b)$ For every $\sigma, \tau\geq 1$, there is a commutative diagram of $\A(1)_*$-comodules

$$\xymatrix{R_{\sigma}^{'}\ar[d]\otimes R_{\tau}^{'}\ar[r]^{\mu}\ar[d]&R_{\sigma+\tau}\ar[r]& R^{'}_{\sigma+\tau}\ar[d]\\
		\mathrm{H}_{*}(\Sigma^{6\sigma}X_{\sigma})\otimes \mathrm{H}_{*}(\Sigma^{6\tau}X_{\tau})\ar[d]\ar[rr]^{\mu}&&\mathrm{H}_{*}(\Sigma^{6\sigma+6\tau}X_{\sigma+\tau})\ar[d]\\
		\mathrm{H}_{*}(\Sigma^{6\sigma}Y)\otimes \mathrm{H}_{*}(\Sigma^{6\tau}Y)\ar[rr]^{\mu}&&\mathrm{H}_{*}(\Sigma^{6\sigma+6\tau}Y)
}.$$ Let us explain the maps in this diagram. The spectrum $X_{\sigma}$  is $V(0), S^{0}\cup_{2}e^{1}\cup_{\eta}e^{2}$ or $Y$ if $\sigma= 1, 2$ or $\sigma>2$ respectively; and in each case the map $R_{\sigma}^{'}\rightarrow \mathrm{H}_{*}(X_{\sigma})$ is the projection appearing in the proof of Lemma \ref{Lem_G1}, Lemma \ref{Lem_SES1} or Lemma \ref{Lem_SES2}, respectively. The other vertical arrows are the inclusions of $X_{\sigma}$ into Y. The bottom horizontal arrow is the multiplication on $\mathrm{H}_{*}(Y)$, described in Lemma \ref{Lem_Y}, and the middle one is induced by the latter. The second upper arrow is the projection on the factor $R_{\sigma+\tau}^{'}$ of the decomposition in Lemma \ref{Lem_Decom}. \\\\
The induced homomorphisms in $\Ext$ over $\A(1)_*$ of all vertical arrows are studied in the proof of Lemmas \ref{Lem_G1}, \ref{Lem_G2}, \ref{Lem_G} and Lemma \ref{Lem_inclu}, according to which the classes $\alpha_{s,t,\sigma}$, where $\sigma \geq 1$ and $(s,t,\sigma)\notin \{(2,9,1), (3,10,1)\}$, are sent non-trivially in a unique way to $\Ext_{\A(1)_*}^{s,t}(\mathrm{H}_{*}(Y))$, hence their products are non-trivial by Lemma \ref{Lem_Y}. This proves $(b)$.
\end{proof}
\noindent
\begin{Remark} \phantomsection \label{Alg_gen} Let us summarise what has been done so far. First, Lemma \ref{Lem_Decom} implies that
\begin{equation*} 
\Ext_{\A(1)_*}^{*,*}(R) \cong (\bigoplus\limits_{i\geq0} G_{i})\otimes \F[v_{2}^{4}]
\end{equation*}
where $v_{2}^{4} \in \Ext^{4,28}(\F,R_{4})$ represented by $y_{3}^{4}$. Next, Theorem \ref{Lem_G} describes completely the products between $G_{i}$'s modulo the ideal generated by $(v_{2}^{4})$. It is then straightforward to verify that $\Ext_{\A(1)_*}^{*,*}(R)$ is generated by the classes of 
\begin{equation}\phantomsection \label{alg_gen}
h_{0}, h_{1}, v, v_{1}^{4}, \alpha_{0,4,1}, \alpha_{1,8,1}, \alpha_{0,12,2}, \alpha_{1,14,2}, \alpha_{3,18,2}, \alpha_{0,18,3}, v_{2}^{4}.
\end{equation}
%To simplify the notation, we replace $\alpha_{0,4,1}, \alpha_{0,12,2},\alpha_{1,14,2}, \alpha_{0,18,3}$ by $ h_{2},\alpha_{0,12,2},\alpha, \alpha_{0,18,3}$, respectively.   \\\\
%%%%%%%%%%%%%%%%%%%%%%
\end{Remark}
\noindent
Let us describe the subalgebra of primitives. 
\begin{Corollary}There is the following isomorphism of graded algebras $$\Ext^{0,*}_{\A(1)_*}(R)\cong \F[\alpha_{0,4,1}, \alpha_{0,12,2},v_{2}^{4},\alpha_{0,18,3}]/(\alpha_{0,18,3}^{2}=\alpha_{0,12,2}^{3}+\alpha_{0,4,1}^{2}v_{2}^{4}).$$ 
\end{Corollary}
\begin{proof} The algebra $\Ext_{\A(1)_*}^{0,*}(\F,R)$ is naturally identified with a subalgebra of $R=\F[y_{1},y_{2},y_{3}]$. Through this identification, $\alpha_{0,4,1}, \alpha_{0,12,2},v_{2}^{4}, \alpha_{0,18,3} $ identify with $y_{1}, y_{2}^{2}, y_{3}^{4}, y_{2}^{3}+ y_{1}y_{3}^{2}$, respectively. Thus  $\F[\alpha_{0,4,1}, \alpha_{0,12,2},v_{2}^{4},\alpha_{0,18,3}]/(\alpha_{0,18,3}^{2}=\alpha_{0,12,2}^{3}+\alpha_{0,4,1}^{2}v_{2}^{4})$ is isomorphic to the subalgebra of $\Ext_{\A(1)_*}^{0,*}(\F,R)$ generated by $\alpha_{0,4,1}, \alpha_{0,12,2},v_{2}^{4}, \alpha_{0,18,3} $. On the other hand, it follows from Remark (\ref{Alg_gen}) that $\alpha_{0,4,1}, \alpha_{0,12,2},v_{2}^{4}, \alpha_{0,18,3} $ generate the whole subalgebra of primitives of $\Ext_{\A(1)_*}^{*,*}(R)$.
 % We prove that this identification is in fact an isomorphism by comparing their Poincar\'e series. \\\\
% The Poincar\'e series of $\F[\alpha_{0,4,1}, \alpha_{0,12,2},v_{2}^{4},\alpha_{0,18,3}]/(\alpha_{0,18,3}^{2}=\alpha_{0,12,2}^{3}+\alpha_{0,4,1}^{2}v_{2}^{4})$ is $$P(x,t) = \frac{1+x^{3}t^{18}}{(1-xt^{4})(1-x^{2}t^{12})(1-x^{4}t^{28})}$$
%Now, we compute the Poincar\'e series of $\Ext_{\A(1)_*}^{0,*}(\F,R)$. According to discussion following Lemma \ref{Lem_G}, it is reduced to compute the Poincar\'e series of $\bigoplus_{i\geq0}G_{i}$. Let denote by $Q_{i}(x,t)$ the Poincar\'e series of $G_{i}$. Then there is a recurrence formula
%\begin{alignat*}{1}
%Q_{0}(x,t)& = 1\\
%Q_{1}(x,t)& = xt^{4}\\
%Q_{i}(x,t)&= xt^{4}Q_{i-1}(x,t)+ x^{i}t^{6i}\ \forall\ i\geq 2
%\end{alignat*} 
%
%  
%   
%
%Then the Poincar\'e series of  $\Ext_{\A(1)_*}^{0,*}(\F,R)$ is $$Q(x,t)=(\sum_{\substack{i\geq0}} Q_{i}(x,t))\frac{1}{1-x^{4}t^{28}}$$
%Let $a_{m,n}$ and $b_{m,n}$ be coefficients of the monomial $x^{m}t^{n}$ in the power series $P$ and $Q$, respectively. Then, $$0\leq a_{m,n}\leq b_{m,n}\forall m, n \in \N$$ because $P\leq Q$. Therefore, it suffices to show that $$\sum_{n\geq 0}a_{m,n} = \sum_{n\geq 0}b_{m,n} \ \ \forall n\in \N$$ or equivalently $$P(x,1)=Q(x,1)$$
%By recurrence, we find that $$Q_{i}(x,1) = i x^{i}$$
%and so 
%$$\sum_{i\geq 0}Q_{i}(x,1) = 1 +\sum_{i\geq 1}i x^{i} = \frac{x^{2}-x+1}{(x-1)^{2}}$$ 
%Hence, $$Q(x,1) = \frac{x^{2}-x+1}{(x-1)^{2}(1-x^{4})}$$
%Finally, $$P(x,1) = \frac{1+x^{3}}{(1-x)(1-x^{2})(1-x^{4})}=\frac{x^{2}-x+1}{(x-1)^{2}(1-x^{4})}$$
This concludes the proof of the lemma.
\end{proof}
%\begin{Lemma}\marginpar{Find a representative for $\alpha$ in $\A(2)_{*}\otimes E\otimes A_1$} $$\Ext^{1,15}_{\A(1)_*}(\F,R^{2})\cong \F\{\alpha\}$$ Plus, $\alpha$ is represented by $\xi_{2}\otimes y_{1}^{2}+\xi_{1}^{3}\otimes y_{1}^{2}+ \xi_{1}\otimes y_{2}^{2}\in \A(1)_*\otimes R^{2}$.
%\end{Lemma}\marginpar{give a proof}
%\begin{proof}
%\end{proof}
\noindent
\textbf{The differentials $d_{1}$.} %\marginpar{Choice 1: Give a demonstration of how to compute $d_{1}$ on some elements and refer the rest to Rognes }
 %\marginpar{Choice 2: Reproduce the whole $E_{1}$-term of DMSS and study all $d_{1}$ } 
Since the DMSS for $\F$ is a spectral sequence of algebras, all $d_{1}$-differentials can be determined on the set of algebra generators of (\ref{alg_gen}). 
\begin{Proposition}\phantomsection \label{d1-S} The $d_{1}$-differential is multiplicative and on generators, it is given as follows:
\begin{itemize}
\item[1)] $d_{1}(h_{0}) = 0$
\item[2)] $d_{1}(h_{1}) = 0$
\item[3)] $d_{1}(\alpha_{0,4,1}) =0$
\item[4)] $d_{1}(\alpha_{1,14,2})=0$
\item[5)] $d_{1}(\alpha_{0,18,3}) = 0$
\item[6)] $d_{1}(v_{1}^{4}) =0$
\item[7)] $d_{1}(\alpha_{0,12,2}) = \alpha_{0,4,1}^{3} $
\item[8)] $d_{1}(\alpha_{1,8,1}) = h_{0}\alpha_{0,4,1}^{2}$
\item[9)] $d_{1}(v)= h_{0}^{3}\alpha_{0,4,1}$
\item[10)]  $d_{1}(\alpha_{3,18,2})=h_{0}^{3}\alpha_{0,18,3}$
\item[11)] $d_{1}(v_{2}^{4}) = \alpha_{0,4,1}\alpha_{0,12,2}^{2}.$
\end{itemize}
\end{Proposition}
\begin{proof} 
\begin{itemize}
1), 2), 4) For degree reasons, there is no room for a non-trivial $d_{1}$-differential on $h_{0}, h_{1}, \alpha_{1,14,2}$\\\\
3) It is easy to see that $\Ext_{\A(2)_{*}}^{1,4}(\F,\F)$ is non-trivial and that $\alpha_{0,4,1}$ is the only class in the $\mathrm{E}_{1}$-term that can contribute to it. Therefore $\alpha_{0,4,1}$ is a permanent cycle.\\\\
5) We see that $h_{0}\alpha_{0,18,3} = \alpha_{0,4,1}\alpha_{1,14,2}$. By the Leibniz rule, $h_{0}d_{1}(\alpha_{0,18,3}) =0$. As $h_{0}$ acts injectively on $G_{3}$, it follows that $d_{1}(\alpha_{0,18,3}) = 0$. \\\\
6) Since $h_{0}^{2}v_{1}^{4} = v^{2}$, $h_{0}^2d_{1}(v_{1}^{4}) = 2vd_1(v) = 0$. This follows because $d_{1}(v_{1}^{4})$ takes values in $\Ext_{\A(1)_*}^{4,8}(\F, R_{1}^{'})$ on which $h_{0}$ acts injectively.\\\\
7) We have that $\alpha_{0,12,2}$ is represented by the $\A(2)$-primitive $[1|y_{2}^{2}]+[x_{1}|y_{1}^{2}] \in E\otimes R_{2}$. By Remark \ref{SSd_1},  $d_{1}(\alpha_{0,12,2})$ is represented by $d([1|y_{2}^{2}]+[x_{1}|y_{1}^{2}] )=[1|y_{1}^{3}]\in E\otimes R_{3}$, hence is equal to $\alpha_{0,4,1}^{3}$. \\\\
8) Because $\alpha_{0,4,1}\alpha_{1,8,1} = h_{0}\alpha_{0,12,2}$, the Leibniz rule implies that $$\alpha_{0,4,1}d_{1}(\alpha_{1,8,1}) = h_{0}d_{1}(\alpha_{0,12,2}) = h_{0}\alpha_{0,4,1}^{3}.$$
 That $\alpha_{0,4,1}$ acts injectively on the $\mathrm{E}_{1}$-term implies that $d_{1}(\alpha_{1,8,1}) = h_{0}\alpha_{0,4,1}^{2}$\.\\\
9) The relation $\alpha_{0,4,1}v = h_{0}^{2}\alpha_{1,8,1}$ implies that $$\alpha_{0,4,1}d_{1}(v) = h_{0}^{2}d_{1}(\alpha_{1,8,1})=h_{0}^{3}\alpha_{0,4,1}^{2}$$ As $\alpha_{0,4,1}$ acts injectively on the $\mathrm{E}_{1}$-term, we obtain that $d_{1}(v) = h_{0}^{3}\alpha_{0,4,1}$. \\\\
10) The relation $v\alpha_{1,14,2} = h_{0}\alpha_{3,18,2}$ shows that $$h_{0}d_{1}(\alpha_{3,18,2}) = \alpha_{1,14,2}d_{1}(v)=\alpha_{1,14,2}h_{0}^{3}\alpha_{0,4,1} = h_{0}^{4}\alpha_{0,18,3}$$ Therefore, $d_{1}(\alpha_{3,18,2}) = h_{0}^{3}\alpha_{0,18,3}$.\\\\
11) We check that $v_{2}^{4}$ is represented by the $\A(2)$-primitive $[1|y_{3}^{4}]+[x_{1}|y_{2}^{4}]$ in $E\otimes R_{4}$.  By Remark \ref{SSd_1}, $d_{1}(v_{2}^{4})$ is represented by $[1|y_{1}y_{2}^{4}]$, hence is equal to $\alpha_{0,4,1}\alpha_{0,12,2}^{2}$.
\end{itemize}
 
\end{proof}
\noindent
%We omit the proof of this proposition because it is not the main concern of this note. We refer interested readers to [theorem 5.6, D-M] and [Rognes] for a proof. \marginpar{need to give precise reference}
It turns out that the DMSS collapses at the $\mathrm{E}_{2}$-term because there is no room for higher differentials. In particular, the classes $\alpha_{1,14,2}, \alpha_{0,4,1}, \alpha_{0,12,2}^{2}, v_{2}^{8}, \alpha_{0,18,3}$
 survive the spectral sequence converging to elements of $\Ext^{*,*}_{\A(2)_{*}}(\F,\F)$ in appropriate bidegrees. Following \cite{DFHH14}, those elements are denoted by $\alpha, h_{2}, g, w_{2}, \beta$, respectively.  Furthermore, $h_{2}, g, w_{2}, \beta$ generate a subalgebra of $\Ext^{*,*}_{\A(2)_{*}}(\F,\F)$ isomorphic to  $\F[h_{2}, g, w_{2}, \beta]/(h_{2}^{3}, h_{2}g, \beta^{4}-g^{3})$. The relation $\beta^{4}= g^{3}$ is a consequence of a $d_{1}$-differential. In effect, the relation $\alpha_{0,18,3}^{2}=\alpha_{0,12,2}^{3}+\alpha_{0,4,1}^{2}v_{2}^{4} $ implies the relation $ \beta^{4}-g^{3}-h_{2}^{4}w_{2}=0$ in $\Ext_{\A(2)_{*}}^{*,*}(\F)$. But $\alpha_{0,4,1}^{4}v_{2}^{8}$ gets hit by the differential $$d_{1}(v_{2}^{8}\alpha_{0,4,1}\alpha_{0,12,2})=v_{2}^{8}\alpha_{0,4,1}d_{1}(\alpha_{0,12,2})=v_{2}^{8}\alpha_{0,4,1}^{4}.$$ Thus the relation $\beta^{4}=g^{3}+h_{2}^{4}w_{2}$ becomes $\beta^{4}=g^{3}$.

%There are the following $d_{2}$-differential in the $E_{2}$-term of the Adams spectral sequence that will be important in the sequel:
\subsection{The Davis-Mahowald spectral sequence for $A_1$}\phantomsection \label{DMSS_A_1}
\noindent
\textbf{The $\A(2)_{*}$-comodule structure of $A_1$.} In \cite{DM81}, Davis and Mahowald constructed four finite spectra, whose mod $2$ cohomology are isomorphic to a free module of rank one over the subalgebra $\A(1) = \langle Sq^{1}, Sq^{2}\rangle$ of the Steenrod algebra $\A$. Let us review the construction of these spectra and their module structure over the subalgebra $\A(2) = \langle Sq^{1}, Sq^{2}, Sq^{4}\rangle$ of $\A$. Recall that  $Y$ is  $V(0)\wedge C_{\eta}$. The $\A$-module structure of $\mathrm{H}^{*}(Y)$ is depicted in Figure \ref{Diagram_Y}.
%\marginpar{need a better presentation}
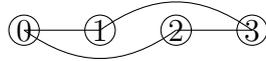
\begin{figure}[h!]
\begin{center}
\begin{tikzpicture}

   \draw[] (0,0) circle [radius=0.2];
   \node at (0,0) {$0$};
   \draw (2,0) circle [radius=0.2];
   \node at (2,0) {$2$};
   \draw (1,0) circle [radius=0.2];
   \node at (1,0) {$1$};
   \draw (3,0) circle [radius=0.2];
   \node at (3,0) {$3$};
   \draw[-] (0,0)--(1,0);
   \draw[-] (2,0)--(3,0);
   \draw (0,0) .. controls (0.7,-0.5) and (1.3,-0.5) .. (2,0);
   \draw (1,0) .. controls (1.7,0.5) and (2.3,0.5) .. (3,0);
\end{tikzpicture}
\end{center}
\caption{Diagram of $\mathrm{H}^{*}(Y)$: the straight lines represent $Sq^{1}$ and the curved lines represent $Sq^{2}$, the numbers represent the degree of the cell.}
\phantomsection \label{Diagram_Y}
\end{figure}
\noindent
An element of $\Ext_{\A(1)}^{1,3}(\mathrm{H}^{*}(Y),\mathrm{H}^{*}(Y))$ can be represented by an $\A(1)$-module $M$ sitting in a short exact sequence of $\A(1)$-modules $$0\rightarrow \mathrm{H}^{*}(\Sigma^{3}Y)\rightarrow M\rightarrow \mathrm{H}^{*}(Y)\rightarrow 0.$$ It can be checked that $M$ must be isomorphic either to $\mathrm{H}^{*}(\Sigma^{3}Y)\oplus \mathrm{H}^{*}(Y)$ or to $\A(1)$ as an $\A(1)$-module. This means that 
\begin{equation}
\Ext_{\A(1)}^{1,3}(\mathrm{H}^{*}(Y),\mathrm{H}^{*}(Y))\cong \F.
\end{equation}
The $\A(1)$-module structure of $\A(1)$ is depicted in Figure \ref{Diagram_A_1}.
\begin{figure}[h!]
\begin{center}
\begin{tikzpicture}
\draw (0,0) circle [radius=0.2];
\node at (0,0) {$0$}; 
\draw (1,0) circle [radius=0.2];
\node at (1,0) {$1$}; 
\draw (2,-1) circle [radius=0.2];
\node at (2,-1) {$2$}; 
\draw (3,-1) circle [radius=0.2];
\node at (3,-1) {$3$}; 
\draw  (3,-0.5) circle [radius=0.2];
\node at (3,-0.5) {$3$}; 
\draw  (4,-0.5) circle [radius=0.2];
\node at (4,-0.5) {$4$}; 
\draw (6,-1.5)circle [radius=0.2];
\node at (6,-1.5) {$6$}; 
\draw (5,-1.5) circle [radius=0.2];
\node at (5,-1.5) {$5$}; 
%%%%%%%%%%%%%% Sq^{1}
\draw[-] (0,0)--(1,0);
\draw[-] (2,-1)--(3,-1);
\draw[-] (3,-0.5)--(4,-0.5);
\draw[-] (5,-1.5)--(6,-1.5);
%%%%%%%%%%%%%%% Sq^{2}
\draw (0,0) .. controls (0.7,-0.7) and (1,-1) ..(2,-1);
\draw (1,0) .. controls (1.7,0.4) and (2.5,0) ..(3,-0.5);
\draw (4,-0.5) .. controls (4.7,-0.5) and (5.5,-1) ..(6,-1.5);
\draw (3,-1) .. controls (3.7,-1.7) and (4.5,-1.7) ..(5,-1.5);
\draw (2,-1) .. controls (2.6,-0.7) and (3.4,-0.8) ..(4,-0.5);
\end{tikzpicture}
\caption{ Diagram of $\A(1)$.}
\phantomsection \label{Diagram_A_1}
\end{center}
\end{figure}
\noindent
One can ask whether $\A(1)$ admits a structure of $\A(2)$-module. If such a structure exists, then according to the Adem relations $Sq^{2}Sq^{1}Sq^{2} = Sq^{4}Sq^{1}+ Sq^{1}Sq^{4}$, there must be a nontrivial action of $Sq^{4}$ on the nontrivial class of degree $1$. It is straightforward to verify that the latter is the only constraint to put an $\A(2)$-module structure on $\A(1)$. There are also possibilities for $Sq^{4}$ to act nontrivially on the classes of degree $0$ and $2$. These give in total four different $\A(2)$-module structures on $A_1$. In other words, the inclusion of Hopf algebras $\A(1)\hookrightarrow \A(2)$ induces a surjective homomorphism 
$$\Ext_{\A(2)}^{1,3}(\mathrm{H}^{*}(Y),\mathrm{H}^{*}(Y))\rightarrow \Ext_{\A(1)}^{1,3}(\mathrm{H}^{*}(Y),\mathrm{H}^{*}(Y))$$ whose kernel contains $4$ element. Therefore, 
$$\Ext_{\A(2)}^{1,3}(\mathrm{H}^{*}(Y),\mathrm{H}^{*}(Y))\cong \F^{\oplus 3}.$$
Next, one observes that restriction along $\A(2)\subset \A$ induces an isomorphism
$$\Ext_{\A}^{1,3}(\mathrm{H}^{*}(Y),\mathrm{H}^{*}(Y))\cong \Ext_{\A(2)}^{1,3}(\mathrm{H}^{*}(Y),\mathrm{H}^{*}(Y)),$$ 
because for any $\A$-module $M$ sitting in a short exact sequence 
$$0\rightarrow \mathrm{H}^{*}(\Sigma^{3}Y)\rightarrow M\rightarrow \mathrm{H}^{*}(Y)\rightarrow 0$$ 
there can not be any non-trivial $Sq^{k}$ for $k\geq 8$ on $M$. It is proved in \cite{DM81} that the four classes of $\Ext_{\A}^{1,3}(\mathrm{H}^{*}(Y),\mathrm{H}^{*}(Y))$ that are sent to the unique non-trivial class of $\Ext_{\A(1)}^{1,3}(\mathrm{H}^{*}(Y),\mathrm{H}^{*}(Y))$ are permanent cycles in the Adams spectral sequence and converge to four $v_{1}$- self-maps of $Y$, i.e., the maps $\Sigma^{2}Y\rightarrow Y$ inducing isomorphisms in $K(1)$-homology theory. As a consequence, the cofibers of these $v_{1}$-self-maps realise the four different $\A$-module structures on $\A(1)$. We will write $A_1$ to refer to any of these four finite spectra. 
Following \cite{BEM17}, we make the following definition.
\begin{Definition}\label{version A1} We define by $A_{1}[i,j],\ i,j\in\{0,1\}  $ the version of $A_1$ having the non-trivial $Sq^{4}$ on the generator of degree $0$, respectively $2$ if and only if $i=1$, respectively $j=1$.
\end{Definition} 
As a $\F$-vector spaces, 
\begin{equation}\phantomsection \label{basis-A_1}\mathrm{H}_{*}(A_{1}[ij])\cong \F\{a_{0}, a_{1}, a_{2}, a_{3}, \overline{a_{3}}, a_{4}, a_{5}, a_{6}\},
\end{equation} where $a_{0}, a_{1}, a_{2}, a_{4}, a_{5}, a_{6}$ are duals to the generators of degree $0, 1, 2, 4, 5, 6$ of $\mathrm{H}^{*}(A_{1}[ij])$, respectively and $a_{3}, \overline{a_{3}}$ are duals to the images of the generator of degree 0 by $Sq^{3}, Sq^{3}+Sq^{2}Sq^{1}$, respectively. By taking duals to the action of $\A(2)$ on $\H^*(A_{1}[ij])$, we obtain
\begin{Proposition} \phantomsection \label{comod_A_1}The left coaction of $\A(2)_{*}$ on $\H_*(A_{1}[ij])$ is given by

 $\Delta(a_{1}) = [1| a_{1}]+[\xi_{1}|a_{0}]$
 
 $\Delta(a_{2})=[1| a_{2}]+[\xi_{1}^{2}| a_{2}]$
 
$\Delta(a_{3})=[1| a_{3}]+[\xi_{1}|a_{2}]+[\xi_{1}^{2}| a_{1}]+[\xi_{1}^{3}| a_{0}]$

 $\Delta(\overline{a_{3}})=[1|\overline{a_{3}}]+[\xi_{1}^{2}| a_{1}]+[\xi_{2}| a_{0}]$
 
$\Delta(a_{4})=[1| a_{4}]+[\xi_{1}|\overline{a_{3}}]+[\xi_{1}^{2}|a_{2}]+[\xi_{1}^{3}|a_{1}]+[\xi_{2}|a_{1}]+[\xi_{2}\xi_{1}|a_{0}] + \alpha_{i,j}[\xi_{1}^{4}| a_{0}]$

$\Delta(a_{5})=[1| a_{5}]+[\xi_{1}^{2}|\overline{a_{3}}]+[\xi_{1}^{2}|a_{3}]+[\xi_{2}| a_{2}]+[\xi_{1}^{4}|a_{1}]+[\xi_{2}\xi_{1}^{2}|a_{0}]$

$\Delta(a_{6})=[1| a_{6}]+[\xi_{1}|a_{5}]+[\xi_{1}^{2}| a_{4}]+[\xi_{1}^{3}|\overline{a_{3}}]+[\xi_{1}^{3}| a_{3}]+[\xi_{2}| a_{3}]+[\xi_{2}\xi_{1}| a_{2}]+\beta_{i,j}[\xi_{1}^{4}| a_{2}]+[\xi_{2}\xi_{1}^{2}| a_{1}]+[\xi_{1}^{5}| a_{1}]+\gamma_{i,j}[\xi_{1}^{6}| a_{0}] + [\xi_{2}\xi_{1}^{3}| a_{0}]+\lambda_{i,j}[\xi_{2}^{2}| a_{0}],$
 where 

 \[ \alpha_{i,j} = \begin{cases} 0 & \mbox{if } (i,j) \in \{(0,0), (0,1)\} \\ 
 				              1 & \mbox{if } (i,j)\in \{(1,0), (1,1)\} \end{cases} \]
				          
 \[ \beta_{i,j} = \begin{cases} 0 & \mbox{if } (i,j) \in \{(0,0), (1,0)\} \\ 
 				              1 & \mbox{if } (i,j)\in \{(0,1), (1,1)\} \end{cases} \]	
$$\gamma_{i,j} = 1+\alpha_{i,j}$$ and $$\lambda_{i,j} = \alpha_{i,j}+\beta_{i,j}$$			              
\end{Proposition}
\begin{proof} The proof is a straightforward translation from $\A(2)$-module structure to $\A(2)_{*}$-comodule structure using the formula of the duals of the Milnor basis in \cite{Mil58}.
\end{proof}
\noindent
\textbf{DMSS for $A_1$.} In what follows, we will apply in many places the shearing homomorphism to find primitives representing certain cohomology classes, see \cite{ABP69}, Theorem 3.1. In general, let $C$ be a Hopf algebra with conjugation $\chi$ and $B$ be Hopf-algebra quotient of $C$. Given a $C$-comodule $M$, consider the composite  
$$C\otimes M\xrightarrow{id\otimes \Delta} C\otimes C\otimes M\xrightarrow{id\otimes \chi\otimes id}C\otimes C\otimes M\xrightarrow{\mu\otimes id} C\otimes M. $$
When restricting to $C\square_{B}M$, this composite factors through $(C\square_{B}k)\otimes M$ inducing the shearing isomorphism of $C$-comodules $$Sh :C\square_{B}M\rightarrow (C\square_{B}k)\otimes M, $$ where $C$ coacts on $C\square_{B}M$ via the left factor and on $(C\square_{B}k)\otimes M$ diagonally. Combined with the change-of-rings isomorphism, we have the following isomorphisms: 
$$\Ext_{B}^{*}(k,M)\cong \Ext_{C}^{*}(k,C\square_{B}M)\cong \Ext_{C}^{*}(k,(C\square_{B}k)\otimes M).$$
In particular, via these isomorphisms, a class $x\in\Ext_{B}^{0}(k,M)$ is sent to $Sh(1\otimes x)$.
\begin{Proposition}\phantomsection \label{Prop_A_1}The $\mathrm{E}_{1}$-term of the Davis-Mahowald spectral sequence converging to $\mathrm{Ext}^{s,t}_{\A(2)_{*}}(\H_*(A_1))$ is given by
$$
\mathrm{E}_{1}^{s,\sigma,*}\cong\left\{ \begin{array}{ll}
 						0 & \mbox{if} \ s > 0\\
						R_{\sigma} & \mbox{if} \ s=0.
\end{array} \right.$$
As a module over $\mathbb{F}_{2}[\alpha_{0,4,1},\alpha_{0,12,2},v_{2}^{4}] (\subset \Ext_{\A(1)_*}^{*,*}(R))  $, $\mathrm{E}_{1}^{*,*,*}$ is free module of rank eight on the following generators of 
\begin{equation}\phantomsection \label{mod_gen_A1} 1, y_{3}, y_{3}^{2}, y_{3}^{3}, y_{2}, y_{2}y_{3}, y_{2}y_{3}^{2}, y_{2}y_{3}^{3}.
\end{equation}
\end{Proposition}
\begin{proof} 
In effect, $\mathrm{E}_{1}^{s,\sigma,t}$ is equal to $\mathrm{Ext}_{\A(1)_*}^{s,t}( R^{\sigma}\otimes \H_*(A_1))$ by definition. The coaction of $\A(1)_*$ on $R^{\sigma}\otimes \H_*(A_1)$ is the usual diagonal coaction on tensor products. In addition, $\H_*(A_1)$ is isomorphic to $\A(1)_*$ as $\A(1)_*$-comodules. By the change-of-rings isomorphism, we obtain that
\begin{equation}
\phantomsection \label{Ext(A_1)}
\mathrm{Ext}_{\A(1)_*}^{s,t}(R^{\sigma}\otimes \H_*(A_1))\cong \mathrm{Ext}^{s,t}_{\mathbb{F}_{2}}(R^{\sigma})\cong R^{\sigma}.
\end{equation}
\noindent
The first part of the proposition follows.\\\\
For the second part, the action of $\Ext_{\A(1)_*}^{s,t}(R)$ on $\mathrm{E}_{1}^{s,t,\sigma}$
$$\xymatrix{
\mathrm{Ext}_{\A(1)_*}^{s,t}(R)\otimes\mathrm{Ext}_{\A(1)_*}^{s',t'}(R\otimes \H_*(A_1))\ar[r]&\mathrm{Ext}_{\A(1)_*}^{s+s',t+t'}(R\otimes \H_*(A_1))
}$$ is induced by the multiplication on $R$: $$R\otimes (R\otimes \H_*(A_1))\rightarrow R\otimes \H_*(A_1).$$
Now let $r\in \Ext_{\A(1)_*}^{0,*}(R)\subset R$ and $s\in R\cong \Ext_{\A(1)_*}^{0,*}(R\otimes \H_*(A_1))$. By applying the shearing isomorphism, the class $s$ is represented by a unique element of the form $s\otimes a_{0} + \sum s_{i}\otimes a_{i}\in R\otimes \H_*(A_1)$ where the $a_{i}$ are in positive degrees. The action of $r$ on $s$ is then represented by $rs\otimes a_{0} + \sum rs_{i}\otimes a_{i}$ which represents $rs\in R\cong \Ext_{\A(1)_*}^{0,*}(R\otimes A_1)$ via \eqref{Ext(A_1)}. In other words, the action of $\Ext_{\A(1)_*}^{0,*}(R)$ on $\Ext_{\A(1)_*}^{0,*}(R\otimes \H_*(A_1))$ is given by the multiplication of the polynomial algebra $R$. The proof follows from the fact that $\F[\alpha_{0,4,1},\alpha_{0,12,2},v_{2}^{4}]$ is identified with the subalgebra of $R$ generated by $y_{1}, y_{2}^{2}, y_{3}^{4}$.
\end{proof}
\noindent
Let us analyse the differentials in this spectral  sequence. As the $d_{r}$-differentials decrease $s$-filtration by $r-1$, i.e., $d_{r}: \mathrm{E}_{r}^{s,\sigma,t}\rightarrow \mathrm{E}_{r}^{s-r+1,\sigma+r,t }$ and $\mathrm{E}_{1}^{s,\sigma,t} =0$ if $s>0$, the spectral sequence collapses at the $\mathrm{E}_{2}$-term and there are no extension problems. Therefore, 
$$\mathrm{E}_{2}^{0,t,\sigma}\cong \mathrm{Ext}_{\A(2)_{*}}^{\sigma,t}(\H_*(A_1)).$$
%%%%%%
We now turn our attention to the $d_{1}$-differentials. As all elements of the $\mathrm{E}_{1}$-term are in $\Ext_{\A(1)_*}^{0,*}(R\otimes \H_*(A_1))$, we can apply the remark after Proposition \ref{Prop_SS}. We have determined the $d_{1}$-differential on the classes $\alpha_{0,4,1}, \alpha_{0,12,2}, v_{2}^{4}$ in Proposition \ref{d1-S}.
%\begin{Lemma}[D-M] In the Danis-Mahowald spectral sequence converging to $\mathrm{Ext}_{A_{2}}(\mathbb{F}_{2},\mathbb{F}_{2})$, there are following $d_{1}$-differentials:%$d_{1}(y_{0}) = 0$, $d_{1}(y_{1}^{2})= y_{0}^{3}$ and $d_{1}(y_{2}^{4}) = y_{0}y_{1}^{4}$ 
%\end{Lemma}
By the Leibniz rule, it remains to determine the $d_{1}$-differential on the classes of (\ref{mod_gen_A1}).
\begin{Proposition} \phantomsection \label{d1-A_1}There are the following $d_{1}$-differentials

\begin{itemize}
\item[1)]$d_{1}(1)=0,$
\item[2)]$d_{1}(y_{2})=0,$
\item[3)]$d_{1}(y_{3}) =0,$
\item[4)]$d_{1}(y_{2}y_{3}) =0,$
\item[5)] $ d_{1}(y_{2}y_{3}^{2})=0,$
\item[6)] $d_{1}(y_{2}y_{3}^{3}) =0,$
\item[7)] $d_{1}(y_{3}^{2})=\alpha_{0,4,1}^{2}y_{2},$
\item[8)]$d_{1}(y_{3}^{3}) = \alpha_{0,4,1}^{2}y_{2}y_{3}.$ 
\end{itemize}
\end{Proposition}
\begin{proof} Parts $1 - 4$ follow from the sparseness of the $\mathrm{E}_{1}$-term. \\\\
5) The only nontrivial $d_{1}$-differential that $y_{2}y_{3}^{2}$ can support is $$d_{1}(y_{2}y_{3}^{2}) = \alpha_{0,4,1}^{2}\alpha_{0,12,2} 1.$$ However, $$d_{1}(\alpha_{0,4,1}^{2}\alpha_{0,12,2})=  \alpha_{0,4,1}^{2}d_{1}(\alpha_{0,12,2})=\alpha_{0,4,1}^{5} 1 \ne 0.$$ This means that $\alpha_{0,4,1}^{2}\alpha_{0,12,2}$ is not a $d_{1}$-cycle, and so cannot be hit by a $d_{1}$-differential. Therefore, $y_{2}y_{3}^{2}$ is a $d_{1}$-cycle.\\\\
%%%%%%%
6) Similarly, a nontrivial $d_{1}$-differential on $y_{2}y_{3}^{3}$ would be $$d_{1}(y_{2}y_{3}^{3})=\alpha_{0,4,1}^{2}\alpha_{0,12,2}y_{3}.$$ However, $$d_{1}(\alpha_{0,4,1}^{2}\alpha_{0,12,2}y_{3}) =  \alpha_{0,4,1}^{5}y_{3} \ne 0$$ by the Leibniz rule. Thus, $y_{2}y_{3}^{3}$ is a $d_{1}$-cycle.\\\\
%%%%%%%%%%%%%
7-8) It suffices to prove that $\nu^{2}y_{2} =0$ and $\nu^{2}y_{2}y_{3}=0$ in $\Ext^{*,*}_{\A(2)_{*}}(\H_*(A_1))$ because the differentials in part 7) and 8) are the only possibilities for the latter to occur. We will proceed using juggling formulas for Massey products, see \cite{Rav86}, Section 4 of Appendix A1. In effect, the classes $1$ and $y_{3}$ being permanent cycles by part 1) and part 3), they converge to classes in $ \Ext^{0,0}_{\A(2)_{*}}(\H_*(A_1))$ and $ \Ext^{1,6}_{\A(2)_{*}}(\H_*(A_1))$, respectively. By sparseness even at the level of the $\mathrm{E}_{1}$-term of the DMSS, $\eta 1  = \eta y_{3} = 0$. Hence the Massey product $\langle \nu, \eta, y_{3}^{i}\rangle$ with $i\in \{0,1\}$ can be formed. We have that  $$\nu^{2}y_{3}^{i} = \langle \eta, \nu, \eta\rangle y_{3}^{i} = \eta \langle \nu,\eta, y_{3}^{i}\rangle.$$ By sparseness of the DMSS, $\alpha_{0,4,1}^{2}y_{3}^{i}$ survives the DMSS and so $\nu^{2}y_{3}^{i}\ne 0$. It follows that $\langle \nu,\eta, y_{3}^{i}\rangle$ is nontrivial and must be equal to $y_{2}y_{3}^{i}$. The fact that $\nu^{3} = 0 \in \Ext^{3,12}_{\A(2)_{*}}(\F)$ allows us to do the following juggling
$$\nu^{2}y_{2}y_{3}^{i}= \nu^{2}\langle \nu,\eta, y_{3}^{i}\rangle = \langle \nu^{2},\nu, \eta\rangle y_{3}^{i}.$$ However, the Massey product $\langle \nu^{2},\nu,\eta\rangle$ lives in the group $\Ext^{3,14}_{\A(2)_{*}}(\F)$ which vanishes by Theorem \ref{Lem_G}. This concludes the proof of parts 7) and 8).
%%%%%%%%%%
\end{proof}
\noindent
\textbf{$\mathrm{E}_{2}$-term of the Adams SS.} We describe $\Ext^{*,*}_{\A(2)_{*}}(\H_*(A_1))$ as a module over $$\mathbb{F}_{2}[h_{2},g, v_{2}^{8}]/(h_{2}^{3},h_{2}g) \subset \mathrm{Ext}_{\A(2)_{*}}^{*,*}(\mathbb{F}_{2}).$$ We recall that $g$ is represented by $\alpha_{0,12,2}^2$ in the DMSS for $\F$. We will denote by $e[s,t]$ where $s,t \in \N$ the unique non-trivial class belonging to $\Ext^{s,s+t}_{\A(2)_{*}}(\H_*(A_1))$.
\begin{Proposition}\phantomsection \label{Prop_AdamsA_1} As a module over $\mathbb{F}_{2}[h_{2},g, v_{2}^{8}]/(h_{2}^{3},h_{2}g)$,  $\mathrm{Ext}_{\A(2)_{*}}^{*,*}(\H_*(A_1))$ is a direct sum of cyclic modules generated by the following elements  \\\\
\begin{tabular}{| c | c | c | c |}
\hline
e[0,0]&e[1,5]&e[1,6]&e[2,11]\\
\hline
1&$y_{2}$&$y_{3}$&$y_{2}y_{3}$\\
\hline
(0)&$(h_{2}^{2})$&(0)&$(h_{2}^{2})$\\
\hline
\end{tabular}
\\\\
\begin{tabular}{| c | c | c | c |}
\hline
e[3,15]&e[3,17]&e[4,21]&e[4,23]\\
\hline
$y_{2}^{3}+y_{1}y_{3}^{2}$&$y_{2}y_{3}^{2}$&$y_{1}y_{3}^{3}+y_{2}^{3}y_{3}$&$y_{2}y_{3}^{3}$\\
\hline
$(h_{2}^{2})$&(0)&$(h_{2}^{2})$&(0)\\
\hline
\end{tabular}\\\\
\begin{tabular}{| c | c | c | c |}
\hline
e[6,30]&e[6,32]&e[7,36]&e[7,38]\\
\hline
$y_{2}^{6}+y_{1}^{2}y_{3}^{4}$&$y_{2}^{4}y_{3}^{2}+y_{1}y_{2}y_{3}^{4}$&$y_{2}^{6}y_{3}+ y_{1}^{2}y_{3}^{5}$&$y_{2}^{4}y_{3}^{3}+y_{1}y_{2}y_{3}^{5}$\\
\hline
$(h_{2})$&$(h_{2})$&$(h_{2})$&$(h_{2})$\\
\hline
\end{tabular}\\\\
\begin{tabular}{| c | c | c | c |}
\hline
e[8,42]&e[9,47]&e[9,48]&e[10,53]\\
\hline
$y_{2}^{6}y_{3}^{2}+y_{1}^{2}y_{3}^{6}+y_{1}y_{2}^{3}y_{3}^{4}$&$y_{2}^{7}y_{3}^{2}+y_{1}^{2}y_{2}y_{3}^{6}$&$y_{2}^{6}y_{3}^{3}+y_{1}^{2}y_{3}^{7}+y_{1}y_{2}^{3}y_{3}^{5}$&$y_{2}^{7}y_{3}^{3}+y_{1}^{2}y_{2}y_{3}^{7}$\\	
\hline
$(h_{2})$&$(h_{2})$&$(h_{2})$&$(h_{2})$\\
\hline
\end{tabular}\\\\
The second row in the table indicates a representative in the DMSS and the third row the annihilator ideal of the corresponding generator.
\end{Proposition}
\begin{proof} As a corollary of Proposition \ref{Prop_A_1}, the $\mathrm{E}_{1}$-term of the DMSS for $\H_*(A_1)$ is isomorphic to a free module of rank $32$ over $\F[h_{2},g, v_{2}^{8}]$. In particular, these $32$ generators are $h_{2}$-free. It turns out that one can choose these $32$-generators in such a way that there are exactly $16$ $h_{2}$-free towers that truncate $16$ others by $d_{1}$-differentials. The question is how one can identify these $16$ $d_{1}$-cycles. For this, we compute the $d_{1}$-differentials on the following $32$ generators of the $\mathrm{E}_{1}$-term: $\{y_{2}^{i}y_{3}^{j}|0\leq i\leq 3, 0\leq j\leq 7\}$. Some of them are $d_{1}$-cycles, for example $y_{2}, y_{3}$. Whereas, some of them are not $d_{1}$-cycle at first, but become so after adding a multiple of $h_{2}$, for example $\alpha_{0,12,2}y_{2}+h_{2}y_{3}^{2} = y_{2}^{3}+y_{1}y_{3}^{2}$. This procedure is straightforward but lengthy, so we omit details here. It can be checked that the generators listed in the table are $d_{1}$-cycles. Finally, since $g$ and $v_{2}^{8}$ are $d_{1}$-cycles, Proposition \ref{Prop_AdamsA_1} follows.  
\end{proof} 
 \subsection {Two products}
 \noindent
Now we turn our attention to the product between $\alpha \in \Ext^{3,15}_{\A(2)_{*}}(\F)$ and $e[4,23]\in\Ext^{4,27}_{\A(2)_{*}}(\H_*(A_1))$. This product is not detected in the DMSS because $\alpha$ has $\sigma$-filtration $1$ in the DMSS whereas all non-trivial groups in the $\E_\infty$-term of the DMSS converging to $\Ext_{\A(2)_{*}}^{*,*}( \H_*(A_1))$ are in $\sigma$-filtration $0$. Therefore, we need first to find a representative of $\alpha$ in the total cochain complex of the double complex $\A(2)_{*}^{\otimes *}\otimes E_{2}\otimes R$ and that of $e[4,23]$ in $\A(2)_{*}^{\otimes *}\otimes E_{2}\otimes R\otimes \H_*(A_1)$, then take the product at the level of cochain complexes and finally check if this product is a coboundary. It is tedious to carry out this procedure because any representative of $e[4,23]$ contains many terms, and so it is not easy to check if the product is a coboundary. Here, by a term of $\A(2)_{*}^{\otimes *}\otimes E_{2}\otimes R_{*}$ and $\A(2)_{*}^{\otimes *}\otimes E_{2}\otimes R_{*}\otimes \H_*(A_1)$, we mean an element of the basis formed by the tensor products of a basis of $\A(2)_{*}$, $E_{2}$, $R_{*}$ and $\H_*(A_1)$ chosen to be the monomial basis and the basis of (\ref{basis-A_1}), respectively. We will keep the same convention when working with $B(2)_{*}$, $F_{2}$, $S_{*}$ instead of $\A(2)_{*}$, $E_{2}$, $R_{*}$. The following two lemmas simplify computations.
\begin{Lemma} The product of $\alpha$ and $e[4,23]$ is equal either to $0$ or to $g e[3,15].$ 
\end{Lemma}
\begin{proof} This is trivial because $ge[3,15]$ is the only non-trivial class in the appropriate bidegree.
\end{proof}
\noindent
We recall from Section \ref{DMSS_Const} that there is a map of pairs $(\A(2)_{*}, E_{2}) \rightarrow (B(2)_{*}, F_{2})$ given by 
$$\A(2)_{*} = \F[\zeta_{1}, \zeta_2, \zeta_{3}]/(\zeta_{1}^{8}, \zeta_2^{4},\zeta_{3}^{2})\rightarrow B(2)_{*} = \F[\zeta_{1}, \zeta_2, \zeta_{3}]/(\zeta_{1}^{4}, \zeta_2^{4},\zeta_{3}^{2})$$
$$\zeta_{i}\mapsto \zeta_{i} \ \ i\in\{1,2,3\}$$
$$E_{2} = E(x_{1}, x_{2}, x_{3})\rightarrow F_{2} = E(x_{2},x_{3})$$
$$x_{1}\mapsto 0, x_{2}\mapsto x_{2}, x_{2}\mapsto x_{2}. $$
The induced map on their Koszul duals is $$R=\F[y_{1},y_{2},y_{3}]\rightarrow S=\F[y_{2},y_{3}]$$
$$y_{1}\mapsto 0, y_{2}\mapsto y_{2}, y_{3}\mapsto y_{3}.$$
By an abuse of notation, we will denote by $p$ these projection maps. The context will make it clear which map is referred to.
\begin{Lemma}\phantomsection \label{Lem_prod3} The map $p_{*} = \Ext^{7,42}_{\A(2)_{*}}(\H_*(A_1))\rightarrow \Ext^{7,42}_{B(2)_{*}}(\H_*(A_1))$ induced by the projection $\A(2)_{*}\rightarrow B(2)_{*}$ sends $ge[3,15]$ to a non-trivial element. 
\end{Lemma} 
\begin{proof} The projection $\A(2)_{*}\rightarrow B(2)_{*}$ induces a morphism of the DMSSs. The morphism of the $\mathrm{E}_{1}$-terms reads 

$$\Ext_{\A(2)_{*}}^{s,t}(E_{2}\otimes R\otimes \H_*(A_1))\rightarrow \Ext_{B(2)_{*}}^{s,t}(F_{2}\otimes S\otimes \H_*(A_1)).$$
By the change-of-rings isomorphism, this morphism identifies with the projection $p : R\rightarrow S$, which is surjective. The class $ge[3,15]$ is detected by $y_{2}^{4}(y_{2}^{3}+ y_{1}y_{3}^{2})\in R^{7}$, which maps to $y_{2}^{7}\in S^{7}$ via $p$. By naturality, $y_{2}^{7}$ is a permanent cycle in the target DMSS. The only class in the $\mathrm{E}_{1}$-term which can support a differential hitting $y_{2}^{7}$ is $y_{3}^{6}$. $y_{3}^{6}$ admits $v_{2}^{4}y_{3}^{2}$ as a lift in the source DMSS. We have that 
$$d_{1}(v_{2}^{4}y_{3}^{2}) = d_{1}(v_{2}^{4})y_{3}^{2} +v_{2}^{4}d_{1}(y_{3}^{2}) = (\alpha_{0,4,1}\alpha_{0,12,2}^{2}) y_{3}^{2} + v_{2}^{4}(\alpha_{0,4,1}y_{2}) = y_{1}y_{2}^{4}y_{3}^{2} +y_{3}^{4}y_{1}y_{2}.$$ 
This uses the Leibniz rule, Proposition \ref{d1-S} part 11), Proposition \ref{d1-A_1} part 7). By naturality, the $d_{1}$-differential in the target DMSS is equal to $p(y_{1}y_{2}^{4}y_{3}^{2} +y_{3}^{4}y_{1}y_{2}),$ which is equal to $0$. Therefore, the image of $ge[3,15]$ is non-trivial.

\end{proof}
\begin{Lemma} The product of $\alpha$ and $e[4,23]$ is non-trivial, hence equal to $ge[3,15]$ if and only if the product of $p_{*}(\alpha)$ and $p_{*}(e[4,23])$ is non-trivial.
\end{Lemma}
\begin{proof} The map $p : \A(2)_{*}\rightarrow B(2)_{*}$ induces the commutative diagram

$$\xymatrix{ \Ext_{\A(2)_{*}}^{3,15}(\F)\otimes\Ext_{\A(2)_{*}}^{4,27}(\H_*(A_1))\ar[d]^{p_{*}}\ar[r] & \Ext_{\A(2)_{*}}^{7,42}(\H_*(A_1))\ar[d]^{p_{*}}\\
                    \Ext_{B(2)_{*}}^{3,15}(\F)\otimes\Ext_{B(2)_{*}}^{4,27}(\H_*(A_1))\ar[r] & \Ext_{B(2)_{*}}^{7,42}(\H_*(A_1)),
 }$$
 where the horizontal maps are the respective multiplications. The result follows from the fact that $p_{*}(ge[3,15])$ is non-trivial by Lemma \ref{Lem_prod3}.  

\end{proof}
\noindent
Now let us compute the product of $p_{*}(\alpha)$ and $p_{*}(e[4,23])$.
\begin{Lemma} In the total cochain complexes of $B(2)_{*}^{\otimes *}\otimes F_{2}\otimes S$ and of $B(2)_{*}^{\otimes *}\otimes F_{2}\otimes S\otimes \H_*(A_1)$, respectively :
\begin{itemize} 
\item[i)] $p_{*}(\alpha)$ is represented by $[\xi_{2}|1|y_{2}^{2}]+ [\xi_{1}^{3}|1|y_{2}^{2}]+[\xi_{1}|1|y_{3}^{2}] \in B(2)\otimes F_{2}\otimes S^{2}$;
\item[ii)] $p_{*}(e[4,23])$ is represented by $[1|y_{2}y_{3}^{3}|a_{0}]+[1|y_{2}^{2}y_{3}^{2}|a_{1}]+[1|y_{2}^{3}y_{3}|a_{2}]+[1|y_{2}^{4}|a_{3}]\in F_{2}\otimes S^{4}\otimes A_1$.
\end{itemize}
\end{Lemma}
\begin{proof} A direct computation shows that these elements are cocycles of the total differentials, which are not coboundaries. One way to prove that they represent the right classes is to prove that they lift to cocycles in the total cochain complexes of $\A(2)_{*}^{\otimes *}\otimes E_{2}\otimes R$ and of $\A(2)_{*}^{\otimes *}\otimes E_{2}\otimes R\otimes \H_*(A_1)$, respectively. \\\\
\noindent
It is easy to check that $[\xi_{2}|1|y_{2}^{2}] + [\xi_{1}^{3}|1|y_{2}^{2}]+[\xi_{1}|1|y_{3}^{2}]+ [\xi_{2}|x_{1}|y_{1}^{2}]+[\xi_{1}^{3}|x_{1}|y_{1}^{2}]+[\xi_{1}|x_{2}|y_{1}^{2}]+[1|1|y_{1}^{2}y_{3}]\in (\A(2)_{*}\otimes E_{2}\otimes R^{2})\oplus (E_{2}\otimes R^{3})$ is a lift for $[\xi_{2}|1|y_{1}^{2}]+ [\xi_{1}^{3}|1|y_{1}^{2}]+[\xi_{1}|1|y_{2}^{2}]$.\\\\
\noindent
For the other element, instead of finding a lift it suffices to show that $p_{*}$ induces an isomorphism $\Ext_{\A(2)_{*}}^{4,27}(\H_*(A_1))\xrightarrow{\cong}\Ext_{B(2)_{*}}^{4,27}(\H_*(A_1))$, so that both are isomorphic to $\F$. This can be proved by a similar argument to that used in the proof of Lemma \ref{Lem_prod3}. In effect, the non-trivial class of $\Ext_{\A(2)_{*}}^{4,27}(\H_*(A_1))$ is detected by $y_{2}y_{3}^{3}$ in the DMSS. Via $p_{*}$, the latter is sent to $y_{2}y_{3}^{3}$ which is the unique non-trivial element of the $E_{1}$-term of the target DMSS in the appropriate tridegree. For degree reasons, $y_{2}y_{3}^{3}$ is not hit by any differential. Therefore, $y_{2}y_{3}^{3}$ survives the target DMSS and it follows that  $\Ext_{\A(2)_{*}}^{4,27}(\H_*(A_1))\xrightarrow{\cong}\Ext_{B(2)_{*}}^{4,27}(\H_*(A_1))\cong \F$.

%Instead of doing that we present a more elegant proof. In fact, it suffices to proof that the maps $p_{*}$ are isomorphisms $\Ext_{\A(2)_{*}}^{3,12}(\F,\F)\xrightarrow{\cong}\Ext_{B(2)_{*}}^{3,12}(\F,\F)$ and $\Ext_{\A(2)_{*}}^{4,23}(\F,A_1)\xrightarrow{\cong}\Ext_{B(2)_{*}}^{4,23}(\F,A_1)$ and so that they are all isomorphic to $\F$. The latter isomorphism can be proved by using the naturality of the DMSS exactly the same way as in the proof of the lemma \ref{Lemma 4.0.7}. For the former, observe that there is an isomorphism of $\A(2)_{*}-$ comodules \linebreak $\A(2)_{*}\square_{B(2)_{*}}\F\cong \mathrm{H}_{*}(C_{\nu})$. The elementary argument of homological algebra shows that $p_{*}$ is the the composite $$\Ext_{\A(2)_{*}}^{s,t}(\F)\rightarrow \Ext_{\A(2)_{*}}(\F,\A(2)_{*}\square_{B(2)_{*}}\F)\xrightarrow{\cong}\Ext_{B(2)_{*}}^{s,t}(\F,\F)$$ where the first map is induces by inclusion to the bottom generator $\F\rightarrow \A(2)\square_{B(2)_{*}}\F$ and the second map is the usual change-of-ring isomorphism. It remains therefore to prove that $\Ext_{\A(2)_{*}}^{3,12}(\F,\F)\xrightarrow{\cong} \Ext_{\A(2)_{*}}^{3,12}(\F,\mathrm{H}_{*}(C_{\nu}))$ 
\end{proof}
\noindent
Set $M = [\xi_{2}|1|y_{2}^{2}]+ [\xi_{1}^{3}|1|y_{2}^{2}]+[\xi_{1}|1|y_{3}^{2}]$ and $N=[1|y_{2}y_{3}^{3}|a_{0}]+[1|y_{2}^{2}y_{3}^{2}|a_{1}]+[1|y_{2}^{3}y_{3}|a_{2}]+[1|y_{2}^{4}|a_{3}]$. We need to show that $MN$, which is a $(d_{v}+d_{h})$-cocycle, represents a non-trivial class in $\Ext_{B(2)_{*}}^{7,42}(\H_*(A_1))$. We see that $MN$ is an element in $B(2)_{*}\otimes F_{2}\otimes S^{6}\otimes A_1$ and $d_{v}(MN) = 0$. This means that $MN$ represents a class in $\Ext_{B(2)_{*}}^{1,42}(F_{2}\otimes S^{6}\otimes \H_*(A_1))$. However, the latter group is trivial because by the change-of-rings theorem, $\Ext^{*,*}_{B(2)_{*}}(\F, F_{2}\otimes S\otimes \H_*(A_1))$ is isomorphic to $S$ which is concentrated only in cohomological degree $0$. There must be an element $P\in F_{2}\otimes S^{6}\otimes \H_*(A_1)$ such that $d_{v}(P)=MN$, and so $d_{h}(P)$ represents the same class in $\Ext_{B(2)_{*}}^{7,42}(\H_*(A_1))$ as $MN$ does. \\\\
 We recall the values of $\lambda_{i,j}$ as introduced in Proposition \ref{comod_A_1}: $\lambda_{1,0} = \lambda_{0,1} =1$ and $\lambda_{0,0} = \lambda_{1,1} =0$. 
\begin{Lemma} \phantomsection \label{LemP} $P$ contains $\lambda_{i,j} [1|x_{2}|y_{2}^{6}|a_{0}]$ as a term.
\end{Lemma}
\begin{proof}
The product MN contains the term $[\xi_{2}|1| y_{2}^{6}|a_{3}]$. One can check that $P$ must contain the term $[1|y_{2}^{6}|a_{6}]$, so that $d_{v}(P)$ contains the term $[\xi_{2}|1| y_{2}^{6}|a_{3}]$. Using the formula for the coaction of $\A(2)_{*}$ on $a_{6}$, one sees that $d_{v}(P)$ contains the term $\lambda_{i,j} [\xi_{2}^{2}|1| y_{2}^{6}|a_{0}]$ which is not a term of $MN$.  In order to compensate this term, $P$ must contain the term $\lambda_{i,j} [1|x_{2}|y_{2}^{6}|a_{0}].$
\end{proof}

\begin{Lemma} \phantomsection \label{LemY}A ($d_{h}+d_{v}$)-cycle in $F_{2}\otimes S^{7}\otimes A_1$ gives rise to a non-trivial class in $\Ext^{7,42}_{B(2)_{*}}(\H_*(A_1))$ if and only if it contains the term $[1|y_{2}^{7}|a_{0}]$.
\end{Lemma}
\begin{proof} It is shown in the proof of Lemma \ref{Lem_prod3} that $$\Ext_{B(2)_{*}}^{7,42}(\H_*(A_1))\cong\F$$ and that this group arises from $$\Ext_{B(2)}^{0,42}(F_{2}\otimes S^{7}\otimes \H_*(A_1))\cong \F\{y_{2}^{7}\}\subset S^{7}.$$ Therefore, by the shearing homomorphism, the only element in $F_{2}\otimes S^{7}\otimes \H_*(A_1)$ that represents the non-trivial class of $\Ext^{7,42}_{B(2)_{*}}(\H_*(A_1))$ must contain the term $[1|y_{2}^{7}|a_{0}]$.  
\end{proof}

\begin{Proposition} \label {exo_prod}The product $\alpha e[4,23]$ is equal to $\lambda_{i,j} ge[3,15]$.
 \end{Proposition}
\begin{proof} $\alpha e[4,23]$ is non-trivial if and only if $d_{h}(P)$ represents a non-trivial class in $\Ext_{B(2)_{*}}^{7,42}(\H_*(A_1))$. Lemma
 \ref{LemP} shows that $d_{h}(P)$ contains the term $\lambda_{i,j} [1|y_{1}^{7}|a_{0}]$. Hence, lemme \ref{LemY} concludes the proof.
\end{proof}
\noindent
The product between $\beta\in\Ext^{3,18}_{\A(2)_{*}}(\F)$ and $e[3,15]\in \Ext^{3,18}_{\A(2)_{*}}(\H_*(A_1))$ is easier because both have $\sigma$-filtration $0$ in the Davis-Mahowald spectral sequence.
\begin{Proposition}\phantomsection \label{prod} $\beta e[3,15] = e[6,30].$
\end{Proposition}
\begin{proof}
the class $\beta$ is represented by $y_{2}^{3}+y_{1}y_{3}^{2}$ in $R^{3}$ and $e[3,15]$ is represented by $[y_{2}^{3}+y_{1}y_{3}^{2}|a_{0}]$ in $R^{3}\otimes A_1$. So the product $\beta e[3,15]$ is represented by $[y_{2}^{6}+y_{1}^{2}y_{3}^{4}|a_{0}]$, which represents $e[6,30]$ by Proposition \ref{Prop_AdamsA_1}. 
\end{proof}
\section{Partial study of the Adams spectral sequence for $tmf\wedge A_1$} \label{G24 3}
\noindent
In this section, we establish some differentials as well as some structures of the ASS for $A_1$. These are essential bits of information allowing us to run the homotopy fixed point spectral sequence in the next section. \\\\
\noindent
Recall that the ASS for $tmf\wedge A_1$ which has $\mathrm{E}_{2}$-term isomorphic to $\Ext_{\A(2)_{*}}^{*,*}(\H_*(A_1))$ is a spectral sequence of modules over that for $tmf$, whose $\mathrm{E}_{2}$-term is isomorphic to $\Ext_{\A(2)_{*}}^{*,*}(\F)$. We first recollect some known properties of the ASS for $tmf$, see \cite{DFHH14}, Chapter 13. 
\begin{Theorem} \phantomsection \label{tmf'sdiff}
\begin{itemize}
\item[(i)] The class $g\in \Ext_{\A(2)_{*}}^{4,24}(\F)$ is a permanent cycle detecting the image of $\overline{\kappa}\in \pi_{20}(S^{0})$ via the Hurewicz map $S^{0}\rightarrow tmf$. 
\item[ii)] There is the following $d_{2}$-differential in the Adams spectral sequence for $tmf$ $$d_{2}(w_{2}) = g\beta\alpha. $$
\item[(iii)] There is the following $d_{3}$-differential in the Adams spectral sequence for $tmf$ $$d_{3}(w_{2}^{2}(v_{2}^{4}\eta)) = g^{6}.$$
\item[(iv)] The class $\Delta^{8}:=w_{2}^{4}$ survives the Adams spectral sequence.

\end{itemize}
\end{Theorem}
\noindent

\begin{Proposition}\phantomsection \label{alternative} In the ASS for $tmf\wedge A_1$, there exists $\lambda\in\F$ such that the following statements are equivalent:
\begin{itemize} 
\item[i)]  $d_{2}(w_{2}e[4,23]) = \lambda g^{2}e[6,30],$
\item[ii)] $d_{2}(w_{2}e[9,48]) = \lambda g^{4}e[3,15],$
\item[iii)] $d_{2}(w_{2}e[10,53]) =\lambda g^{5}e[0,0],$
\item[iv)] $d_{2}(w_{2}e[7,38]) =\lambda g^{4}e[1,5].$
\end{itemize}
\end{Proposition}

\begin{proof} We will prove that $i)\Rightarrow ii)\Rightarrow iii)\Rightarrow iv)\Rightarrow i)$. The charts of Figures (\ref{tmf-A_1}) and (\ref{tmf-A_1-2}) will make the proof easier to follow.  First, we observe that all of the classes $e[4,23]$, $e[7,38]$, $e[9,48]$, $e[10,53]$ are permanent cycles, by sparseness. 

$i)\Rightarrow ii)$ Suppose $d_{2}(w_{2}e[4,23]) = g^{2}e[6,30]$. Then $d_{2}(g^{2}w_{2}e[4,23]) = g^{4}e[6,30]$ by $g$-linearity. It follows that there is no room for a non-trivial differential on $w_{2}^{2}e[3,15]$. In order words, $w_{2}^{2}e[3,15]$ is a permanent cycle. Because of part $iii)$ of Theorem \ref{tmf'sdiff},  a $g^{k}$-multiple of $w_{2}^{2}e[3,15]$ must be hit by a differential for some $k$ less than $7$. One can check that the only possibility is that $d_{2}(w_{2}^{3}e[9,48])=g^{4}w_{2}^{2}e[3,15]$. Since $w_{2}^{2}$ is a $d_{2}$-cycle in the ASS for $tmf$, this differential implies that $d_{2}(w_{2}e[9,48])= g^{4}e[3,15]$.

$ii)\Rightarrow iii)$ Suppose $d_{2}(w_{2}e[9,48]) = g^{4}e[3,15]$. Then the class $w_{2}^{2}e[0,0]$ is a permanent cycle, by sparseness. Again, a $g^{k}$-multiple of $w_{2}^{2}e[0,0]$ for some $k$ smaller than $7$ must be hit by a differential. Inspection shows that the classes $w_{2}^{3}e[10,53]$ and $w_{2}^{4}e[1,5]$ are the only ones that have the appropriate bidegree to support such a differential. However, $w_{2}^{4}e[1,5]$ is a permanent cycle, because $w_{2}^{4}$ and $e[1,5]$ are permanent cycles in their respective ASS. Thus, we have that $d_{2}(w_{2}e[10,53])=g^{5}e[0,0]$.

$iii\Rightarrow iv)$ Suppose $d_{2}(w_{2}e[10,53]) = g^{5}e[0,0]$. Then the class $w_{2}^{2}e[1,5]$ is a permanent cycle, as there is no room for a non-trivial differential on it. Then $g^{k}w_{2}^{2}e[1,5]$ must be hit by a differential for some $k$ less than $7$. Inspection shows that the only possibility is that $d_{2}(w_{2}^{3}e[7,38]) = g^{4}w_{2}^{2}e[1,5]$. As $w_{2}^{2}$ is a $d_{2}$-cycle, it follows that $d_{2}(w_{2}e[7,38]) = g^{4}e[1,5]$.

$iv)\Rightarrow i)$  Suppose $d_{2}(w_{2}e[7,38]) = g^{4}e[1,5]$. By $g$-linearity, we get that $d_{2}(gw_{2}e[7,38]) = g^{5}e[1,5]$. It follows by sparseness that $w_{2}^{2}e[6,30]$ is a permanent cycle. Then the class $g^{k}w_{2}^{2}e[6,30]$ is hit by a differential for some $k$ less than $7$. Inspection shows that the only possibility is that $d_{2}(w_{2}^{3}e[4,23])=g^{2}w_{2}^{2}e[6,30]$. Therefore, $d_{2}(w_{2}e[4,23]) = g^{2}e[6,30]$ by $w_{2}^{2}$-linearity.
  
\end{proof}

\begin{Theorem} \phantomsection \label{ASS_d_2_bis} In the Adams spectral sequence for $tmf\wedge A_{1}[ij]$, there are the following differential $d_2$:
\begin{itemize} 
\item[i)]  $d_{2}(w_{2}e[4,23]) = \lambda_{i,j} g^{2}e[6,30],$
\item[ii)] $d_{2}(w_{2}e[9,48]) = \lambda_{i,j} g^{4}e[3,15],$
\item[iii)] $d_{2}(w_{2}e[10,53]) =\lambda_{i,j} g^{5}e[0,0],$
\item[iv)] $d_{2}(w_{2}e[7,38]) =\lambda_{i,j} g^{4}e[1,5].$
\end{itemize}
\end{Theorem}
\begin{proof} By the Leibniz rule and part (ii) of Theorem \ref{tmf'sdiff}, $$d_{2}(w_{2}e[4,23])=d_{2}(w_{2})e[4,23] = g\beta\alpha e[4,23] = \lambda_{i,j}g^2e[6,30], $$ where the last equality follows from Proposition \ref{exo_prod} and Proposition \ref{prod}. Thus, the theorem follows from Proposition \ref{alternative}.
\end{proof}
\begin{Proposition} \label {d_3-diff}There are the following $d_{3}$-differentials in the Adams spectral sequence for $tmf\wedge A_1$
$$d_{3}(w_{2}^{2}e[10,53]) = g^{5}e[9,48]$$
$$d_{3}(w_{2}^{3}e[1,5]) = g^{5}w_{2}e[0,0].$$
\end{Proposition}
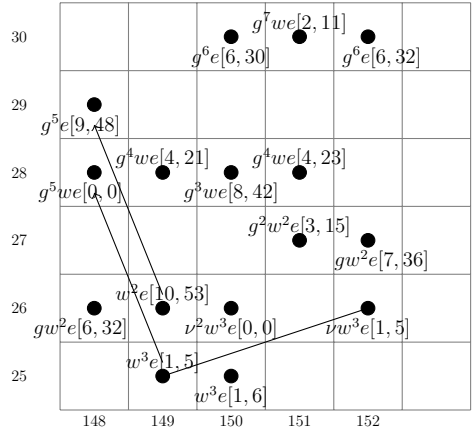
\begin{figure}[h!]
\begin{center}
\begin{tikzpicture}[scale =0.9]
\clip(-1.5,-1.5) rectangle (6.5,6.5);
\draw[color=gray] (0,0) grid [step=1] (6,6);

\foreach \n in {148,149,150,151,152}
{
\def\nn{\n-0}
\node[below, scale=0.5] at (\nn-148+0.5,0) {$\n$};
}
\foreach \s in {25,26,27,28,29,30}
{\def\ss{\s-0};
\node [left, scale=0.5] at (-0.4,\ss-25+0.5){$\s$};
}
\draw [fill] ( 0.5, 1.5) circle [radius=0.1];
\node[scale=0.65] at (0.3,1.2) {$gw^{2}e[6,32]$};
\draw[fill] (0.5,3.5) circle [radius=0.1];
\node[scale=0.65] at (0.3,3.2) {$g^{5}we[0,0]$};
\draw[fill] (0.5, 4.5) circle [radius=0.1];
\node[scale=0.65] at (0.3,4.2) {$g^{5}e[9,48]$};

\draw[fill] (1.5,0.5) circle [radius =0.1];
\node[scale=0.65] at (1.5,0.7) {$w^{3}e[1,5]$};
\draw[fill] (1.5,1.5) circle [radius =0.1];
\node[scale=0.65] at (1.5,1.7) {$w^{2}e[10,53]$};
\draw[fill] (1.5,3.5) circle [radius =0.1];
\node[scale=0.65] at (1.5,3.7) {$g^{4}we[4,21]$};

\draw[fill] (2.5,0.5) circle [radius =0.1];
\node[scale=0.65] at (2.5,0.2) {$w^{3}e[1,6]$};
\draw[fill] (2.5,1.5) circle [radius =0.1];
\node[scale=0.65] at (2.5,1.2) {$\nu^{2}w^{3}e[0,0]$};
\draw[fill] (2.5,3.5) circle [radius =0.1];
\node[scale=0.65] at (2.5,3.2) {$g^{3}we[8,42]$};
\draw[fill] (2.5,5.5) circle [radius =0.1];
\node[scale=0.65] at (2.5,5.2) {$g^{6}e[6,30]$};

\draw[fill] (3.5,2.5) circle [radius =0.1];
\node[scale=0.65] at (3.5,2.7) {$g^{2}w^{2}e[3,15]$};
\draw[fill] (3.5,3.5) circle [radius =0.1];
\node[scale=0.65] at (3.5,3.7) {$g^{4}we[4,23]$};
\draw[fill] (3.5,5.5) circle [radius =0.1];
\node[scale=0.65] at (3.5,5.7) {$g^{7}we[2,11]$};

\draw[fill] (4.5,1.5) circle [radius =0.1];
\node[scale=0.65] at (4.5,1.2) {$\nu w^{3}e[1,5]$};
\draw[fill] (4.5,2.5) circle [radius =0.1];
\node[scale=0.65] at (4.7,2.2) {$gw^{2}e[7,36]$};
\draw[fill] (4.5,5.5) circle [radius =0.1];
\node[scale=0.65] at (4.7,5.2) {$g^{6}e[6,32]$};

\draw[-] (1.5,0.5)--(4.5,1.5);
\draw[->] (1.5,0.7)--(0.5,3.2);
\draw[->] (1.5,1.7)--(0.5,4.2);

\end{tikzpicture}
\end{center}
\caption{The Adams spectral sequence in the range $148 \leq t-s \leq 152$}
\phantomsection \label{ASS_chart}
\end{figure}
\begin{proof} 
We can check from the chart that $e[9,48]$ and $we[0,0]$ are permanent cycles. Then $g^{l}e[9,48]$ and $g^{k}we[0,0]$ must be targets of some differentials for some $l$ and $k$ less than $7$. Inspection of the $\mathrm{E}_{2}$-term shows that either $$d_{2}(w_{2}^{2}e[10,53]) = g^{5}we[0,0] \ \mbox{and}\ d_{4}(w_{2}^{3}e[1,5]) = g^{5}e[9,48]$$ or $$d_{3}(w_{2}^{2}e[10,53]) = g^{5}e[9,48] \ \mbox{and} \ d_{3}(w_{2}^{3}e[1,5]) = g^{5}w_{2}e[0,0].$$ However, the former possibility is ruled out because of the Leibniz rule: $$d_{2}(w_{2}^{2}e[10,53]) = d_{2}(w_{2}^{2})e[10,53] = 2w_{2}d_{2}(w_{2})e[10,53] = 0,$$ where the first equality follows from the fact that $e[10,53]$ is a permanent cycle, by spareness.  
\end{proof}
\begin{Corollary}\phantomsection \label{Toda bracket} The Toda bracket $\langle g^{5},e[9,48],\nu\rangle$ can be formed and contains only elements which are divisible by $g$.
\end{Corollary}
\noindent
For references on Toda bracket, see \cite{Tod62}, \cite{Koc90}. 
\begin{proof} In the $\mathrm{E}_{4}$-term of the ASS, the Massey product $\langle g^{5},e[9,48],\nu\rangle$ has cohomological filtration $27$ and is equal to zero with zero indeterminacy. On the other hand the corresponding Toda bracket can be formed with indeterminacy containing only multiples of $g$. We can check that all conditions of Moss's convergence theorem \cite{Mos70} are met. This implies that the Toda bracket $\langle g^{5},e[9,48],\nu\rangle$ contains an element detected in filtration $27$ by $0$, thus is a multiple of $g$. Therefore, this Toda bracket contains only multiples of $g$.
\end{proof}
\noindent
Finally, we need to have control of the action of the class $\Delta^{8} = w_2^4 \in \Ext_{\A(2)_{*}}^{32,224}(\F)$ on the $\mathrm{E}_{\infty}$-term of the ASS for $tmf\wedge A_1$. This will allow us to compare $\pi_{*}(tmf\wedge A_1)$ with $\pi_{*}(E_{C}^{hG_{24}}\wedge A_1)$ (see Corollary \ref{Cor_Compare}) and hence to discuss higher differentials in the HFPSS for $E_{C}^{hG_{24}}\wedge A_1$.
\begin{Proposition} \phantomsection \label{Period}The class $w_{2}^{4}$ acts freely on the $\mathrm{E}_{\infty}$-term of the ASS for $tmf\wedge A_1$. As a consequence, the element $\Delta^{8}\in\pi_{192}(tmf)$ acts freely on the homotopy groups of $tmf\wedge A_1$.
\end{Proposition}
\begin{proof} Using the description of the $\mathrm{E}_{2}$-term of the ASS for $tmf\wedge A_1$ in Theorem \ref{Prop_AdamsA_1} and an elementary bidegree inspection, we can see that, if a class $y$ is in an appropriate bidegree to support a differential hitting a class of the form $w_{2}^{4}x$ for some class $x$, then $y$ is divisible by $w_{2}^{4}$. Knowing that $w_{2}^{4}$ is a permanent cycle in the ASS for $tmf$, we conclude that, if a class $x$ survives the $\mathrm{E}_{r}$-term, then the multiple of $x$ by all powers of $w_{2}^{4}$ also survive that term. Therefore, the Proposition follows by induction.
\end{proof}
\begin{Proposition}\phantomsection \label{Period+} For every element $x\in \pi_{*}(tmf\wedge A_1)$, the element $\Delta^{8}x$ is divisible by $\overline{\kappa}$ (resp. $\nu$) if and only if $x$ is divisible by $\overline{\kappa}$ (resp. $\nu$).
\end{Proposition}
\begin{proof} The argument is similar to that used in the proof of Proposition \ref{Period}. A bidegree inspection shows that, if a class $y \in \Ext_{\A(2)_{*}}^{*,*}(\H_*(A_1))$ is in an appropriate bidegree whose (exotic) product with $g$ (resp. $\nu$) might detect $\Delta^{8}x$, then $y$ is divisible by $w_{2}^{4}$. We conclude the proof by using the fact that the class $w_{2}^{4}$ acts freely on the ASS for $tmf\wedge A_1$, by Proposition \ref{Period}.
\end{proof}
\noindent
\begin{landscape}
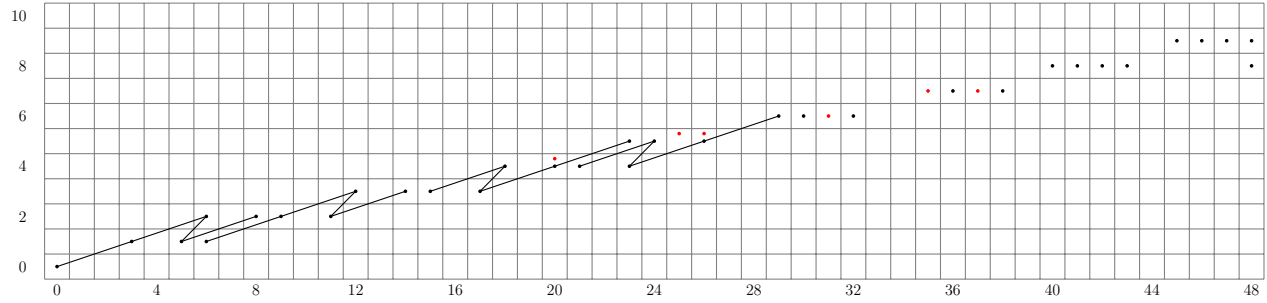
\begin{figure}[h!]
 \center
	\begin{tikzpicture}[scale =0.33]
		\clip(-1.5,-1.5) rectangle (49,12);
		\draw[color=gray] (0,0) grid [step=1] (49,11);
		\foreach \n in {0,4,...,48}
			{\def\nn{\n-0}
			\node[below,scale=0.5] at (\n+0.5,0) {$\n$};
			}
		\foreach \s in {0,2,...,10}
			{\def\ss{\s-0};
			\node [left,scale=0.5] at (-0.5,\s+0.5){$\s$};
			}
		%%%% 0-29
		\draw[fill] (0.5,0.5) circle [radius=0.05]; %%%e_{0}
		\draw[fill] (3.5,1.5) circle [radius=0.05];%%%\nu e_{0}
		\draw[fill] (6.5,2.5) circle [radius=0.05]; %%%\nu2 e_{0}
		\draw[-] (0.5,0.5)--(3.5,1.5);
		\draw[-] (3.5,1.5)--(6.5,2.5);		
		\draw[fill] (5.5,1.5) circle [radius=0.05]; %%%e_{5}
		\draw[fill] (8.5,2.5) circle [radius=0.05]; %%%\nu e_{5}
		\draw[-] (5.5,1.5)--(8.5,2.5);
		\draw[-] (5.5,1.5)--(6.5,2.5);
		\draw[fill] (6.5,1.5) circle [radius=0.05]; %%% e_{6}
		\draw[fill] (9.5,2.5) circle [radius=0.05]; %%%\nu e_{6}
		\draw[fill] (12.5,3.5) circle [radius=0.05]; %%%\nu2 e_{6}
		\draw[-] (6.5,1.5)--(9.5,2.5);
		\draw[-] (9.5,2.5)--(12.5,3.5);
		\draw[fill] (11.5,2.5) circle [radius=0.05]; %%% e_{11}
		\draw[fill] (14.5,3.5) circle [radius=0.05]; %%%\nu e_{11}
		\draw[-] (11.5,2.5)--(14.5,3.5);
		\draw[-] (11.5,2.5)--(12.5,3.5);
		\draw[fill] (15.5,3.5) circle [radius=0.05];%%% e_{15}
		\draw[fill] (18.5,4.5) circle [radius=0.05]; %%%\nu e_{15}
		\draw[-] (15.5,3.5) --(18.5,4.5);
		\draw[fill] (17.5,3.5) circle [radius=0.05]; %%% e_{17}
		\draw[fill] (20.5,4.5) circle [radius=0.05]; %%%\nu e_{17}
		\draw[fill] (23.5,5.5) circle [radius=0.05]; %%%\nu2 e_{17}
		\draw[-] (17.5,3.5)--(20.5,4.5);
		\draw[-] (20.5,4.5)--(23.5,5.5);
		\draw[-] (17.5,3.5)--(18.5,4.5);
		\draw[fill] (21.5,4.5) circle [radius=0.05]; %%% e_{21}
		\draw[fill] (24.5,5.5) circle [radius=0.05];%%%\nu e_{21}
		\draw[-] (21.5,4.5)--(24.5,5.5);
		\draw[fill] (23.5,4.5) circle [radius=0.05]; %%% e_{23}
		\draw[fill] (26.5,5.5) circle [radius=0.05]; %%%\nu e_{23}
		\draw[fill] (29.5,6.5) circle [radius=0.05]; %%%\nu2 e_{23}
		\draw[-] (23.5,4.5)--(26.5,5.5);
		\draw[-] (26.5,5.5)--(29.5,6.5);
		\draw[-] (23.5,4.5)--(24.5,5.5);
		\draw[fill, red] (20.5,4.8) circle [radius=0.05]; %%% g
		\draw[fill, red] (25.5,5.8) circle [radius=0.05]; %%% g e5
		\draw[fill, red] (26.5,5.8) circle [radius=0.05]; %%% g e_{6}
		%%%%%30--62
		\draw[fill] (30.5,6.5) circle [radius=0.05]; %%% e30
		\draw[fill,red] (31.5,6.5) circle [radius=0.05]; %%% ge11
		\draw[fill] (32.5,6.5) circle [radius=0.05]; %%% e32
		\draw[fill,red] (35.5,7.5) circle [radius=0.05]; %%% ge15
		\draw[fill] (36.5,7.5) circle [radius=0.05]; %%% e36
		\draw[fill,red] (37.5,7.5) circle [radius=0.05]; %%% ge17
		\draw[fill] (38.5,7.5) circle [radius=0.05]; %%% e38
		\draw[fill] (40.5,8.5) circle [radius=0.05]; %%% g^2
		\draw[fill] (41.5,8.5) circle [radius=0.05]; %%% ge21	
		\draw[fill] (42.5,8.5) circle [radius=0.05]; %%% e42	
		\draw[fill] (43.5,8.5) circle [radius=0.05]; %%% ge23
		\draw[fill] (45.5,9.5) circle [radius=0.05]; %%% g^2e5
		\draw[fill] (46.5,9.5) circle [radius=0.05]; %%% g^2e6
		\draw[fill] (47.5,9.5) circle [radius=0.05]; %%% e47
		\draw[fill] (48.5,9.5) circle [radius=0.05]; %%% e48
		\draw[fill] (48.5,8.5) circle [radius=0.05]; %%% we0
	\end{tikzpicture}
	\caption[scale=0.5]{Adams spectral sequence for $A_1$ in the range $0\leq t-s\leq 48$ }
	\phantomsection \label{tmf-A_1}
\end{figure}
\begin{figure}[h!]
        \center
	\begin{tikzpicture}[scale=0.32]
	\clip(-1.5,-1.5) rectangle (54,14);
	\draw[color=gray] (0,0) grid [step=1] (54,13);
	\foreach \n in {48,52,...,100}
		{\def\nn{\n-0}
		\node[below,scale=0.5] at (\n+0.5-48,0) {$\n$};
			}
	\foreach \s in {8,10,...,20}
		{\def\ss{\s-0};
		\node [left,scale=0.5] at (-0.5,\s+0.5-8){$\s$};
			}
	 	\draw[fill] (48.5-48,9.5-8) circle [radius=0.05]; %%% e48
		\draw[fill] (48.5-48,8.5-8) circle [radius=0.05]; %%% we0
		\draw[fill] (51.5-48,9.5-8) circle [radius=0.05]; %%% nuwe0
		\draw[fill] (54.5-48,10.5-8) circle [radius=0.05]; %%% nu^2we0
		\draw[-] (48.5-48,8.5-8)--(51.5-48,9.5-8);
		\draw[-] (51.5-48,9.5-8)--(54.5-48,10.5-8);
		\draw[fill,red] (50.5-48,10.5-8) circle [radius=0.05]; %%% ge30
		\draw[fill,red] (51.5-48,10.5-8) circle [radius=0.05]; %%% g^2e11
		\draw[fill,red] (52.5-48,10.5-8) circle [radius=0.05]; %%% ge32
		\draw[fill] (53.5-48,10.5-8) circle [radius=0.05]; %%% e53
		\draw[fill] (53.5-48,9.5-8) circle [radius=0.05]; %%% we5
		\draw[fill] (56.5-48,10.5-8) circle [radius=0.05]; %%% nuwe5
		\draw[-] (53.5-48,9.5-8)--(56.5-48,10.5-8);
		\draw[-] (53.5-48,9.5-8)--(54.5-48,10.5-8);
		\draw[fill] (54.5-48,9.5-8) circle [radius=0.05]; %%% we6
		\draw[fill] (57.5-48,10.5-8) circle [radius=0.05]; %%% nuwe6
		\draw[fill] (60.5-48,11.5-8) circle [radius=0.05]; %%% nu^2we6
		\draw[-] (54.5-48,9.5-8)--(57.5-48,10.5-8);
		\draw[-] (57.5-48,10.5-8)--(60.5-48,11.5-8);
		\draw[fill,red] (55.5-48,11.5-8) circle [radius=0.05]; %%% g^2e15
		\draw[fill,red] (56.5-48,11.5-8) circle [radius=0.05]; %%% ge36
		\draw[fill,red] (57.5-48,11.5-8) circle [radius=0.05]; %%% g^2e17
		\draw[fill,red] (58.5-48,11.5-8) circle [radius=0.05]; %%% ge38
		\draw[fill] (59.5-48,10.5-8) circle [radius=0.05]; %%% we11
		\draw[fill] (62.5-48,11.5-8) circle [radius=0.05]; %%% nuwe11
		\draw[-] (59.5-48,10.5-8)--(62.5-48,11.5-8);
		\draw[-] (59.5-48,10.5-8)--(60.5-48,11.5-8);
		\draw[fill,red] (60.5-48,12.5-8) circle [radius=0.05]; %%% g^3
		\draw[fill,red] (61.5-48,12.5-8) circle [radius=0.05]; %%% g^2e21
		\draw[fill,red] (62.5-48,12.5-8) circle [radius=0.05]; %%% ge42
		\draw[fill,red] (63.5-48,12.5-8) circle [radius=0.05]; %%% g^2e23
		\draw[fill] (63.5-48,11.5-8) circle [radius=0.05]; %%% we15
		\draw[fill] (66.5-48,12.5-8) circle [radius=0.05]; %%% nuwe15
		\draw[-] (63.5-48,11.5-8)--(66.5-48,12.5-8);
		\draw[fill] (65.5-48,11.5-8) circle [radius=0.05]; %%% we17
		\draw[fill] (68.5-48,12.5-8) circle [radius=0.05]; %%% nuwe17
		\draw[fill] (71.5-48,13.5-8) circle [radius=0.05]; %%% nu^2we17
		\draw[-] (65.5-48,11.5-8)--(68.5-48,12.5-8);
		\draw[-] (68.5-48,12.5-8)--(71.5-48,13.5-8);
		\draw[-] (65.5-48,11.5-8)--(66.5-48,12.5-8);
		\draw[fill,red] (65.5-48,13.5-8) circle [radius=0.05]; %%% g^3e5
		\draw[fill,red] (66.5-48,13.5-8) circle [radius=0.05]; %%% g^3e6
		\draw[fill,red] (67.5-48,13.5-8) circle [radius=0.05]; %%% ge47
		\draw[fill,red] (68.5-48,13.5-8) circle [radius=0.05]; %%% ge48
		\draw[fill,red] (68.5-48,12.8-8) circle [radius=0.05]; %%% gw
		\draw[fill] (69.5-48,12.5-8) circle [radius=0.05]; %%% we21
		\draw[fill] (72.5-48,13.5-8) circle [radius=0.05]; %%% nuwe21
		\draw[-] (69.5-48,12.5-8)--(72.5-48,13.5-8);
		\draw[fill] (71.5-48,12.5-8) circle [radius=0.05]; %%% we23
		\node [scale=0.5] at (71.5-48,12.2-8) {$we[4,23]$};
		\draw[->] (71.4-48,12.6-8)--(70.6-48,14.4-8); %%%%d_{2}
		\draw[fill] (74.5-48,13.5-8) circle [radius=0.05]; %%% nuwe23
		\draw[fill] (77.5-48,14.5-8) circle [radius=0.05]; %%% nuwe23
		\draw[-] (71.5-48,12.5-8)--(72.5-48,13.5-8);
		\draw[-] (71.5-48,12.5-8)--(74.5-48,13.5-8);
		\draw[-] (74.5-48,13.5-8)--(77.5-48,14.5-8);
		\draw[fill,red] (70.5-48,14.5-8) circle [radius=0.05]; %%% g^2e30
		\node [scale=0.5] at (70.5-48,14.8-8) {$g^{2}e[6,30]$};
		\draw[fill,red] (71.5-48,14.5-8) circle [radius=0.05]; %%% g^3e11
		\draw[fill,red] (72.5-48,14.5-8) circle [radius=0.05]; %%% g^2e32
		\draw[fill,red] (73.5-48,14.5-8) circle [radius=0.05]; %%% ge53
		\draw[fill,red] (73.5-48,13.5-8) circle [radius=0.05]; %%% gwe5
		\draw[fill,red] (74.5-48,13.8-8) circle [radius=0.05]; %%% gwe6
		\draw[fill,red] (75.5-48,15.5-8) circle [radius=0.05]; %%% g^3e15
		\draw[fill,red] (76.5-48,15.5-8) circle [radius=0.05]; %%% g^2e36
		\draw[fill,red] (77.5-48,15.5-8) circle [radius=0.05]; %%% g^3e15
		\draw[fill,red] (78.5-48,15.5-8) circle [radius=0.05]; %%% g^2e38
		\draw[fill] (78.5-48,14.5-8) circle [radius=0.05]; %%% we30
		\draw[fill,red] (79.5-48,14.5-8) circle [radius=0.05]; %%% gwe11
		\draw[fill] (80.5-48,14.5-8) circle [radius=0.05]; %%% we32
		\draw[fill,red] (80.5-48,16.5-8) circle [radius=0.05]; %%% g^4
		\draw[fill,red] (81.5-48,16.5-8) circle [radius=0.05]; %%% g^3e21
		\draw[fill,red] (82.5-48,16.5-8) circle [radius=0.05]; %%% g^2e42
		\draw[fill,red] (83.5-48,16.5-8) circle [radius=0.05]; %%% g^3e23
		\draw[fill,red] (83.5-48,15.5-8) circle [radius=0.05]; %%% gwe15
		\draw[fill] (84.5-48,15.5-8) circle [radius=0.05]; %%% we36
		\draw[fill,red] (85.5-48,15.5-8) circle [radius=0.05]; %%% gwe17
		\draw[fill,red] (85.5-48,17.5-8) circle [radius=0.05]; %%% g^4e5
		\node [scale=0.5] at (85.5-48,17.8-8) {$g^{4}e[1,5]$};
		\draw[fill] (86.5-48,15.5-8) circle [radius=0.05]; %%% we38
		\draw[->] (86.4-48,15.6-8)--(85.6-48,17.4-8);%%%d_{2}
		\node [scale=0.5] at (86.5-48,15.2-8) {$we[7,38]$};
		\draw[fill,red] (86.5-48,17.5-8) circle [radius=0.05]; %%% g^4e6
		\draw[fill,red] (87.5-48,17.5-8) circle [radius=0.05]; %%% g^2e47
		\draw[fill,red] (88.5-48,16.5-8) circle [radius=0.05]; %%% g^2w
		\draw[fill,red] (88.5-48,17.5-8) circle [radius=0.05]; %%% g^2e48
		\draw[fill,red] (89.5-48,16.5-8) circle [radius=0.05]; %%% gwe21
		\draw[fill,red] (90.5-48,16.5-8) circle [radius=0.05]; %%% we42
		\draw[fill,red] (90.5-48,18.5-8) circle [radius=0.05]; %%% g^3e30
		\node [scale=0.5] at (90.5-48,18.8-8) {$g^{3}e[6,30]$};
		\draw[fill,red] (91.5-48,16.5-8) circle [radius=0.05]; %%% gwe23
		\draw[->] (91.4-48,16.6-8)--(90.6-48,18.4-8);%%%d_{2}
		\node [scale=0.5] at (91.5-48,16.2-8) {$gwe[4,23]$};
		\draw[fill,red] (91.5-48,18.5-8) circle [radius=0.05]; %%%g^4e11
		\draw[fill,red] (92.5-48,18.5-8) circle [radius=0.05]; %%%g^3e32
		\draw[fill,red] (93.5-48,17.5-8) circle [radius=0.05]; %%%g^2we5
		\draw[fill,red] (93.5-48,18.5-8) circle [radius=0.05]; %%%g^2e53
		\draw[fill,red] (94.5-48,17.5-8) circle [radius=0.05]; %%%g^2we5
		\draw[fill,red] (95.5-48,17.5-8) circle [radius=0.05]; %%%we47
		\draw[fill,red] (95.5-48,19.5-8) circle [radius=0.05]; %%%g^4e15
		\node [scale=0.5] at (95.5-48,19.8-8) {$g^{4}e[3,15]$};
		\draw[fill] (96.5-48,16.5-8) circle [radius=0.05]; %%%w^2
		\node [scale=0.5] at (96.5-48,16.2-8) {$w^{2}$};
		\draw[->,densely dashed] (96.4-48,16.6-8)--(95.6-48,19.4-8);%%%d_{3}
		\draw[fill] (99.5-48,17.5-8) circle [radius=0.05];%%%nuw^{2}
		\draw[-] (96.5-48,16.5-8)--(99.5-48,17.5-8);
		\draw[fill] (96.5-48,17.5-8) circle [radius=0.05]; %%%we48
		\node [scale=0.5] at (97.2-48,17.8-8) {$we[9,48]$};
		\draw[->] (96.4-48,17.6-8)--(95.6-48,19.4-8);%%%d_{2}
		\draw[fill,red] (96.5-48,19.5-8) circle [radius=0.05]; %%%g^3e36
		\draw[fill,red] (97.5-48,19.5-8) circle [radius=0.05]; %%%g^4e17
		\draw[fill,red] (98.5-48,18.5-8) circle [radius=0.05]; %%%gwe30
		\draw[fill,red] (98.5-48,19.5-8) circle [radius=0.05]; %%%g^3e38
		\draw[fill,red] (99.5-48,18.5-8) circle [radius=0.05]; %%%g^2we11
		\draw[fill,red] (100.5-48,18.5-8) circle [radius=0.05]; %%%gwe32
		\draw[fill,red] (100.5-48,20.5-8) circle [radius=0.05]; %%%g^5
		\node [scale=0.5] at (100.5-48,20.8-8) {$g^{5}$};
		\draw[fill] (101.5-48,17.5-8) circle [radius=0.05]; %%%w^2e5
		\node [scale=0.5]at (101.5-48,17-8) {$w^{2}e[1,5]$};
		\draw[->,densely dashed] (101.4-48,17.6-8)--(100.6-48,20.4-8);%%%d_{3}
		\draw[fill] (101.5-48,18.5-8) circle [radius=0.05]; %%%we53
		\node [scale=0.5] at (102.2-4!,18.8-8) {$we[10,53]$};
		\draw[->] (101.4-48,18.6-8)--(100.6-48,20.4-8);%%%d_{2}
		\draw[fill,red] (101.5-48,20.5-8) circle [radius=0.05]; %%%g^4e21
		\draw[fill,red] (95.5-48,19.5-8) circle [radius=0.05]; %%%g^4e15
	\end{tikzpicture}
	\caption[scale = 0.3]{Adams spectral sequence for $A_1$ in the range $48\leq t-s\leq 101$. The arrows in bold are differentials for the models $A_{1}[10]$ and $A_{1}[01]$ and the dashed arrows for the models $A_{1}[00]$ and $A_{1}[11]$ }
	\phantomsection \label{tmf-A_1-2}
\end{figure}
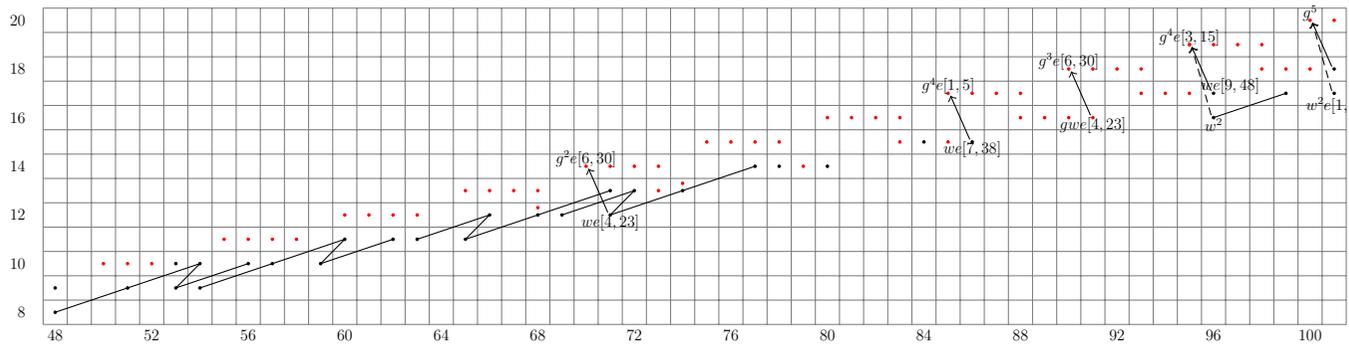
\end{landscape}
\section{The homotopy fixed point spectral sequence for $E_{C}^{hG_{24}}\wedge A_1$}\label{G24 4}
\subsection{Preliminaries and recollection on cohomology of $G_{24}$ }
\noindent
 The action of $G_{24}$ on $(E_C)_*\cong \W(\FF_4)[[u_1]][u^{\pm 1}]$ is documented \cite{Bea17}, 2.4. Let us recall this action.
 \begin{Theorem}\label{actionG_24}
 The action of $G_{24}$ on $\mathbb{W}(\FF_{4})[[u_{1}]][u^{\pm 1}]$ is given by
 $$\omega(u^{-1}) = \zeta^{2}u^{-1}\ \ \ \ \ \omega(v_{1}) = v_{1}$$
 $$i(u^{-1}) =\frac{-u^{-1}+v_{1}}{\zeta^{2}-\zeta}\ \ \ \ i(v_{1})=\frac{v_{1}+2u^{-1}}{\zeta^{2}-\zeta}$$
  $$j(u^{-1}) =\frac{-u^{-1}+\zeta^{2} v_{1}}{\zeta^{2}-\zeta}\ \ \ \ j(v_{1})=\frac{v_{1}+2\zeta^{2}u^{-1}}{\zeta^{2}-\zeta}$$
   $$k(u^{-1}) =\frac{-u^{-1}+\zeta v_{1}}{\zeta^{2}-\zeta}\ \ \ \ k(v_{1})=\frac{v_{1}+2\zeta u^{-1}}{\zeta^{2}-\zeta}.$$
  \end{Theorem}
Equations \ref{K(2)-loc TMF} and \ref{tmf connec model} give us a way to get access to the homotopy groups of $E_{C}^{hG_{24}}\wedge A_1$. 
\begin{Theorem}\phantomsection \label{compar} There is a homotopy equivalence $$[(\Delta^{8})^{-1}]tmf\wedge A_1  \simeq (E_{C}^{hG_{24}})^{h\Gal}\wedge A_1,$$
where $\Gal$ denotes the Galois group $\Gal(\mathbb{F}_{4}/\F)$.
\end{Theorem}
\begin{proof} 
We have
\begin{align*}
[(\Delta^{8})^{-1}]tmf\wedge A_1  &\simeq TMF\wedge A_1\ (\mbox{Equation}\ \ref{tmf connec model}) \\
						   &\simeq L_{2}(TMF)\wedge A_1\  (\mbox{$TMF$ is $E(2)$-local})\\
						   &\simeq L_{2}(TMF\wedge A_1)\ (\mbox{$L_{2}$ is smashing}) \\
						   &\simeq L_{K(2)} (TMF)\wedge A_1 \\
						   &\simeq (E_{C}^{hG_{24}})^{h\Gal(\FF_{4}/\F)} \wedge A_1\ (\mbox{Equation}\ \ref{K(2)-loc TMF}). 	
\end{align*}
The fourth equivalence is Lemma 7.2 of \cite{HS99} applied to the $K(2)$-localisation and $A_1$, which is finite spectrum of type $2$.
\end{proof}
\begin{Corollary}\label{Compar_alg} There is a homotopy equivalence
$$\Gal_+\wedge [(\Delta^{8})^{-1}]tmf\wedge A_1 \simeq E_C^{hG_{24}}\wedge A_1. $$
Therefore, $$ \mathbb{W}(\FF_{4})\otimes_{\mathbb{Z}_{2}}(\Delta^{8})^{-1}(\pi_{*}(tmf\wedge A_1))\cong \pi_{*}(E_{C}^{hG_{24}}\wedge A_1).$$
\end{Corollary}
\begin{proof} This is a consequence of Theorem \ref{compar} and Lemma 1.37 of \cite{BG18}.
\end{proof}
\noindent
Let us denote by 
\begin{equation}\label{map_compa}\Theta : \mathbb{W}(\FF_{4})\otimes_{\Z_{2}}\pi_{*}(tmf\wedge A_1)\rightarrow \pi_{*}(E_{C}^{hG_{24}}\wedge A_1),
\end{equation} given by pre-composing the isomorphism of Corollary \ref{Compar_alg} with the natural homomorphism $\pi_*(tmf\wedge A_1)\rightarrow \pi_*([(\Delta^8)^{-1}]tmf\wedge A_1).$ The following corollary recapitulates the relationship between $\pi_{*}(tmf\wedge A_1)$ and $\pi_{*}(E_{C}^{hG_{24}}\wedge A_1)$. 
\begin{Corollary} \phantomsection \label{Cor_Compare} The homomorphism $\Theta$ is injective. Moreover, it remains injective after quotienting out by the ideal of $\pi_{*}(S^{0})$ generated by $(\overline{\kappa}, \nu).$
\end{Corollary}
\begin{proof} This follows from Theorem \ref{compar}, Proposition \ref{Period} and Proposition \ref{Period+}.
\end{proof}
\noindent
\noindent
We continue to recollect some necessary information about the HFPSS converging to $\pi_{*}(E_{C}^{hG_{24}})$:
\begin{equation}\phantomsection \label{HFPSS_S0}
\mathrm{H}^{s}(G_{24}, (E_{C})_{t}) \Longrightarrow \pi_{t-s}(E_{C}^{hG_{24}}).
\end{equation}
The elements $\eta\in \pi_{1}(S^{0})$, $\nu\in \pi_{3}(S^{0})$, $\overline{\kappa}\in\pi_{20}(S^{0})$ are sent non-trivially to elements of the same name in $\pi_{*}(E_{C}^{hG_{24}})$ via the Hurewicz map $S^{0}\rightarrow E_{C}^{hG_{24}}$. As the latter factors through the unit map of $tmf$, the element $\overline{\kappa}^{6}=0$ in $\pi_{*}(E_{C}^{hG_{24}})$ because $\overline{\kappa}^{6}=0$ in $\pi_{*}(tmf)$ (see \cite{Bau08}). These elements are detected by $\eta\in\mathrm{H}^{1}(G_{24}, (E_{C})_{2})$, $\nu \in \mathrm{H}^{1}(G_{24}, (E_{C})_{4})$, $\overline{\kappa}\in \mathrm{H}^{4}(G_{24},(E_{C})_{24})$, respectively. Furthermore, there is a class $\Delta\in \mathrm{H}^{0}(G_{24}, (E_{C})_{24})$ such that $\Delta^{8}$ is a permanent cycle detecting the periodicity of $E_C^{hG_{24}}$. \\\\
\noindent
The HFPSS for $E_{C}^{hG_{24}}\wedge A_1$ is a spectral sequence of modules over that of (\ref{HFPSS_S0}):
\begin{equation}\phantomsection \label{HFPSS_A_1} \mathrm{H}^{s}(G_{24}, (E_{C})_{t}A_1)\Longrightarrow \pi_{t-s}(E_{C}^{hG_{24}}\wedge A_1).
\end{equation}
In Section \ref{HFPSS-E_2-term}, we will compute $\mathrm{H}^{*}(G_{24}, (E_{C})_{*}A_1)$ as a module over a certain subalgebra of $\mathrm{H}^{*}(G_{24}, (E_{C})_{*})$. Let $\pi : (E_{C})_{*}\rightarrow \FF_{4}[u^{\pm 1}]$ be the quotient of $(E_{C})_{*}$ by the maximal ideal $(2,u_{1})$. As the ideal $(2,u_{1})$ is preserved by the action of $\mathbb{S}_{C}$, the ring $\FF_{4}[u^{\pm 1}]$ inherits an action of $\mathbb{S}_{C}$, and so of its subgroup $G_{24}$. We need the computation of the ring structure of  $\mathrm{H}^{*}(G_{24}, \FF_{4}[u^{\pm 1}])$, which is due to Hans-Werner Henn, see \cite{Bea17}, Appendix A.
\begin{Proposition} There are classes 
$z\in \mathrm{H}^{4}(G_{24}, \FF_{4}[u^{\pm 1}]_{0})$, $a\in \mathrm{H}^{1}(G_{24}, \FF_{4}[u^{\pm 1}]_{2})$, $b\in \mathrm{H}^{1}(G_{24}, \FF_{4}[u^{\pm 1}]_{4})$, $v_{2}\in \mathrm{H}^{0}(G_{24}, (\FF_{4}[u^{\pm1}])_{6})$ such that there is an isomorphism of graded algebras
$$\mathrm{H}^{*}(G_{24}, \FF_{4}[u^{\pm 1}]) \cong \mathbb{F}_{4}[v_{2}^{\pm 1}, z, a, b]/(ab, b^{3}=v_{2}a^{3}).$$
\end{Proposition}
\begin{Proposition}\phantomsection \label{induced-coh} The homomorphism of graded algebras $$\mathrm{H}^{*}(G_{24}, E_{C*}) \rightarrow \mathrm{H}^{*}(G_{24}, \FF_{4}[u^{\pm 1}])$$ induced by the projection $(E_{C})_{*}\rightarrow \FF_{4}[u^{\pm 1}]$ sends $\eta$ to $a$,  $\nu$ to $b$, $\overline{\kappa}$ to $v_{2}^{4}z$, and $\Delta$ to $v_{2}^{4}$.
\end{Proposition}

%\textbf{Remark} This comparison does not tell us anything on how the HFPSS for $E_{C}^{G_{24}}\wedge A_1$ behaves in comparison to the ASS for $tmf\wedge A_1$; with this we mean that we do not know any apparent way to derive differentials in the HFPSS from those of the ASS. However, this comparison gives us partial information about $\pi_{*}(E_{C}^{hG_{24}}\wedge A_1)$. Together with other means we can deduce differentials in the HFPSS.
\subsection{On the cohomology groups $\mathrm{H}^{*}(G_{24},(E_{C})_{*}(A_1))$}\phantomsection \label{HFPSS-E_2-term}
\noindent
We first determine $(E_{C})_{*}(A_1)$ using the cofiber sequences through which $A_1$ are defined. The cofiber sequence
$\Sigma S^{0}\xrightarrow{\eta}S^{0}\rightarrow C_{\eta}$ gives rise to a short exact sequence of $E_{C}$-homology $$0\rightarrow (E_{C})_{*}\rightarrow (E_{C})_{*}(C_{\eta})\rightarrow (E_{C})_{*}(S^{2})\rightarrow 0,$$ since $(E_{C})_{*}$ is concentrated in even degrees. Hence, as an $(E_{C})_{*}$-module $$(E_{C})_{*}(C_{\eta})\cong  \mathbb{W}(\FF_{4})[[u_{1}]][u^{\pm 1}]\{e_{0},e_{2}\},$$ where $e_{0}$ is the image of $1\in (E_{C})_{0}$ and $e_{2}$ is a lift of $\Sigma^{2}1\in (E_{C})_{2}(S^{2})$. Next, the long exact sequence in $E_{C}$-homology associated to $ C_{\eta}\xrightarrow{2} C_{\eta}\rightarrow Y$ is the short exact sequence $$0\rightarrow (E_{C})_{*}(C_{\eta})\xrightarrow{\times 2}(E_{C})_{*}(C_{\eta}) \rightarrow(E_{C})_{*}(Y)\rightarrow 0 $$ since multiplication by $2$ on $(E_{C})_{*}(C_{\eta})\cong\mathbb{W}(\FF_{4})[[u_{1}]][u^{\pm 1}]\{e_{0},e_{2}\}$ is injective. Therefore $$(E_{C})_{*}(Y)\cong \FF_{4}[[u_{1}]][u^{\pm 1}]\{e_{0},e_{2}\}.$$
Now $A_1$ is the cofiber of some $v_{1}$-self map of $Y$: $\Sigma^{2}Y\xrightarrow{v_{1}}Y\rightarrow A_1$. The following lemma describe the induced homomorphism in $E_{C}$-homology of these $v_{1}$-self maps.
\begin{Lemma} \phantomsection \label{E-module} The homomorphism $(E_{C})_{*}(v_{1})$ is given by multiplication by $u_{1}u^{-1}$. Therefore, $$(E_{C})_{*}(A_1)\cong \FF_{4}[u^{\pm 1}]\{e_{0},e_{2}\}.$$
\end{Lemma}
\begin{proof} Let $K(1)$ be the first Morava K-theory at the prime $2$ such that $K(1)_{*}\cong \F[v_{1}^{\pm 1}]$ where $|v_{1}|=2$ and $BP$ be the Brown-Peterson spectrum at the prime $2$. There is a map of ring spectra $BP\rightarrow K(1)$ that classifies the complex orientation of $K(1)$. Recall that the coefficient ring of $BP$ is given by $$BP_*\cong \Z_{(2)}[v_1, v_2, ...],$$ where $|v_i| = 2(2^{i}-1)$, see \cite{Ada74}, Part II. The induced homomorphism of coefficient rings sends $v_{1}$ to $v_{1}$. The map $BP\rightarrow K(1)$ gives rise to the commutative diagram 
$$\xymatrix{ BP_{*}(\Sigma^{2}Y) \ar[rr]^{BP_{*}(v_{1})}\ar[d]&& BP_{*}(Y)\ar[d]\\
		    K(1)_{*}(\Sigma^{2}Y) \ar[rr]^{K(1)_{*}(v_{1})}&& K(1)_{*}(Y)	
}$$
By definition, a $v_{1}$-self-map of $Y$ induces in $K(1)$-homology multiplication by $v_{1}$. The above diagram forces, then for degree reasons, that $BP_{*}(v_{1})$ is given by  multiplication by $v_{1}\in BP_{2}$. Now, let $ c: BP\rightarrow E_{C}$ be the map of ring spectra that classifies the $2$-typification of the formal group law of $E_{C}$. One can show that the $2$-series of the latter has leading term $u_{1}u^{-1}x^{2}$ modulo $(2)$, see \cite{Bea17}, Proposition 6.1.1. This implies that the induced homomorphism $c_{*}: BP_{*}\rightarrow (E_{C})_{*}$ sends $v_{1}$ to $u_{1}u^{-1}$ modulo $2$. By naturality,  $(E_{C})_{*}(v_{1})$ is also given by multiplication by $u_{1}u^{-1}$.
\end{proof}
\noindent
We now describe the action of $G_{24}$ on $(E_{C})_{*}(A_1)$. For any $2$-local finite spectrum $X$, the map $c$, introduced in the proof of Lemma \ref{E-module}, induces a map of ANSS 
$$\xymatrix{ \Ext_{BP_{*}BP}^{s,t}(BP_{*},BP_{*}X) \ar[r]\ar@{=>}[d] & \Ext_{(E_{C})_{*}E_{C}}^{s,t}((E_{C})_{*}, (E_{C})_{*}X)\ar@{=>}[d]\\
				\pi_{t-s}(X)\ar[r] &\pi_{t-s}(L_{K(2)}X)
}$$ where $(E_{C})_{*}E_{C}$ stands for $\pi_{*}(L_{K(2)}(E_{C}\wedge E_{C}))$. By Morava's change-of-ring theorem (see \cite{Dev95}), one has $$\Ext_{(E_{C})_{*}E_{C}}^{s,t}((E_{C})_{*}, (E_{C})_{*})\cong \mathrm{H}_c^{s}(\mathbb{G}_{C}, (E_{C})_{t}).$$ 
Now the map $c$ induces a map of short exact sequences
$$\xymatrix{ 0\ar[r]&BP_{*}\ar[r]^{\times 2}\ar[d]^{c_{*}}&BP_{*}\ar[r]\ar[d]^{c_{*}}&BP_{*}/(2)\ar[r]\ar[d]^{c_{*}}&0\\
		0\ar[r]&E_{C*}\ar[r]^{\times 2}&E_{C*}\ar[r]&E_{C*}/(2)\ar[r]&0.
}$$
Therefore, we obtain the commutative diagram 
$$\xymatrix { \mathrm{Ext}^{0,*}_{BP_{*}BP}(BP_{*},BP_{*}/2) \ar[d]^{c_{*}}\ar[r]^{\delta_{BP}}&\mathrm{Ext}^{1,*}_{BP_{*}BP}(BP_{*},BP_{*})\ar[d]^{c_{*}}\\
		 \mathrm{H}^{0}_{c}(\G_{C},E_{C*}/2)\ar[r]^{\delta_{E_{C}}}&\mathrm{H}^{1}_{c}(\G_{C},E_{C*}),
}$$
where $\delta_{BP}$ and $\delta_{E_C}$ denote the respective connecting homomorphisms.
By \cite{Rav86}, Theorem 4.3.6, one has that $$\mathrm{Ext}^{0,2}_{BP_{*}BP}(BP_{*},BP_{*}/2) = \Z_{(2)}\{v_{1}\}\ \mbox{and}\ \delta_{BP}(v_{1}) = \eta\in \Ext^{1,2}_{BP_{*}BP}(BP_{*}, BP_{*}),$$ where $\eta$ is a permanent cycle representing the Hopf element $\eta\in \pi_{1}(S^{0})$. By naturality, $\delta_{E_{C}}(v_{1}) = c_{*}(\eta)$. Therefore, as a cocycle in $Map_{c}(\G_{C},(E_{C})_{2})$, $c_{*}(\eta)$ is given by 
$$\mathbb{G}_{C}\rightarrow (E_{C})_{2},\ \ g\mapsto \frac{g(v_{1})-v_{1}}{2}$$ 
On the other hand, let us consider the short exact sequence $$0\rightarrow E_{C*}\rightarrow E_{C*}(C_{\eta})\rightarrow E_{C*}(S^{2})\rightarrow 0$$ representing the class $c_{*}(\eta)$, so that the connecting homomorphism sends $\Sigma^{2} 1$ to $c_{*}(\eta)$.
Thus, if $e_{2}$ is a lift of $\Sigma^{2}1$ in $E_{C*}(C_{\eta})$, then $c_{*}(\eta)$ is represented by the cocycle 
$$\mathbb{G}_{C}\rightarrow (E_{C})_{2}, \ \ g\mapsto g(e_{2})-e_{2}.$$ 
This implies that one can modify $e_{2}$ so that $$\frac{g(v_{1})-v_{1}}{2} = g(e_{2})-e_{2}\ \forall \ g\in \G_{C}.$$ With this choice of $e_{2}$, we see that $E_{C*}(C_{\eta}) = E_{C*}\{e_{0},e_{2}\}$ and the action of $\G_{C}$ on $e_{2}$ is given by the formula 
\begin{equation}\phantomsection \label{action} g(e_{2}) = e_{2} + \frac{g(v_{1}) - v_{1}}{2} e_{0}
\end{equation}
Note that when determining $(E_{C})_{*}(A_1)$, we did not specify any lift $e_{2}$ of $\Sigma^{2} 1$. From now on, we will fix $e_{2}$ such that the formula of (\ref{action}) holds.
\begin{Proposition}\label{action G24 on EA1} As an $(E_{C})_*$-module, $(E_{C})_*(A_1)$ is isomorphic to $ \FF_{4}[u^{\pm 1}]\{e_{0}, e_{2}\} $ and the action of $G_{24}$ is given by 
$$\omega(u^{-1}) = \zeta^{2}u^{-1}, \ \ \ \omega(e_{0}) =e_{0}, \ \ \ \omega(e_{2}) = e_{2}$$
$$i(u^{-1}) = u^{-1}, \ \ \ i(e_{0}) = e_{0}, \ \ \ i(e_{2}) = e_{2} +  u^{-1}e_{0}$$
$$j(u^{-1}) = u^{-1}, \ \ \ j(e_{0}) = e_{0}, \ \ \ j(e_{2}) = e_{2} +  \zeta^{2}u^{-1}e_{0}$$
$$k(u^{-1}) = u^{-1}, \ \ \ k(e_{0}) = e_{0}, \ \ \ k(e_{2}) = e_{2} + \zeta u^{-1}e_{0}$$
\end{Proposition}
\begin{proof} The first part of the statement is the content of Lemma \ref{E-module}. The second part follows from the action of $G_{24}$ on $v_{1}$ given in Theorem \ref{actionG_24}  and the formula (\ref{action}).
\end{proof}

%\begin{Lemma} The induced map $$\mathrm{H}^{1}_{c}(G_{24}, E_{C*})\rightarrow \mathrm{H}^{1}_{c}(G_{24}, \FF_{4}[u^{\pm 1}])$$ given by the quotient $$p : E_{C*}\rightarrow \FF_{4}[u^{\pm 1}]$$ sends $\eta$ to $a$. 
%\end{Lemma}
%\begin{proof} In effect, the image of $\eta$ is represented by the following cocycle in $Map (G_{24}, \FF_{4}[u^{\pm1}])$
%$$\xymatrix{ G_{24} \ar[rr]&& \FF_{4}[u^{\pm 1}] \\	
%			i \ar[rr]&&	p(\frac{i(v_{1})-v_{1}}{2}) =   \xi u^{-1}		
%}$$ 
%
%In particular, it is not a coboundary because $Q_{8}$ acts trivially on $\FF_{4}[u^{\pm 1}]$. Therefore, up to a nontrivial scalar, $\eta$ maps to $a$. 
%\end{proof}

\begin{Corollary} \phantomsection \label{G-SS}$E_{C*}(A_1)$ sits in a non-split short exact sequence of $G_{24}$-modules 
\begin{equation}\phantomsection \label{G-SES}0\rightarrow \FF_{4}[u^{\pm 1}]\{e_{0}\}\rightarrow E_{C*}(A_1)\rightarrow\FF_{4}[u^{\pm 1}]\{e_{2}\}\rightarrow 0.
\end{equation}
\end{Corollary}
\begin{proof} This is immediate in view of the explicite description of the action of $G_{24}$ on $(E_{C})_{*}A_1$.
\end{proof}
\noindent
The cohomology group $\mathrm{H}^{*}(G_{24}, \FF_{4}[u^{\pm 1}]\{e_{i}\}) \ i\in \{0, 2\}$ is free of rank one as a module over $\mathrm{H}^{*}(G_{24}, \FF_{4}[u^{\pm 1}])$. For $i\in \{0,2\}$, we choose the generators $e[0,i] \in$ $\mathrm{H}^{0}(G_{24}, (\FF_{4}[u^{\pm 1}]\{e_{i}\})_{i})$ of these modules.
\begin{Corollary} The connecting homomorphism induced from the short exact sequence (\ref{G-SES}) in Corollary \ref{G-SS} $$\mathrm{H}^{*}(G_{24}, \FF_{4}[u^{\pm 1}]\{e_{2}\})\xrightarrow{\delta} \mathrm{H}^{*+1}(G_{24}, \FF_{4}[u^{\pm 1}]\{e_{0}\})$$ is $\mathrm{H}^{*}(G_{24}, \FF_{4}[u^{\pm 1}])$-linear and sends $e[0,2]$ to $ae[0,0]$ up to a unit of $\FF_{4}$, where, as a reminder, $a\in \H^1(G_{24}, (\FF_4[u^{\pm 1}])_{2})$.  
\end{Corollary}
\begin{proof} That $\delta$ is $\mathrm{H}^{*}(G_{24}, \FF_{4}[u^{\pm 1}])$-linear is a well-known property of the connecting homomorphism (See \cite{Bro82}, V.3). Next,  since the short exact sequence in Corollary \ref{G-SS} does not split, the connecting homomorphism $\delta$ sends $e[2,0]$ to a non-trivial class and hence to $ae[0,0]$ up to a unit of $\FF_{4}$.
\end{proof}
\noindent
 Using the description of $\mathrm{H}^{*}(G_{24}, \mathbb{F}_{4}[u^{\pm 1}])$ and the long exact sequence associated to the short exact sequence of Corollary \ref{G-SS}, we obtain the following description of $\mathrm{H}^{*}(G_{24},(E_{C})_{*}(A_1))$:
\begin{Proposition} As a module over $\mathrm{H}^{*}(G_{24},\FF_{4}[u^{\pm}])$, there is an isomorphism $$\mathrm{H}^{*}(G_{24},(E_{C})_{*}(A_1))\cong\mathbb{F}_{4}[v_{2}^{\pm 1}, z, a,b]/(a,b^3)\{e[0,0],e[1,5]\},$$
where $e[0,0]\in \mathrm{H}^{0}(G_{24},(E_{C})_{0}(A_1))$ and $e[1,5]\in\mathrm{H}^{1}(G_{24},(E_{C})_{6}(A_1)).$
\end{Proposition}
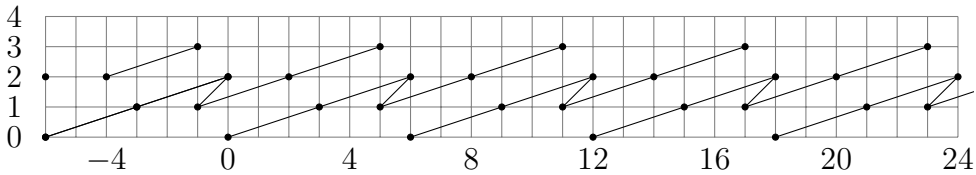
\begin{figure}[h!]
\begin{tikzpicture}[scale=0.4]
\clip(-1.5,-1.5) rectangle (30.5,4.5);
\draw[color=gray] (0,0) grid [step=1] (30,4);
\foreach \n in {-4,0,...,24}
{
\def\nn{\n-0}
\node[below] at (\nn+6,0) {$\n$};
}
\foreach \s in {0,1,...,4}
{\def\ss{\s-0};
\node [left] at (-0.4,\ss,0){$\s$};
}
\draw [fill] ( 0.00, 0.00) circle [radius=0.1];
\draw [fill] (3,1) circle [radius=0.1];
\draw [fill] (6,2) circle [radius=0.1];
\draw [-] (0,0)--(3,1);
\draw [-] (3,1)--(6,2);
%\draw[->,red] (11.8,0.2)--(11.2,2.8);
\foreach \t in {0}
\foreach \s in {0,6,...,24}
{\def\ss{\s-0};
\draw [fill]  (\s-\t,\t) circle [radius=0.1];
\draw [fill] (\s+3-\t,1+\t) circle [radius=0.1];
\draw [-] (\s-\t,\t)--(\s+3-\t,1+\t);
\draw[fill] (\s+5-\t,1+\t) circle [radius=0.1];
\draw[fill] (\s+8-\t,2+\t) circle [radius=0.1];
\draw[-] (\s+5-\t,1+\t)--(\s+8-\t,2+\t); 
\draw [fill] (\s-\t+6,2+\t) circle [radius=0.1];
\draw [-] (\s+3-\t,1+\t)--(\s+6-\t,2+\t);
\draw[fill] (\s+11-\t,3+\t) circle [radius=0.1];
\draw[-] (\s+8-\t,2+\t)--(\s+11-\t,3+\t); 
\draw [-] (\s+5-\t,1+\t)--(\s-\t+6,2+\t);
}
\draw[fill] (0,2) circle [radius=0.1];
\draw[fill] (2,2) circle [radius=0.1];
\draw[fill] (5,3) circle [radius=0.1];
\draw[-] (2,2)--(5,3);
%\foreach \t in {0,4,...,24}
%\foreach \s in {0,24,...,72}
%{\draw[->,red] (12+\s-\t,\t)--(11+\s-\t,\t+3);
%\draw[->,red] (18+\s-\t,\t)--(17+\s-\t,\t+3);
%}
%\foreach \t in {0,8,...,24}
%\foreach \s in {0,24,...,72}
%{\draw[->,blue] (24+\s-\t,\t)--(23+\s-\t,\t+5);
%\draw[->,blue] (30+\s-\t,\t)--(29+\s-\t,\t+5);
%\draw[->,blue] (27+\s-\t,1+\t)--(26+\s-\t,\t+6);
%\draw[->,blue] (29+\s-\t,1+\t)--(28+\s-\t,\t+6);
%\draw[->,blue] (33+\s-\t,1+\t)--(32+\s-\t,\t+6);
%\draw[->,blue] (35+\s-\t,1+\t)--(34+\s-\t,\t+6);
%\draw[-\rangle,blue] (39+\s-\t,1+\t)--(38+\s-\t,\t+6);figure
%\draw[-\rangle,blue] (41+\s-\t,1+\t)--(40+\s-\t,\t+6);
%\draw[-\rangle,blue] (44+\s-\t,2+\t)--(43+\s-\t,\t+7);
%\draw[-\rangle,blue] (45+\s-\t,1+\t)--(44+\s-\t,\t+6);
%\draw[-\rangle,blue] (47+\s-\t,1+\t)--(46+\s-\t,\t+6);
%}
%\foreach \t in {4,12}
%\foreach \s in {0}
%{\draw[-\rangle,blue] (\s-\t,\t)--(\s-\t-1,\t+5);
%\draw[-\rangle,blue] (2,4)--(1,9);
%\draw[-\rangle,blue] (3+\s-\t,1+\t)--(2+\s-\t,\t+6);
%\draw[-\rangle,blue] (9+\s-\t,1+\t)--(8+\s-\t,\t+6);
%\draw[-\rangle,blue] (15+\s-\t,1+\t)--(14+\s-\t,\t+6);
%\draw[-\rangle,blue] (17+\s-\t,1+\t)--(16+\s-\t,\t+6);
%\draw[-\rangle,blue] (21+\s-\t,1+\t)--(20+\s-\t,\t+6);
%\draw[-\rangle,blue] (5+\s-\t,1+\t)--(4+\s-\t,\t+6);
%\draw[-\rangle,blue] (11+\s-\t,1+\t)--(10+\s-\t,\t+6);
%\draw[-\rangle,blue] (20+\s-\t,2+\t)--(19+\s-\t,\t+7);
%\draw[-\rangle,blue] (21+\s-\t,1+\t)--(20+\s-\t,\t+6);
%\draw[-\rangle,blue] (23+\s-\t,1+\t)--(22+\s-\t,\t+6);}
%\node[left] at (19.8,4) {$\overline{\kappa}$};
\end{tikzpicture}
\caption{ $\mathrm{H}^{s}(G_{24}, (E_{C})_{t}(A_1))$ depicted in the coordinate (s, t-s))}
\end{figure}
\noindent
The above proposition also gives the action of $\mathrm{H}^{*}(G_{24}, (E_{C})_{*})$ on $\mathrm{H}^{*}(G_{24}, (E_{C})_{*}A_1)$. In effect, 
the action of $E_{C*}$ on $E_{C*}(A_1)$ factors though $\FF_{4}[u^{\pm 1}]$ via $ E_{C*}\xrightarrow{\pi} \FF_{4}[u^{\pm 1}]$. As a consequence, the action of $\mathrm{H}^{*}(G_{24},E_{C*})$ on $\mathrm{H}^{*}(G_{24},E_{C*}(A_1))$ factors through the induced homomorphism in cohomology of $G_{24}$. In particular, il follows from Proposition \ref{induced-coh} that the classes $\Delta, \overline{\kappa}, \nu$ act on $\mathrm{H}^{*}(G_{24},E_{C*}(A_1))$ as $v_{2}^{4}, v_{2}^{4}z, b$ do, respectively.

\subsection{Differentials of the homotopy fixed point spectral sequence for $E_{C}^{hG_{24}}\wedge A_1$}
\noindent
The HFPSS for $E_C^{hG_{24}}\wedge A_1$ has the following features. The spectrum $E_C\wedge A_1$ is a $G_{24}$-$E_C$-module in the sense that $E_C\wedge A_1$ is an $E_C$-module and the structure maps are $G_{24}$-equivariant. This guarantees that the HFPSS for $E_C^{hG_{24}}\wedge A_1$ is a module over that for $\EG24$. In particular, all differentials are $\overline{\kappa}$-linear. This element plays a central role here: the group $G_{24}$ is a group with periodic cohomology (see \cite{Bro82}, Chapter VI) and $\overline{\kappa}\in \H^4(G_{24}, (E_C)_*)$ is a cohomological periodicity class. These features induce more structure on the HFPSS. 
\begin{Definition}\phantomsection\label{regular pair} Let $R$ be a ring spectrum and $G$ be a finite group acting on $R$ by maps of ring spectra. The pair $(G,R)$ is said to be regular \index{regular pair} if $G$ is a group with periodic cohomology and there exists a cohomological periodicity class $u\in \H^*(G, R_*)$ which is a permanent cycle in the HFPSS for $R^{hG}$.
\end{Definition}
\begin{Lemma} Let $(G,R)$ be a regular pair as in Definition \ref{regular pair} and $X$ be a $G$-$R$ spectrum. Suppose $u\in \H^k(G, R_*)$ is a cohomological periodicity class which is a permanent cycle in the HFPSS for $R^{hG}$. Then the $\mathrm{E}_r$-term of the  HFPSS for $X^{hG}$ has the following properties:
\begin{itemize}
	\item[(i)] All classes of cohomological filtration at least $k$ are divisible by $u$;
	\item[(ii)] All classes of cohomological filtration at least $r$ are $u$-free.
\end{itemize}
\end{Lemma}
\begin{proof} We will prove by induction on $r$ that the $\E_r$-term of the HFPSS for $X^{hG}$ has the properties $(i)$ and $(ii)$.
The $\E_2$-term is isomorphic to $\H^*(G, \pi_*(X))$. We recall that the natural map from the cohomology to the Tate cohomology $\iota: \H^s(G, \pi_tX)\rightarrow \hat{\H}^s(G, \pi_t(X))$ is an epimorphism and is an isomorphism when $s>0$, see \cite{Bro82}, Chapter VI. Because $G$ has periodic cohomology, we have $$\hat{\H}^s(G, \pi_tX) \cong \hat{\H}^s(G, \pi_tX)[u^{-1}],$$ which means that the group $\hat{\H}^s(G, \pi_tX)$ is $u$-free and is divisible by $u$. Since $\iota: \H^s(G, \pi_tX)\rightarrow \hat{\H}^s(G, \pi_t(X))$ is an isomorphism when $s>0$, all classes of positive cohomological degree of $\H^s(G, \pi_tX)$ are $u$-free. 

Now suppose $x$ is a class of $\H^s(G, \pi_tX)$  with $s\geq k$. Then the class $u^{-1}\iota(x)\in \hat{\H}^{s-k}(G,\pi_tX)$ has a pre-image $y\in \H^{s-k}(G,\pi_tX)$ (because $s-k\geq 0$), i.e. $$\iota(y) = u^{-1}\iota(x).$$ This implies that $$\iota(uy) = u\iota(y) = \iota(x),$$ and thus since $s>0$, $$uy=x.$$
%$$\xymatrix{ \H^s(G,\pi_tX) \ar[r]^{\cong} & \hat{\H}^s (G, \pi_tX)\\
%			\H^{s-k}(G,\pi_tX)\ar[u]^{\cup u}\ar@{->>}[r] & \hat{\H}^{s-k}(G,\pi_tX)\ar[u]_{\cup u}
%}$$
%the lower horizontal map is surjective because $s-k\geq$
Thus, the $\E_2$-term has the properties $(i)$ and $(ii)$. Suppose that the $\E_r$-term satisfies $(i)$ and $(ii)$. Let $[x]\in \E_{r+1}$ be a non-trivial class represented by $x\in \E_r$. Suppose that $x$ has its cohomological filtration $s \geq k$. By the induction hypothesis, there exists $y\in \E_r^{s-k, *}$ such that $uy = x$. We show that $y$ is a $d_r$-cycle. Because $x$ is a $d_r$-cycle, we have by $u$-linearity that $ud_r(y) = d_r(uy) = d_r(x) = 0$. However, the cohomological filtration of $d_r(y)$ is at least $r$, and so it is $u$-free by the induction hypothesis, and so $d_r(y)=0$. Therefore, $[x]$ is divisible by $u$.\\\\
\noindent
Now we prove that $\E_{r+1}$ has the property $(ii)$. Suppose that $[x]$ is $u$-torsion and has its cohomological filtration at least $r+1$. Without loss of generality, we can assume that $u[x]=0$. Then there exists $y\in \E_r$ such that $d_{r}(y) = u x$. The cohomological filtration of $y$ is at least to $r+1 + k - r = k+1$, hence $y$ is divisible by $u$, i.e., there exists $z\in \E_r$ such that $uz=y$, and then by $u$-linearity, $$ud_r(z) = d_r(uz) = d_r(y) = ux.$$ However, $d_r(z) - x$ has cohomolgical filtration at least $r+1$, it must be $u$-free by hypothesis $(ii)$, hence is equal to zero, i.e.,  $[x]$ is trivial in $\E_{r+1}$.\\\\
\noindent
We conclude that the $\E_{r+1}$-term satisfies $(i)$ and $(ii)$, thus finishing the proof by induction.
\end{proof}
\begin{Corollary}\phantomsection\label{Organ} Let $(G,R)$ be a regular pair and $X$ be a $G$-$R$ spectrum. Suppose $u\in \H^k(G, R_*)$ is a cohomological periodicity class which is a permanent cycle in the HFPSS for $R^{hG}$. Then we have, in the HFPSS for $X^{hG}$,
\begin{itemize}
	\item[1.] At the $\E_r$-term, $u$-torsion classes are permanent cycles.
	\item[2.] Any $u$-free tower is truncated by at most one other $u$-free tower by the same differential. More precisely, if $x$ is a class of cohomological filtration less than $k$, then there exists at most one class $y$ of cohomological filtration less than $k$ such that there exists an unique integer $l$ and a unique integer $r$ such that $d_r(u^my) = u^{m+l}x$ for all non-negative integers $m$. Moreover, all classes $u^{i}x$ for $i\in\{0, 1, ..., m-1\}$ survive the spectral sequence.
	\item[3.] Suppose some power of $u$ is hit by a differential in the HFPSS for $R^{hG}$. Then any $u$-free tower consisting of permanent cycles is truncated by a unique $u$-free tower. Moreover, the HFPSS has a horizontal vanishing line.
	\item[4.] Every element of $\pi_*(X^{hG})$ that is detected in filtration at least $k$ is divisible by $\overline{u}$ where $\overline{u}$ is an element of $\pi_*(R^{hG})$ detected by $u$. 
\end{itemize}
\end{Corollary}
\noindent
\begin{Remark} This situation turns out to be abundant once the group in question is a group with periodic cohomology. For example, all finite subgroups of $\G_C$ have these properties.
\end{Remark}
\noindent
We return to the HFPSS for $E_C^{hG_{24}}\wedge A_1$. We will call the set  $\{\overline{\kappa}^{l}x|l\in\N\}$ associated to a class $x$ in some page of the HFPSS the $\overline{\kappa}$-family of that class. 
\noindent
The following proposition gives us the horizontal vanishing line of the HFPSS for $E_{C}^{hG_{24}}\wedge A_1$.
\begin{Proposition}\phantomsection \label{Degen} The HFPSS for $E_{C}^{hG_{24}}\wedge A_1$ has a horizontal vanishing line of height $23$, i.e., $\E_{24}^{s,t} = 0$ if $s>23$. As a consequence, it collapses at the $\mathrm{E}_{24}$-term.
\end{Proposition}
\begin{proof} As $\overline{\kappa}^{6}=0$ in $\pi_{*}(E_{C}^{hG_{24}})$, the class $\overline{\kappa}^{6}$ must be hit by a differential which is of length at most $23$. This is because $\overline{\kappa}^{6}$ has cohomological filtration $24$ and all even differentials are trivial. Hence $\overline{\kappa}^{6}$ is trivial in the $\mathrm{E}_{24}$-term of the HFPSS for $E_{C}^{hG_{24}}$. Next, because the $\mathrm{E}_{24}$-term of the HFPSS for $E_{C}^{hG_{24}}\wedge A_1$ is a module over that for $E_{C}^{hG_{24}}$, the class $\overline{\kappa}^{6}$ acts trivially on the $\mathrm{E}_{24}$-term of the HFPSS for $E_{C}^{hG_{24}}\wedge A_1$. Since all classes which are not a multiple of $\overline{\kappa}$ have cohomological filtration at most $3$, the HFPSS has the horizontal vanishing line of height $23$.
\end{proof}
\begin{Proposition} \phantomsection \label{perm24}The following classes are permanent cycles $$e[0,0], e[1,5], e[0,6], e[1,11], e[1,15], e[1,17], e[1,21], e[1,23].$$
\end{Proposition}
\begin{proof} Firstly, the class $e[0,0]$ is a permanent cycle because it detects the inclusion $S^{0}\rightarrow A_1$ into the bottom cell of $A_1$. Next, we recapitulate, in the following table, the associated graded object with respect to the induced Adams filtration on the groups $\pi_{*}(tmf\wedge A_1)/(\overline{\kappa})$ in the following stems.
$$\begin{tabular}{| c | c | c | c |c | c|c }
Dim&6&15&17&21&23\\
\hline
Value& $\F\oplus \F$&$\F$&$\F$&$\F$&$\F\oplus \F $ \\	
\end{tabular}$$
By Corollary \ref{Cor_Compare}, the groups $\pi_{*}(E_{C}^{hG_{24}}\wedge A_1)/(\overline{\kappa})$ in these dimensions must have order twice as big as the respective groups. Inspection in the $\mathrm{E}_{2}$-term of the HFPSS through dimensions from $0$ to $23$ and in cohomological filtration less than $4$ show that the classes $ e[0,6], e[1,15], e[1,21], e[1,23]$ are permanent cycles. \\\\
\noindent
Note that the groups $\pi_{0}(tmf\wedge A_1)$ and $\pi_{6}(tmf\wedge A_1)$ are annihilated by $\eta$. This means that $e[0,0]$ and $e[0,6]$ detects two elements which are annihilated by $\eta$. It follows that the Toda brackets $\langle \nu,\eta, e[0,0]\rangle$ and  $\langle \nu,\eta, e[0,6]\rangle$ can be formed. By juggling, $$\eta\langle \nu,\eta, e[0,0]\rangle = \langle \eta,\nu,\eta\rangle e[0,0]= \nu^{2}e[0,0]$$ and $$\eta\langle \nu,\eta, e[0,6]\rangle = \langle \eta,\nu,\eta\rangle e[0,6]= \nu^{2}e[0,6].$$ 
Observe that $\nu^{2}e[0,0]$ and $\nu^{2}e[0,6]$ are nontrivial and are detected in cohomological filtration $2$. Consequently, both $\langle \nu,\eta, e[0,0]\rangle$ and $\langle \nu,\eta, e[0,6]\rangle$ are nontrivial and are represented by classes in cohomological filtration at most $1$. Therefore $e[1,5]$ and  $e[1,11]$ are permanent cycles. \\\\
\noindent
The unique nontrivial element of $\pi_{11}(tmf\wedge A_1)/(\overline{\kappa})$ is annihilated by $\nu^{2}$. This implies that the class $\nu^{2}e[1,11]$ is the target of some differential. Since $\pi_{17}(E_{C}^{hG_{24}}\wedge A_1)/(\overline{\kappa})$ has order at least equal to $4$, the class $e[1,17]$ must be a permanent cycle representing the only element in dimension $17$ of $\pi_{*}(E_{C}^{hG_{24}}\wedge A_1) /(\overline{\kappa})$.
\end{proof}
\noindent
$\underline{d_{3}-\mbox{differentials}}$
\begin{Proposition} As a module over $\FF_{4}[\Delta^{\pm 1}, \overline{\kappa}, \nu]/(\nu^{3})$, the term $\mathrm{E}_{2}=\mathrm{E}_{3}$ is free on the generators
 \begin{equation}\phantomsection \label{genE_3}
 e[0,0], e[1,5], e[0,6], e[1,11], e[0,12], e[1,17], e[0,18], e[1,23].
 \end{equation}
\end{Proposition}
\begin{Proposition} \phantomsection \label{d_3}The $d_{3}$-differential in the HFPSS for $E_{C}^{hG_{24}}\wedge A_1$ is trivial on all of the generators of (\ref{genE_3}) with the exception of 
\begin{itemize} 
\item[i)]$d_{3}(e[0,12])=\nu^{2}e[1,5]$
\item[ii)]$d_{3}(e[0,18])=\nu^{2}e[1,11].$
\end{itemize}
\end{Proposition}
\begin{proof} That $e[0,0], e[1,5], e[0,6], e[1,11], e[1,17], e[1,23]$ are $d_{3}$-cycles follows from Proposition \ref{perm24}. For the two other classes, the proof of Proposition \ref{perm24} implies that the elements $\Theta(e[1,5])$ and $\Theta(e[2,11])$ are detected by $e[1,5]$ and $e[1,11]$, respectively. Moreover, the elements $e[1,5]$ and $e[2,11]$ are annihilated by $\nu^{2}$ in $\pi_{*}(tmf\wedge A_1)$. It follows that, in the HFPSS, the classes $\nu^{2}e[1,5]$ and $\nu^{2}e[1,11]$ must be hit by some differentials. The only possibilities are $d_{3}(e[0,12])=\nu^{2}e[1,5]$ and $d_{3}(e[0,18])=\nu^{2}e[1,11]$.
\end{proof}
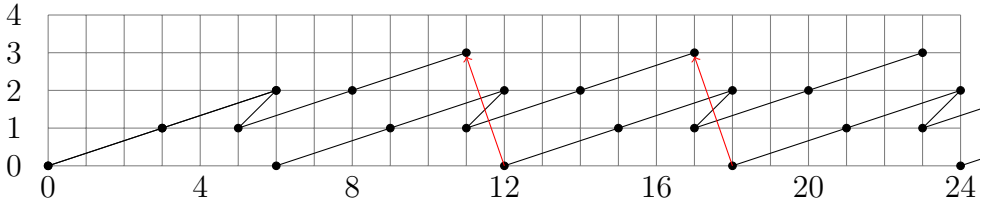
\begin{figure}[h!]
\begin{tikzpicture}[scale=0.5]
\clip(-1.5,-1.5) rectangle (24.5,4.5);
\draw[color=gray] (0,0) grid [step=1] (24,4);
\foreach \n in {0,4,...,24}
{
\def\nn{\n-0}
\node[below] at (\nn,0) {$\n$};
}
\foreach \s in {0,1,...,4}
{\def\ss{\s-0};
\node [left] at (-0.4,\ss,0){$\s$};
}
\draw [fill] ( 0.00, 0.00) circle [radius=0.1];
\draw [fill] (3,1) circle [radius=0.1];
\draw [fill] (6,2) circle [radius=0.1];
\draw [-] (0,0)--(3,1);
\draw [-] (3,1)--(6,2);
%\draw[->,red] (11.8,0.2)--(11.2,2.8);
\foreach \t in {0}
\foreach \s in {0,6,...,24}
{\def\ss{\s-0};
\draw [fill]  (\s-\t,\t) circle [radius=0.1];
\draw [fill] (\s+3-\t,1+\t) circle [radius=0.1];
\draw [-] (\s-\t,\t)--(\s+3-\t,1+\t);
\draw[fill] (\s+5-\t,1+\t) circle [radius=0.1];
\draw[fill] (\s+8-\t,2+\t) circle [radius=0.1];
\draw[-] (\s+5-\t,1+\t)--(\s+8-\t,2+\t); 
\draw [fill] (\s-\t+6,2+\t) circle [radius=0.1];
\draw [-] (\s+3-\t,1+\t)--(\s+6-\t,2+\t);
\draw[fill] (\s+11-\t,3+\t) circle [radius=0.1];
\draw[-] (\s+8-\t,2+\t)--(\s+11-\t,3+\t); 
\draw [-] (\s+5-\t,1+\t)--(\s-\t+6,2+\t);
}
%%%%%%%%
\foreach \t in {12,18}
\draw[->,red] (\t,0)--(\t-1,2.9);
\end{tikzpicture}
\caption{Differentials $d_{3}$}
\end{figure}
\begin{Corollary} As a module over $\mathbb{F}_{4}[\Delta^{\pm 1},\overline{\kappa},\nu]/(\nu^{3})$, the term $\mathrm{E}_{4}=\mathrm{E}_{5}$ is a direct sum of cyclic modules generated by the classes 
\begin{equation} \phantomsection \label{genE_5}
e[0,0],e[1,5],e[0,6],e[1,11],e[1,15], e[1,17], e[1,21], e[1,23]
\end{equation} 
with the relations 
\begin{equation}\phantomsection \label{relE_5}
\nu^{2}e[1,5] = \nu^{2}e[1,11] =\nu^{2}e[1,15] = \nu^{2}e[1,21] = 0.
\end{equation}
\end{Corollary}
\begin{proof} This is straightforward from Proposition \ref{d_3} and from the fact that $\Delta, \overline{\kappa}, \nu$ are $d_{3}$-cycles in the HFPSS for $E_{C}^{hG_{24}}$ .
\end{proof}
\noindent
\underline{$d_{5}$-\mbox{differentials}.}
We need the $d_{5}$-differential, in the HFPSS for $E_{C}^{hG_{24}}$, $d_{5}(\Delta) = \overline{\kappa}\nu$ (see \cite{Bau08}, Section 8.3).
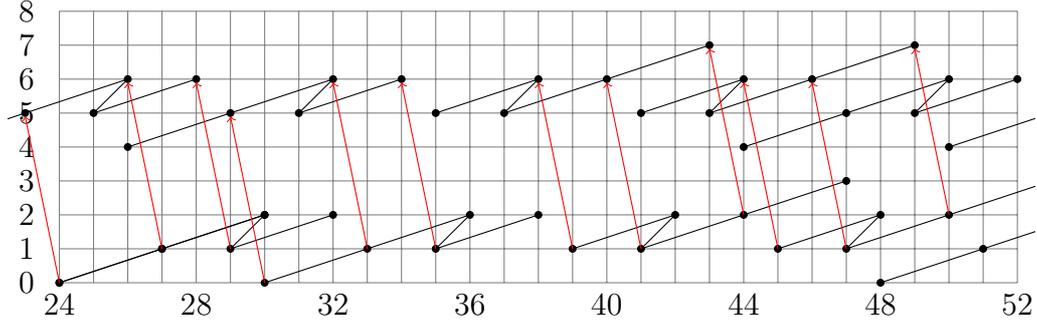
\begin{figure}[h!]
\begin{tikzpicture}[scale=0.45]
\clip(-1.5,-1.5) rectangle (28.5,8.5);
\draw[color=gray] (0,0) grid [step=1] (28,8);
\foreach \n in {24,28,...,54}
{
\def\nn{\n-0}
\node[below] at (\n-24,0) {$\n$};
}
\foreach \s in {0,1,...,8}
{\def\ss{\s-0};
\node [left] at (-0.4,\ss,0){$\s$};
}
\draw [fill] ( 0.00, 0.00) circle [radius=0.1];
\draw [fill] (3,1) circle [radius=0.1];
\draw [fill] (6,2) circle [radius=0.1];
\draw [-] (0,0)--(3,1);
\draw [-] (3,1)--(6,2);
%\draw[->,red] (11.8,0.2)--(11.2,2.8);
\foreach \t in {0,4}
\foreach \s in {0,6,24,30}
{\def\ss{\s-0};
\draw [fill]  (\s-\t,\t) circle [radius=0.1];
\draw [fill] (\s+3-\t,1+\t) circle [radius=0.1];
\draw [-] (\s-\t,\t)--(\s+3-\t,1+\t);
\draw[fill] (\s+5-\t,1+\t) circle [radius=0.1];
\draw[fill] (\s+8-\t,2+\t) circle [radius=0.1];
\draw[-] (\s+5-\t,1+\t)--(\s+8-\t,2+\t); 
\draw [fill] (\s-\t+6,2+\t) circle [radius=0.1];
\draw [-] (\s+3-\t,1+\t)--(\s+6-\t,2+\t);
%\draw[fill] (\s+11-\t,3+\t) circle [radius=0.1];
%\draw[-] (\s+8-\t,2+\t)--(\s+11-\t,3+\t); 
\draw [-] (\s+5-\t,1+\t)--(\s-\t+6,2+\t);
}

\foreach \t in {0,4}
\foreach \s in {12,18,36,42}
{\def\ss{\s-0};
%\draw [fill]  (\s-\t,\t) circle [radius=0.1];
\draw [fill] (\s+3-\t,1+\t) circle [radius=0.1];
%\draw [-] (\s-\t,\t)--(\s+3-\t,1+\t);
\draw[fill] (\s+5-\t,1+\t) circle [radius=0.1];
\draw[fill] (\s+8-\t,2+\t) circle [radius=0.1];
\draw[-] (\s+5-\t,1+\t)--(\s+8-\t,2+\t); 
\draw [fill] (\s-\t+6,2+\t) circle [radius=0.1];
\draw [-] (\s+3-\t,1+\t)--(\s+6-\t,2+\t);
\draw[fill] (\s+11-\t,3+\t) circle [radius=0.1];
\draw[-] (\s+8-\t,2+\t)--(\s+11-\t,3+\t); 
\draw [-] (\s+5-\t,1+\t)--(\s-\t+6,2+\t);
}

%%%%%%%%
\draw[->,red] (0,0)--(-1,4.9);
\draw[->,red] (3,1)--(2,5.9);
\draw[->,red] (5,1)--(4,5.9);
\draw[->,red] (6,0)--(5,4.9);
\draw[->,red] (9,1)--(8,5.9);
\draw[->,red] (11,1)--(10,5.9);
\draw[->,red] (15,1)--(14,5.9);
\draw[->,red] (17,1)--(16,5.9);
\draw[->,red] (20,2)--(19,6.9);
\draw[->,red] (21,1)--(20,5.9);
\draw[->,red] (23,1)--(22,5.9);
\draw[->,red] (26,2)--(25,6.9);
\end{tikzpicture}
\caption{Differentials $d_{5}$}
\end{figure}

\begin{Proposition}\phantomsection \label{E_7}
As a module over $\mathbb{F}_{4}[(\Delta^{8})^{\pm 1},\overline{\kappa},\nu]/(\overline{\kappa}\nu)$, $\mathrm{E}_{6}=\mathrm{E}_{7}$ is a direct sum of cyclic modules generated by the following classes for $i\in{0,2,4,6}$ with the respective annihilator ideal:
 
 $\begin{array}{lllllllllllllllll}
 generator& \Delta^{i}e[0,0]& \Delta^{i}e[1,5]&\Delta^{i}e[0,6]&\Delta^{i}e[1,11]\\
 ideal&(\nu^{3})&(\nu^{2})&(\nu^{3})&(\nu^{2})\\
 generator&\Delta^{i}e[1,15]&\Delta^{i}e[1,17]&\Delta^{i}e[1,21]&\Delta^{i}e[1,23]\\
 ideal&(\nu^{2})&(\nu^{3})&(\nu^{2})&(\nu^{3})\\
 generator&\Delta^{i}e[2,30]&\Delta^{i}e[2,32]&\Delta^{i}e[2,36]&\Delta^{i}e[2,38]\\
 ideal&(\nu)&(\nu)&(\nu)&(\nu)\\
 generator&\Delta^{i}e[2,42]&\Delta^{i}e[3,47]&\Delta^{i}e[2,48]&\Delta^{i}e[3,53]\\
 ideal&(\nu)&(\nu)&(\nu)&(\nu).\\
 \end{array} $
\end{Proposition}
\begin{proof} Notice that, if $x$ is a class in the $\mathrm{E}_{5}$-term, then $d_{5}(\Delta^{2k}x) = \Delta^{2k}d_{5}(x)$  $\forall k\in\Z$. This says in particular that the $\mathrm{E}_{6}$-term is $\Delta^{2}$-periodic. Next, if $x$ is a $d_{5}$-cycle and is annihilated by $\nu^{i}$, then $d_{5}(\Delta x) = \overline{\kappa}\nu x$ and $d_{5}(\Delta\nu^{i-1}x) = 0$. Together with the fact that all of the generators of (\ref{genE_5}) are permanent cycles (Proposition \ref{perm24}), it is straightforward to verify that the classes together with their annihilation ideal given in the statement of the Proposition generate the $\mathrm{E}_{6}$-term as a module over $\mathbb{F}_{4}[(\Delta^{8})^{\pm 1},\overline{\kappa},\nu]/(\overline{\kappa}\nu)$.
\end{proof}
\begin{Remark}\label{Rmk_Org} Since  $\Delta^{8}$ is a permanent cycle in the HFPSS for $E_{C}^{hG_{24}}$, the HFPSS for $E_{C}^{hG_{24}}\wedge A_1$ is linear with respect to $\Delta^{8}$. Note that all $\overline{\kappa}$-free generators in the $\mathrm{E}_{7}$-term are of the form $(\Delta^{8})^{k} x$ where $k\in \Z$ and $x$ is one of the generators listed in Proposition \ref{E_7}. Then, by Corollary \ref{Organ}, these free $\overline{\kappa}$-families pair up so that each non-permanent $\overline{\kappa}$-family truncates one and only one permanent $\overline{\kappa}$-family. By $\Delta^{8}$-linearity, among these $64$ generators, only half of them are permanent cycles and the others support a differential. It reduces the problem into two steps: first identify all permanent $\overline{\kappa}$-families, then identify by which $\overline{\kappa}$-family they are truncated.
\end{Remark}
%%%%% 0--53
\begin{figure}[h!]
\begin{tikzpicture}[scale=0.24]
\clip(-1.5,-1.5) rectangle (55,5);
\draw[color=gray] (0,0) grid [step=1] (54,4);

\foreach \n in {0,4,...,54}
{
\def\nn{\n-0}
\node[below,scale=0.5] at (\nn,0) {$\n$};
}
\foreach \s in {0,2,...,4}
{\def\ss{\s-0};
\node [left,scale=0.5] at (-0.4,\ss,0){$\s$};
}
\draw [fill] ( 0.00, 0.00) circle [radius=0.1];
\draw [fill] (3,1) circle [radius=0.1];
\draw [fill] (6,2) circle [radius=0.1];
\draw [-] (0,0)--(3,1);
\draw [-] (3,1)--(6,2);
%\draw[->,red] (11.8,0.2)--(11.2,2.8);
\foreach \t in {0}
\foreach \s in {0,6}
{\def\ss{\s-0};
\draw [fill]  (\s-\t,\t) circle [radius=0.1];
\draw [fill] (\s+3-\t,1+\t) circle [radius=0.1];
\draw [-] (\s-\t,\t)--(\s+3-\t,1+\t);
\draw[fill] (\s+5-\t,1+\t) circle [radius=0.1];
\draw[fill] (\s+8-\t,2+\t) circle [radius=0.1];
\draw[-] (\s+5-\t,1+\t)--(\s+8-\t,2+\t); 
\draw [fill] (\s-\t+6,2+\t) circle [radius=0.1];
\draw [-] (\s+3-\t,1+\t)--(\s+6-\t,2+\t);
%\draw[fill] (\s+11-\t,3+\t) circle [radius=0.1];
%\draw[-] (\s+8-\t,2+\t)--(\s+11-\t,3+\t); 
\draw [-] (\s+5-\t,1+\t)--(\s-\t+6,2+\t);
}

\foreach \t in {0}
\foreach \s in {12,18}
{\def\ss{\s-0};
%\draw [fill]  (\s-\t,\t) circle [radius=0.1];
\draw [fill] (\s+3-\t,1+\t) circle [radius=0.1];
%\draw [-] (\s-\t,\t)--(\s+3-\t,1+\t);
\draw[fill] (\s+5-\t,1+\t) circle [radius=0.1];
\draw[fill] (\s+8-\t,2+\t) circle [radius=0.1];
\draw[-] (\s+5-\t,1+\t)--(\s+8-\t,2+\t); 
\draw [fill] (\s-\t+6,2+\t) circle [radius=0.1];
\draw [-] (\s+3-\t,1+\t)--(\s+6-\t,2+\t);
\draw[fill] (\s+11-\t,3+\t) circle [radius=0.1];
\draw[-] (\s+8-\t,2+\t)--(\s+11-\t,3+\t); 
\draw [-] (\s+5-\t,1+\t)--(\s-\t+6,2+\t);
}
\draw[fill] (0,2) circle [radius=0.1];
\draw[fill] (5,3) circle [radius=0.1];
\draw[fill] (30,2) circle [radius=0.1];
\draw[fill] (32,2) circle [radius=0.1];
\draw[fill] (36,2) circle [radius=0.1];
\draw[fill] (38,2) circle [radius=0.1];
\draw[fill] (42,2) circle [radius=0.1];
\draw[fill] (47,3) circle [radius=0.1];
\draw[fill] (48,2) circle [radius=0.1];
\draw[fill] (53,3) circle [radius=0.1];
\draw[fill] (48,0) circle[radius=0.1];
\draw[fill] (51,1) circle[radius =0.1];
\draw[-] (48,0)--(51,1);
\draw[-] (51,1)--(54,2);
\draw[fill] (54,2) circle[radius=0.1];
\draw[fill] (53,1) circle[radius=0.1];
\draw[-](53,1)--(54,2);
\end{tikzpicture}
\caption{The $\mathrm{E}_{7}$-term for $s\leq 3$ and $t-s\leq 54$ }
\end{figure}
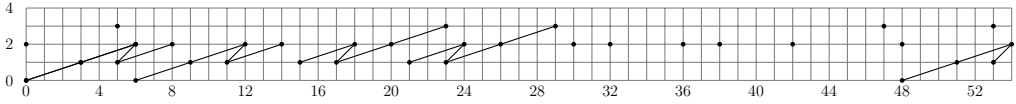
%%%%%%%%
\noindent

\begin{Proposition} The generators $$e[2,30], e[2,32], e[2,36], e[2,38], e[2,42], e[3,47], e[2,48], e[3,53]$$ are permanent cycles.
\end{Proposition}
\begin{proof} We give the proof for $e[2,30]$ and the other generators are proven in a similar manner. In the $\mathrm{E}_{6}$-term, the Massey product $\langle  \overline{\kappa},\nu,\nu^{2}e[0,0]\rangle$ can be formed. Since $d_{5}(\Delta) = \overline{\kappa}\nu$ and $\nu^3e[0,0] =0 \in \E_5$, we see that 
$$e[2,30] = \Delta\nu^2e[0,0]\in \langle \overline{\kappa},\nu,\nu^{2} e[0,0]\rangle.$$ The indeterminacy consists of $\ovk\E_6^{-2, 8} + \E_6^{0,26}\nu^2e[0,0]$, where $\E_6^{-2,8}$ is in the $\E_6$-term of the HFPSS for $E_C^{hG_{24}}\wedge A_1$ and $\E_6^{0,26}$ for $E_C^{hG_{24}}$. The latter are zero groups, hence the indeterminacy is zero. Thus, $$\langle \overline{\kappa},\nu,\nu^{2} e[0,0]\rangle = e[2,30].$$ At the level of the homotopy groups of $\pi_{*}(E_{C}^{hG_{24}}\wedge A_1)$ one can form the corresponding Toda bracket $\langle \overline{\kappa},\nu,\nu^{2} e[0,0]\rangle$ because $\nu\overline{\kappa} =0$ in $\pi_{*}(E_{C}^{hG_{24}})$ and inspection in $\pi_{*}(tmf\wedge A_1)$ tells us that $\nu^{3}e[0,0] =0$. Furthermore, all hypotheses of Moss's convergence theorem are verified. Therefore, $e[2,30]$ is a permanent cycle representing the Toda bracket $\langle e[0,0],\nu^{3},\overline{\kappa}\rangle$. For the sake of completeness, we record the Toda bracket expressions for the other elements $$\langle \overline{\kappa}, \nu, \nu e[1,5]\rangle = e[2,32],\ \langle \overline{\kappa},\nu, \nu^{2}e[0,6]\rangle = e[2,36],$$ $$\langle \overline{\kappa},\nu,\nu e[1,11]\rangle = e[2,38],\ \langle \overline{\kappa},\nu,\nu e[1,15] \rangle = e[2,42],$$ $$\langle \overline{\kappa},\nu, \nu^{2}e[1,17]\rangle = e[3,47], \ \langle \overline{\kappa}, \nu, \nu e[1,21]\rangle = e[2,48],$$ $$\langle \overline{\kappa}, \nu, \nu^{2}e[2,23]\rangle = e[3,53].$$
\end{proof}
\noindent
We have already identified $16$ out of $32$ permanent cycles. The next $16$ ones are not the same for different versions of $A_1$. The difference reflects the different behavior of the  $d_{2}$-differential in the ASS for different models of $A_1$ (see Proposition \ref{ASS_d_2_bis}). 
\begin{Proposition} In the HFPSS for all four versions of $A_1$, the following $12$ generators are permanent cycles :
$$\Delta^{2}e[0,0], \Delta^{2}e[1,5], \Delta^{2}e[0,6], \Delta^{2}e[1,11], \Delta^{2}e[1,15], \Delta^{2}e[1,17]$$
$$\Delta^{2}e[1,21], \Delta^{2}e[2,30], \Delta^{2}e[2,32], \Delta^{2}e[2,36],, \Delta^{2}e[2,42], \Delta^{2}e[3,47].$$
The remaining four permanent cycles for $A_{1}[00]$ and $A_{1}[11]$ are 
 $$\Delta^{2}e[1,23], \Delta^{2}e[2,38], \Delta^{2}e[2,48], \Delta^{2}e[3,53],$$ 
 whereas the remaining four permanent cycles for $A_{1}[10]$ and $A_{1}[01]$ are

$$\Delta^{4}e[1,15], \Delta^{4}e[0,0], \Delta^{4}e[1,5], \Delta^{4}e[2,30].$$ 

\end{Proposition}
\begin{proof} The graded associated object of the groups $\pi_{*}(tmf\wedge A_1)/(\overline{\kappa},\nu)$, with respect to the Adams filtration, in the following stems are given in the following table:
$$
\begin{tabular}{c | c | c | c | c |c | c|c|c|c|c|c|c }
Stem&48&53&54&59&63&65&69&78&80&84&90&95\\
\hline
Value&$\F\oplus\F$&$\F\oplus\F$&$\F$&$\F$&$\F$&$\F$&$\F$&$\F$&$\F$&$\F$&$\F$&$\F$ \\	
\end{tabular}
$$
In view of Corollary \ref{Cor_Compare} and Corollary \ref{Organ}, inspection in the $\mathrm{E}_{7}$-term shows that the following $12$ classes are permanent cycles in the HFPSS for all four versions of $A_1$.
$$\Delta^{2}e[0,0], \Delta^{2}e[1,5], \Delta^{2}e[0,6], \Delta^{2}e[1,11], \Delta^{2}e[1,15], \Delta^{2}e[1,17],$$
$$\Delta^{2}e[1,21], \Delta^{2}e[2,30], \Delta^{2}e[2,32], \Delta^{2}e[2,36],, \Delta^{2}e[2,42], \Delta^{2}e[3,47].$$
Next, in the ASS for $tmf\wedge A_{1}[00]$ and $tmf\wedge A_{1}[11]$, there is no differential until stem $96$. Again, inspection in the $\E_{2}$-term shows that $$\pi_{71}(tmf\wedge A_{1}[00])/(\overline{\kappa},\nu)=\pi_{71}(tmf\wedge A_{1}[11])/(\overline{\kappa},\nu) \cong \F$$ and 
$$\pi_{86}(tmf\wedge A_{1}[00])/(\overline{\kappa},\nu)=\pi_{86}(tmf\wedge A_{1}[11])/(\overline{\kappa},\nu) \cong \F$$
It follows that the classes $\Delta^{2}e[1,23]$ and $\Delta^{2}e[2,38]$ are permanent cycles in the HFPSS for $E_{C}^{hG_{24}}\wedge A_{1}[00]$ and $E_{C}^{hG_{24}}\wedge A_{1}[11]$.\\\\\
\noindent
On the other hand, in the ASS for $tmf\wedge A_{1}[10]$ and $tmf\wedge A_{1}[01]$, Lemma \ref{ASS_d_2_bis} and $g$-linearity imply that $d_{2}(g^{2}w_{2}e[4,23]) = g^{4}e[6,30]$ and $d_{2}(g^{2}w_{2}e[7,38])=g^{6}e[1,5]$. Hence, $w_{2}^{2}e[3,15]$ and $w_{2}^{2}e[6,30]$ survive to the $\E_{\infty}$-term, by sparseness. It then follows that $\Delta^{4}e[1,15]$ and $\Delta^{4}e[2,30]$ are permanent cycles in the HFPSS for $A_{1}[10]$ and $A_{1}[01]$.\\\\
\noindent
For $A_{1}[00]$ and $A_{1}[11]$, the classes $w_{2}e[9,48]$ and $w_{2}e[10,53]$ do not support differentials, by Lemma \ref{ASS_d_2_bis}, hence persist to the $E_{\infty}$-term, by sparseness. They are also not divisible neither by $\overline{\kappa}$ nor by $\nu$. Lastly, both $w_{2}e[9, 48]$ and $w_{2}e[10, 53]$ are annihilated by $\nu$. The only classes in the HFPSS that match those properties are $\Delta^{2}e[2,48]$ and $\Delta^{2}e[3,53]$, respectively. Thus, the latter are the last two of the $32$ permanent cycles in the HFPSS for $A_{1}[00]$ and $A_{1}[11]$.\\\\
\noindent
For $A_{1}[10]$ and $A_{1}[01]$, the classes $w_{2}e[9,48]$ and $w_{2}e[10,53]$ support nontrivial $d_{2}$ differentials. Thus $w_{2}^{2}e[0,0]$ and $w_{2}^{2}e[1,5]$ survive to the $E_{\infty}$-term. For degree reasons, both $w_{2}^{2}e[0,0]$ and $w_{2}^{2}e[1,5]$ are not divisible either by $\overline{\kappa}$ or by $\nu$, and moreover their multiples by $\nu$ are not divisible by $\overline{\kappa}$. In the HFPSS for $E_C^{hG_{24}}\wedge A_1[10]$ and $E_C^{hG_{24}}\wedge A_1[10]$, $\Delta^{4}e[0,0]$ and $\Delta^{4}e[1,5]$ are the only classes verifying the respective properties, hence are permanent cycles. 

\end{proof}
\noindent 
Having determined all permanent $\overline{\kappa}$-families, we consider differentials. We recall, from Remark \ref{Rmk_Org}, that each permanent $\overline{\kappa}$-family is truncated by one and only one non-permanent $\overline{\kappa}$-family. We can proceed as follows: take a permanent cycle, say $x$; then locate all non-permanent classes that can support a differential killing $\overline{\kappa}^{n}x$ for some $n\leq 6$. Precisely, one of the following situations will happen:

1) There is no ambiguity: i.e., there is only one generator that can support a differential killing $\overline{\kappa}^{n}x$ for some $n\leq 6$, so this differential occurs.

2) There are two generators that can support a differential killing multiples of $x$ by different powers of $\overline{\kappa}$. In order to decide, we inspect the $\overline{\kappa}$-exponent of $x$ using the ASS.

3) There are two generators that can support a differential killing the multiple of $x$ by the same power of $\overline{\kappa}$. In this case, inspection on the $\overline{\kappa}$-exponent of $x$ does not help. We will treat each of the particularity case by case. Some Toda brackets will be involved to resolve these cases.\\\\
\noindent
A permanent cycle is said to be of type 1, 2, 3 respectively if its $\overline{\kappa}$-family is as in the situation 1, 2, 3 above respectively. The HFPSS for different versions of $A_1$ do not behave in the same manner. It turns out the HFPSS for the versions $A_{1}[10]$ and $A_{1}[01]$ behave in the same way and $A_1[00]$ and $A_1[11]$ in the same way. We will treat the HFPSS for $A_{1}[10]$ and $A_{1}[01]$ in detail and then point out the changes needed for $A_{1}[00]$ and $A_{1}[11]$.\\\\
\noindent
\textbf{Differentials (continued) for $A_{1}[01]$ and $A_{1}[10]$.} The reader is invited to follow the discussion of the differentials using Figures (\ref{0-48}) to (\ref{144-197}) below.

\underline{The $d_{9}$-differentials}

\begin{Proposition} There are the following $d_{9}$-differentials:
 \begin{itemize}
  \item[(1)] $d_{9}(\Delta^{2}e[1,23]) = \overline{\kappa}^{2}e[2,30]$
  \item[(2)] $d_{9}(\Delta^{6}e[1,23]) = \overline{\kappa}^{2}\Delta^{4}e[2,30].$
 \end{itemize}
\end{Proposition}
\begin{proof} The classes $e[2,30]$ and $\Delta^{4}e[2,30]$ are of type $1$ and the only possibilities are $d_{9}(\Delta^{2}e[1,23]) = \overline{\kappa}^{2}e[2,30]$ and $d_{9}(\Delta^{6}e[1,23]) = \overline{\kappa}^{2}\Delta^{4}e[2,30]$, respectively. 
\end{proof}
%\begin{Corollary} The Toda bracket $\langle \nu,\overline{\kappa}^{2}, e[2,30]\rangle$ is nontrivial and is represented by $e[2,74]$ which is $\nu\Delta^{2}e[2,23]$ in the $\mathrm{E}_{7}$-term.  The Toda bracket $\langle\overline{\kappa}^{2},\nu, \Delta^{4}e[2,30]\rangle$ is nontrivial and is represented by $e[2,170]$ which is $\nu\Delta^{6}e[2,23]$ in the $\mathrm{E}_{7}$-term. 
%\end{Corollary}
%\begin{Remark} It is worth pointing out that $\nu\Delta^{2}e[2,23]$ is not a multiple of $\nu$ in the homotopy groups because $\Delta^{2}e[2,23]$ is not permanent cycle.
%\end{Remark}
\underline{The $d_{15}$-differentials}

\begin{Proposition} There are the following $d_{15}$-differentials:
 \begin{itemize}
   \item[(1)] $d_{15}(\Delta^{2}e[2,38]) = \overline{\kappa}^{4}e[1,5]$
   \item[(2)] $d_{15}(\Delta^{2}e[2,48]) = \overline{\kappa}^{4}e[1,15]$
   \item[(3)] $d_{15}(\Delta^{6}e[2,38]) = \overline{\kappa}^{4}\Delta^{4}e[1,5]$
   \item[(4)] $d_{15}(\Delta^{6}e[2,48]) = \overline{\kappa}^{4}\Delta^{4}e[1,15].$
   \end{itemize}
\end{Proposition}
\begin{proof} It is readily checked from the chart that all $e[1,5]$, $e[1,15]$, $\Delta^{4}e[1,5]$, $\Delta^{4}e[1,15]$ are of type $1$ and their $\overline{\kappa}$-family is truncated as indicated in the proposition.
\end{proof}

\underline{The $d_{17}$-differentials}

\begin{Proposition} \phantomsection \label{d_17}There are the following $d_{17}$-differentials:
 \begin{itemize} 
    \item[(1)] $d_{17}(\Delta^{2}e[3,53]) = \overline{\kappa}^{5}e[0,0]$
    \item[(2)] $d_{17}(\Delta^{4}e[0,6]) = \overline{\kappa}^{4}e[1,21]$
    \item[(3)] $d_{17}(\Delta^{4}e[1,17]) = \overline{\kappa}^{4}e[2,32]$
    \item[(4)] $d_{17}(\Delta^{4}e[1,21]) = \overline{\kappa}^{4}e[2,36]$
     \item[(5)] $d_{17}(\Delta^{4}e[2,32]) = \overline{\kappa}^{4}e[3,47]$
     \item[(6)] $d_{17}(\Delta^{6}e[0,6]) = \overline{\kappa}^{4}\Delta^{2}e[1,21]$
     \item[(7)]$d_{17}(\Delta^{6}e[1,17]) = \overline{\kappa}^{4}\Delta^{2}e[2,32]$
     \item[(8)]$d_{17}(\Delta^{6}e[1,21])=\overline{\kappa}^{4}\Delta^{2}e[2,36]$
    \item[(9)]$d_{17}(\Delta^{6}e[2,32]) = \overline{\kappa}^{4}\Delta^{2}e[3,47]$
    \item[(10)]$d_{17}(\Delta^{6}e[3,53]) = \overline{\kappa}^{5}\Delta^{4}e[0,0]$
    
    \item[(11)] $d_{17}(\Delta^{4}e[1,23]) = \overline{\kappa}^{4}e[2,38]$
   
    \item[(12)] $d_{17}(\Delta^{4}e[2,38]) = \overline{\kappa}^{4}e[3,53]$
    \item[(13)] $d_{17}(\Delta^{6}e[0,0]) = \overline{\kappa}^{4}\Delta^{2}e[1,15]$
    
    \item[(14)]$d_{17}(\Delta^{6}e[1,15]) = \overline{\kappa}^{4}\Delta^{2}e[2,30].$

 \end{itemize}
\end{Proposition}
\begin{proof} 
\begin{itemize}
  \item[(1)-(10)] All of the generators of 

  $$e[0,0], e[1,21], e[2,32], e[2,36], e[3,47] $$
   $$\Delta^{2}e[1,21], \Delta^{2}e[2,32], \Delta^{2}e[2,36], \Delta^{2}e[3,47], \Delta^{4}e[0,0] $$
  are of type $1$.
  \item[(11)] $e[2,38]$ is of type $2$. The differentials that can truncate its $\overline{\kappa}$-family are $d_{17}(\Delta^{4}e[1,23]) = \overline{\kappa}^{4}e[2,38]$ and $d_{25}(\Delta^{6}e[1,15]) = \overline{\kappa}^{6}e[2,38]$. The latter can not happen because the spectral sequence collapses at the $\mathrm{E}_{24}$-term. Therefore, one must have that $d_{17}(\Delta^{4}e[1,23]) = \overline{\kappa}^{4}e[2,38]$.
   \item[(12)] $e[3,53]$ is of type $2$. Its $\overline{\kappa}$-family can be truncated by $d_{17}(\Delta^{4}e[2,38]) = \overline{\kappa}^{4}e[3,53]$ or $d_{25}(\Delta^{6}e[2,30])=\overline{\kappa}^{6}e[3,53]$. As above, there can not be any $d_{25}$-differential in the spectral sequence. Hence, one must have that $d_{17}(\Delta^{4}e[2,38]) = \overline{\kappa}^{4}e[3,53]$.
  \item[(13)] $\Delta^{2}e[1,15]$ is of type $3$. In its $\overline{\kappa}$-family, only $\overline{\kappa}^{4}\Delta^{2}e[1,15]$ can be a target of differentials, $d_{17}(\Delta^{6}e[0,0]) = \overline{\kappa}^{4}\Delta^{2}e[1,15]$ and $d_{15}(\Delta^{4}e[2,48]) = \overline{\kappa}^{4}\Delta^{2}e[1,15]$. However, if $d_{15}(\Delta^{4}e[2,48]) = \overline{\kappa}^{4}\Delta^{2}e[1,15]$ then the only class that can truncate the $\overline{\kappa}$-family of $e[1,23]$ is $\Delta^{6}e[0,0]$ and by a $d_{25}$-differential: $d_{25}(\Delta^{6}e[0,0]) = \overline{\kappa}^{6}e[1,23]$. This contradicts the fact that the spectral sequence collapses at the $\mathrm{E}_{24}$-term. Thus, one must have that $d_{17}(\Delta^{6}e[0,0]) = \overline{\kappa}^{4}\Delta^{2}e[1,15]$.
 
   \item[(14)] $\Delta^{2}e[2,30]$ is of type 2. Its $\overline{\kappa}$-family can be truncated by a $d_{9}$-differential on $\Delta^{4}e[1,23]$ or by  a $d_{17}$-differential on $\Delta^{6}e[1,15]$. However, the former possibility can not occur because of part $(11)$. Therefore, $d_{17}(\Delta^{6}e[1,15]) = \overline{\kappa}^{4}\Delta^{2}e[2,30]$.
  
\end{itemize}
\end{proof}
\underline{The $d_{19}$-differentials}
\begin{Proposition} \phantomsection \label{d_19}There are the following $d_{19}$-differentials:
   \begin{itemize}
       \item[(1)] $d_{19}(\Delta^{4}e[1,11])=\overline{\kappa}^{5}e[0,6]$
       \item[(2)] $d_{19}(\Delta^{4}e[3,47])=\overline{\kappa}^{5}e[2,42]$
        \item[(3)] $d_{19}(\Delta^{6}e[1,11]) = \overline{\kappa}^{5}\Delta^{2}e[0,6]$
       \item[(4)] $d_{19}(\Delta^{6}e[3,47]) = \overline{\kappa}^{5}\Delta^{2}e[2,42]$
       \item[(5)] $d_{19}(\Delta^{6}e[1,5])=\overline{\kappa}^{5}\Delta^{2}e[0,0]$
       \item[(6)] $d_{19}(\Delta^{4}e[3,53])= \overline{\kappa}^{5}e[2,48].$
      
   \end{itemize}
\end{Proposition}
\begin{proof} 
\begin{itemize}
   \item[(1)-(4)] All of the classes $$e[0,6], e[2,42], \Delta^{2}e[0,6], \Delta^{2}e[2,42]$$ are of type $1$.
   \item[(5)] The class $\Delta^{2}e[0,0]$ is of type $3$ and its $\overline{\kappa}$-family can be truncated either by $d_{17}(\Delta^{4}e[3,53])=\overline{\kappa}^{5}\Delta^{2}e[0,0]$ or by $d_{19}(\Delta^{6}e[1,5]) =\overline{\kappa}^{5}\Delta^{2}e[0,0] $. Suppose $d_{17}(\Delta^{4}e[3,53]) = \overline{\kappa}^{5}\Delta^{2}e[0,0]$. This would leave us with the differential $d_{21}(\Delta^{6}e[1,5]) = \overline{\kappa}^{5}e[2,48]$. It would imply the Massey product in the $\mathrm{E}_{22}$-term $$\langle \overline{\kappa}^{5},e[2,48],\nu\rangle = \nu\Delta^{6}e[1,5]$$ with zero indeterminacy in the $\mathrm{E}_{22}$-term. All conditions of Moss's convergence theorem are met, the Toda bracket $\langle \overline{\kappa}^{5},e[2,48],\nu\rangle$ could then be formed and would contain an element represented by $\nu\Delta^{6}e[1,5]$. This contradicts Corollary \ref{Toda bracket}. This contradiction proves that $$d_{19}(\Delta^{6}e[1,5])=\overline{\kappa}^{5}\Delta^{2}e[0,0].$$
   \item[(6)] The class $e[2,48]$ is of type $2$ and its $\overline{\kappa}$-family is truncated either by $d_{19}(\Delta^{4}e[3,53]) = \overline{\kappa}^{5}e[2,48]$ or by $d_{21}(\Delta^{6}e[1,5]) = \overline{\kappa}^{5}e[2,48]$. However, part (5) of Proposition \ref{d_19} rules out the latter.
\end{itemize}
\end{proof}
\underline{The $d_{23}$-differentials}
\begin{Proposition} There are the following $d_{23}$-differentials:
  \begin{itemize} 
  	\item[(1)] $d_{23}(\Delta^{4}e[2,36]) = \overline{\kappa}^{6}e[1,11]$
	\item[(2)] $d_{23}(\Delta^{4}e[2,42]) = \overline{\kappa}^{6}e[1,17]$
	\item[(3)] $d_{23}(\Delta^{4}e[2,48]) = \overline{\kappa}^{6}e[1,23]$
	
	\item[(4)] $d_{23}(\Delta^{6}e[2,36]) = \overline{\kappa}^{6}\Delta^{2}e[1,11]$
	\item[(5)] $d_{23}(\Delta^{6}e[2,42]) = \overline{\kappa}^{6}\Delta^{2}e[1,17]$
	\item[(6)] $d_{23}(\Delta^{6}e[2,30]) = \overline{\kappa}^{6}\Delta^{2}e[1,5].$
  \end{itemize}
\end{Proposition}
\begin{proof}
  \begin{itemize} 
       \item[(1)-(5)] All of the classes $$e[1,11], e[1,17], e[1,23], \Delta^{2}e[1,11], \Delta^{2}e[1,17]$$ are of type $1$.
       \item[(6)] The class $\Delta^{2}e[1,5]$ is of type $2$. The two possibilities are $d_{15}(\Delta^{4}e[2,38])=\overline{\kappa}^{4}\Delta^{2}e[1,5]$ and $d_{23}(\Delta^{6}e[2,30]) = \overline{\kappa}^{6}\Delta^{2}e[1,5]$. However, part $(12)$ of Proposition \ref{d_17} rules out the former because the class $\Delta^{4}e[2,38]$ must pair up with the class $e[3,38]$, by a $d_{17}$-differential $d_{17}(\Delta^{4}e[2,38]) = \overline{\kappa}^{4}e[3,53] $.  
       \end{itemize}
\end{proof}
\noindent
 The above differentials from $d_{9}$ to $d_{23}$, together with the $\overline{\kappa}$-  and $\Delta^{8}$-linearity exhaust all differentials. In the statement of Theorem \ref{Einfty A01} and \ref{Einfty A11}, we write $e_{t-s}$ for the permanent cycle $e[s,t-s]$ in bidegree $(s,t)$ listed in Proposition \ref{E_7}, for the sake of presentation. 
\begin{Theorem}\label{Einfty A01} As a module over $\FF_4[\Delta^{\pm 8}, \overline{\kappa},\nu]/(\ovk\nu)$, the $\E_{\infty}$-term of the HFPSS for $E_C^{hG_{24}}\wedge A_1$ for $A_1 = A_1[10]$ and $A_1[01]$ is a direct sum of cyclic modules generated by the following elements and with the respective annihilator ideal:
%$$\begin{array}{llllllll}
%
% (0,0)&(1,5)&(0,6)&(1,11)&(1,15)&(1,17)&(1,21)&(1,23)\\
% e[0,0]&e[1,5]&e[0,6]&e[1,11]&e[1,15]&e[1,17]&e[1,21]&e[1,23]\\
% (\overline{\kappa}^5, \nu^3)&(\overline{\kappa}^4,\nu^2)&(\overline{\kappa}^5,\nu^3)&(\ovk^6,\nu^2)&(\ovk^4,\nu^2)&(\ovk^6,\nu^3)&(\ovk^4,\nu^2)&(\ovk^6,\nu^3)\\
% \\
%(2,30)&(2,32)&(2,36)&(2,38)&(2,42)&(3,47)&(2,48)&(3,53)\\
%e[2,30]&e[2,32]&e[2,36]&e[2,38]&e[2,42]&e[3,47]&e[2,48]&e[3,53]\\
%(\ovk^2,\nu)&(\ovk^4,\nu)&(\ovk^4,\nu)&(\ovk^4,\nu)&(\ovk^5,\nu)&(\ovk^4,\nu)&(\ovk^5,\nu)&(\ovk^4,\nu)\\
% \end{array}$$
 $$\begin{array}{llllllll}

 (0,0)&(1,5)&(0,6)&(1,11)&(1,15)&(1,17)&(1,21)&(1,23)\\
 e_0&e_{5}&e_{6}&e_{11}&e_{15}&e_{17}&e_{21}&e_{23}\\
 (\overline{\kappa}^5, \nu^3)&(\overline{\kappa}^4,\nu^2)&(\overline{\kappa}^5,\nu^3)&(\ovk^6,\nu^2)&(\ovk^4,\nu^2)&(\ovk^6,\nu^3)&(\ovk^4,\nu^2)&(\ovk^6,\nu^3)\\
 \\
(2,30)&(2,32)&(2,36)&(2,38)&(2,42)&(3,47)&(2,48)&(3,53)\\
e_{30}&e_{32}&e_{36}&e_{38}&e_{42}&e_{47}&e_{48}&e_{53}\\
(\ovk^2,\nu)&(\ovk^4,\nu)&(\ovk^4,\nu)&(\ovk^4,\nu)&(\ovk^5,\nu)&(\ovk^4,\nu)&(\ovk^5,\nu)&(\ovk^4,\nu)\\
\\
%%%%%%%%%%%%%%%%%%%%%%%%%%%%%%
(0,48)&(1,53)&(0,54)&(1,59)&(1,63)&(1,65)&(1,69)&(2,74)\\
 \Delta^2e_0& \Delta^2e_{5}& \Delta^2e_{6}& \Delta^2e_{11}& \Delta^2e_{15}& \Delta^2e_{17}& \Delta^2e_{21}& \Delta^2\nu e_{23}\\
 (\ovk^5,\nu^3)& (\ovk^6,\nu^2)& (\ovk^5,\nu^3)& (\ovk^6,\nu^2)& (\ovk^4,\nu^2)& (\ovk^6,\nu^3)& (\ovk^4,\nu^2)& (\ovk,\nu^2)\\
\\
 %%%%%%%%%%%%%%%%%%%%%%%%%%%%%%
(2,78)&(2,80)&(2,84)&(2,90)&(3,95)&(0,96)&(1,101)\\
\Delta^2e_{30}&\Delta^2e_{32}&\Delta^2e_{36}&\Delta^2e_{42}&\Delta^2e_{47}&\Delta^4e_{0}&\Delta^4e_{5}\\
(\ovk^4,\nu)&(\ovk^4,\nu)&(\ovk^4,\nu)&(\ovk^5,\nu)&(\ovk^4,\nu)&(\ovk^5,\nu^3)&(\ovk^4,\nu^2)\\
\\
%%%%%%%%%%%%%%%%%%%%%%%%%%%%%%%
(1,105)&(2,110)&(1,111)&(2,116)&(2,120)&(2,122)&(2,126)\\
 \Delta^4\nu e_{6}& \Delta^4\nu e_{11}& \Delta^4 e_{15}& \Delta^4\nu e_{17}& \Delta^4\nu e_{21}& \Delta^4\nu e_{23}&(\Delta^4 e_{30})\\
 (\ovk,\nu^2)& (\ovk,\nu)& (\ovk^4,\nu^2)& (\ovk,\nu^2)& (\ovk,\nu)& (\ovk,\nu^2)&(\ovk^2, \nu)\\
\\
%%%%%%%%%%%%%%%%%%%%%%%%%%%%%%
(1,147)&(2,152)&(1,153)&(2,158)&(2,162)&(2,164)&(2,168)&(2,170)\\
 \Delta^6\nu e_0& \Delta^6\nu e_{5}& \Delta^6\nu e_{6}& \Delta^6\nu e_{11}& \Delta^6\nu e_{15}& \Delta^6\nu e_{17}& \Delta^6\nu e_{21}& \Delta^6\nu e_{23}\\
 (\ovk,\nu^2)& (\ovk,\nu)& (\ovk,\nu^2)& (\ovk,\nu)& (\ovk,\nu)& (\ovk,\nu^2)& (\ovk,\nu)& (\ovk,\nu^2).\\
 \end{array}$$
\end{Theorem}
\noindent
\textbf{The case of $A_{1}[00]$ and $A_{1}[11]$.} The analysis of the HFPSS for $A_{1}[00]$ and $A_{1}[11]$ can be done in the same manner as that for $A_{1}[10]$ and $A_{1}[01]$. All differentials are identical except for $8$ ones involving $16$ of the generators of Proposition \ref{E_7}. We will be content to point out all modifications, see Figures from \ref{0-48 nondual} to \ref{144-197 nondual}. 

$$d_{17}(\Delta^{4}e[1,15]) = \overline{\kappa}^{4}e[2,30]\ \mbox {instead of } d_{9}(\Delta^{2}e[1,23]) = \overline{\kappa}^{2}e[2,30],$$
$$d_{17}(\Delta^{6}e[1,23]) = \overline{\kappa}^{4}\Delta^{2}e[2,38]\ \mbox {instead of } d_{9}(\Delta^{6}e[1,23]) = \overline{\kappa}^{2}\Delta^{4}e[2,30],$$
$$d_{17}(\Delta^{4}e[0,0]) = \overline{\kappa}^{4}e[1,15]\ \mbox {instead of } d_{15}(\Delta^{2}e[2,48]) = \overline{\kappa}^{4}e[1,15],$$
$$d_{17}(\Delta^{6}e[2,38]) = \overline{\kappa}^{4}\Delta^{2}e[3,53]\ \mbox {instead of } d_{15}(\Delta^{6}e[2,38]) = \overline{\kappa}^{4}\Delta^{2}e[1,5],$$
$$d_{19}(\Delta^{4}e[1,5]) = \overline{\kappa}^{5}e[0,0]\ \mbox {instead of } d_{17}(\Delta^{2}e[3,53]) = \overline{\kappa}^{5}e[0,0],$$
$$d_{19}(\Delta^{6}e[3,53]) = \overline{\kappa}^{5}\Delta^{2}e[2,48]\ \mbox {instead of } d_{17}(\Delta^{6}e[3,53]) = \overline{\kappa}^{5}\Delta^{4}e[0,0],$$
$$d_{23}(\Delta^{6}e[2,48]) = \overline{\kappa}^{6}\Delta^{2}e[1,23]\ \mbox {instead of } d_{15}(\Delta^{6}e[2,48]) = \overline{\kappa}^{4}\Delta^{4}e[1,15],$$
$$d_{23}(\Delta^{4}e[2,30]) = \overline{\kappa}^{6}e[1,5]\ \mbox {instead of } d_{15}(\Delta^{2}e[2,38]) = \overline{\kappa}^{4}e[1,5].$$
%%%%%%%%%%
\noindent
\begin{Theorem}\label{Einfty A11} As a module over $\FF_4[\Delta^{\pm 8}, \overline{\kappa},\nu]/(\ovk\nu)$, the $\E_{\infty}$-term of the HFPSS for $E_C^{hG_{24}}\wedge A_1$ for $A_1 = A_1[00]$ and $A_1[11]$ is a direct sum of cyclic modules generated by the following elements and with the respective annihilator ideals:
$$\begin{array}{llllllll}
 (0,0)&(1,5)&(0,6)&(1,11)&(1,15)&(1,17)&(1,21)&(1,23)\\
 e_0&e_{5}&e_{6}&e_{11}&e_{15}&e_{17}&e_{21}&e_{23}\\
 (\overline{\kappa}^5, \nu^3)&(\overline{\kappa}^6,\nu^2)&(\overline{\kappa}^5,\nu^3)&(\ovk^6,\nu^2)&(\ovk^4,\nu^2)&(\ovk^6,\nu^3)&(\ovk^4,\nu^2)&(\ovk^6,\nu^3)\\
 \\
(2,30)&(2,32)&(2,36)&(2,38)&(2,42)&(3,47)&(2,48)&(3,53)\\
e_{30}&e_{32}&e_{36}&e_{38}&e_{42}&e_{47}&e_{48}&e_{53}\\
(\ovk^4,\nu)&(\ovk^4,\nu)&(\ovk^4,\nu)&(\ovk^4,\nu)&(\ovk^5,\nu)&(\ovk^4,\nu)&(\ovk^5,\nu)&(\ovk^4,\nu)\\
\\
%%%%%%%%%%%%%%%%%%%%%%%%%%%%%%
(0,48)&(1,53)&(0,54)&(1,59)&(1,63)&(1,65)&(1,69)&(1,71)\\
 \Delta^2e_0& \Delta^2e_{5}& \Delta^2e_{6}& \Delta^2e_{11}& \Delta^2e_{15}& \Delta^2e_{17}& \Delta^2e_{21}& \Delta^2e_{23}\\
 (\ovk^5,\nu^3)& (\ovk^6,\nu^2)& (\ovk^5,\nu^3)& (\ovk^6,\nu^2)& (\ovk^4,\nu^2)& (\ovk^6,\nu^3)& (\ovk^4,\nu^2)& (\ovk^6,\nu^3)\\
\\
 %%%%%%%%%%%%%%%%%%%%%%%%%%%%%%
(2,78)&(2,80)&(2,84)&(2,86)&(2,90)&(3,95)&(2,96)&(3,101)\\
\Delta^2e_{30}&\Delta^2e_{32}&\Delta^2e_{36}&\Delta^2e_{38}&\Delta^2e_{42}&\Delta^2e_{47}&\Delta^2e_{48}&\Delta^2e_{53}\\
(\ovk^4,\nu)&(\ovk^4,\nu)&(\ovk^4,\nu)&(\ovk^4,\nu)&(\ovk^5,\nu)&(\ovk^4,\nu)&(\ovk^5,\nu)&(\ovk^4,\nu)\\
\\
%%%%%%%%%%%%%%%%%%%%%%%%%%%%%%%
(1,99)&(2,104)&(1,105)&(2,110)&(2,114)&(2,116)&(2,120)&(2,122)\\
 \Delta^4\nu e_0& \Delta^4\nu e_{5}& \Delta^4\nu e_{6}& \Delta^4\nu e_{11}& \Delta^4\nu e_{15}& \Delta^4\nu e_{17}& \Delta^4\nu e_{21}& \Delta^4\nu e_{23}\\
 (\ovk,\nu^2)& (\ovk,\nu)& (\ovk,\nu^2)& (\ovk,\nu)& (\ovk,\nu)& (\ovk,\nu^2)& (\ovk,\nu)& (\ovk,\nu^2)\\
\\
%%%%%%%%%%%%%%%%%%%%%%%%%%%%%%
(1,147)&(2,152)&(1,153)&(2,158)&(2,162)&(2,164)&(2,168)&(2,170)\\
 \Delta^6\nu e_0& \Delta^6\nu e_{5}& \Delta^6\nu e_{6}& \Delta^6\nu e_{11}& \Delta^6\nu e_{15}& \Delta^6\nu e_{17}& \Delta^6\nu e_{21}& \Delta^6\nu e_{23}\\
 (\ovk,\nu^2)& (\ovk,\nu)& (\ovk,\nu^2)& (\ovk,\nu)& (\ovk,\nu)& (\ovk,\nu^2)& (\ovk,\nu)& (\ovk,\nu^2).\\
\end{array}$$
\end{Theorem}
\begin{Remark} We emphasise that the relations given in Theorem \ref{Einfty A01} and \ref{Einfty A11} are only the relations in the $\E_{\infty}$-term. In fact, we can see by sparseness that, the annihilator exponents of $\ovk$ are still true in $\pi_*(E_C^{hG_{24}}\wedge A_1)$. Whereas, there are exotic extensions by $\nu$, i.e., multiplications by $\nu$ that are not detected in the $\E_\infty$-term. These can be determined by two different methods: by using the Tate spectral sequence as in \cite{BO16}, Section 2.3 or by computing the Gross-Hopkins dual of $E_C^{hG_{24}}\wedge A_1$; however, we do not discuss this point here.
\end{Remark}
\noindent
Using the structure of the $\mathrm{E}_\infty$-term, we can read off the action of the ideal $(\overline{\kappa}, \nu)$ on $\pi_*(E_C^{hG_{24}}\wedge A_1)$. From this, we obtain the following Corollary.
 \begin{Theorem}\label{connective-iso} We have
 
 a) The map $$\Theta^{'}: \W(\FF_4)\otimes_{\Z_{2}}\pi_k(tmf\wedge A_1)/(\overline{\kappa},\nu)\rightarrow \pi_{k}(E_C^{hG_{24}}\wedge A_1)/(\overline{\kappa},\nu),$$  induced by $\Theta$ in (\ref{map_compa}), is an isomorphism for $k\geq 0$, independent of the version of $A_1$.
 
 b) The map $$\Theta : \W(\FF_4)\otimes_{\Z_2}\pi_k(tmf\wedge A_1)\rightarrow \pi_k(E_C^{hG_{24}}\wedge A_1)$$ is also an isomorphism for $k\geq 0$, independent of the version of $A_1$.
 
 c) Multiplication by $\Delta^8$ induces isomorphisms $$\pi_k(tmf\wedge A_1)\rightarrow \pi_{k+192}(tmf\wedge A_1)$$ and $$\pi_k(tmf\wedge A_1)/(\ovk,\nu)\rightarrow \pi_{k+192}(tmf\wedge A_1)/(\ovk,\nu)$$ for $k\geq 0$.
\end{Theorem}
\begin{proof} For part a), Corollary \ref{Cor_Compare} asserts that $\Theta^{'}$ is injective. To show that the latter is surjective, it suffices to show that its source and target have the same order. The order of the target can be seen from Theorem \ref{Einfty A01} and \ref{Einfty A11}; in particular, it has order $0$ or $4$ in all stems, except for the stems $48$ and $53$ modulo $192$, in which it has order $8$. The remaining part of the proof is an inspection of the ASS for $tmf\wedge A_1$, together with the fact that $\Theta$ is injective, by Corollary \ref{Cor_Compare}, and is linear with respect to $\overline{\kappa}$ and $\nu$, to show that $\W\otimes_{\Z_2}\pi_*(tmf\wedge A_1)$ has the same order as of $\pi_*(E_C^{hG_{24}}\wedge A_1)$, in non-negative stems. Because of the dependance of the structure of $\pi_*(E_C^{hG_{24}}\wedge A_1)$ on the version of $A_1$, we consider them separately: we only give a detailed treatment for $A_1[00]$ and $A_1[11]$ and claim that the treatment for $A_1[01]$ and $A_1[10]$ is completely similar. For the remaining part of the proof, $A_1$ will be $A_1[00]$ or $A_1[11]$.\\\\
By sparseness and part $(i)$ of Theorem \ref{ASS_d_2_bis}, all classes $w_2^le[i,j]$ for $l=0,1$ and $e[i,j]$, the classes in the table of Proposition \ref{Prop_AdamsA_1} survive to the $\E_\infty$-term of the ASS for $tmf\wedge A_1$. Moreover, for degree reasons, these classes must converge to non-trivial elements of $\pi_*(tmf\wedge A_1)/(\ovk,\nu)$ in the appropriate stems. Therefore, $\W\otimes_{\Z_2}\pi_*(tmf\wedge A_1)/(\ovk,\nu)$ has the same order as of $\pi_*(E_C^{hG_{24}}\wedge A_1)$ up to stem $96$ and in stem $101$.\\\\
All of the classes $$w_2^2e[0,0], w_2^2e[1,5], w_2^2e[1,6], w_2^2e[2,11],$$ $$w_2^2e[3,15], w_2^2e[3,17], w_2^2e[4,21], w_2^2e[4,23]$$
are $d_2$-cycles in the ASS and the $d_3$-differentials on them can only hit $g$-multiple classes. Thus, by $\nu$-linearity and the fact that $g\nu=0$ in $\Ext_{\A(2)_*}^{5,28}(\F)$, the classes 
$$\nu w_2^2e[0,0], \nu w_2^2e[1,5], \nu w_2^2e[1,6], \nu w_2^2e[2,11],$$ $$\nu w_2^2e[3,15], \nu w_2^2e[3,17], \nu w_2^2e[4,21], \nu w_2^2e[4,23]$$ are $d_3$-cycles and hence survive to the $\E_\infty$-term, by sparseness. As above, these classes must converge to non-trivial elements of $\pi_*(tmf\wedge A_1)/(\ovk,\nu)$ in the appropriate stems. It follows that $\W\otimes_{\Z_2}\pi_*(tmf\wedge A_1)/(\ovk,\nu)$ has the same order as of $\pi_*(E_C^{hG_{24}}\wedge A_1)$ for stems from $96$ to $144$.\\\\
Consider the classes
\begin{align}\label{144-192}
\nu w_2^3e[0,0], \nu w_2^3e[1,5], \nu w_2^3e[1,6], \nu w_2^3e[2,11], \nonumber \\
\nu w_2^3e[3,15], \nu w_2^3e[3,17], \nu w_2^3e[4,21], \nu w_2^3e[4,23].
\end{align}
As above, these classes survive to the $\E_4$-term of the ASS for $tmf\wedge A_1$. By sparseness, $\nu w_2^3e[4,23]$ survives to the $\E_\infty$-term and converges to a non-trivial element of $\pi_{170}(tmf\wedge A_1)/(\ovk,\nu)$. By sparseness, the other classes can only support $d_4$-differentials hitting the classes
$$g^7e[1,6], g^{7}e[2,11], g^6e[6,32], g^7e[3,17], g^7e[4,21], g^7e[4,23], g^6e[9,47], $$ respectively. However, the class $g^ke[i,j]$ for $(i,j)\in$ $ \{(1,6), (2,11), (6,32), (3,17),$ $(4,21), (4,23), (9,47)\}$ is killed by a differential for a certain integer $k$ less than $7$, hence $g^7e[i,j]$ for $(i,j)\in \{(1,6), (2,11), (6,32), (3,17), (4,21), (4,23), (9,47)\}$ is killed by a differential on a certain $g$-multiple class. This means that $$\nu w_2^3e[0,0], \nu w_2^3e[1,5], \nu w_2^3e[2,11], \nu w_2^3e[3,15], \nu w_2^3e[3,17]$$ survive to the $\E_\infty$-term, hence, as above, to non-trivial elements of $\pi_*(tmf\wedge A_1)/(\ovk,\nu)$. Next, the map $\Theta$ sends $e[6,32]$ and $e[9,47]$ to $e[2,32]$ and $e[3,47]$, respectively. The latter are both annihilated by $\ovk^4$, so that $g^4e[6,32]$ and $g^4e[9,47]$ are hit by certain differentials in the ASS, hence $g^6e[6,32]$ and $g^6e[9,47]$ are hit by differentials supported on $g$-multiple classes. As above, this implies that $\nu w_2^3e[1,6]$ and $\nu w_2^3 e[4,23]$ survive to non-trivial elements of $\pi_*(tmf\wedge A_1)/(\ovk,\nu)$. In total, we have proved that all classes of (\ref{144-192}) converge to non-trivial elements of $\pi_*(tmf\wedge A_1)/(\ovk,\nu)$; as a consequence, $\W\otimes_{\Z_2}\pi_*(tmf\wedge A_1)/(\ovk,\nu)$ has the same order as of $\pi_*(E_C^{hG_{24}}\wedge A_1)/(\ovk,\nu)$ in stems from $144$ to $192$. \\\\
Together with the fact that $\pi_*(E_C^{hG_{24}}\wedge A_1)/(\ovk,\nu)$ is $\Delta^8$-periodic, we conclude that $\Theta^{'}$ is a surjection, hence is an isomorphism.\\\\
For part b), there is a commutative diagram
$$\xymatrix{ \W(\FF_4)\otimes_{\Z_{2}}\pi_*(tmf\wedge A_1)\ar[r]^{\Theta}\ar[d]& \pi_*(E_C^{hG_{24}}\wedge A_1)\ar[d]\\
 \W(\FF_4)\otimes_{\Z_2}\pi_*(tmf\wedge A_1)/(\ovk,\nu)\ar[r]^{\Theta^{'}}& \pi_*(E_C^{hG_{24}}\wedge A_1)/(\ovk,\nu).
}$$
Part b) then follows from part a) and the fact that $\pi_*(tmf\wedge A_1)$ is bounded below.
\\\\
Part c) follows from part a) and part b) and the fact that $\Delta^8$ is invertible in $\pi_*(E_C^{hG_{24}}).$
\end{proof}
\begin{landscape}
The Figures (\ref{0-48}) to (\ref{144-197}) represent the HFPSS for $E_{C}^{hG_{24}}\wedge A_{1}[10]$ and $E_{C}^{hG_{24}}\wedge A_{1}[01]$ from the $\mathrm{E}_{7}$-term on. Each black dot $\bullet$ represents a class generating a group $\FF_{4}$ which survives to the $\mathrm{E}_{\infty}$-term. Each circle $\circ$ represent a class which either is hit by a differential or supports a differential higher than $d_{5}$. We only represent the differentials on generators listed in Proposition \ref{E_7} but not those generated by $\overline{\kappa}$-linearity.
\begin{figure}[h!]
\begin{tikzpicture}[scale=0.42]
\clip(-1.5,-1.5) rectangle (49,11);
\draw[color=gray] (0,0) grid [step=1] (48,10);

\foreach \n in {0,4,...,48}
{
\def\nn{\n-0}
\node[below,scale=0.5] at (\nn,0) {$\n$};
}
\foreach \s in {0,2,...,10}
{\def\ss{\s-0};
\node [left,scale=0.5] at (-0.4,\ss,0){$\s$};
}
\draw [fill] ( 0.00, 0.00) circle [radius=0.2];
\draw [fill] (3,1) circle [radius=0.2];
\draw [fill] (6,2) circle [radius=0.2];
\draw [-] (0,0)--(3,1);
\draw [-] (3,1)--(6,2);
%\draw[->,red] (11.8,0.2)--(11.2,2.8);
\foreach \t in {0}
\foreach \s in {0,6}
{\def\ss{\s-0};
\draw [fill]  (\s-\t,\t) circle [radius=0.2];
\draw [fill] (\s+3-\t,1+\t) circle [radius=0.2];
\draw [-] (\s-\t,\t)--(\s+3-\t,1+\t);
\draw[fill] (\s+5-\t,1+\t) circle [radius=0.2];
\draw[fill] (\s+8-\t,2+\t) circle [radius=0.2];
\draw[-] (\s+5-\t,1+\t)--(\s+8-\t,2+\t); 
\draw [fill] (\s-\t+6,2+\t) circle [radius=0.2];
\draw [-] (\s+3-\t,1+\t)--(\s+6-\t,2+\t);
%\draw[fill] (\s+11-\t,3+\t) circle [radius=0.1];
%\draw[-] (\s+8-\t,2+\t)--(\s+11-\t,3+\t); 
\draw [-] (\s+5-\t,1+\t)--(\s-\t+6,2+\t);
}

\foreach \t in {0}
\foreach \s in {12,18}
{\def\ss{\s-0};
%\draw [fill]  (\s-\t,\t) circle [radius=0.1];
\draw [fill] (\s+3-\t,1+\t) circle [radius=0.2];
%\draw [-] (\s-\t,\t)--(\s+3-\t,1+\t);
\draw[fill] (\s+5-\t,1+\t) circle [radius=0.2];
\draw[fill] (\s+8-\t,2+\t) circle [radius=0.2];
\draw[-] (\s+5-\t,1+\t)--(\s+8-\t,2+\t); 
\draw [fill] (\s-\t+6,2+\t) circle [radius=0.2];
\draw [-] (\s+3-\t,1+\t)--(\s+6-\t,2+\t);
\draw[fill] (\s+11-\t,3+\t) circle [radius=0.2];
\draw[-] (\s+8-\t,2+\t)--(\s+11-\t,3+\t); 
\draw [-] (\s+5-\t,1+\t)--(\s-\t+6,2+\t);
}
%\draw[fill] (0,2) circle [radius=0.2];
%\draw[fill] (5,3) circle [radius=0.2];
\draw[fill] (30,2) circle [radius=0.2];
\draw[fill] (32,2) circle [radius=0.2];
\draw[fill] (36,2) circle [radius=0.2];
\draw[fill] (38,2) circle [radius=0.2];
\draw[fill] (42,2) circle [radius=0.2];
\draw[fill] (47,3) circle [radius=0.2];
\draw[fill] (48,2) circle [radius=0.2];
\draw[fill] (53,3) circle [radius=0.2];
\draw[fill] (48,0) circle[radius=0.2];

%%%multiple of \ove{\kappa}
\draw[fill] (20,4) circle [radius=0.2]; %e0
\draw[fill] (25,5) circle [radius=0.2];%e5
\draw[fill] (26,4) circle [radius=0.2];%e6
\draw[fill] (31,5) circle [radius=0.2];%e11
\draw[fill] (35,5) circle [radius=0.2];%e15
\draw[fill] (37,5) circle [radius=0.2];%e17
\draw[fill] (41,5) circle [radius=0.2];%e21
\draw[fill] (43,5) circle [radius=0.2];%e23
\draw[fill] (50,6) circle [radius=0.2];%e30
\draw[fill] (52,6) circle [radius=0.2];%e32
%%%multiple of \kappa^2
\draw[fill] (40,8) circle [radius=0.2];%e0
\draw[fill] (45,9) circle [radius=0.2];%e5
\draw[fill] (46,8) circle [radius=0.2];%e6
\end{tikzpicture}
\caption{HFPSS for $A_{1}[10]\ \mbox{and}\ A_{1}[01$] from $\mathrm{E}_{7}$-term with $0\leq t-s\leq 48$}
\phantomsection \label{0-48}
\end{figure}
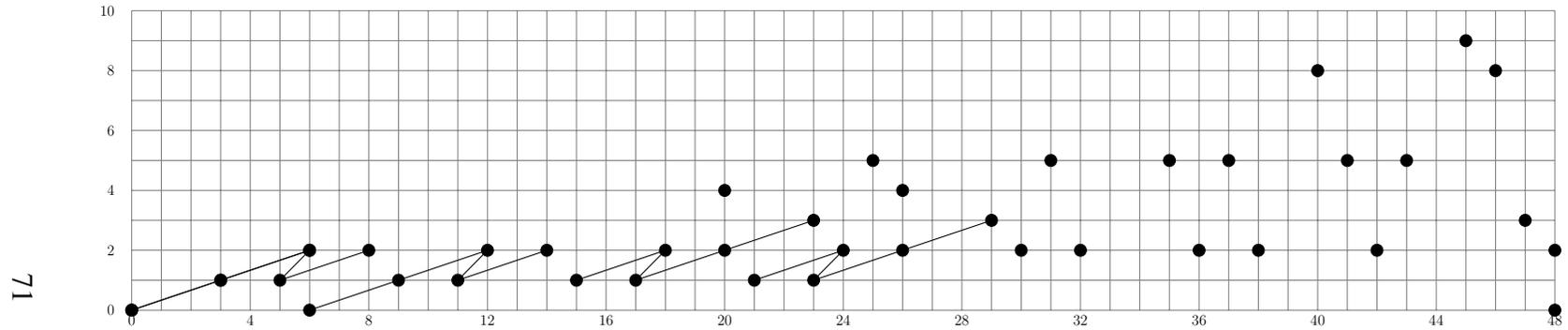
%%%%48-96
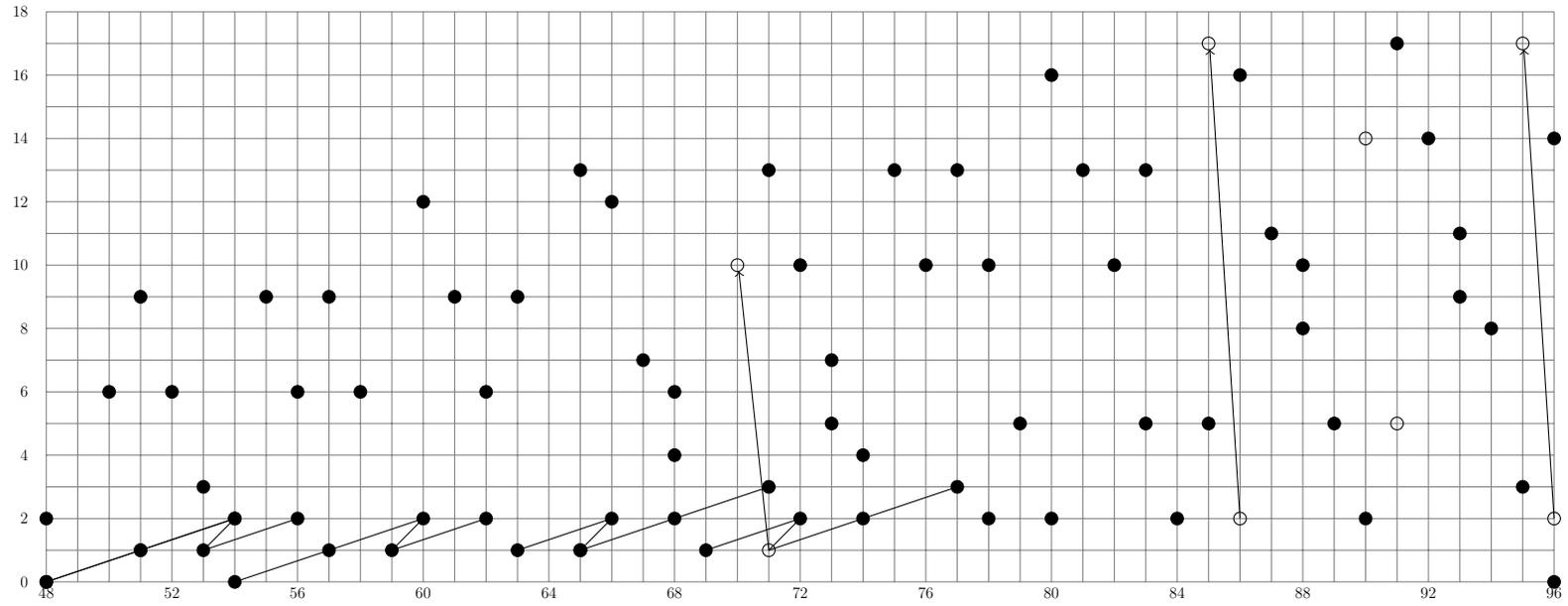
\begin{figure}[h!]
\begin{tikzpicture}[scale=0.42]
\clip(-1.5,-1.5) rectangle (49,19);
\draw[color=gray] (0,0) grid [step=1] (48,18);

\foreach \n in {48,52,...,96}
{
\def\nn{\n-0}
\node[below,scale=0.5] at (\nn-48,0) {$\n$};
}
\foreach \s in {0,2,...,18}
{\def\ss{\s-0};
\node [left,scale=0.5] at (-0.4,\ss,0){$\s$};
}
\draw [fill] ( 0.00, 0.00) circle [radius=0.2];
\draw [fill] (3,1) circle [radius=0.2];
\draw [fill] (6,2) circle [radius=0.2];
\draw [-] (0,0)--(3,1);
\draw [-] (3,1)--(6,2);
%\draw[->,red] (11.8,0.2)--(11.2,2.8);
\foreach \t in {0}
\foreach \s in {0,6}
{\def\ss{\s-0};
\draw [fill]  (\s-\t,\t) circle [radius=0.2];
\draw [fill] (\s+3-\t,1+\t) circle [radius=0.2];
\draw [-] (\s-\t,\t)--(\s+3-\t,1+\t);
\draw[fill] (\s+5-\t,1+\t) circle [radius=0.2];
\draw[fill] (\s+8-\t,2+\t) circle [radius=0.2];
\draw[-] (\s+5-\t,1+\t)--(\s+8-\t,2+\t); 
\draw [fill] (\s-\t+6,2+\t) circle [radius=0.2];
\draw [-] (\s+3-\t,1+\t)--(\s+6-\t,2+\t);
%\draw[fill] (\s+11-\t,3+\t) circle [radius=0.1];
%\draw[-] (\s+8-\t,2+\t)--(\s+11-\t,3+\t); 
\draw [-] (\s+5-\t,1+\t)--(\s-\t+6,2+\t);
}

\foreach \t in {0}
\foreach \s in {12,18}
{\def\ss{\s-0};
%\draw [fill]  (\s-\t,\t) circle [radius=0.1];
\draw [fill] (\s+3-\t,1+\t) circle [radius=0.2];
%\draw [-] (\s-\t,\t)--(\s+3-\t,1+\t);
\draw[] (\s+5-\t,1+\t) circle [radius=0.2];
\draw[fill] (12+5,1) circle [radius=0.2];
\draw[fill] (\s+8-\t,2+\t) circle [radius=0.2];
\draw[-] (\s+5-\t,1+\t)--(\s+8-\t,2+\t); 
\draw [fill] (\s-\t+6,2+\t) circle [radius=0.2];
\draw [-] (\s+3-\t,1+\t)--(\s+6-\t,2+\t);
\draw[fill] (\s+11-\t,3+\t) circle [radius=0.2];
\draw[-] (\s+8-\t,2+\t)--(\s+11-\t,3+\t); 
\draw [-] (\s+5-\t,1+\t)--(\s-\t+6,2+\t);
}
\draw[fill] (0,2) circle [radius=0.2];
\draw[fill] (5,3) circle [radius=0.2];
\draw[fill] (30,2) circle [radius=0.2];
\draw[fill] (32,2) circle [radius=0.2];
\draw[fill] (36,2) circle [radius=0.2];
\draw[] (38,2) circle [radius=0.2];
\draw[fill] (42,2) circle [radius=0.2];
\draw[fill] (47,3) circle [radius=0.2];
\draw[] (48,2) circle [radius=0.2];
\draw[fill] (53,3) circle [radius=0.2];
\draw[fill] (48,0) circle[radius=0.2];
%%%multiple of \ove{\kappa}
\draw[fill] (50-48,6) circle [radius=0.2]; % e30
\draw[fill] (52-48,6) circle [radius=0.2]; % e32
\draw[fill] (56-48,6) circle [radius=0.2]; % e36
\draw[fill] (58-48,6) circle [radius=0.2]; % e38
\draw[fill] (62-48,6) circle [radius=0.2]; % e42
\draw[fill] (67-48,7) circle [radius=0.2]; % e47
\draw[fill] (68-48,6) circle [radius=0.2]; % e48
\draw[fill] (73-48,7) circle [radius=0.2]; % e53
\draw[fill] (20,4) circle [radius=0.2]; %delta^2 e0
\draw[fill] (25,5) circle [radius=0.2];%delta^2 e5
\draw[fill] (26,4) circle [radius=0.2];%delta^2 e6
\draw[fill] (31,5) circle [radius=0.2];%delta^2 e11
\draw[fill] (35,5) circle [radius=0.2];%delta^2 e15
\draw[fill] (37,5) circle [radius=0.2];%delta^2 e17
\draw[fill] (41,5) circle [radius=0.2];%delta^2 e21
\draw[] (43,5) circle [radius=0.2];%delta^2 e23

%%%multiple of \kappa^2
\draw[fill] (51-48,9) circle [radius=0.2];% e11
\draw[fill] (55-48,9) circle [radius=0.2];% e15
\draw[fill] (57-48,9) circle [radius=0.2];% e17
\draw[fill] (61-48,9) circle [radius=0.2];% e21
\draw[fill] (63-48,9) circle [radius=0.2];% e23
\draw[] (70-48,10) circle [radius=0.2]; % e30
\draw[fill] (72-48,10) circle [radius=0.2]; % e32
\draw[fill] (76-48,10) circle [radius=0.2]; % e36
\draw[fill] (78-48,10) circle [radius=0.2]; % e38
\draw[fill] (82-48,10) circle [radius=0.2]; % e42
\draw[fill] (87-48,11) circle [radius=0.2]; % e47
\draw[fill] (88-48,10) circle [radius=0.2]; % e48
\draw[fill] (93-48,11) circle [radius=0.2]; % e53
\draw[fill] (40,8) circle [radius=0.2];%\delta^2 e0
\draw[fill] (45,9) circle [radius=0.2];%delta^2 e5
\draw[fill] (46,8) circle [radius=0.2];%delta^2 e6
%%%multiple of \kappa^3
\draw[fill] (60-48,12) circle [radius=0.2]; %e0
\draw[fill] (65-48,13) circle [radius=0.2];%e5
\draw[fill] (66-48,12) circle [radius=0.2];%e6
\draw[fill] (71-48,13) circle [radius=0.2];% e11
\draw[fill] (75-48,13) circle [radius=0.2];% e15
\draw[fill] (77-48,13) circle [radius=0.2];% e17
\draw[fill] (81-48,13) circle [radius=0.2];% e21
\draw[fill] (83-48,13) circle [radius=0.2];% e23
\draw[] (90-48,14) circle [radius=0.2]; % e30
\draw[fill] (92-48,14) circle [radius=0.2]; % e32
\draw[fill] (96-48,14) circle [radius=0.2]; % e36
%%%multiple of \kappa^4
\draw[fill] (80-48,16) circle [radius=0.2]; %e0
\draw[] (85-48,17) circle [radius=0.2];%e5
\draw[fill] (86-48,16) circle [radius=0.2];%e6
\draw[fill] (91-48,17) circle [radius=0.2];% e11
\draw[] (95-48,17) circle [radius=0.2];% e15
%%%%%%differentials
\draw[->] (71-48,1)--(70-48+0.05,10-0.2); %d9(delta^2e23)=k^2e30
\draw[->] (38,2)--(37+0.04,17-0.2);%d15(delta^2 e38)
\draw[->] (48,2)--(47+0.04,17-0.2);%d15(delta^2 e48)
\end{tikzpicture}
\caption{HFPSS for $A_{1}[10]\ \mbox{and} A_{1}[01]$ from $\mathrm{E}_{7}$-term with $48\leq t-s\leq 96$}
\phantomsection \label{48-96}
\end{figure}

%%%%96-144
\begin{figure}[h!]
\begin{tikzpicture}[scale=0.42]
\clip(-1.5,-1.5) rectangle (49,27);
\draw[color=gray] (0,0) grid [step=1] (48,26);

\foreach \n in {96,100,...,144}
{
\def\nn{\n-0}
\node[below,scale=0.5] at (\nn-96,0) {$\n$};
}
\foreach \s in {0,2,...,26}
{\def\ss{\s-0};
\node [left,scale=0.5] at (-0.4,\ss,0){$\s$};
}
\draw [fill] ( 0.00, 0.00) circle [radius=0.2];
\draw [fill] (3,1) circle [radius=0.2];
\draw [fill] (6,2) circle [radius=0.2];
\draw [-] (0,0)--(3,1);
\draw [-] (3,1)--(6,2);
%\draw[->,red] (11.8,0.2)--(11.2,2.8);
\foreach \t in {0}
\foreach \s in {0,6}
{\def\ss{\s-0};
\draw[]  (\s-\t,\t) circle [radius=0.2];
\draw [fill]  (0,0) circle [radius=0.2];
\draw [fill] (\s+3-\t,1+\t) circle [radius=0.2];
\draw [-] (\s-\t,\t)--(\s+3-\t,1+\t);
\draw[] (\s+5-\t,1+\t) circle [radius=0.2];
\draw[fill] (5,1) circle [radius=0.2];
\draw[fill] (\s+8-\t,2+\t) circle [radius=0.2];
\draw[-] (\s+5-\t,1+\t)--(\s+8-\t,2+\t); 
\draw [fill] (\s-\t+6,2+\t) circle [radius=0.2];
\draw [-] (\s+3-\t,1+\t)--(\s+6-\t,2+\t);
%\draw[fill] (\s+11-\t,3+\t) circle [radius=0.1];
%\draw[-] (\s+8-\t,2+\t)--(\s+11-\t,3+\t); 
\draw [-] (\s+5-\t,1+\t)--(\s-\t+6,2+\t);
}

\foreach \t in {0}
\foreach \s in {12,18}
{\def\ss{\s-0};
%\draw [fill]  (\s-\t,\t) circle [radius=0.1];
\draw[] (\s+3-\t,1+\t) circle [radius=0.2];
\draw [fill] (12+3,1) circle [radius=0.2];
\draw[] (\s+5-\t,1+\t) circle [radius=0.2];
\draw[fill] (\s+8-\t,2+\t) circle [radius=0.2];
\draw[-] (\s+5-\t,1+\t)--(\s+8-\t,2+\t); 
\draw [fill] (\s-\t+6,2+\t) circle [radius=0.2];
\draw [-] (\s+3-\t,1+\t)--(\s+6-\t,2+\t);
\draw[fill] (\s+11-\t,3+\t) circle [radius=0.2];
\draw[-] (\s+8-\t,2+\t)--(\s+11-\t,3+\t); 
\draw [-] (\s+5-\t,1+\t)--(\s-\t+6,2+\t);
}
\draw[] (0,2) circle [radius=0.2];
\draw[] (5,3) circle [radius=0.2];%delta^2e53
\draw[fill] (30,2) circle [radius=0.2];
\draw[] (32,2) circle [radius=0.2];
\draw[] (36,2) circle [radius=0.2];
\draw[] (38,2) circle [radius=0.2];
\draw[] (42,2) circle [radius=0.2];
\draw[] (47,3) circle [radius=0.2];
\draw[] (48,2) circle [radius=0.2];
\draw[fill] (53,3) circle [radius=0.2];
\draw[] (48,0) circle[radius=0.2];

%%%multiple of \ove{\kappa}
\draw[fill] (50-48,6) circle [radius=0.2]; %delta^2 e30
\draw[fill] (52-48,6) circle [radius=0.2]; %delta^2 e32
\draw[fill] (56-48,6) circle [radius=0.2]; %delta^2 e36
\draw[] (58-48,6) circle [radius=0.2]; %delta^2 e38
\draw[fill] (62-48,6) circle [radius=0.2]; %delta^2 e42
\draw[fill] (67-48,7) circle [radius=0.2]; %delta^2 e47
\draw[] (68-48,6) circle [radius=0.2]; %delta^2 e48
\draw[] (73-48,7) circle [radius=0.2]; %delta^2 e53
\draw[fill] (20,4) circle [radius=0.2]; %delta^4 e0
\draw[fill] (25,5) circle [radius=0.2];%delta^4 e5
\draw[] (26,4) circle [radius=0.2];%delta^4 e6
\draw[] (31,5) circle [radius=0.2];%delta^4 e11
\draw[fill] (35,5) circle [radius=0.2];%delta^4 e15
\draw[] (37,5) circle [radius=0.2];%delta^4 e17
\draw[] (41,5) circle [radius=0.2];%delta^4 e21
\draw[] (43,5) circle [radius=0.2];%delta^4 e23

%%%multiple of \kappa^2
\draw[fill] (51-48,9) circle [radius=0.2];% delta^2 e11
\draw[fill] (55-48,9) circle [radius=0.2];% delta^2 e15
\draw[fill] (57-48,9) circle [radius=0.2];% delta^2 e17
\draw[fill] (61-48,9) circle [radius=0.2];% delta^2 e21
\draw[] (63-48,9) circle [radius=0.2];% delta^2 e23
\draw[fill] (70-48,10) circle [radius=0.2]; % delta^2 e30
\draw[fill] (72-48,10) circle [radius=0.2]; % delta^2 e32
\draw[fill] (76-48,10) circle [radius=0.2]; % delta^2 e36
\draw[] (78-48,10) circle [radius=0.2]; % delta^2 e38
\draw[fill] (82-48,10) circle [radius=0.2]; % delta^2 e42
\draw[fill] (87-48,11) circle [radius=0.2]; % delta^2 e47
\draw[] (88-48,10) circle [radius=0.2]; % delta^2 e48
\draw[] (93-48,11) circle [radius=0.2]; % delta^2 e53
\draw[fill] (40,8) circle [radius=0.2];%\delta^4 e0
\draw[fill] (45,9) circle [radius=0.2];%delta^4 e5
\draw[] (46,8) circle [radius=0.2];%delta^4 e6
%%%multiple of \kappa^3
\draw[fill] (96-96,14) circle [radius=0.2]; %e36
\draw[fill] (98-96,14) circle [radius=0.2]; %e38
\draw[fill] (102-96,14) circle [radius=0.2]; %e42
\draw[fill] (107-96,15) circle [radius=0.2]; %e47
\draw[fill] (108-96,14) circle [radius=0.2]; %e48
\draw[fill] (113-96,15) circle [radius=0.2]; %e53
\draw[fill] (60-48,12) circle [radius=0.2];%delta^2 e0
\draw[fill] (65-48,13) circle [radius=0.2];%delta^2 e5
\draw[fill] (66-48,12) circle [radius=0.2];%delta^2 e6
\draw[fill] (71-48,13) circle [radius=0.2];% delta^2 e11
\draw[fill] (75-48,13) circle [radius=0.2];% delta^2 e15
\draw[fill] (77-48,13) circle [radius=0.2];% delta^2 e17
\draw[fill] (81-48,13) circle [radius=0.2];% delta^2 e21
\draw[] (83-48,13) circle [radius=0.2];% delta^2 e23
\draw[fill] (90-48,14) circle [radius=0.2]; % delta^2e30
\draw[fill] (92-48,14) circle [radius=0.2]; % delta^2 e32
\draw[fill] (96-48,14) circle [radius=0.2]; % delta^2 e36
%%%multiple of \kappa^4
\draw[fill] (80-48,16) circle [radius=0.2]; %delta^2 e0
\draw[fill] (85-48,17) circle [radius=0.2];%delta^2 e5
\draw[fill] (86-48,16) circle [radius=0.2];%delta^2 e6
\draw[fill] (91-48,17) circle [radius=0.2];%delta^2 e11
\draw[] (95-48,17) circle [radius=0.2];%delta^2 e15
\draw[fill] (97-96,17) circle [radius=0.2];%e17
\draw[] (101-96,17) circle [radius=0.2];%e21
\draw[fill] (103-96,17) circle [radius=0.2];%  e23
\draw[] (110-96,18) circle [radius=0.2]; %  e30
\draw[] (112-96,18) circle [radius=0.2]; %  e32
\draw[] (116-96,18) circle [radius=0.2]; % e36
\draw[] (118-96,18) circle [radius=0.2]; % e38
\draw[fill] (122-96,18) circle [radius=0.2]; %  e42
\draw[] (127-96,19) circle [radius=0.2]; % e47
\draw[fill] (128-96,18) circle [radius=0.2]; % e48
\draw[] (133-96,19) circle [radius=0.2]; % e53
%%%multiple of \ove{\kappa}^5
\draw[] (100-96,20) circle [radius=0.2]; %e0
\draw[] (105-96,21) circle [radius=0.2];%e5
\draw[] (106-96,20) circle [radius=0.2];%e6
\draw[fill] (111-96,21) circle [radius=0.2];%e11
\draw[] (115-96,21) circle [radius=0.2];%e15
\draw[fill] (117-96,21) circle [radius=0.2];%e17
\draw[] (121-96,21) circle [radius=0.2];%e21
\draw[fill] (123-96,21) circle [radius=0.2];%e23
\draw[] (130-96,22) circle [radius=0.2];%e30
\draw[] (132-96,22) circle [radius=0.2];%e32
\draw[] (136-96,22) circle [radius=0.2]; % e36
\draw[] (138-96,22) circle [radius=0.2]; % e38
\draw[] (142-96,22) circle [radius=0.2]; % e42
%%%multiple of \ove{\kappa}^6
\draw[] (120-96,24) circle [radius=0.2]; %e0
\draw[] (125-96,25) circle [radius=0.2];%e5
\draw[] (126-96,24) circle [radius=0.2];%e6
\draw[] (131-96,25) circle [radius=0.2];%e11
\draw[] (135-96,25) circle [radius=0.2];%e15
\draw[] (137-96,25) circle [radius=0.2];%e17
\draw[] (141-96,25) circle [radius=0.2];%e21
\draw[] (143-96,25) circle [radius=0.2];%e23
%%%%%differentials
\draw[->] (5,3)--(4+0.03,20-0.2);%d17(delta^2 e53)
\draw[->] (6,0)--(5+0.03,17-0.2);%d17(delta^4 e6)
\draw[->] (17,1)--(16+0.03,18-0.2);%d17(delta^4 e17)
\draw[->] (21,1)--(20+0.03,18-0.2);%d17(delta^4 e21)
\draw[->] (32,2)--(31+0.03,19-0.2);%d17(delta^4 e32)
\draw[->] (23,1)--(22+0.03,18-0.2);%d17(delta^4 e23)
\draw[->] (38,2)--(37+0.03,19-0.2);%d17(delta^4 e38)
\draw[->] (48,0)--(47+0.03,17-0.2);%d17(delta^6 e0)
\draw[->] (11,1)--(10+0.02,20-0.2);%d19(delta^4 e11)
\draw[->] (47,3)--(46+0.02,22-0.2);%d19(delta^4 e47)
\draw[->] (36,2)--(35+0.01,25-0.2);%d23(delta^4 e36)
\draw[->] (42,2)--(41+0.01,25-0.2);%d23(delta^4 e42)
\draw[->] (48,2)--(47+0.01,25-0.2);%d23(delta^4 e48)
\end{tikzpicture}
\caption{HFPSS for $A_{1}[10]\  \mbox{and} A_{1}[01]$ from $\mathrm{E}_{7}$-term with $96\leq t-s\leq 144$}
\phantomsection \label{96-144}
\end{figure}
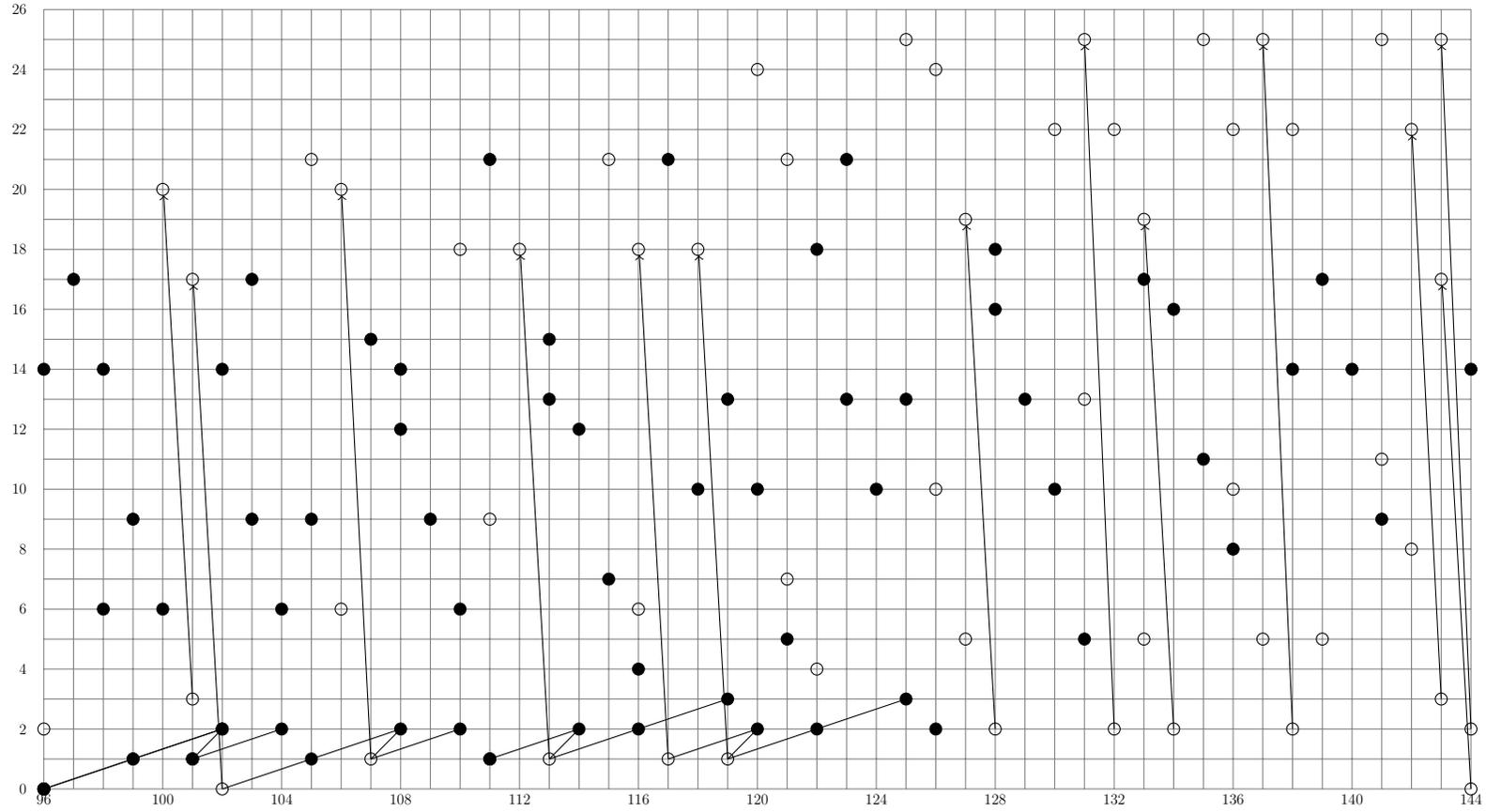
%%%%144-192
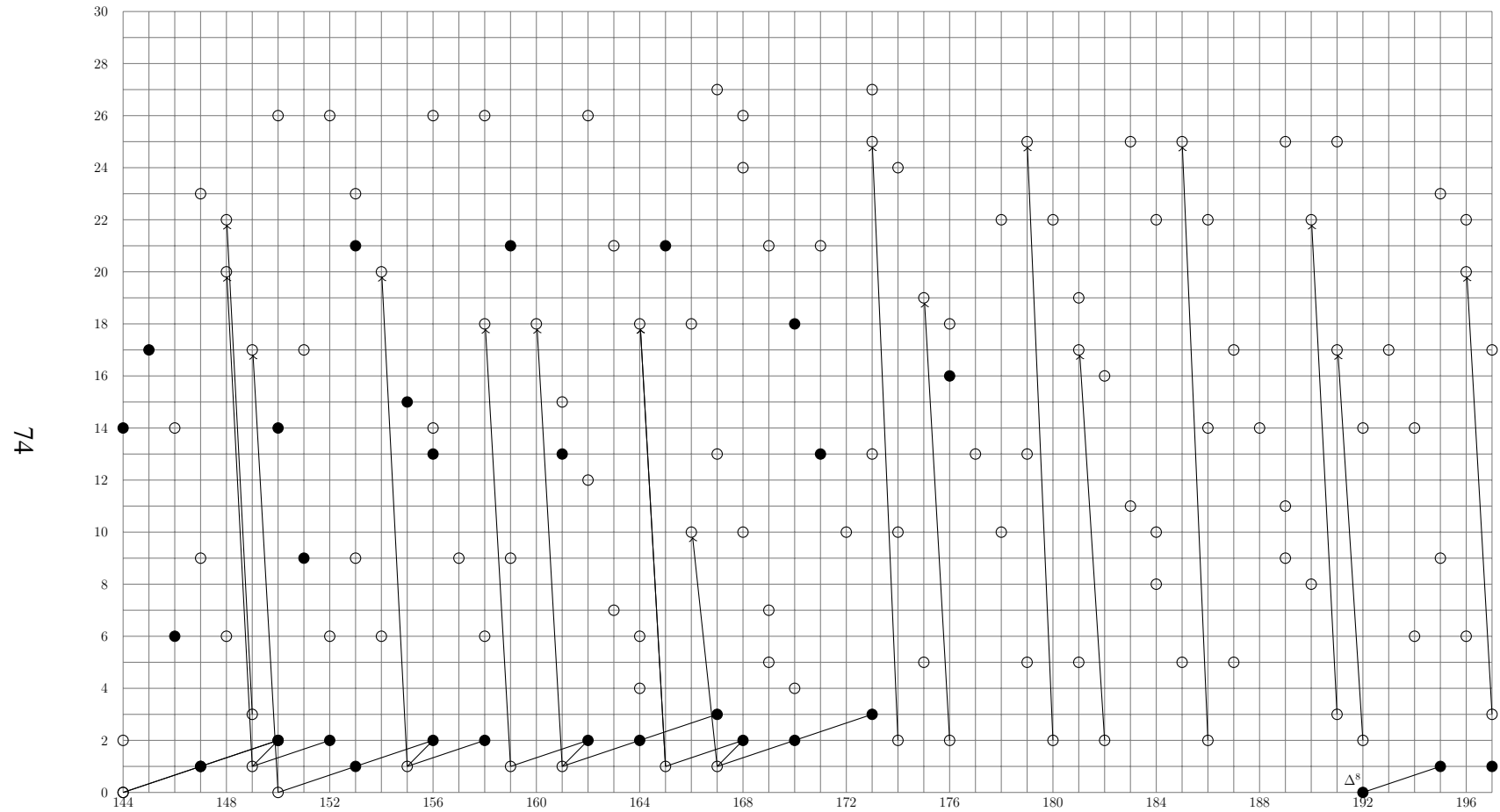
\begin{figure}[h!]
\begin{tikzpicture}[scale=0.39]
\clip(-1.5,-1.5) rectangle (54,31);
\draw[color=gray] (0,0) grid [step=1] (53,30);

\foreach \n in {144,148,...,197}
{
\def\nn{\n-0}
\node[below,scale=0.5] at (\nn-144,0) {$\n$};
}
\foreach \s in {0,2,...,30}
{\def\ss{\s-0};
\node [left,scale=0.5] at (-0.4,\ss,0){$\s$};
}
\draw [] ( 0.00, 0.00) circle [radius=0.2];
\draw [fill] (3,1) circle [radius=0.2];
\draw [fill] (6,2) circle [radius=0.2];
\draw [-] (0,0)--(3,1);
\draw [-] (3,1)--(6,2);
%\draw[->,red] (11.8,0.2)--(11.2,2.8);
\foreach \t in {0}
\foreach \s in {0,6}
{\def\ss{\s-0};
\draw[]  (\s-\t,\t) circle [radius=0.2];
\draw [fill] (\s+3-\t,1+\t) circle [radius=0.2];
\draw [-] (\s-\t,\t)--(\s+3-\t,1+\t);
\draw[] (\s+5-\t,1+\t) circle [radius=0.2];
\draw[fill] (\s+8-\t,2+\t) circle [radius=0.2];
\draw[-] (\s+5-\t,1+\t)--(\s+8-\t,2+\t); 
\draw [fill] (\s-\t+6,2+\t) circle [radius=0.2];
\draw [-] (\s+3-\t,1+\t)--(\s+6-\t,2+\t);
%\draw[fill] (\s+11-\t,3+\t) circle [radius=0.1];
%\draw[-] (\s+8-\t,2+\t)--(\s+11-\t,3+\t); 
\draw [-] (\s+5-\t,1+\t)--(\s-\t+6,2+\t);
}

\foreach \t in {0}
\foreach \s in {12,18}
{\def\ss{\s-0};
%\draw [fill]  (\s-\t,\t) circle [radius=0.1];
\draw[] (\s+3-\t,1+\t) circle [radius=0.2];
\draw[] (\s+5-\t,1+\t) circle [radius=0.2];
\draw[fill] (\s+8-\t,2+\t) circle [radius=0.2];
\draw[-] (\s+5-\t,1+\t)--(\s+8-\t,2+\t); 
\draw [fill] (\s-\t+6,2+\t) circle [radius=0.2];
\draw [-] (\s+3-\t,1+\t)--(\s+6-\t,2+\t);
\draw[fill] (\s+11-\t,3+\t) circle [radius=0.2];
\draw[-] (\s+8-\t,2+\t)--(\s+11-\t,3+\t); 
\draw [-] (\s+5-\t,1+\t)--(\s-\t+6,2+\t);
}
\draw[] (0,2) circle [radius=0.2];
\draw[] (5,3) circle [radius=0.2];
\draw[] (30,2) circle [radius=0.2];
\draw[] (32,2) circle [radius=0.2];
\draw[] (36,2) circle [radius=0.2];
\draw[] (38,2) circle [radius=0.2];
\draw[] (42,2) circle [radius=0.2];
\draw[] (47,3) circle [radius=0.2];
\draw[] (48,2) circle [radius=0.2];
\draw[] (53,3) circle [radius=0.2];
\draw[fill] (48,0) circle[radius=0.2];
\draw[fill] (51,1) circle[radius=0.2];
\draw[-](48,0)--(51,1);
\draw[fill] (53,1) circle[radius=0.2];
\node[scale = 0.5] at (47.6,0.5) {$\Delta^8$};
%%%multiple of \ove{\kappa}
\draw[fill] (50-48,6) circle [radius=0.2]; %delta^4 e30
\draw[] (52-48,6) circle [radius=0.2]; %delta^4 e32
\draw[] (56-48,6) circle [radius=0.2]; %delta^4 e36
\draw[] (58-48,6) circle [radius=0.2]; %delta^4 e38
\draw[] (62-48,6) circle [radius=0.2]; %delta^4 e42
\draw[] (67-48,7) circle [radius=0.2]; %delta^4 e47
\draw[] (68-48,6) circle [radius=0.2]; %delta^4 e48
\draw[] (73-48,7) circle [radius=0.2]; %delta^4 e53
\draw[] (20,4) circle [radius=0.2]; %delta^6 e0
\draw[] (25,5) circle [radius=0.2];%delta^6 e5
\draw[] (26,4) circle [radius=0.2];%delta^6 e6
\draw[] (31,5) circle [radius=0.2];%delta^6 e11
\draw[] (35,5) circle [radius=0.2];%delta^6 e15
\draw[] (37,5) circle [radius=0.2];%delta^6 e17
\draw[] (41,5) circle [radius=0.2];%delta^6 e21
\draw[] (43,5) circle [radius=0.2];%delta^6 e23
\draw[] (50,6) circle [radius=0.2];%delta^6 e30
\draw[] (52,6) circle [radius=0.2];%delta^6 e32
%%%multiple of \kappa^2
\draw[] (51-48,9) circle [radius=0.2];% delta^4 e11
\draw[fill] (55-48,9) circle [radius=0.2];% delta^4 e15
\draw[] (57-48,9) circle [radius=0.2];% delta^4 e17
\draw[] (61-48,9) circle [radius=0.2];% delta^4 e21
\draw[] (63-48,9) circle [radius=0.2];% delta^4 e23
\draw[] (70-48,10) circle [radius=0.2]; % delta^4 e30
\draw[] (72-48,10) circle [radius=0.2]; % delta^4 e32
\draw[] (76-48,10) circle [radius=0.2]; % delta^4 e36
\draw[] (78-48,10) circle [radius=0.2]; %delta^4 e38
\draw[] (82-48,10) circle [radius=0.2]; % delta^4 e42
\draw[] (87-48,11) circle [radius=0.2]; %delta^4 e47
\draw[] (88-48,10) circle [radius=0.2]; % delta^4 e48
\draw[] (93-48,11) circle [radius=0.2]; % delta^4 e53
\draw[] (40,8) circle [radius=0.2];%delta^6 e0
\draw[] (45,9) circle [radius=0.2];%delta^6 e5
\draw[] (46,8) circle [radius=0.2];%delta^6 e6
\draw[] (51,9) circle [radius=0.2];%delta^6 e11
%%%multiple of \kappa^3
\draw[fill] (96-96,14) circle [radius=0.2]; %delta^2e36
\draw[] (98-96,14) circle [radius=0.2]; %delta^2e38
\draw[fill] (102-96,14) circle [radius=0.2]; %delta^2 e42
\draw[fill] (107-96,15) circle [radius=0.2]; %delta^2 e47
\draw[] (108-96,14) circle [radius=0.2]; %delta^2 e48
\draw[] (113-96,15) circle [radius=0.2]; %delta^2 e53
\draw[fill] (60-48,13) circle [radius=0.2];%delta^4 e0
\draw[fill] (65-48,13) circle [radius=0.2];%delta^4 e5
\draw[] (66-48,12) circle [radius=0.2];%delta^4 e6
\draw[] (71-48,13) circle [radius=0.2];% delta^4 e11
\draw[fill] (75-48,13) circle [radius=0.2];% delta^4 e15
\draw[] (77-48,13) circle [radius=0.2];% delta^4 e17
\draw[] (81-48,13) circle [radius=0.2];% delta^4 e21
\draw[] (83-48,13) circle [radius=0.2];% delta^4 e23
\draw[] (90-48,14) circle [radius=0.2]; % delta^4e30
\draw[] (92-48,14) circle [radius=0.2]; % delta^4 e32
\draw[] (96-48,14) circle [radius=0.2]; % delta^4 e36
\draw[] (98-48,14) circle [radius=0.2]; %delta^4 e38

%%%multiple of \kappa^4
\draw[fill] (80-48,16) circle [radius=0.2]; %delta^4 e0
\draw[] (85-48,17) circle [radius=0.2];%delta^4 e5
\draw[] (86-48,16) circle [radius=0.2];%delta^4 e6
\draw[] (91-48,17) circle [radius=0.2];%delta^4 e11
\draw[] (95-48,17) circle [radius=0.2];%delta^4 e15
\draw[] (97-48,17) circle [radius=0.2];%delta^4 e17
\draw[] (101-48,17) circle [radius=0.2];%delta^4 e21
\draw[fill] (97-96,17) circle [radius=0.2];%  delta^2 e17
\draw[] (101-96,17) circle [radius=0.2];% delta^2 e21
\draw[] (103-96,17) circle [radius=0.2];% delta^2 e23
\draw[] (110-96,18) circle [radius=0.2]; % delta^2 e30
\draw[] (112-96,18) circle [radius=0.2]; % delta^2 e32
\draw[] (116-96,18) circle [radius=0.2]; % delta^2 e36
\draw[] (118-96,18) circle [radius=0.2]; % delta^2 e38
\draw[fill] (122-96,18) circle [radius=0.2]; % delta^2 e42
\draw[] (127-96,19) circle [radius=0.2]; % delta^2 e47
\draw[] (128-96,18) circle [radius=0.2]; %delta^2 e48
\draw[] (133-96,19) circle [radius=0.2]; % delta^2 e53

%%%multiple of \ove{\kappa}^5
\draw[] (147-144,23) circle [radius=0.2]; %e47
\draw[] (148-144,22) circle [radius=0.2]; %e48
\draw[] (153-144,23) circle [radius=0.2]; %e53
\draw[] (100-96,20) circle [radius=0.2]; %delta^2 e0
\draw[fill] (105-96,21) circle [radius=0.2];%delta^2 e5
\draw[] (106-96,20) circle [radius=0.2];%delta^2 e6
\draw[fill] (111-96,21) circle [radius=0.2];%delta^2 e11
\draw[] (115-96,21) circle [radius=0.2];%delta^2 e15
\draw[fill] (117-96,21) circle [radius=0.2];%delta^2 e17
\draw[] (121-96,21) circle [radius=0.2];%delta^2 e21
\draw[] (123-96,21) circle [radius=0.2];%delta^2 e23
\draw[] (130-96,22) circle [radius=0.2];%delta^2 e30
\draw[] (132-96,22) circle [radius=0.2];%delta^2 e32
\draw[] (136-96,22) circle [radius=0.2]; % delta^2 e36
\draw[] (138-96,22) circle [radius=0.2]; % delta^2 e38
\draw[] (142-96,22) circle [radius=0.2]; % delta^2 e42
\draw[] (147-96,23) circle [radius=0.2]; % delta^2 e47
\draw[] (148-96,22) circle [radius=0.2]; % delta^2 e48
\draw[] (148-96,20) circle [radius=0.2]; % delta^4 e0
%%%multiple of \ove{\kappa}^6
\draw[] (150-144,26) circle [radius=0.2]; % e30
\draw[] (152-144,26) circle [radius=0.2]; % e32
\draw[] (156-144,26) circle [radius=0.2]; % e36
\draw[] (158-144,26) circle [radius=0.2]; % e38
\draw[] (162-144,26) circle [radius=0.2]; % e42
\draw[] (167-144,27) circle [radius=0.2]; % e47
\draw[] (168-144,26) circle [radius=0.2]; % e48
\draw[] (173-144,27) circle [radius=0.2]; % e53
\draw[] (168-144,24) circle [radius=0.2]; %delta^2 e0
\draw[] (125+48-144,25) circle [radius=0.2];%delta^2e5
\draw[] (126+48-144,24) circle [radius=0.2];%delta^2 e6
\draw[] (131+48-144,25) circle [radius=0.2];%delta^2e11
\draw[] (135+48-144,25) circle [radius=0.2];%delta^2 e15
\draw[] (137+48-144,25) circle [radius=0.2];%delta^2e17
\draw[] (141+48-144,25) circle [radius=0.2];%delta^2 e21
\draw[] (143+48-144,25) circle [radius=0.2];%delta^2e23
%%%differentials
\draw[->] (23,1)--(22+0.05,10-0.2);%d9(delta^6 e23)
\draw[->] (38,2)--(37+0.04,17-0.2);%d15(delta^6 e38)
\draw[->] (48,2)--(47+0.04,17-0.2);%d15(delta^6 e48)
\draw[->] (6,0)--(5+0.03,17-0.2);%d17(delta^6 e6)
\draw[->] (17,1)--(16+0.03,18-0.2);%d17(delta^6 e17)
\draw[->] (21,1)--(20+0.03,18-0.2);%d17(delta^6 e21)
\draw[->] (21,1)--(20+0.03,18-0.2);%d17(delta^6 e21)
\draw[->] (32,2)--(31+0.03,19-0.2);%d17(delta^6 e32)
\draw[->] (53,3)--(52+0.03,20-0.2);%d17(delta^6 e53)
\draw[->] (15,1)--(14+0.03,18-0.2);%d17(delta^6 e53)
\draw[->] (11,1)--(10+0.02,20-0.2);%d19(delta^6 e11)
\draw[->] (47,3)--(46+0.02,22-0.2);%d19(delta^6 e47)
\draw[->] (5,1)--(4+0.02,20-0.2);%d19(delta^6 e5)
\draw[->] (5,3)--(4+0.02,22-0.2);%d19(delta^4 e53)
\draw[->] (36,2)--(35+0.01,25-0.2);%d23(delta^6 e36)
\draw[->] (42,2)--(41+0.01,25-0.2);%d23(delta^6 e42)
\draw[->] (30,2)--(29+0.01,25-0.2);%d23(delta^6 e30)
\end{tikzpicture}
\caption{HFPSS for $A_{1}[10]\ \mbox{and} \ A_{1}[01]$ from $\mathrm{E}_{7}$-term with $144\leq t-s\leq 197$}
\phantomsection \label{144-197}
\end{figure}
\end{landscape}

\begin{landscape}
The Figures (\ref{0-48 nondual}) to (\ref{144-197 nondual}) represent the HFPSS for $E_{C}^{hG_{24}}\wedge A_{1}[00]$ and $E_{C}^{hG_{24}}\wedge A_{1}[11]$ from the $\mathrm{E}_{7}$-term on. Each black dot $\bullet$ represents a class generating a group $\FF_{4}$ which survives to the $\mathrm{E}_{\infty}$-term. Each circle $\circ$ represent a class which either is hit by a differential or supports a differential higher than $d_{5}$. We only represent the differentials on generators listed in Proposition \ref{E_7} but not those generated by $\overline{\kappa}$-linearity.

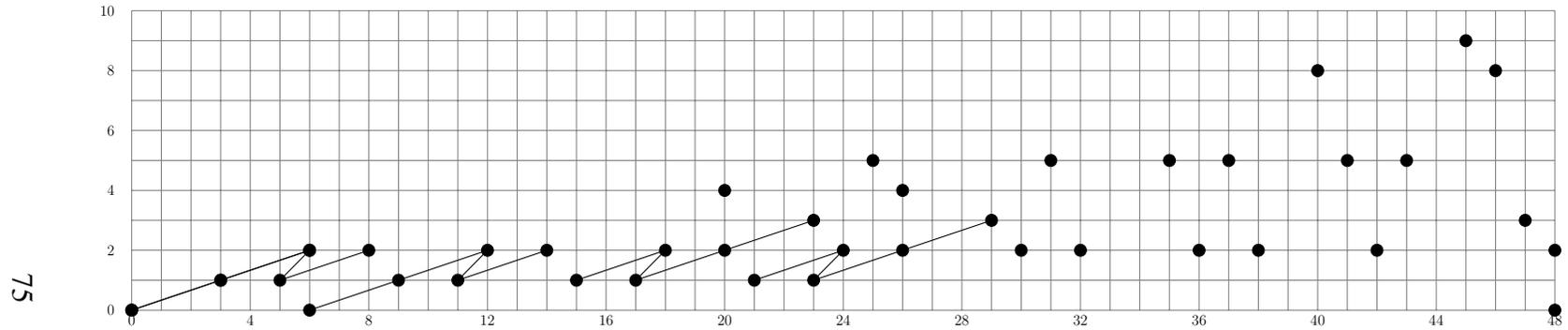
\begin{figure}[h!]
\begin{tikzpicture}[scale=0.42]
\clip(-1.5,-1.5) rectangle (49,11);
\draw[color=gray] (0,0) grid [step=1] (48,10);

\foreach \n in {0,4,...,48}
{
\def\nn{\n-0}
\node[below,scale=0.5] at (\nn,0) {$\n$};
}
\foreach \s in {0,2,...,10}
{\def\ss{\s-0};
\node [left,scale=0.5] at (-0.4,\ss,0){$\s$};
}
\draw [fill] ( 0.00, 0.00) circle [radius=0.2];
\draw [fill] (3,1) circle [radius=0.2];
\draw [fill] (6,2) circle [radius=0.2];
\draw [-] (0,0)--(3,1);
\draw [-] (3,1)--(6,2);
%\draw[->,red] (11.8,0.2)--(11.2,2.8);
\foreach \t in {0}
\foreach \s in {0,6}
{\def\ss{\s-0};
\draw [fill]  (\s-\t,\t) circle [radius=0.2];
\draw [fill] (\s+3-\t,1+\t) circle [radius=0.2];
\draw [-] (\s-\t,\t)--(\s+3-\t,1+\t);
\draw[fill] (\s+5-\t,1+\t) circle [radius=0.2];
\draw[fill] (\s+8-\t,2+\t) circle [radius=0.2];
\draw[-] (\s+5-\t,1+\t)--(\s+8-\t,2+\t); 
\draw [fill] (\s-\t+6,2+\t) circle [radius=0.2];
\draw [-] (\s+3-\t,1+\t)--(\s+6-\t,2+\t);
%\draw[fill] (\s+11-\t,3+\t) circle [radius=0.1];
%\draw[-] (\s+8-\t,2+\t)--(\s+11-\t,3+\t); 
\draw [-] (\s+5-\t,1+\t)--(\s-\t+6,2+\t);
}

\foreach \t in {0}
\foreach \s in {12,18}
{\def\ss{\s-0};
%\draw [fill]  (\s-\t,\t) circle [radius=0.1];
\draw [fill] (\s+3-\t,1+\t) circle [radius=0.2];
%\draw [-] (\s-\t,\t)--(\s+3-\t,1+\t);
\draw[fill] (\s+5-\t,1+\t) circle [radius=0.2];
\draw[fill] (\s+8-\t,2+\t) circle [radius=0.2];
\draw[-] (\s+5-\t,1+\t)--(\s+8-\t,2+\t); 
\draw [fill] (\s-\t+6,2+\t) circle [radius=0.2];
\draw [-] (\s+3-\t,1+\t)--(\s+6-\t,2+\t);
\draw[fill] (\s+11-\t,3+\t) circle [radius=0.2];
\draw[-] (\s+8-\t,2+\t)--(\s+11-\t,3+\t); 
\draw [-] (\s+5-\t,1+\t)--(\s-\t+6,2+\t);
}
%\draw[fill] (0,2) circle [radius=0.2];
%\draw[fill] (5,3) circle [radius=0.2];
\draw[fill] (30,2) circle [radius=0.2];
\draw[fill] (32,2) circle [radius=0.2];
\draw[fill] (36,2) circle [radius=0.2];
\draw[fill] (38,2) circle [radius=0.2];
\draw[fill] (42,2) circle [radius=0.2];
\draw[fill] (47,3) circle [radius=0.2];
\draw[fill] (48,2) circle [radius=0.2];
\draw[fill] (53,3) circle [radius=0.2];
\draw[fill] (48,0) circle[radius=0.2];

%%%multiple of \ove{\kappa}
\draw[fill] (20,4) circle [radius=0.2]; %e0
\draw[fill] (25,5) circle [radius=0.2];%e5
\draw[fill] (26,4) circle [radius=0.2];%e6
\draw[fill] (31,5) circle [radius=0.2];%e11
\draw[fill] (35,5) circle [radius=0.2];%e15
\draw[fill] (37,5) circle [radius=0.2];%e17
\draw[fill] (41,5) circle [radius=0.2];%e21
\draw[fill] (43,5) circle [radius=0.2];%e23
\draw[fill] (50,6) circle [radius=0.2];%e30
\draw[fill] (52,6) circle [radius=0.2];%e32
%%%multiple of \kappa^2
\draw[fill] (40,8) circle [radius=0.2];%e0
\draw[fill] (45,9) circle [radius=0.2];%e5
\draw[fill] (46,8) circle [radius=0.2];%e6
\end{tikzpicture}
\caption{HFPSS for $A_{1}[00]\ \mbox{and}\ A_{1}[11]$ from $\mathrm{E}_{7}$-term with $0\leq t-s\leq 48$}
\phantomsection \label{0-48 nondual}
\end{figure}
\end{landscape}

%%%%48-96
\begin{landscape}
\begin{figure}[h!]
\begin{tikzpicture}[scale=0.42]
\clip(-1.5,-1.5) rectangle (49,19);
\draw[color=gray] (0,0) grid [step=1] (48,18);

\foreach \n in {48,52,...,96}
{
\def\nn{\n-0}
\node[below,scale=0.5] at (\nn-48,0) {$\n$};
}
\foreach \s in {0,2,...,18}
{\def\ss{\s-0};
\node [left,scale=0.5] at (-0.4,\ss,0){$\s$};
}
\draw [fill] ( 0.00, 0.00) circle [radius=0.2];
\draw [fill] (3,1) circle [radius=0.2];
\draw [fill] (6,2) circle [radius=0.2];
\draw [-] (0,0)--(3,1);
\draw [-] (3,1)--(6,2);
%\draw[->,red] (11.8,0.2)--(11.2,2.8);
\foreach \t in {0}
\foreach \s in {0,6}
{\def\ss{\s-0};
\draw [fill]  (\s-\t,\t) circle [radius=0.2];
\draw [fill] (\s+3-\t,1+\t) circle [radius=0.2];
\draw [-] (\s-\t,\t)--(\s+3-\t,1+\t);
\draw[fill] (\s+5-\t,1+\t) circle [radius=0.2];
\draw[fill] (\s+8-\t,2+\t) circle [radius=0.2];
\draw[-] (\s+5-\t,1+\t)--(\s+8-\t,2+\t); 
\draw [fill] (\s-\t+6,2+\t) circle [radius=0.2];
\draw [-] (\s+3-\t,1+\t)--(\s+6-\t,2+\t);
%\draw[fill] (\s+11-\t,3+\t) circle [radius=0.1];
%\draw[-] (\s+8-\t,2+\t)--(\s+11-\t,3+\t); 
\draw [-] (\s+5-\t,1+\t)--(\s-\t+6,2+\t);
}

\foreach \t in {0}
\foreach \s in {12,18}
{\def\ss{\s-0};
%\draw [fill]  (\s-\t,\t) circle [radius=0.1];
\draw [fill] (\s+3-\t,1+\t) circle [radius=0.2];
%\draw [-] (\s-\t,\t)--(\s+3-\t,1+\t);
\draw[fill] (\s+5-\t,1+\t) circle [radius=0.2];
\draw[fill] (12+5,1) circle [radius=0.2];
\draw[fill] (\s+8-\t,2+\t) circle [radius=0.2];
\draw[-] (\s+5-\t,1+\t)--(\s+8-\t,2+\t); 
\draw [fill] (\s-\t+6,2+\t) circle [radius=0.2];
\draw [-] (\s+3-\t,1+\t)--(\s+6-\t,2+\t);
\draw[fill] (\s+11-\t,3+\t) circle [radius=0.2];
\draw[-] (\s+8-\t,2+\t)--(\s+11-\t,3+\t); 
\draw [-] (\s+5-\t,1+\t)--(\s-\t+6,2+\t);
}
\draw[fill] (0,2) circle [radius=0.2];
\draw[fill] (5,3) circle [radius=0.2];
\draw[fill] (30,2) circle [radius=0.2];
\draw[fill] (32,2) circle [radius=0.2];
\draw[fill] (36,2) circle [radius=0.2];
\draw[fill] (38,2) circle [radius=0.2];
\draw[fill] (42,2) circle [radius=0.2];
\draw[fill] (47,3) circle [radius=0.2];
\draw[fill] (48,2) circle [radius=0.2];
\draw[fill] (53,3) circle [radius=0.2];
\draw[] (48,0) circle[radius=0.2];
%%%multiple of \ove{\kappa}
\draw[fill] (50-48,6) circle [radius=0.2]; % e30
\draw[fill] (52-48,6) circle [radius=0.2]; % e32
\draw[fill] (56-48,6) circle [radius=0.2]; % e36
\draw[fill] (58-48,6) circle [radius=0.2]; % e38
\draw[fill] (62-48,6) circle [radius=0.2]; % e42
\draw[fill] (67-48,7) circle [radius=0.2]; % e47
\draw[fill] (68-48,6) circle [radius=0.2]; % e48
\draw[fill] (73-48,7) circle [radius=0.2]; % e53
\draw[fill] (20,4) circle [radius=0.2]; %delta^2 e0
\draw[fill] (25,5) circle [radius=0.2];%delta^2 e5
\draw[fill] (26,4) circle [radius=0.2];%delta^2 e6
\draw[fill] (31,5) circle [radius=0.2];%delta^2 e11
\draw[fill] (35,5) circle [radius=0.2];%delta^2 e15
\draw[fill] (37,5) circle [radius=0.2];%delta^2 e17
\draw[fill] (41,5) circle [radius=0.2];%delta^2 e21
\draw[fill] (43,5) circle [radius=0.2];%delta^2 e23

%%%multiple of \kappa^2
\draw[fill] (51-48,9) circle [radius=0.2];% e11
\draw[fill] (55-48,9) circle [radius=0.2];% e15
\draw[fill] (57-48,9) circle [radius=0.2];% e17
\draw[fill] (61-48,9) circle [radius=0.2];% e21
\draw[fill] (63-48,9) circle [radius=0.2];% e23
\draw[fill] (70-48,10) circle [radius=0.2]; % e30
\draw[fill] (72-48,10) circle [radius=0.2]; % e32
\draw[fill] (76-48,10) circle [radius=0.2]; % e36
\draw[fill] (78-48,10) circle [radius=0.2]; % e38
\draw[fill] (82-48,10) circle [radius=0.2]; % e42
\draw[fill] (87-48,11) circle [radius=0.2]; % e47
\draw[fill] (88-48,10) circle [radius=0.2]; % e48
\draw[fill] (93-48,11) circle [radius=0.2]; % e53
\draw[fill] (40,8) circle [radius=0.2];%\delta^2 e0
\draw[fill] (45,9) circle [radius=0.2];%delta^2 e5
\draw[fill] (46,8) circle [radius=0.2];%delta^2 e6
%%%multiple of \kappa^3
\draw[fill] (60-48,12) circle [radius=0.2]; %e0
\draw[fill] (65-48,13) circle [radius=0.2];%e5
\draw[fill] (66-48,12) circle [radius=0.2];%e6
\draw[fill] (71-48,13) circle [radius=0.2];% e11
\draw[fill] (75-48,13) circle [radius=0.2];% e15
\draw[fill] (77-48,13) circle [radius=0.2];% e17
\draw[fill] (81-48,13) circle [radius=0.2];% e21
\draw[fill] (83-48,13) circle [radius=0.2];% e23
\draw[fill] (90-48,14) circle [radius=0.2]; % e30
\draw[fill] (92-48,14) circle [radius=0.2]; % e32
\draw[fill] (96-48,14) circle [radius=0.2]; % e36
%%%multiple of \kappa^4
\draw[fill] (80-48,16) circle [radius=0.2]; %e0
\draw[fill] (85-48,17) circle [radius=0.2];%e5
\draw[fill] (86-48,16) circle [radius=0.2];%e6
\draw[fill] (91-48,17) circle [radius=0.2];% e11
\draw[fill] (95-48,17) circle [radius=0.2];% e15
%%%%%%differentials
%\draw[->] (71-48,1)--(70-48+0.05,10-0.2); %d9(delta^2e23)=k^2e30
%\draw[->] (38,2)--(37+0.04,17-0.2);%d15(delta^2 e38)
\draw[->] (48,0)--(47+0.04,17-0.2);%d1(delta^2 e48)
\end{tikzpicture}
\caption{HFPSS for $A_{1}[00]\ \mbox{and} A_{1}[11]$ from $\mathrm{E}_{7}$-term with $48\leq t-s\leq 96$}
\phantomsection \label{48-96 nondual}
\end{figure}
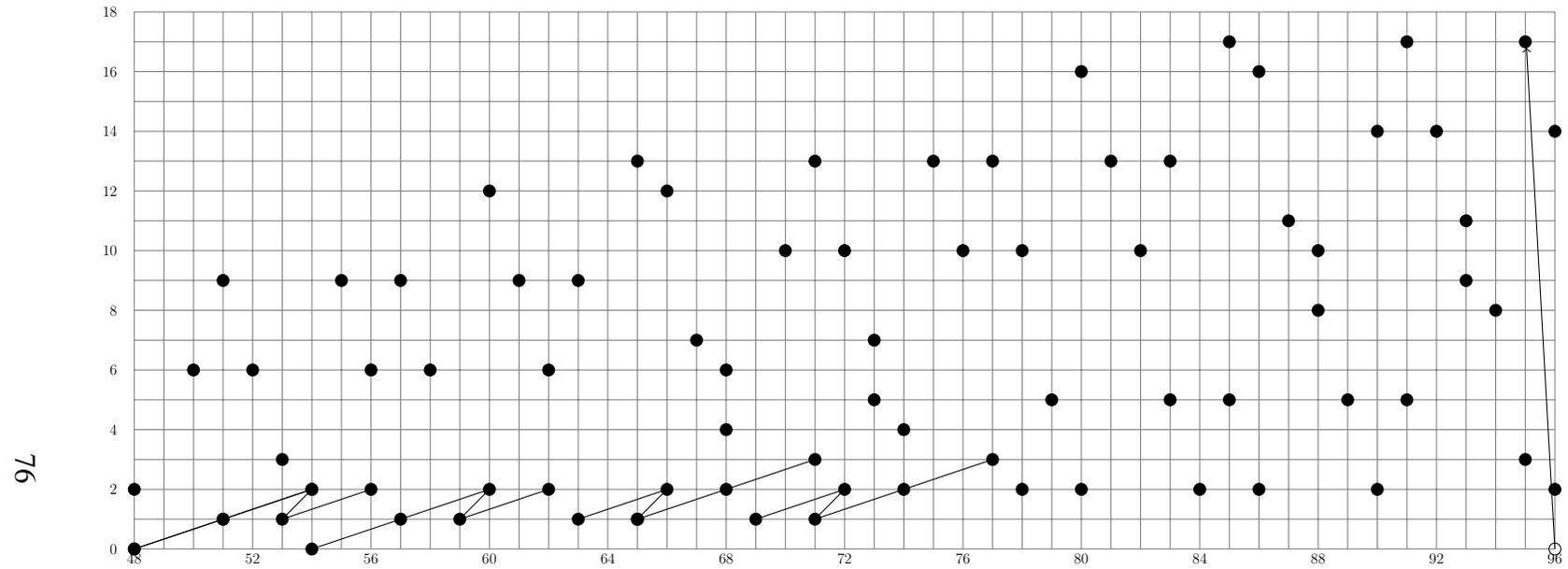
\end{landscape}
%%%%96-144

\begin{landscape}
\begin{figure}[h!]
\begin{tikzpicture}[scale=0.42]
\clip(-1.5,-1.5) rectangle (49,27);
\draw[color=gray] (0,0) grid [step=1] (48,26);

\foreach \n in {96,100,...,144}
{
\def\nn{\n-0}
\node[below,scale=0.5] at (\nn-96,0) {$\n$};
}
\foreach \s in {0,2,...,26}
{\def\ss{\s-0};
\node [left,scale=0.5] at (-0.4,\ss,0){$\s$};
}
%\draw [fill] ( 0.00, 0.00) circle [radius=0.2];
\draw [fill] (3,1) circle [radius=0.2];
\draw [fill] (6,2) circle [radius=0.2];
\draw [-] (0,0)--(3,1);
\draw [-] (3,1)--(6,2);
%\draw[->,red] (11.8,0.2)--(11.2,2.8);
\foreach \t in {0}
\foreach \s in {0,6}
{\def\ss{\s-0};
\draw[]  (\s-\t,\t) circle [radius=0.2];
%\draw [fill]  (0,0) circle [radius=0.2];
\draw [fill] (\s+3-\t,1+\t) circle [radius=0.2];
\draw [-] (\s-\t,\t)--(\s+3-\t,1+\t);
\draw[] (\s+5-\t,1+\t) circle [radius=0.2];
%\draw[fill] (5,1) circle [radius=0.2];
\draw[fill] (\s+8-\t,2+\t) circle [radius=0.2];
\draw[-] (\s+5-\t,1+\t)--(\s+8-\t,2+\t); 
\draw [fill] (\s-\t+6,2+\t) circle [radius=0.2];
\draw [-] (\s+3-\t,1+\t)--(\s+6-\t,2+\t);
%\draw[fill] (\s+11-\t,3+\t) circle [radius=0.1];
%\draw[-] (\s+8-\t,2+\t)--(\s+11-\t,3+\t); 
\draw [-] (\s+5-\t,1+\t)--(\s-\t+6,2+\t);
}

\foreach \t in {0}
\foreach \s in {12,18}
{\def\ss{\s-0};
%\draw [fill]  (\s-\t,\t) circle [radius=0.1];
\draw[] (\s+3-\t,1+\t) circle [radius=0.2];
\draw [] (12+3,1) circle [radius=0.2];
\draw[] (\s+5-\t,1+\t) circle [radius=0.2];
\draw[fill] (\s+8-\t,2+\t) circle [radius=0.2];
\draw[-] (\s+5-\t,1+\t)--(\s+8-\t,2+\t); 
\draw [fill] (\s-\t+6,2+\t) circle [radius=0.2];
\draw [-] (\s+3-\t,1+\t)--(\s+6-\t,2+\t);
\draw[fill] (\s+11-\t,3+\t) circle [radius=0.2];
\draw[-] (\s+8-\t,2+\t)--(\s+11-\t,3+\t); 
\draw [-] (\s+5-\t,1+\t)--(\s-\t+6,2+\t);
}
\draw[fill] (0,2) circle [radius=0.2];
\draw[fill] (5,3) circle [radius=0.2];%delta^2e53
\draw[] (30,2) circle [radius=0.2];
\draw[] (32,2) circle [radius=0.2];
\draw[] (36,2) circle [radius=0.2];
\draw[] (38,2) circle [radius=0.2];
\draw[] (42,2) circle [radius=0.2];
\draw[] (47,3) circle [radius=0.2];
\draw[] (48,2) circle [radius=0.2];
\draw[fill] (53,3) circle [radius=0.2];
\draw[] (48,0) circle[radius=0.2];

%%%multiple of \ove{\kappa}
\draw[fill] (50-48,6) circle [radius=0.2]; %delta^2 e30
\draw[fill] (52-48,6) circle [radius=0.2]; %delta^2 e32
\draw[fill] (56-48,6) circle [radius=0.2]; %delta^2 e36
\draw[fill] (58-48,6) circle [radius=0.2]; %delta^2 e38
\draw[fill] (62-48,6) circle [radius=0.2]; %delta^2 e42
\draw[fill] (67-48,7) circle [radius=0.2]; %delta^2 e47
\draw[] (68-48,6) circle [radius=0.2]; %delta^2 e48
\draw[] (73-48,7) circle [radius=0.2]; %delta^2 e53
\draw[fill] (20,4) circle [radius=0.2]; %delta^4 e0
\draw[fill] (25,5) circle [radius=0.2];%delta^4 e5
\draw[] (26,4) circle [radius=0.2];%delta^4 e6
\draw[] (31,5) circle [radius=0.2];%delta^4 e11
\draw[fill] (35,5) circle [radius=0.2];%delta^4 e15
\draw[] (37,5) circle [radius=0.2];%delta^4 e17
\draw[] (41,5) circle [radius=0.2];%delta^4 e21
\draw[] (43,5) circle [radius=0.2];%delta^4 e23

%%%multiple of \kappa^2
\draw[fill] (51-48,9) circle [radius=0.2];% delta^2 e11
\draw[fill] (55-48,9) circle [radius=0.2];% delta^2 e15
\draw[fill] (57-48,9) circle [radius=0.2];% delta^2 e17
\draw[fill] (61-48,9) circle [radius=0.2];% delta^2 e21
\draw[fill] (63-48,9) circle [radius=0.2];% delta^2 e23
\draw[fill] (70-48,10) circle [radius=0.2]; % delta^2 e30
\draw[fill] (72-48,10) circle [radius=0.2]; % delta^2 e32
\draw[fill] (76-48,10) circle [radius=0.2]; % delta^2 e36
\draw[fill] (78-48,10) circle [radius=0.2]; % delta^2 e38
\draw[fill] (82-48,10) circle [radius=0.2]; % delta^2 e42
\draw[fill] (87-48,11) circle [radius=0.2]; % delta^2 e47
\draw[] (88-48,10) circle [radius=0.2]; % delta^2 e48
\draw[] (93-48,11) circle [radius=0.2]; % delta^2 e53
\draw[fill] (40,8) circle [radius=0.2];%\delta^4 e0
\draw[fill] (45,9) circle [radius=0.2];%delta^4 e5
\draw[] (46,8) circle [radius=0.2];%delta^4 e6
%%%multiple of \kappa^3
\draw[fill] (96-96,14) circle [radius=0.2]; %e36
\draw[fill] (98-96,14) circle [radius=0.2]; %e38
\draw[fill] (102-96,14) circle [radius=0.2]; %e42
\draw[fill] (107-96,15) circle [radius=0.2]; %e47
\draw[fill] (108-96,14) circle [radius=0.2]; %e48
\draw[fill] (113-96,15) circle [radius=0.2]; %e53
\draw[fill] (60-48,13) circle [radius=0.2];%delta^2 e0
\draw[fill] (65-48,13) circle [radius=0.2];%delta^2 e5
\draw[fill] (66-48,12) circle [radius=0.2];%delta^2 e6
\draw[fill] (71-48,13) circle [radius=0.2];% delta^2 e11
\draw[fill] (75-48,13) circle [radius=0.2];% delta^2 e15
\draw[fill] (77-48,13) circle [radius=0.2];% delta^2 e17
\draw[fill] (81-48,13) circle [radius=0.2];% delta^2 e21
\draw[fill] (83-48,13) circle [radius=0.2];% delta^2 e23
\draw[fill] (90-48,14) circle [radius=0.2]; % delta^2e30
\draw[fill] (92-48,14) circle [radius=0.2]; % delta^2 e32
\draw[fill] (96-48,14) circle [radius=0.2]; % delta^2 e36
%%%multiple of \kappa^4
\draw[fill] (80-48,16) circle [radius=0.2]; %delta^2 e0
\draw[fill] (85-48,17) circle [radius=0.2];%delta^2 e5
\draw[fill] (86-48,16) circle [radius=0.2];%delta^2 e6
\draw[fill] (91-48,17) circle [radius=0.2];%delta^2 e11
\draw[] (95-48,17) circle [radius=0.2];%delta^2 e15
\draw[fill] (97-96,17) circle [radius=0.2];%e17
\draw[] (101-96,17) circle [radius=0.2];%e21
\draw[fill] (103-96,17) circle [radius=0.2];%  e23
\draw[] (110-96,18) circle [radius=0.2]; %  e30
\draw[] (112-96,18) circle [radius=0.2]; %  e32
\draw[] (116-96,18) circle [radius=0.2]; % e36
\draw[] (118-96,18) circle [radius=0.2]; % e38
\draw[fill] (122-96,18) circle [radius=0.2]; %  e42
\draw[] (127-96,19) circle [radius=0.2]; % e47
\draw[fill] (128-96,18) circle [radius=0.2]; % e48
\draw[] (133-96,19) circle [radius=0.2]; % e53
%%%multiple of \ove{\kappa}^5
\draw[] (100-96,20) circle [radius=0.2]; %e0
\draw[fill] (105-96,21) circle [radius=0.2];%e5
\draw[] (106-96,20) circle [radius=0.2];%e6
\draw[fill] (111-96,21) circle [radius=0.2];%e11
\draw[] (115-96,21) circle [radius=0.2];%e15
\draw[fill] (117-96,21) circle [radius=0.2];%e17
\draw[] (121-96,21) circle [radius=0.2];%e21
\draw[fill] (123-96,21) circle [radius=0.2];%e23
\draw[] (130-96,22) circle [radius=0.2];%e30
\draw[] (132-96,22) circle [radius=0.2];%e32
\draw[] (136-96,22) circle [radius=0.2]; % e36
\draw[] (138-96,22) circle [radius=0.2]; % e38
\draw[] (142-96,22) circle [radius=0.2]; % e42
%%%multiple of \ove{\kappa}^6
\draw[] (120-96,24) circle [radius=0.2]; %e0
\draw[] (125-96,25) circle [radius=0.2];%e5
\draw[] (126-96,24) circle [radius=0.2];%e6
\draw[] (131-96,25) circle [radius=0.2];%e11
\draw[] (135-96,25) circle [radius=0.2];%e15
\draw[] (137-96,25) circle [radius=0.2];%e17
\draw[] (141-96,25) circle [radius=0.2];%e21
\draw[] (143-96,25) circle [radius=0.2];%e23
%%%%%differentials
\draw[->] (5,1)--(4+0.03,20-0.2);%d17(delta^2 e53)
\draw[->] (6,0)--(5+0.03,17-0.2);%d17(delta^4 e6)
\draw[->] (15,1)--(14+0.03,18-0.2);%d17(delta^4 e17)
\draw[->] (17,1)--(16+0.03,18-0.2);%d17(delta^4 e17)
\draw[->] (21,1)--(20+0.03,18-0.2);%d17(delta^4 e21)
\draw[->] (32,2)--(31+0.03,19-0.2);%d17(delta^4 e32)
\draw[->] (23,1)--(22+0.03,18-0.2);%d17(delta^4 e23)
\draw[->] (30,2)--(29+0.03,25-0.2);%d17(delta^4 e38)
\draw[->] (38,2)--(37+0.03,19-0.2);%d17(delta^4 e38)
\draw[->] (48,0)--(47+0.03,17-0.2);%d17(delta^6 e0)
\draw[->] (11,1)--(10+0.02,20-0.2);%d19(delta^4 e11)
\draw[->] (47,3)--(46+0.02,22-0.2);%d19(delta^4 e47)
\draw[->] (36,2)--(35+0.01,25-0.2);%d23(delta^4 e36)
\draw[->] (42,2)--(41+0.01,25-0.2);%d23(delta^4 e42)
\draw[->] (48,2)--(47+0.01,25-0.2);%d23(delta^4 e48)
\end{tikzpicture}
\caption{HFPSS for $A_{1}[00]\  \mbox{and} A_{1}[11]$ from $\mathrm{E}_{7}$-term with $96\leq t-s\leq 144$}
\phantomsection \label{96-144 nondual}
\end{figure}
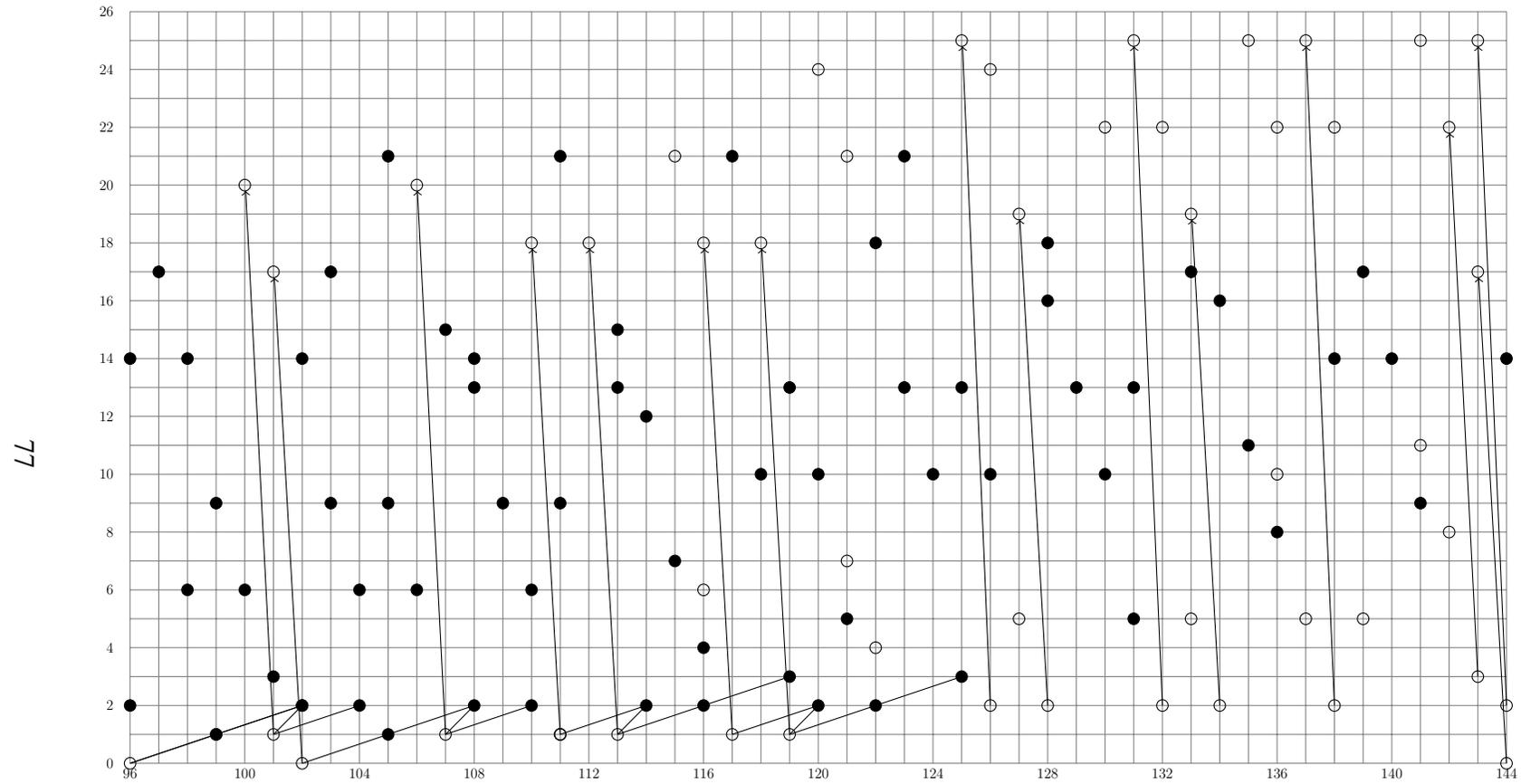
\end{landscape}

%%%%144-192

\begin{landscape}
\begin{figure}[h!]
\begin{tikzpicture}[scale=0.39]
\clip(-1.5,-1.5) rectangle (54,31);
\draw[color=gray] (0,0) grid [step=1] (53,30);

\foreach \n in {144,148,...,197}
{
\def\nn{\n-0}
\node[below,scale=0.5] at (\nn-144,0) {$\n$};
}
\foreach \s in {0,2,...,30}
{\def\ss{\s-0};
\node [left,scale=0.5] at (-0.4,\ss,0){$\s$};
}
\draw [] ( 0.00, 0.00) circle [radius=0.2];
\draw [fill] (3,1) circle [radius=0.2];
\draw [fill] (6,2) circle [radius=0.2];
\draw [-] (0,0)--(3,1);
\draw [-] (3,1)--(6,2);
%\draw[->,red] (11.8,0.2)--(11.2,2.8);
\foreach \t in {0}
\foreach \s in {0,6}
{\def\ss{\s-0};
\draw[]  (\s-\t,\t) circle [radius=0.2];
\draw [fill] (\s+3-\t,1+\t) circle [radius=0.2];
\draw [-] (\s-\t,\t)--(\s+3-\t,1+\t);
\draw[] (\s+5-\t,1+\t) circle [radius=0.2];
\draw[fill] (\s+8-\t,2+\t) circle [radius=0.2];
\draw[-] (\s+5-\t,1+\t)--(\s+8-\t,2+\t); 
\draw [fill] (\s-\t+6,2+\t) circle [radius=0.2];
\draw [-] (\s+3-\t,1+\t)--(\s+6-\t,2+\t);
%\draw[fill] (\s+11-\t,3+\t) circle [radius=0.1];
%\draw[-] (\s+8-\t,2+\t)--(\s+11-\t,3+\t); 
\draw [-] (\s+5-\t,1+\t)--(\s-\t+6,2+\t);
}

\foreach \t in {0}
\foreach \s in {12,18}
{\def\ss{\s-0};
%\draw [fill]  (\s-\t,\t) circle [radius=0.1];
\draw[] (\s+3-\t,1+\t) circle [radius=0.2];
\draw[] (\s+5-\t,1+\t) circle [radius=0.2];
\draw[fill] (\s+8-\t,2+\t) circle [radius=0.2];
\draw[-] (\s+5-\t,1+\t)--(\s+8-\t,2+\t); 
\draw [fill] (\s-\t+6,2+\t) circle [radius=0.2];
\draw [-] (\s+3-\t,1+\t)--(\s+6-\t,2+\t);
\draw[fill] (\s+11-\t,3+\t) circle [radius=0.2];
\draw[-] (\s+8-\t,2+\t)--(\s+11-\t,3+\t); 
\draw [-] (\s+5-\t,1+\t)--(\s-\t+6,2+\t);
}
\draw[] (0,2) circle [radius=0.2];
\draw[] (5,3) circle [radius=0.2];
\draw[] (30,2) circle [radius=0.2];
\draw[] (32,2) circle [radius=0.2];
\draw[] (36,2) circle [radius=0.2];
\draw[] (38,2) circle [radius=0.2];
\draw[] (42,2) circle [radius=0.2];
\draw[] (47,3) circle [radius=0.2];
\draw[] (48,2) circle [radius=0.2];
\draw[] (53,3) circle [radius=0.2];
\draw[fill] (48,0) circle[radius=0.2];
\draw[fill] (51,1) circle[radius=0.2];
\draw[-](48,0)--(51,1);
\draw[fill] (53,1) circle[radius=0.2];
\node[scale = 0.5] at (47.6,0.5) {$\Delta^8$};
%%%multiple of \ove{\kappa}
\draw[fill] (50-48,6) circle [radius=0.2]; %delta^4 e30
\draw[] (52-48,6) circle [radius=0.2]; %delta^4 e32
\draw[] (56-48,6) circle [radius=0.2]; %delta^4 e36
\draw[] (58-48,6) circle [radius=0.2]; %delta^4 e38
\draw[] (62-48,6) circle [radius=0.2]; %delta^4 e42
\draw[] (67-48,7) circle [radius=0.2]; %delta^4 e47
\draw[] (68-48,6) circle [radius=0.2]; %delta^4 e48
\draw[] (73-48,7) circle [radius=0.2]; %delta^4 e53
\draw[] (20,4) circle [radius=0.2]; %delta^6 e0
\draw[] (25,5) circle [radius=0.2];%delta^6 e5
\draw[] (26,4) circle [radius=0.2];%delta^6 e6
\draw[] (31,5) circle [radius=0.2];%delta^6 e11
\draw[] (35,5) circle [radius=0.2];%delta^6 e15
\draw[] (37,5) circle [radius=0.2];%delta^6 e17
\draw[] (41,5) circle [radius=0.2];%delta^6 e21
\draw[] (43,5) circle [radius=0.2];%delta^6 e23
\draw[] (50,6) circle [radius=0.2];%delta^6 e30
\draw[] (52,6) circle [radius=0.2];%delta^6 e32
%%%multiple of \kappa^2
\draw[] (51-48,9) circle [radius=0.2];% delta^4 e11
\draw[fill] (55-48,9) circle [radius=0.2];% delta^4 e15
\draw[] (57-48,9) circle [radius=0.2];% delta^4 e17
\draw[] (61-48,9) circle [radius=0.2];% delta^4 e21
\draw[] (63-48,9) circle [radius=0.2];% delta^4 e23
\draw[] (70-48,10) circle [radius=0.2]; % delta^4 e30
\draw[] (72-48,10) circle [radius=0.2]; % delta^4 e32
\draw[] (76-48,10) circle [radius=0.2]; % delta^4 e36
\draw[] (78-48,10) circle [radius=0.2]; %delta^4 e38
\draw[] (82-48,10) circle [radius=0.2]; % delta^4 e42
\draw[] (87-48,11) circle [radius=0.2]; %delta^4 e47
\draw[] (88-48,10) circle [radius=0.2]; % delta^4 e48
\draw[] (93-48,11) circle [radius=0.2]; % delta^4 e53
\draw[] (40,8) circle [radius=0.2];%delta^6 e0
\draw[] (45,9) circle [radius=0.2];%delta^6 e5
\draw[] (46,8) circle [radius=0.2];%delta^6 e6
\draw[] (51,9) circle [radius=0.2];%delta^6 e11
%%%multiple of \kappa^3
\draw[fill] (96-96,14) circle [radius=0.2]; %delta^2e36
\draw[fill] (98-96,14) circle [radius=0.2]; %delta^2e38
\draw[fill] (102-96,14) circle [radius=0.2]; %delta^2 e42
\draw[fill] (107-96,15) circle [radius=0.2]; %delta^2 e47
\draw[] (108-96,14) circle [radius=0.2]; %delta^2 e48
\draw[] (113-96,15) circle [radius=0.2]; %delta^2 e53
\draw[fill] (60-48,13) circle [radius=0.2];%delta^4 e0
\draw[fill] (65-48,13) circle [radius=0.2];%delta^4 e5
\draw[] (66-48,12) circle [radius=0.2];%delta^4 e6
\draw[] (71-48,13) circle [radius=0.2];% delta^4 e11
\draw[fill] (75-48,13) circle [radius=0.2];% delta^4 e15
\draw[] (77-48,13) circle [radius=0.2];% delta^4 e17
\draw[] (81-48,13) circle [radius=0.2];% delta^4 e21
\draw[] (83-48,13) circle [radius=0.2];% delta^4 e23
\draw[] (90-48,14) circle [radius=0.2]; % delta^4e30
\draw[] (92-48,14) circle [radius=0.2]; % delta^4 e32
\draw[] (96-48,14) circle [radius=0.2]; % delta^4 e36
\draw[] (98-48,14) circle [radius=0.2]; %delta^4 e38

%%%multiple of \kappa^4
\draw[fill] (80-48,16) circle [radius=0.2]; %delta^4 e0
\draw[] (85-48,17) circle [radius=0.2];%delta^4 e5
\draw[] (86-48,16) circle [radius=0.2];%delta^4 e6
\draw[] (91-48,17) circle [radius=0.2];%delta^4 e11
\draw[] (95-48,17) circle [radius=0.2];%delta^4 e15
\draw[] (97-48,17) circle [radius=0.2];%delta^4 e17
\draw[] (101-48,17) circle [radius=0.2];%delta^4 e21
\draw[fill] (97-96,17) circle [radius=0.2];%  delta^2 e17
\draw[] (101-96,17) circle [radius=0.2];% delta^2 e21
\draw[fill] (103-96,17) circle [radius=0.2];% delta^2 e23
\draw[] (110-96,18) circle [radius=0.2]; % delta^2 e30
\draw[] (112-96,18) circle [radius=0.2]; % delta^2 e32
\draw[] (116-96,18) circle [radius=0.2]; % delta^2 e36
\draw[] (118-96,18) circle [radius=0.2]; % delta^2 e38
\draw[fill] (122-96,18) circle [radius=0.2]; % delta^2 e42
\draw[] (127-96,19) circle [radius=0.2]; % delta^2 e47
\draw[] (128-96,18) circle [radius=0.2]; %delta^2 e48
\draw[] (133-96,19) circle [radius=0.2]; % delta^2 e53

%%%multiple of \ove{\kappa}^5
\draw[] (147-144,23) circle [radius=0.2]; %e47
\draw[] (148-144,22) circle [radius=0.2]; %e48
\draw[] (153-144,23) circle [radius=0.2]; %e53
\draw[] (100-96,20) circle [radius=0.2]; %delta^2 e0
\draw[fill] (105-96,21) circle [radius=0.2];%delta^2 e5
\draw[] (106-96,20) circle [radius=0.2];%delta^2 e6
\draw[fill] (111-96,21) circle [radius=0.2];%delta^2 e11
\draw[] (115-96,21) circle [radius=0.2];%delta^2 e15
\draw[fill] (117-96,21) circle [radius=0.2];%delta^2 e17
\draw[] (121-96,21) circle [radius=0.2];%delta^2 e21
\draw[fill] (123-96,21) circle [radius=0.2];%delta^2 e23
\draw[] (130-96,22) circle [radius=0.2];%delta^2 e30
\draw[] (132-96,22) circle [radius=0.2];%delta^2 e32
\draw[] (136-96,22) circle [radius=0.2]; % delta^2 e36
\draw[] (138-96,22) circle [radius=0.2]; % delta^2 e38
\draw[] (142-96,22) circle [radius=0.2]; % delta^2 e42
\draw[] (147-96,23) circle [radius=0.2]; % delta^2 e47
\draw[] (148-96,22) circle [radius=0.2]; % delta^2 e48
\draw[] (148-96,20) circle [radius=0.2]; % delta^4 e0
%%%multiple of \ove{\kappa}^6
\draw[] (150-144,26) circle [radius=0.2]; % e30
\draw[] (152-144,26) circle [radius=0.2]; % e32
\draw[] (156-144,26) circle [radius=0.2]; % e36
\draw[] (158-144,26) circle [radius=0.2]; % e38
\draw[] (162-144,26) circle [radius=0.2]; % e42
\draw[] (167-144,27) circle [radius=0.2]; % e47
\draw[] (168-144,26) circle [radius=0.2]; % e48
\draw[] (173-144,27) circle [radius=0.2]; % e53
\draw[] (168-144,24) circle [radius=0.2]; %delta^2 e0
\draw[] (125+48-144,25) circle [radius=0.2];%delta^2e5
\draw[] (126+48-144,24) circle [radius=0.2];%delta^2 e6
\draw[] (131+48-144,25) circle [radius=0.2];%delta^2e11
\draw[] (135+48-144,25) circle [radius=0.2];%delta^2 e15
\draw[] (137+48-144,25) circle [radius=0.2];%delta^2e17
\draw[] (141+48-144,25) circle [radius=0.2];%delta^2 e21
\draw[] (143+48-144,25) circle [radius=0.2];%delta^2e23
%%%%%%%%%%%%%%%%%%%%%%%%%%%%%

%%%differentials
\draw[->] (23,1)--(22+0.05,18-0.2);%d17(delta^6 e23)
\draw[->] (38,2)--(37+0.04,19-0.2);%d17(delta^6 e38)
\draw[->] (48,2)--(47+0.04,25-0.2);%d23(delta^6 e48)
\draw[->] (6,0)--(5+0.03,17-0.2);%d17(delta^6 e6)
\draw[->] (17,1)--(16+0.03,18-0.2);%d17(delta^6 e17)
\draw[->] (21,1)--(20+0.03,18-0.2);%d17(delta^6 e21)
\draw[->] (21,1)--(20+0.03,18-0.2);%d17(delta^6 e21)
\draw[->] (32,2)--(31+0.03,19-0.2);%d17(delta^6 e32)
\draw[->] (53,3)--(52+0.03,22-0.2);%d17(delta^6 e53)
\draw[->] (15,1)--(14+0.03,18-0.2);%d17(delta^6 e15)
\draw[->] (11,1)--(10+0.02,20-0.2);%d19(delta^6 e11)
\draw[->] (47,3)--(46+0.02,22-0.2);%d19(delta^6 e47)
\draw[->] (5,1)--(4+0.02,20-0.2);%d19(delta^6 e5)
\draw[->] (5,3)--(4+0.02,22-0.2);%d19(delta^4 e53)
\draw[->] (36,2)--(35+0.01,25-0.2);%d23(delta^6 e36)
\draw[->] (42,2)--(41+0.01,25-0.2);%d23(delta^6 e42)
\draw[->] (30,2)--(29+0.01,25-0.2);%d23(delta^4 e30)
%%%%%%%%%%%%%%%%%
\end{tikzpicture}
\caption{HFPSS for $A_{1}[00]\ \mbox{and} \ A_{1}[11]$ from $\mathrm{E}_{7}$-term with $144\leq t-s\leq 197$}
\phantomsection \label{144-197 nondual}
\end{figure}
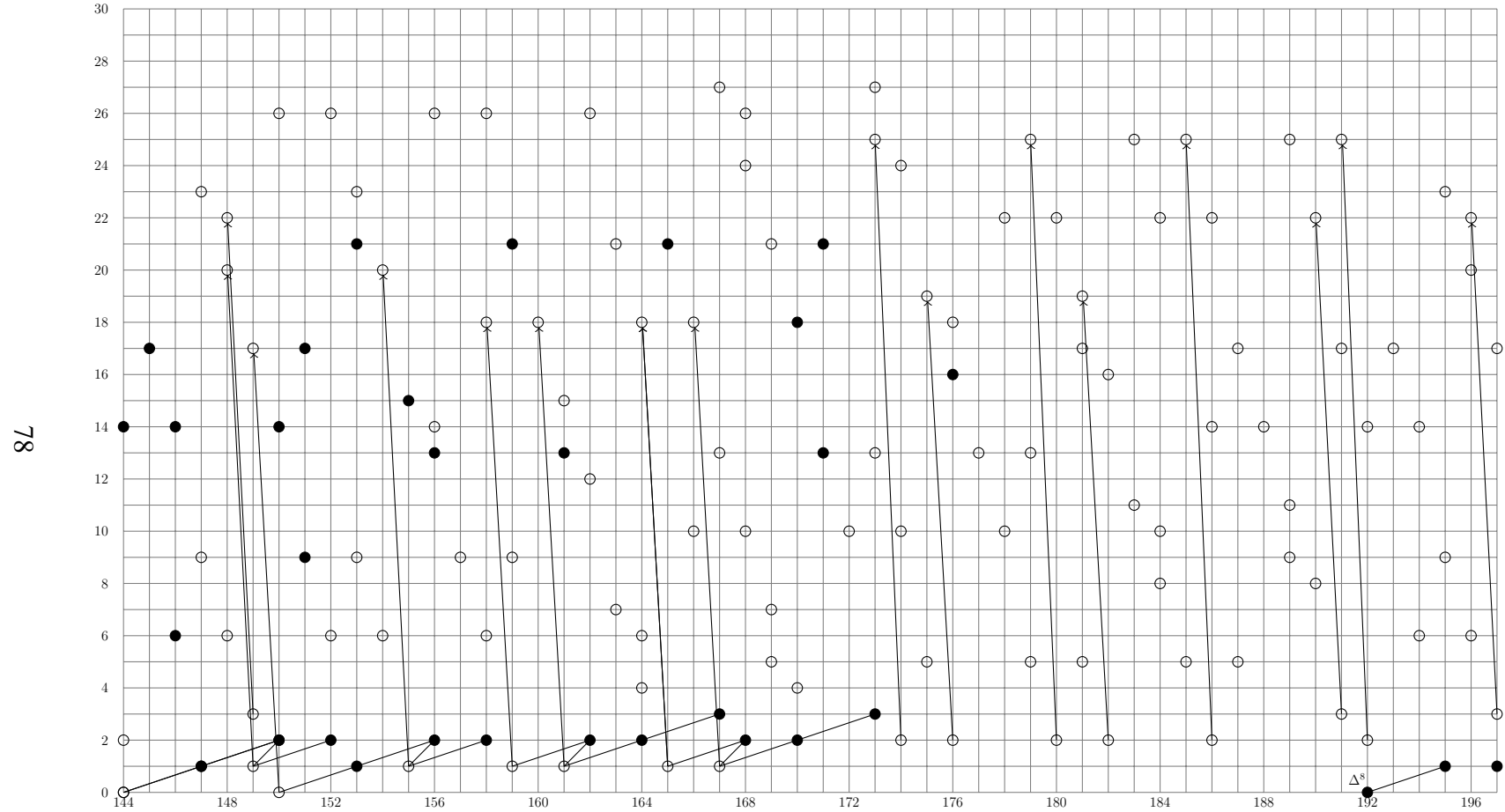
\end{landscape}
%\end{document}%\addcontentsline{toc}{chapter}{Bibliography}
\bibliographystyle{alpha}
\bibliography{mybibs}{}

\end{document}